\documentclass[12pt]{article}
\usepackage{amsmath,amsthm,amssymb}
\usepackage{indentfirst}
\usepackage{mathrsfs}
\usepackage{bm}
\usepackage{graphicx}
\usepackage{tikz-cd}
\usepackage{tabularx} 
\usepackage[margin=1in]{geometry} 
\usepackage{blindtext}
\usepackage{mathtools}
\usepackage{graphicx} 
\usepackage{hyperref}
\hypersetup{unicode=true,bookmarksnumbered=false,bookmarksopen=false,breaklinks=false,pdfborder={0 0 1},colorlinks=true,linkcolor=blue}
\usepackage{csquotes}
\usepackage{booktabs} 
\usepackage{array} 
\usepackage{paralist} 
\usepackage{subfig} 
\usepackage{amsfonts, bbm, dsfont}
\usepackage{stmaryrd}
\usepackage{float}
\usepackage{elocalloc}
\usepackage{enumitem}
\setenumerate[1]{itemsep=0pt,partopsep=0pt,parsep=\parskip,topsep=5pt}
\setitemize[1]{itemsep=0pt,partopsep=0pt,parsep=\parskip,topsep=5pt}
\usepackage{url}
\usepackage{tkz-graph}
\usepackage{eucal}
\usepackage{multirow}
\usepackage{cleveref}
\Crefname{paragraph}{\S}{Paragraphs}

\numberwithin{equation}{section}

\crefformat{section}{\S#2#1#3} 
\crefformat{subsection}{\S#2#1#3}
\crefformat{subsubsection}{\S#2#1#3}

\usepackage[backend=biber,style=alphabetic,maxnames=99,maxalphanames=10]{biblatex}
\NewBibliographyString{diplomathesis}
\DefineBibliographyStrings{english}{diplomathesis = {diploma thesis},}
\usepackage{dynkin-diagrams}
\pgfkeys{/Dynkin diagram,
edge length=0.8cm,
fold radius=0.8cm,
root radius=0.1cm,
indefinite edge/.style={
draw=black,
fill=white,
thin,
densely dashed}}
\newtheoremstyle{special}
    {\topsep}
    {\topsep}
    {\itshape}
    {}
    {\bfseries}
    {}
    {0.5em}
    {{\thmname{#1$^{\bm*}\!$}\thmnumber{ #2.}\thmnote{\ \textmd{(#3)}}}}

\newtheorem{thm}{Theorem}[subsection]
\newtheorem{cor}[thm]{Corollary}
\newtheorem{lemma}[thm]{Lemma}
\newtheorem{prop}[thm]{Proposition}
\newtheorem{conj}[thm]{Conjecture}
\newtheorem{prob}{Question}
\newtheorem{axiom}{Axiom}
\newtheorem{introthm}{Theorem}

\newtheorem{introconj}[introthm]{Conjecture}
\theoremstyle{remark}
\newtheorem{ex}[thm]{Example}
\newtheorem{exs}[thm]{Examples}
\newtheorem{rmk}[thm]{Remark}
\theoremstyle{definition}
\newtheorem{defi}[thm]{Definition}
\newtheorem{notation}[thm]{Notation}
\theoremstyle{special}
\newtheorem{spprop}[thm]{Proposition}
\newtheorem{spthm}[thm]{Theorem}

\newtheorem{splemma}[thm]{Lemma}
\newtheorem{spintrothm}[introthm]{Theorem}

\newcommand{\sym}{\operatorname{Sym}}
\newcommand{\spin}{\operatorname{Spin}}

\newcommand{\alg}{\mathrm{alg}}

\newcommand{\Hom}{\mathrm{Hom}}

\newcommand{\G}{\mathscr{G}}
\newcommand{\Z}{\mathbb{Z}}
\newcommand{\C}{\mathbb{C}}
\newcommand{\gal}{\mathrm{Gal}}
\newcommand{\Q}{\mathbb{Q}}

\newcommand{\oct}{\mathbb{O}}
\newcommand{\R}{\mathbb{R}}
\newcommand{\GL}{\mathrm{GL}}
\newcommand{\SL}{\mathrm{SL}}
\newcommand{\A}{\mathbb{A}}
\newcommand{\tr}{\mathrm{Tr}}
\newcommand{\frob}{\mathrm{Frob}}
\newcommand{\aut}{\mathrm{Aut}}
\newcommand{\orth}{\operatorname{O}}
\newcommand{\sorth}{\operatorname{SO}}
\newcommand{\symp}{\operatorname{Sp}}
\newcommand{\su}{\operatorname{SU}}
\newcommand{\psu}{\operatorname{PSU}}
\newcommand{\st}{\mathrm{St}}
\newcommand{\modulo}{\operatorname{mod}}
\newcommand{\pgl}{\operatorname{PGL}}
\newcommand{\cusp}{\mathrm{cusp}}
\newcommand{\disc}{\mathrm{disc}}
\newcommand{\weights}{\mathrm{Weights}}
\newcommand{\lietype}[2]{\mathrm{#1}_{#2}}
\newcommand{\reg}{\mathrm{reg}}
\newcommand{\grpF}{\mathbf{F}_{4}}
\newcommand{\triv}{\mathbf{1}}
\newcommand{\jord}{\mathrm{J}}
\newcommand{\midline}{\,\middle\vert\,}
\newcommand{\set}[2]{\left\{#1\midline #2\right\}}
\newcommand{\vrep}[1]{\mathrm{V}_{#1}}
\newcommand\scalemath[2]{\scalebox{#1}{\mbox{\ensuremath{\displaystyle #2}}}}

\begin{document}
\title{Level one automorphic representations of an anisotropic exceptional group over $\Q$ of type \texorpdfstring{$\mathrm{F}_{4}$}{}}
\author{Yi Shan}
\date{\today}
\maketitle
\begin{abstract}
    Up to isomorphism, there is a unique connected semisimple algebraic group over $\Q$ of Lie type $\lietype{F}{4}$, 
    with compact real points and split over $\Q_{p}$ for all primes $p$. 
    Let $\grpF$ be such a group.
    In this paper, we study the level one automorphic representations of $\grpF$ in the spirit of the work of Chenevier, Renard and Ta\"ibi \cites{ChenevierRenard}{TaibiDimension}{Chenevier2020DiscreteSM}.
    
    First, we give an explicit formula for the number of these representations having any given archimedean component. 
    For this, we study the automorphism group of the two definite exceptional Jordan algebras of rank $27$ over $\Z$ studied by Gross in \cite{G96}, 
    as well as the dimension of the invariants of these groups in all irreducible representations of $\grpF(\R)$.

    Then, assuming standard conjectures by Arthur and Langlands for $\grpF$ \cites{ArthurUnipRep}{ChenevierLannes}, 
    we refine this counting by studying the contribution of the representations whose global Arthur parameter has any possible image (or ``Sato-Tate group"). 
    This includes a detailed description of all those images, as well as precise statements for the Arthur's multiplicity formula in each case. 
    As a consequence, we obtain a conjectural but explicit formula for the number of algebraic, cuspidal, level one automorphic representation of $\GL_{26}$ over $\Q$ with Sato-Tate group $\grpF(\R)$ of any given weight (assumed ``$\lietype{F}{4}$-regular").
    The first example of such representations occurs in motivic weight $36$.
\end{abstract}
\setcounter{tocdepth}{1}
\tableofcontents
\section{Introduction}\label{section introduction}
\subsection{Galois representations with given image}\label{section Galois representations with given image}
 The absolute Galois group $\mathrm{Gal}(\overline{\Q}/\Q)$ encodes a lot of arithmetic information about number fields,
and a natural way to study $\mathrm{Gal}(\overline{\Q}/\Q)$ is to consider its representations,
especially those arising from algebraic geometry.
Motivated by the inverse Galois problem,
the following question has been studied by a lot of mathematicians:
\begin{prob}\label{prob Galois representation with given image}
    Let $\ell$ be a prime number and $H$ a connected reductive algebraic group over $\overline{\Q_{\ell}}$.
    Is there an $\ell$-adic Galois representation $\rho:\gal(\overline{\Q}/\Q)\rightarrow H(\overline{\Q}_{\ell})$ 
    such that it is semisimple and geometric (in the sense of Fontaine-Mazur \cite[Conjecture 1.1]{TaylorICM}),
    and its image is Zariski dense in $H(\overline{\Q_{\ell}})$?
\end{prob}
In the case $H=\GL_{2}\simeq \mathrm{GSp}_{2}$ or $\mathrm{GSp}_{4}$,
or more generally,
a (similitude) classical group,
there are many well-known constructions and examples.
For instance, one can use the Poincar\'e pairing on $\ell$-adic cohomologies of algebraic varieties to construct Galois representations with images in classical groups.
The case of exceptional groups,
i.e. groups with Lie types $\lietype{G}{2},\,\lietype{F}{4},\,\lietype{E}{6},\,\lietype{E}{7}$ and $\lietype{E}{8}$,
is harder,
but we still have some examples 
in \cites{Dettweiler_Reiter_2010}{GrossSavin}{YunMotive}{PatrikisDeformation}{BOXER_CALEGARI_EMERTON_LEVIN_MADAPUSI_PERA_PATRIKIS_2019}.
Notice that when $H$ has Lie type $\lietype{G}{2}$ or $\lietype{E}{8}$,
this question is related to Serre's question on motives \cite[Question 8.8, \S 1]{SerreMotive}.

Composing $\gal(\overline{\Q}/\Q)\rightarrow H(\overline{\Q_{\ell}})$ with an irreducible faithful algebraic representation $H\hookrightarrow \GL_{n}$,
we obtain an $n$-dimensional geometric $\ell$-adic representation.
One can associate two invariants with a geometric $\ell$-adic Galois representation $\rho:\gal(\overline{\Q}/\Q)\rightarrow\GL_{n}(\overline{\Q}_{\ell})$: 
the \emph{(Artin) conductor} $\mathrm{N}(\rho)\in\mathbb{N}$,
and the \emph{Hodge-Tate weights} $\mathrm{HT}(\rho)$, a multiset of $n$ integers (see, for example, \cite{TaylorICM}). 
In the aforementioned works,
the conductors of the geometric $\ell$-adic representations
that they construct are usually not controlled.
One may refine \Cref{prob Galois representation with given image} naturally by fixing these two invariants:
\begin{prob}\label{prob refinement of Galois representations counting}
    Let $\ell$ be a prime number, $n\geq 1$ and $H$ a connected reductive subgroup of $\GL_{n}$ over $\overline{\Q_{\ell}}$.
    What is the number (up to equivalence) of geometric $\ell$-adic Galois representations $\rho:\gal(\overline{\Q}/\Q)\rightarrow \GL_{n}(\overline{\Q_{\ell}})$
    of given conductor and Hodge-Tate weights such that the Zariski closure of $\mathrm{Im}(\rho)$ is $H(\overline{\Q}_{\ell})$?
\end{prob}
For $(H,n)=(\GL_{2},2)$ or $(\mathrm{SO}_{2g+1},2g+1)$,
this question is for instance related to the dimension of spaces of classical or Siegel modular forms.
We have less knowledge of the cases of other groups $H$.
When the conductor $N=1$,
\Cref{prob refinement of Galois representations counting} is solved \emph{conjecturally} by Chenevier and Renard in \cite{ChenevierRenard} 
for the following groups 
($n$ is chosen to be the dimension of the standard representation when $H$ is a (similitude) classical group, and to be $7$ when $H$ has type $\lietype{G}{2}$):
\[\GL_{2}\simeq \mathrm{GSp}_{2},\,\mathrm{GSp}_{4},\,\sorth_{4},\,\sorth_{5},\,\mathrm{GSp}_{6},\,\mathrm{GSp}_{8},\,\sorth_{8},\,\lietype{G}{2},\]
via the conjectural connection between $n$-dimensional geometric $\ell$-adic representations and cuspidal automorphic representations of $\GL_{n}$.
See also \cites{TaibiDimension}{Chenevier2020DiscreteSM} for higher dimensions.
In \cite{lachausseeThesis},
Lachauss\'ee extends the results for $\mathrm{GSp}_{2g},\,1\leq g\leq 4$ to the case of Artin conductor $N=2$. 
Now we concentrate on the case of conductor one (see \Cref{intrormk reason of conductor 1} for more explanations about this conductor one assumption). 

In this paper,
following \cite{ChenevierRenard}, 
we give a \emph{conjectural} solution to \Cref{prob refinement of Galois representations counting} 
when $N=1$,
$H$ has Lie type $\lietype{F}{4}$,
and $n=26$.
For a $26$-dimensional geometric $\ell$-adic Galois representation $\rho$
such that $\overline{\mathrm{Im}(\rho)}$ has type $\lietype{F}{4}$,
its multiset of Hodge-Tate weights only depends on $4$ variables $a,b,c,d\in \mathbb{N}$,
and has the form

{\scriptsize
\begin{align*}
    \mathrm{HT}(a,b,c,d):=\left\{
        \begin{array}{c}
            0,0,\pm a,\pm b,\pm(a+b),\pm(b+c),\pm(a+b+c),\pm(b+c+d),\pm(a+b+c+d),\pm(a+2b+c),\\    
            \pm(a+2b+c+d),\pm(a+2b+2c+d),\pm(a+3b+2c+d),\pm(2a+3b+2c+d).
        \end{array} 
        \right\}
\end{align*}}
As a conjectural corollary of our results in this paper,
we propose the following conjecture
on $\lietype{F}{4}$-type geometric $\ell$-adic representations:
\begin{introconj}\label{introconj F4 Galois representation}
    The number of equivalence classes of $26$-dimensional conductor one geometric $\ell$-adic Galois representations $\rho$ such that 
    \begin{itemize}
        \item the Zariski closure of $\mathrm{Im}(\rho)$ is a connected reductive group of type $\lietype{F}{4}$,
        \item and $\mathrm{HT}(\rho)=\mathrm{HT}(a,b,c,d)$, $a,b,c,d\geq 1$,
    \end{itemize}
    is $\lietype{F}{4}(a-1,b-1,c-1,d-1)$,
    where $\lietype{F}{4}(\lambda)$ is the computable function on $\mathbb{N}^{4}$ given by \Cref{prop conjectural number of F4 image representations}.     
\end{introconj}
\begin{rmk}\label{rmk formulas under conditions}
   The formula for $\lietype{F}{4}(\lambda)$ has so many terms that 
   we will not write down the full formula in this paper.
   However, under some hypothesis
   this formula becomes much simpler.
   For instance, 
   when $a>b+c+d+3$, $b,c,d>0$ and $c,d$ are both odd,
   a short formula for $\lietype{F}{4}(a,b,c,d)$ is given in \Cref{rmk simpler version of formula}.
\end{rmk}
\begin{ex}\label{introex tempered representation}
    Among quadruples $(a,b,c,d)$ with nonzero $\lietype{F}{4}(a,b,c,d)$,
    there exists a unique one $(1,2,0,2)$ that has the smallest $2a+3b+2c+d$
    (the largest Hodge-Tate weight).
    Moreover, $\lietype{F}{4}(1,2,0,2)=1$,
    so according to \Cref{introconj F4 Galois representation}
    there should be a unique $26$-dimensional conductor one geometric $\ell$-adic representation $\rho$
    such that 
    \begin{itemize}
        \item $\overline{\mathrm{Im}(\rho)}$ has type $\lietype{F}{4}$,
        \item and its multiset of Hodge-Tate weights $\mathrm{HT}(\rho)$ is:
    \[\mathrm{HT}(2,3,1,3)=\{0,0,\pm 2,\pm 3,\pm 4,\pm 5,\pm 6,\pm 7,\pm 9,\pm 9,\pm 12,\pm 13,\pm 16,\pm 18\}.\]
    \end{itemize}
    For people preferring non-negative Hodge-Tate weights,
    one can twist $\rho$ by $\omega_{\ell}^{-18}$, where $\omega_{\ell}$ denotes the $\ell$-adic cyclotomic character of $\gal(\overline{\Q}/\Q)$,
    and obtain a representation with motivic weight $36$.
    Hence we expect a $26$-dimensional geometric $\ell$-adic representation whose Zariski image is the product of an $\lietype{F}{4}$-type group with $\overline{\Q_{\ell}}^{\times}$
    to appear in the $36$th degree $\ell$-adic cohomology of some algebraic variety.
    A very interesting open problem is to find such a variety!
\end{ex}

\subsection{An automorphic variant of \Cref{prob refinement of Galois representations counting}}\label{section automorphic variant of question}

Now we send \Cref{prob refinement of Galois representations counting} to the automorphic side.
Let $G$ be a connected reductive group over $\Q$ with a reductive $\Z$-model (see \Cref{section reductive integral models of Q groups}).
As we will talk about Galois representations,
it will be convenient to assume that $\widehat{G}$ is defined over $\overline{\Q}$,
and we fix two embeddings:
$\iota_{\infty}:\overline{\Q}\rightarrow \C$ and $\iota_{\ell}:\overline{\Q}\hookrightarrow \overline{\Q_{\ell}}$.
We also fix a maximal compact subgroup $G_{c}$ of $\widehat{G}(\C)$.

Let $\pi$ be an \emph{$L$-algebraic}
\footnote{For the definition of $L$-algebraicity, see \cite[Definition 2.3.1]{BuzzardGeeConj}.
For a representation which is algebraic in the sense of \Cref{def algebraic and regular representations} but not $L$-algebraic,
one should replace $\widehat{G}$ by some ``similitude" group.
}
level one automorphic representation of $G$.
By a conjecture of Buzzard and Gee \cite[Conjecture 3.2.1]{BuzzardGeeConj},
one should be able to associate with $\pi$ a compatible conductor one geometric $\ell$-adic representation
$\rho_{\pi,\iota}:\gal(\overline{\Q}/\Q)\rightarrow \widehat{G}(\overline{\Q_{\ell}})$,
which depends on the choice of embeddings $\iota=(\iota_{\infty},\iota_{\ell})$.
By the standard conjectures of Fontaine-Mazur and Langlands,
every conductor one geometric $\ell$-adic representation into $\widehat{G}(\overline{\Q}_{\ell})$
should arise in this way.
If any two element-conjugate homomorphisms from a connected compact Lie group into $G_{c}$ are conjugate 
(see \Cref{section local global conjugacy} for a detailed explanation),
the following question gives an automorphic variant of \Cref{prob refinement of Galois representations counting} for $H=\widehat{G}\times_{\iota_{\ell}}\overline{\Q_{\ell}}$:
\begin{prob}\label{prob automorphic dimension for reductive groups}
    Let $G$ be a connected reductive group over $\Q$
    admitting a reductive $\Z$-model.
    \begin{enumerate}[label=(\arabic*)]
        \item (\emph{Counting})
        Count the number (up to equivalence) of level one algebraic
        \footnote{One can remove this algebraicity condition by restricting to semisimple $\Q$-groups.} 
        discrete automorphic representations for $G$ with an arbitrary given archimedean component.
        \item (\emph{Refinement}) Refine this counting by \emph{``Sato-Tate groups"} of automorphic representations.
    \end{enumerate}
\end{prob}
\begin{rmk}[``Sato-Tate groups'']\label{rmk Sato Tate groups}
    In the above question,
    the ``Sato-Tate group'' $\mathrm{H}(\pi)$ of a level one automorphic representation $\pi$ for $G$
    is a certain conjugacy class of subgroups of $G_{c}$
    that we will explain carefully in \Cref{section Sato Tate groups of Arthur parameters},
    and we can briefly introduce it as follows.
    Based on Arthur's parametrization of automorphic representations,
    one can \emph{conjecturally} associate with $\pi$ a group homomorphism 
    \[\psi_{\pi}:\mathcal{L}_{\Z}\times\su(2)\rightarrow G_{c},\]
    where $\mathcal{L}_{\Z}$ is the \emph{hypothetical Langlands group}, 
    which is connected and compact (see \Cref{section global parametrization and Langlands group}).
    We define $\mathrm{H}(\pi)$ to be the conjugacy class of the image of $\psi_{\pi}$ in $G_{c}$.  
    When the restriction of $\psi_{\pi}$ to $1\times\su(2)\subset \mathcal{L}_{\Z}\times\su(2)$ is trivial,
    this notion $\mathrm{H}(\pi)$ coincides with the usual notion of Sato-Tate groups.
    In general, we decided to include the $\su(2)$ factor in the definition
    as it proves convenient for stating some of our results.
\end{rmk}
The point of the refinement part in \Cref{prob automorphic dimension for reductive groups} is that
in general many level one discrete automorphic representations $\pi$ for $G$,
for example the \emph{endoscopic} ones,
will have a Sato-Tate group strictly smaller than $G_{c}$.
For these $\pi$,
$\overline{\mathrm{Im}(\rho_{\pi,\iota})}$ should be a proper subgroup of $\widehat{G}(\overline{\Q_{\ell}})$.
Hence we have to find a way to exclude these representations to obtain the desired number in \Cref{prob refinement of Galois representations counting}. 

In \cite{ChenevierRenard},
Chenevier and Renard solve the part (1) of \Cref{prob automorphic dimension for reductive groups} for a number of classical groups of small ranks,
namely, $G$ is one of the following groups:
\[\SL_{2}=\symp_{2},\,\symp_{4},\,\sorth_{2,2},\,\sorth_{3,2},\,\sorth_{7},\,\sorth_{8}\text{ and }\sorth_{9},\]
and also for a connected semisimple $\Q$-group of type $\lietype{G}{2}$ with compact real points.
For the part (2) of \Cref{prob automorphic dimension for reductive groups},
their method relies in an important way on Arthur's classification of automorphic representations \cites{ArthurUnipRep}{arthur2013endoscopic}.
Their results for $\sorth_{7},\sorth_{8},\sorth_{9}$ and $\lietype{G}{2}$ are conditional to Arthur's conjectures for these groups,
since $\sorth_{7},\sorth_{8}$ and $\sorth_{9}$ are not quasi-split,
and $\lietype{G}{2}$ is not covered by Arthur's results. 
In \cites{TaibiDimension},
Ta\"ibi uses Arthur's $\mathrm{L}^{2}$-Lefschetz formula to make these results unconditional (except for $\lietype{G}{2}$) 
and he also extends them to the following split classical groups:
\[\symp_{2g}\text{ with }g\leq 7,\,\sorth_{n+1,n}\text{ with }n\leq 8\text{ and }\sorth_{2m,2m}\text{ with }m\leq 4.\]
In particular,
Ta\"ibi's solution to \Cref{prob automorphic dimension for reductive groups} for $\symp_{8}$ will be important in our work.

In this paper,
we apply the method of \cite{ChenevierRenard} to $\grpF$,
the unique (up to isomorphism) connected semisimple algebraic group over $\Q$ of type $\lietype{F}{4}$,
with compact real points and split over $\Q_{p}$ for every prime $p$.
The construction of $\grpF$ is explicitly given in \Cref{section compact Lie group F4}.
For this group,
automorphic representations are automatically $L$-algebraic.
Moreover, it turns our that there is no local-global conjugacy problems for connected subgroups of $(\grpF)_{c}=\grpF(\R)$ (see \Cref{prop compact Lie group F4 is acceptable}).
As a consequence,
\Cref{introconj F4 Galois representation} follows from standard conjectures and our result on automorphic representations (\Cref{introthm formula for cuspidal representations with Sato-Tate group F4}).

\begin{rmk}\label{intrormk related works}
    The automorphic representations for $\grpF$ (and their local components) have been studied in \cites{Savin1994}{SavinMagaardExceptionalTheta}{GWTExceptionalSiegelWeil}{pollack2023exceptional}{karasiewicz2023dual} via exceptional theta correspondences,
    and we will explain some links between these correspondences with our work in \Cref{section explain the table of lower motivic weights}.   
    Let us mention that automorphic representations for $\grpF$ have also been studied in the past by Seth Padowitz in \cite[\S 9]{PadowitzThesis}.
    Padowitz rather considers the automorphic representations which are Steinberg at a fixed \emph{non-empty} set of primes and unramified elsewhere,
    and he tries to enumerate them using the stable trace formula,
    in the spirit of works of Gross-Pollack \cite{GrossPollack}.
    The results are only partial,
    as several stable local orbital integrals there are not determined
    \footnote{Another minor problem is that the author asserts on \cite[P.42]{PadowitzThesis} that the $26$-dimensional irreducible representation of $\grpF$ is ``excellent'' in his sense,
    which is not correct.
    See \Cref{rmk char poly not enough for conjugacy} for a counterexample.},
    and we hope to go back to this question in the future.
\end{rmk}

\subsection{Counting level one automorphic representations}\label{section counting automorphic representations}
 In \cite{G96},
Gross proves the following result for $\grpF$, 
which is important in our solution to the part (1) of \Cref{prob automorphic dimension for reductive groups} for $\grpF$:
\begin{introthm}\label{introthm two reductive integral models of F4}(\Cref{prop exactly two reductive models})
    Up to $\Z$-isomorphism,
    there are two smooth affine group schemes over $\Z$
    with generic fiber isomorphic to $\grpF$,
    whose special fiber over $\Z/p\Z$ is reductive for all primes $p$.    
\end{introthm}
The $\Z$-group schemes in \Cref{introthm two reductive integral models of F4}
are reductive $\Z$-models of $\grpF$.
Their constructions are related to integral structures of the $27$-dimensional definite exceptional Jordan algebra over $\Q$.
Gross proves this result via the mass formula for $\grpF$ and some results in \cite{ATLAS},
and we will give a new proof in \Cref{section reductive integral models of F4} without using \cite{ATLAS}.

In our proof of \Cref{introthm two reductive integral models of F4}, 
we study the $\Z$-points of two reductive $\Z$-models in \Cref{introthm two reductive integral models of F4},
which are finite groups inside the compact Lie group $\grpF(\R)$.
With the help of \cite{PARI2} and \cite{GAP4},
for each of these finite groups,
we give an explicit set of generators in \Cref{section generators of integral points}
and enumerate its conjugacy classes in \Cref{section enumeration conj classes}.

Since the method of counting in \cite{ChenevierRenard} can be applied to any algebraic $\Q$-group that has compact real points and admits a reductive $\Z$-model,
we recall and apply this method to $\grpF$ in \Cref{section ideas and obstructions of the computation}, \Cref{section Kac coordinates} and \Cref{section comparison of conj classes}.
This formula leads to the answer for the part (1) of \Cref{prob automorphic dimension for reductive groups} in the case of $\grpF$, 
which is also the main computational result in this paper:
\begin{introthm}\label{introthm dimension formula for automorphic representations}
    (\Cref{thm main computational result} and \Cref{cor dimension formula for the multiplicity space})
    (1) For an irreducible representation $\vrep{\lambda}$ of $\grpF(\R)$ with highest weight $\lambda$,
    we have an explicit and computable formula for 
    the number $d(\lambda)$ of equivalence classes of level one automorphic representations $\pi$ with $\pi_{\infty}\simeq \vrep{\lambda}$.

    (2) For dominant weights $\lambda=\sum_{i=1}^{4}\lambda_{i}\varpi_{i}$\footnote{Here we follow the notations in \cite[\S VI.4.9]{Lie}.}
    satisfying $2\lambda_{1}+3\lambda_{2}+2\lambda_{3}+\lambda_{4}\leq 13$,
    we list the numbers $d(\lambda)$ in \Cref{nonzeroAMF}, \Cref{TableAppendix}.
\end{introthm}

\subsection{Candidates for Sato-Tate groups}\label{section classification Sato-Tate}
 The part (2) of \Cref{prob automorphic dimension for reductive groups}
involves a classification of all possible Sato-Tate groups for level one automorphic representations of $\grpF$.
For this $\Q$-group,
its Langlands dual group $\widehat{\grpF}$ is isomorphic to $\grpF\times_{\Q}\C$,
and as mentioned in \Cref{rmk Sato Tate groups},
Sato-Tate groups in this case are conjugacy classes of subgroups of the compact Lie group $\grpF(\R)$.
The following result gives us $13$ candidates for Sato-Tate groups strictly smaller than $\grpF(\R)$: 
\begin{introthm}\label{introthm classification of subgroups of F4}
    (\Cref{thm classification result of subgroups satisfying three conditions})
    There are $13$ conjugacy classes of proper connected subgroups $H$ of $\grpF(\R)$ such that:
    \begin{itemize}
        \item the centralizer of $H$ in $\grpF(\R)$ is isomorphic to the product of finitely many copies of $\Z/2\Z$;
        \item the zero weight appears twice in the restriction of the $26$-dimensional irreducible representation of $\grpF(\R)$ to $H$.
    \end{itemize}
\end{introthm}
We prove this classification result step by step in \Cref{section maximal subgroups of F4}, \Cref{section A1 subgroups classification}, \Cref{section classification simple subgroups} and \Cref{section connected subgroups satisfying our conditions},
following Dynkin's strategy in \cite{Dynkin1952}.
It is worth mentioning two important ingredients in the proof:
\begin{itemize}
    \item A local-global conjugacy result (\Cref{prop compact Lie group F4 is acceptable}) for $\grpF(\R)$,
    which we have already mentioned in the end of \Cref{section automorphic variant of question}.  
    This relies on a result about Lie algebras (\Cref{thm Lie algebra F4 is acceptable}) proved by Losev in \cite{Losev2005OnIO}.
    \item A useful criterion (\Cref{prop just one rep is enough for conjugacy}) given in \Cref{section one representation is enough} for the conjugacy of two homomorphisms from a connected compact Lie group into $\grpF(\R)$.
\end{itemize}
\begin{ex}\label{introex subgroups of F4}
    Among the conjugacy classes of subgroups classified in \Cref{introthm classification of subgroups of F4},
    we have 
    \[\spin(9),\,\spin(8),\,\lietype{G}{2}\times\sorth(3),\,\left(\symp(3)\times\su(2)\right)/\mu_{2}^{\Delta},\,\left(\symp(2)\times\su(2)\times\su(2)\right)/\mu_{2}^{\Delta},\]
    where the notations will be explained in \Cref{notation compact Lie groups} and \Cref{notation quotient by diagonal}.
    The remaining subgroups are all centrally isogenous to products of $n$ copies of $\su(2),\,n\leq 4$.
    Note that among the subgroups listed above,
    only $\spin(9)$ and $\left(\symp(3)\times\su(2)\right)/\mu_{2}^{\Delta}$ are maximal proper connected regular subgroups of $\grpF(\R)$.
\end{ex}

\subsection{Arthur's conjectures}\label{section Arthur conjectures in the intro}
 As in \cite{ChenevierRenard},
for the part (2) of \Cref{prob automorphic dimension for reductive groups}, 
we need some conjectures on automorphic representations.
For a connected reductive algebraic group $G$ over $\Q$,
Arthur introduces in \cite{ArthurUnipRep} a conjectural parametrization of discrete automorphic representations,
via \emph{discrete global Arthur parameters} for $G$.
In the level one case,
these parameters are $\widehat{G}(\C)$-conjugacy classes of admissible morphisms 
\[\psi:\mathcal{L}_{\Z}\times\SL_{2}(\C)\rightarrow\widehat{G}(\C),\]
where $\mathcal{L}_{\Z}$ is the hypothetical Langlands group of $\Z$ (see \Cref{section global parametrization and Langlands group} for more details),
and $\widehat{G}$ is the Langlands dual group of $G$.
Arthur proposes a conjectural formula for the multiplicity of an irreducible $G(\A)$-representation 
in the discrete automorphic spectrum of $G$,
in terms of the associated global Arthur parameters.

In \cite{arthur2013endoscopic},
Arthur reformulates his conjectures 
for any quasi-split classical group $G$,
avoiding the appearance of the hypothetical Langlands group $\mathcal{L}_{\Z}$.
In this case,
he relates the global Arthur parameters for $G$
to cuspidal automorphic representations of linear groups,
and proves the endoscopic classifications,
relying in particular on the works of M\oe glin-Waldspurger \cite{moeglin2014stabilisation}, Ng\^o \cite{NgoAFL} and many others.
We refer to \cite[\S 8]{ChenevierLannes}
for precise statements of Arthur's results in \cite{arthur2013endoscopic}
in the case of level one cohomological automorphic representations of classical groups. 

Of course $\grpF$ is not a classical group,
and Arthur's general conjectures \cite{ArthurUnipRep} are still open in this case.
Nevertheless, they can still be formulated quite precisely if we admit the existence of $\mathcal{L}_{\Z}$.
See also \cite[\S 6.4]{ChenevierLannes} for some generalities of Arthur's conjectures in the level one case.
\begin{notation}
    In the rest of this paper,
we will mark any result conditional to  
the existence of $\mathcal{L}_{\Z}$ and Arthur's multiplicity formula (\Cref{conj Arthur multiplicity formula})
with a star $*$.
\end{notation}
Now we briefly explain Arthur's conjectures for $\grpF$.
For a level one automorphic representation $\pi$ of $\grpF$
with global Arthur parameter $\psi:\mathcal{L}_{\Z}\times\SL_{2}(\C)\rightarrow \grpF(\C)$,
we may compose $\psi$ with the $26$-dimensional irreducible representation $\mathrm{r}:\grpF(\C)\rightarrow \GL_{26}(\C)$
\footnote{The image of $\mathrm{r}$ is even inside $\sorth_{26}(\C)\subset\SL_{26}(\C)\subset\GL_{26}(\C)$.},
and thus obtain a representation of $\mathcal{L}_{\Z}\times\SL_{2}(\C)$.
This representation is decomposed as:
\begin{align}
    \mathrm{r}\circ\psi\simeq \pi_{1}[d_{1}]\oplus\cdots\oplus\pi_{k}[d_{k}],\tag{$\star$}\label{eqn parameter for F4}
\end{align}
where $\pi_{i}$ is an $n_{i}$-dimensional irreducible representation of $\mathcal{L}_{\Z}$ and $[d_{i}]$ stands for the irreducible $d_{i}$-dimensional representation of $\SL_{2}(\C)$,
and $\sum_{i=1}^{k}n_{i}d_{i}=26$.
We identify $\pi_{i}$ as a level one cuspidal representations of $\pgl_{n_{i}}$,
and observe that it is always self-dual and algebraic in this case (see \Cref{section cuspidal representations of GL(n)}).
In a similar way as in \cite{arthur2013endoscopic},
we view the global Arthur parameter $\psi$ as a linear combination of $\pi_{i}[d_{i}]$'s.

In \Cref{section Arthur parameters of F4},
we derive from \Cref{introthm classification of subgroups of F4} that 
the Sato-Tate group of any $\pi_{i}$ appearing in the decomposition \eqref{eqn parameter for F4}
is one of the following compact Lie groups:
\begin{align*}
    \su(2),\,\symp(2),\,\symp(3),\,\sorth(8),\,\sorth(9),\,\lietype{G}{2},\grpF(\R).    \tag{$\star\star$}\label{eqn possible Sato Tate groups}
\end{align*}
Cuspidal representations with Sato-Tate group $\grpF(\R)$
conjecturally correspond to the desired $\ell$-adic representations in \Cref{prob refinement of Galois representations counting},
and those with other Sato-Tate groups in \eqref{eqn possible Sato Tate groups}
are related to level one automorphic representations for the following $\Q$-groups:
\[\pgl_{2},\,\sorth_{3,2},\,\sorth_{7},\,\sorth_{8},\,\symp_{8},\lietype{G}{2},\]
which have already been studied in \cites{ChenevierRenard}{TaibiDimension}{Chenevier2020DiscreteSM}.

Conversely, for a global Arthur parameter $\psi:\mathcal{L}_{\Z}\times\SL_{2}(\C)\rightarrow\grpF(\C)$ whose ``archimedean component" is an Adams-Johnson parameter (see \Cref{def Adams Johnson parameters} and \Cref{rmk relation with Adams Johnson parameter}),
the multiplicity of its corresponding irreducible $\grpF(\A)$-representation $\pi$ in the automorphic spectrum 
can be calculated via Arthur's formula in \cite{ArthurUnipRep}.
In the level one case,
this formula involves two characters on
the centralizer $\mathrm{C}_{\psi}$ of $\mathrm{Im}(\psi)$ in $\grpF(\C)$,
which is an elementary abelian $2$-group.
The first character is Arthur's character $\varepsilon_{\psi}$,
and we will recall its definition in \Cref{section Arthur epsilon character}.
The second character is a local character at the archimedean place,
an explicit formula for which will be given in \Cref{section multiplicity formula for F4}.

\subsection{Refinement of the counting}\label{section refinement of counting}
 With all these preparations, 
we are ready to refine the counting in \Cref{introthm dimension formula for automorphic representations}.
For a global Arthur parameter $\psi:\mathcal{L}_{\Z}\times\SL_{2}(\C)\rightarrow\grpF(\C)$,
one can associate two invariants:
\begin{itemize}
    \item its Sato-Tate group $\mathrm{H}(\psi):=\psi(\mathcal{L}_{\Z}\times \su(2))$, viewed as a conjugacy of subgroups in the compact group $\grpF(\R)$;
    \item its ``weights'', i.e. eigenvalues of its infinitesimal character under the $26$-dimensional irreducible representation $\mathrm{r}:\grpF\rightarrow\SL_{26}$. 
\end{itemize}
Given any conjugacy class of proper subgroups $H$ of $\grpF(\R)$ classified in \Cref{introthm classification of subgroups of F4},
in \Cref{section classification of Arthur parameters of F4} we classify all the possible decompositions \eqref{eqn parameter for F4} of $\mathrm{r}\circ\psi$ for 
global Arthur parameters $\psi$ with $\mathrm{H}(\psi)=H$. 
If $\psi$ corresponds to an irreducible level one $\grpF(\A)$-representation $\pi$,
an important part of our work is to give an exact formula for the multiplicity of $\pi$,
for each case of Sato-Tate groups.
Roughly speaking,
the multiplicity depends on how the weights of $\psi$ 
are distributed in the summands $\pi_{i}[d_{i}]$'s of \eqref{eqn parameter for F4}.
In conclusion, we have the following result:
\begin{spintrothm}\label{introthm multiplicity formula for F4}
    (\Cref{thm arthur parameter with nonzero multiplicity})
    \begin{enumerate}[label=(\alph*)]
        \item The Sato-Tate group of a level one automorphic representation for $\grpF$ is either
        $\grpF(\R)$ or one of the proper subgroups of $\grpF(\R)$ classified in \Cref{introthm classification of subgroups of F4} except $\spin(8)$.
        \item For global Arthur parameters of $\grpF$ with a given Sato-Tate group,
        the multiplicity of its corresponding irreducible level one $\grpF(\A)$-representation ($0$ or $1$)
        is given explicitly by the formulas in \Cref{prop multiplicity principal A1} to \Cref{prop multiplicity F4 tempered}.
    \end{enumerate}
\end{spintrothm}
\begin{rmk}\label{rmk not endoscopic group}
    We observe that not all subgroups in \Cref{introthm classification of subgroups of F4}
    come from \emph{endoscopic groups} of $\grpF$,
    in the sense of \cite{arthur2013endoscopic}.
    For example, the subgroup $\lietype{G}{2}\times\sorth(3)$ has trivial centralizer in $\grpF(\R)$,
    thus it can not be the centralizer of any element in $\grpF(\R)$.
    As a result,
    our conjectural refinement is finer than Arthur's endoscopic classification in \cite{arthur2013endoscopic}. 
\end{rmk}
Given an irreducible representation $\vrep{\lambda}$ of $\grpF(\R)$,
from \Cref{introthm dimension formula for automorphic representations}
we know the number of equivalence classes of level one automorphic representations $\pi$ for $\grpF$ with $\pi_{\infty}\simeq\vrep{\lambda}$.
The weights of the global Arthur parameter $\psi_{\pi}$ of $\pi$ are determined by $\vrep{\lambda}$.
We can enumerate all the possible global Arthur parameters with these weights,
and then use the multiplicity formulas in \Cref{introthm multiplicity formula for F4} to determine their multiplicities.
In this way,
we obtain a \emph{conjectural} refinement of the counting in \Cref{introthm dimension formula for automorphic representations}.
\begin{ex}
    In \Cref{table packet 34} and \Cref{table packet 36}, 
    we list some parameters with ``small'' archimedean components.
    For example,
    there are two different level one automorphic representations of $\grpF$ with trivial archimedean components,
    whose Arthur parameters are:
    \[[9]\oplus [17]\text{ and }\Delta_{11}[6]\oplus [5]\oplus [9].\]
    The first parameter corresponds to the trivial representation,
    and its Sato-Tate group
    \footnote{As we mentioned in \Cref{rmk Sato Tate groups},
    the notion of Sato-Tate groups in the introduction coincides with the usual notion if and only if the restriction of the global Arthur parameter to $\SL_{2}(\C)$ is trivial.
    Here these two Arthur parameters fail to satisfy this condition.
    }
    is the principal $\su(2)$ in $\grpF(\R)$.
    The Sato-Tate group of the second parameter 
    is isomorphic to $\left(\su(2)\times\su(2)\right)/\mu_{2}^{\Delta}$
    \footnote{Beware that there are many distinct conjugacy classes of subgroups of $\grpF(\R)$ isomorphic to $\su(2)$.},
    the information about which can be found in \Cref{section Arthur principal A1 in Sp(3) with Sp(1)}.
    The Hecke eigenvalues of its corresponding automorphic representation for $\grpF$
    are thus related to the Fourier coefficients of Ramanujan's $\Delta$ function,
    i.e. the unique level one classical cuspidal modular form with weight $11$.
\end{ex}
As a consequence of \Cref{introthm multiplicity formula for F4},
we obtain a conjectural solution to \Cref{prob refinement of Galois representations counting},
stated in terms of automorphic representations:
\begin{spintrothm}\label{introthm formula for cuspidal representations with Sato-Tate group F4}(\Cref{prop conjectural number of F4 image representations} and \Cref{prop smallest motivic weight 36 tempered stable representation})
    The number of algebraic, cuspidal, level one automorphic representations of $\GL_{26}$ over $\Q$
    satisfying:
    \begin{itemize}
        \item the Sato-Tate group is $\grpF(\R)$,
        \item and the multiset of weights
        \footnote{See \Cref{section cuspidal representations of GL(n)} for the precise definition of weights for an algebraic cuspidal level one automorphic representation of $\GL_{n}$.}  
        is $\mathrm{HT}(a,b,c,d)$ for $a,b,c,d\geq 1$,
    \end{itemize} 
    is $\lietype{F}{4}(a-1,b-1,c-1,d-1)$,
    where $\lietype{F}{4}(\lambda)$ is an explicit function on $\mathbb{N}^{4}$ given by \Cref{prop conjectural number of F4 image representations}.
\end{spintrothm}
\begin{ex}\label{introex number of F4 representations}
    The quadruples $(a,b,c,d)\in \mathbb{N}^{4}$ such that 
    \begin{itemize}
        \item the largest weight $2a+3b+2c+d+8$ in the multiset $\mathrm{HT}(a+1,b+1,c+1,d+1)$ is not larger than $22$,
        \item and $\lietype{F}{4}(a,b,c,d)\neq 0$,
    \end{itemize}
    are listed in \Cref{TableMult}, \Cref{TableAppendix}.
    We also list the values of $\lietype{F}{4}(a,b,c,d)$ for these quadruples.
    As a direct consequence,
    we predict the existence of the geometric $\ell$-adic representation in \Cref{introex tempered representation}.
\end{ex}
\begin{rmk}\label{intrormk reason of conductor 1}
    One may want to remove the level one condition,
    like in \cite{lachausseeThesis}.
    For the part (1) of \Cref{prob automorphic dimension for reductive groups} for $\grpF$,
    one can calculate the dimension of invariants under other congruence subgroups,
    and obtain results similar to \Cref{introthm dimension formula for automorphic representations} for higher levels.
    However, for the part (2) of \Cref{prob automorphic dimension for reductive groups} for $\grpF$, 
    what we use is a simplified version of Arthur's recipe in \cite{ArthurUnipRep}.
    When allowing ramifications at some finite place $p$,
    one needs some properties of \emph{local Arthur packets} for $\grpF(\Q_{p})$,
    which are still unknown to us. 
\end{rmk}
Let us end the introduction with a short summary of the contents of this paper.
In \Cref{section groups of type F4},
we recall the definition of $\grpF$
and some results of Gross \cite{G96} on reductive $\Z$-models of $\grpF$.
We also give a new proof for \Cref{introthm two reductive integral models of F4}.
We prove \Cref{introthm dimension formula for automorphic representations} 
in \Cref{section dimension of invariant subspaces of F4 representations}.
In \Cref{section subgroups of F4},
we study the subgroups of the compact Lie group $\grpF(\R)$ 
and prove \Cref{introthm classification of subgroups of F4}.
In \Cref{section automorphic representations and Arthur's conjectures}, we recall the theory of level one automorphic representations 
and the conjectures by Arthur and Langlands,
mainly following \cites{ChenevierRenard}{ChenevierLannes}. 
Then we apply these conjectures to $\grpF$ and prove \Cref{introthm multiplicity formula for F4} and \Cref{introthm formula for cuspidal representations with Sato-Tate group F4} in \Cref{section Arthur classification for F4}.
In \Cref{TableAppendix},
some figures and tables used in this article are provided.

\subsection*{Acknowledgements} 
 I would like to thank my thesis advisor, Ga\"etan Chenevier, for introducing me to this subject.
His advice and inspiring discussions have greatly contributed to this work.
I also thank Wee Teck Gan and Aaron Pollack for helpful suggestions.  

\section{The compact Lie group \texorpdfstring{$\lietype{F}{4}$}{} and its reductive integral models}\label{section groups of type F4}
 In this section we introduce the compact Lie group of type $\lietype{F}{4}$ that we will discuss in this paper,
and give a classification of its reductive $\Z$-models.
\subsection{The compact group \texorpdfstring{$\lietype{F}{4}$}{PDFstring} and its rational structure}\label{section compact Lie group F4}
 To construct Lie groups of exceptional types, 
we need to recall the notion of octonions,
and our main reference is \cite[\S 5]{NonRed}.
\begin{defi}\label{def general definition of octonion algebra}
    An \emph{octonion algebra} $C$ over a field $k$ is a non-associative $k$-algebra of $k$-dimension $8$ with $2$-sided identity element $e$ 
    such that there exists a non-degenerate quadratic form $N$ on $C$ satisfying 
    $N(xy)=N(x)N(y), x,y\in C$.
    The quadratic form $N$ is referred as the \emph{norm} on $C$.
\end{defi}
When considering octonion algebras over $\R$, we have the following classification result:
\begin{prop}\label{prop unique real definite octonion}
    \cite[Theorem 15.1]{AdamsExceptionalGrp}
    Up to $\R$-algebra isomorphism,
    there is a unique octonion algebra $\oct_{\R}$ over $\R$ 
    whose norm $\mathrm{N}$ is positive definite,
    which is named as the \emph{real octonion division algebra}.
\end{prop}
The multiplication law $\oct_{\R}\times\oct_{\R}\rightarrow \oct_{\R}$ can be given as follows:
as a vector space $\oct_{\R}$ admits a basis $\{\mathrm{e},\mathrm{e}_{1},\ldots,\mathrm{e}_{7}\}$ such that $\mathrm{e}$ is the identity element and as an $\R$-algebra $\oct_{\R}$ 
is generated by $\{\mathrm{e}_{1},\ldots,\mathrm{e}_{7}\}$ subject to the relations 
\begin{itemize}
    \item for all $i$, $\mathrm{e}_{i}^{2}=-\mathrm{e}$;
    \item considering the subscripts as elements in $\Z/7\Z$,
        the subspace of $\oct_{\R}$ generated by $\{\mathrm{e},\mathrm{e}_{i},\mathrm{e}_{i+1},\mathrm{e}_{i+3}\}$ is an associative algebra with relations 
        \[\mathrm{e}_{i}^{2}=\mathrm{e}_{i+1}^{2}=\mathrm{e}_{i+3}^{2}=-\mathrm{e}, \mathrm{e}_{i}\mathrm{e}_{i+1}=-\mathrm{e}_{i+1}\mathrm{e}_{i}=\mathrm{e}_{i+3}.\] 
\end{itemize}
We identify the real numbers $\R$ with the subalgebra $\R \mathrm{e}$ of $\oct_{\R}$
and the identity element of $\oct_{\R}$ will be denoted as $1$. 
Now we recall some basic properties of $\oct_{\R}$, for which we refer to \cite[\S 5]{NonRed}.
There is an anti-involution of algebra $x\mapsto \overline{x}$ called the \emph{conjugation} on $\oct_{\R}$, 
defined by $\overline{1}=1$ and $\overline{\mathrm{e}_{i}}=-\mathrm{e}_{i}$ for each $i$. 
The \emph{trace} and \emph{norm} on $\oct_{\R}$ are defined as: 
\[\mathrm{Tr}(x)=x+\overline{x},\ \mathrm{N}(x)=x\cdot\overline{x}=\overline{x}\cdot x.\]
The multiplication law on $\oct_{\R}$ implies that 
\begin{align}\label{eqn trace-commutativity of octonion}
    \tr(xy)=\tr(yx)=\tr(\overline{x}\cdot\overline{y})\text{ for all }x,y\in\oct_{\R}.
\end{align}
For an element $x=x_{0}+\sum\limits_{i=1}^{7}x_{i}\mathrm{e}_{i}\in\oct_{\R}$,
its norm $\mathrm{N}(x)$ equals
$\sum_{i=0}^{7}x_{i}^{2}$,
from which we can see that $\mathrm{N}$ is a positive definite quadratic form.
Its associated symmetric bilinear form is 
$\langle x,y\rangle:=\mathrm{N}(x+y)-\mathrm{N}(x)-\mathrm{N}(y)=x\cdot \overline{y}+y\cdot \overline{x}=\mathrm{Tr}(x\cdot \overline{y})$.

Although the multiplication law of $\oct_{\R}$ is not associative,
it is still trace-associative in the sense that 
\[\tr((x\cdot y)\cdot z)=\tr(x\cdot(y\cdot z))\text{ for all }x,y,z\in\oct_{\R},\]
and we can define $\tr(xyz):=\tr((x\cdot y)\cdot z)=\tr(x\cdot (y\cdot z))$.

For our construction, we still have to recall the exceptional Jordan algebra,
following \cite[\S 6]{NonRed}:
\begin{defi}\label{def real exceptional Jordan algebra}
    The \emph{(positive definite) real exceptional Jordan algebra}, denoted by $\mathrm{J}_{\R}$, 
    is the $27$-dimensional $\R$-vector space consisting of ``Hermitian'' matrices in $\mathrm{M}_{3}(\oct_\R)$,
    i.e. matrices of the form 
    \begin{align*}
        \left(\begin{matrix}
        a&z&\overline{y}\\
        \overline{z}&b&x\\
        y&\overline{x}&c
        \end{matrix}\right),\ a,b,c\in\R,\ x,y,z\in\oct_{\R},
    \end{align*}
    equipped with the $\R$-bilinear multiplication law 
    \[\mathrm{J}_{\R}\times \mathrm{J}_{\R}\rightarrow \mathrm{J}_{\R}, A\circ B:=\frac{1}{2}(AB+BA),\]
    where $AB$ and $BA$ denote the usual product of octonionic matrices,
    and with $2$-sided identity element $\mathrm{I}$ given by the standard matrix identity element $\mathrm{diag}(1,1,1)$.
\end{defi}
As an $\R$-algebra, $\mathrm{J}_{\R}$ is commutative but not associative.
\begin{notation}
    To compress the space,
    when we do not need to emphasize the matrix structure of elements in $\mathrm{J}_{\R}$,
    we denote the element 
    \[\left(\begin{matrix}
        a&z&\overline{y}\\
        \overline{z}&b&x\\
        y&\overline{x}&c
        \end{matrix}\right),\ a,b,c\in\R,\ x,y,z\in\oct_{\R}\]
    by $[a,b,c\,;x,y,z]$ for short.    
\end{notation}

The \emph{trace} of $A=[a,b,c\,;x,y,z]\in \mathrm{J}_{\R}$ is defined as $\mathrm{Tr}(A):=a+b+c$.
The underlying vector space of $\mathrm{J}_{\R}$ is equipped with the non-degenerate positive definite quadratic form:
\begin{align}\label{eqn quadratic form on exceptional Jordan algebra}
\mathrm{Q}(A):=\mathrm{Tr}(A\circ A)/2=\frac{1}{2}(a^{2}+b^{2}+c^{2})+\mathrm{N}(x)+\mathrm{N}(y)+\mathrm{N}(z).
\end{align}
Its associated bilinear form is $\mathrm{B}_{\mathrm{Q}}(A,B):=\mathrm{Q}(A+B)-\mathrm{Q}(A)-\mathrm{Q}(B)=\mathrm{Tr}(A\circ B)$.
The \emph{determinant} of the matrix $A$ is defined by
\begin{align}\label{eqn determinant on Jordan algebra}
\det(A):=abc+\mathrm{Tr}(xyz)-a\mathrm{N}(x)-b\mathrm{N}(y)-c\mathrm{N}(z).
\end{align}
It defines a cubic form on $\mathrm{J}_{\R}$. 

We denote by $\lietype{F}{4}$ the subgroup $\aut(\mathrm{J}_{\R},\circ)$ of $\GL(\mathrm{J}_{\R})$ consisting of elements 
$g\in\GL(\mathrm{J}_{\R})$ such that for all $A,B\in \mathrm{J}_{\R}$,
$g(A\circ B)=g(A)\circ g(B)$.
It is a compact Lie group of type $\lietype{F}{4}$ \cite[Theorem 16.7]{AdamsExceptionalGrp}.

In this paper, we deal with automorphic forms so we want a reductive group over $\Q$ whose real points is isomorphic to $\lietype{F}{4}$.
For this purpose, we first define the following $\Q$-algebras:
\begin{defi}\label{def rational octonion and exceptional Jordan algebra}
    \emph{Cayley's definite octonion algebra} $\oct_{\Q}$ is the sub-$\Q$-algebra of $\oct_{\R}$
    generated by $\{\mathrm{e}_{1},\ldots,\mathrm{e}_{7}\}$.
    The \emph{(positive definite) rational exceptional Jordan algebra} $\mathrm{J}_{\Q}$ is 
    the sub-$\Q$-space of $\mathrm{J}_{\R}$ consisting of $[a,b,c\,;x,y,z],a,b,c\in\Q,x,y,z\in \oct_{\Q}$
    equipped with the multiplication $\circ$.  
\end{defi}
The main object considered in this paper is the following algebraic group:
\begin{defi}\label{def algebraic group F4}
    We define $\mathbf{F}_{4}$ to be the closed subgroup of the algebraic $\Q$-group $\GL_{\mathrm{J}_{\Q}}$,  
    which as a functor sends a commutative unital $\Q$-algebra $R$ to the group 
    \[\mathbf{F}_{4}(R):=\aut(\mathrm{J}_{\Q}\otimes_{\Q}R,\circ)=\{g\in\GL(\mathrm{J}_{\Q}\otimes_{\Q}R)\,|\,g(A\circ B)=g(A)\circ g(B),\forall A,B\in \mathrm{J}_{\Q}\otimes_{\Q}R\}.\]
\end{defi}
From the definition we have $\mathbf{F}_{4}(\R)=\lietype{F}{4}$.
By \cite[Theorem 7.2.1]{OctSpr},
$\mathbf{F}_{4}$ is a semisimple and simply-connected group over $\Q$,
and it is of adjoint type.
\begin{rmk}\label{rmk alternative description F4}
    We have an alternative description of $\mathbf{F}_{4}$ that will be used later:
    the closed subgroup $\aut_{(\mathrm{J}_{\Q},\det,\mathrm{I})/\Q}$ of $\GL_{\mathrm{J}_{\Q}}$ consisting of linear automorphisms that preserve both the cubic form $\det$ and the identity element $\mathrm{I}$.
    The closed subgroups $\mathbf{F}_{4}=\aut_{(\mathrm{J}_{\Q},\circ)/\Q}$ and $\aut_{(\mathrm{J}_{\Q},\det,\mathrm{I})/\Q}$ inside $\GL_{\mathrm{J}_{\Q}}$ are both smooth and they have the same geometric points according to \cite[Proposition 5.9.4]{OctSpr},
    so they coincide.
\end{rmk}

\subsection{Reductive \texorpdfstring{$\Z$}{}-models of reductive \texorpdfstring{$\Q$}{}-groups}\label{section reductive integral models of Q groups}
 Now we recall some results in \cites{G96}{GrossZmodelAMS}.
In this subsection, let $G$ be a connected reductive algebraic group over $\Q$.
Denote the product $\prod_{p} \Z_{p}$ by $\widehat{\Z}$ and 
let $\A_{f}=\widehat{\Z}\otimes_{\Z}\Q$ be the ring of finite ad\`eles,
and $\A=\R\times \A_{f}$.
\begin{defi}\label{def reductive integral models}
    A \emph{reductive $\Z$-model} of $G$ is a pair $(\G,\iota)$ consisting of:
    \begin{itemize}
        \item an affine smooth group scheme $\G$ of finite type over $\Z$ such that $\G\otimes_{\Z}\Z/p\Z$ is reductive over $\Z/p\Z$ for each prime number $p$,
        \item an isomorphism $\iota:\G\otimes_{\Z}\Q\simeq G$ of algebraic groups over $\Q$.
    \end{itemize}
    Two reductive $\Z$-models $(\G_{1},\iota_{1})$ and $(\G_{2},\iota_{2})$ are said to be isomorphic if there exists an isomorphism $f:\G_{1}\rightarrow \G_{2}$ over $\Z$ such that the following diagram commutes:
    \[\begin{tikzcd}
        \G_{1}\otimes_{\Z}\Q\ar[rr,"f_{\Q}"]\ar[dr,"\iota_{1}"']& & \G_{2}\otimes_{\Z}\Q\ar[dl,"\iota_{2}"]\\
        & G &
    \end{tikzcd}\]
\end{defi}
\begin{rmk}\label{rmk omit the isomorphism in the def of reductive model}
    When there is no confusion about $\iota$, we simply say that $\G$ is a reductive $\Z$-model of $G$.
\end{rmk}
From the theory of \emph{Chevalley groups} in \cite[\S XXV]{SGA3}, 
every group $G$ split over $\Q$ admits a reductive $\Z$-model.
Indeed, we can take the Chevalley group with the same root datum of $G$ to be its reductive $\Z$-model.

When $G$ is not split, 
in general the existence of reductive $\Z$-models of $G$ is no longer ensured.
Now we consider the case when $G$ is \emph{anisotropic},
i.e. $G$ does not contain any non-trivial split $\Q$-torus.
When $G$ has a reductive $\Z$-model, being anisotropic is equivalent to that $G(\R)$ is compact,
which is due to \cite[Theorem 5.5(1)]{AlgGrpandNumberTheory} and \cite[Proposition 2.1]{G96}.
In \cite[\S 1]{G96}, Gross proves the following result:
\begin{thm}\label{thm classification compact group admitting integral models}
    Let $G$ be an anisotropic semisimple simply-connected $\Q$-group
    such that $G$ admits reductive $\Z$-models
    and the root system of $G_{\C}$ is irreducible,
    then the Lie type of $G$ is among:
    \[\lietype{B}{(d-1)/2}\,(d\equiv \pm 1\modulo 8),\lietype{D}{d/2}\,(d\equiv 0\modulo 8),\lietype{G}{2},\lietype{F}{4},\lietype{E}{8}.\]
\end{thm}
The next question is to classify reductive $\Z$-models of a given anisotropic group $G$ up to some equivalence relation.
\begin{defi}\label{def genus of reductive integral models}
    Let $(\G,\mathrm{id})$ be a reductive $\Z$-model of its generic fiber $G:=\G\otimes_{\Z} \Q$.
    A reductive $\Z$-model $(\G^{\prime},\iota^{\prime})$ of $G$ is said to be in the same \emph{genus} as $\G$,
    if $\iota^{\prime}(\G^{\prime}(\widehat{\Z}))$ and $\G(\widehat{\Z})$ are conjugate in $G(\A_{f})$.
\end{defi}
\begin{rmk}\label{rmk equivalent condition for genus of trductive integral models}
    This condition is equivalent to that for each prime $p$,
    $\iota^{\prime}(\G^{\prime}(\Z_{p}))$ is conjugate to $\G(\Z_{p})$ in $G(\Q_{p})$.
\end{rmk}
By \cite[Proposition 1.4]{GrossZmodelAMS},
the equivalence classes of reductive $\Z$-models in the genus of $\G$ can be identified with the coset space $G(\A_{f})/\G(\widehat{\Z})$.

The group $G(\Q)$ acts on reductive $\Z$-models in the genus of $\G$ by the formula:
\[g(\G^{\prime},\iota^{\prime})=(\G^{\prime},\mathrm{ad}(g)\circ \iota^{\prime}),\]
where $\mathrm{ad}(g)$ is the conjugation by $g$.
This induces an action of $G(\Q)$ on the equivalence classes of reductive $\Z$-models in the genus of $\G$.
We say two reductive $\Z$-models in the genus of $\G$ are \emph{$G(\Q)$-conjugate} if their equivalence classes are in the same $G(\Q)$-orbit.

Now the set of $G(\Q)$-orbits on the equivalence classes of reductive $\Z$-models in the genus of $\G$ can be identified
with the double coset space $G(\Q)\backslash G(\A_{f})/\G(\widehat{\Z})$,
which is finite by Borel's famous result \cite{Borel63}.

\subsection{Reductive \texorpdfstring{$\Z$}{}-models of \texorpdfstring{$\lietype{F}{4}$}{}}\label{section reductive integral models of F4}
 For our $\Q$-group $\mathbf{F}_{4}$,
the $\mathbf{F}_{4}(\Q)$-orbits of equivalence classes of reductive $\Z$-models of $\mathbf{F}_{4}$ in some genus
is determined by Gross in \cite[Proposition 5.3]{G96},
using the mass formula \cite[Proposition 2.2]{G96}.
In this subsection we provide an alternative proof for his result,
which will be helpful for our computations in \Cref{section dimension of invariant subspaces of F4 representations}.

\subsubsection{Integral structures of \texorpdfstring{$\oct_{\Q}$}{PDFstring} and \texorpdfstring{$\mathrm{J}_{\Q}$}{PDFstring}}\label{section integral structure of octonion and Jordan algebra}
 Parallel to the construction of $\mathbf{F}_{4}$ in \Cref{section compact Lie group F4},
we want to define integral structures of $\oct_{\Q}$ and $\mathrm{J}_{\Q}$ and then use them to construct reductive $\Z$-models of $\grpF$.
\begin{defi}\label{def coxeter integral octonion}
    \emph{Coxeter's integral order} $\oct_{\Z}$ is the $\Z$-lattice of rank $8$ inside $\oct_{\Q}$ spanned by the lattice $\Z\oplus\Z \mathrm{e}_{1}\oplus\cdots\oplus\Z \mathrm{e}_{7}$ and the four elements
    \begin{align*}
    \mathrm{h}_{1}&=(1+\mathrm{e}_{1}+\mathrm{e}_{2}+\mathrm{e}_{4})/2,\mathrm{h}_{2}=(1+\mathrm{e}_{1}+\mathrm{e}_{3}+\mathrm{e}_{7})/2,\\
    \mathrm{h}_{3}&=(1+\mathrm{e}_{1}+\mathrm{e}_{5}+\mathrm{e}_{6})/2,\mathrm{h}_{4}=(\mathrm{e}_{1}+\mathrm{e}_{2}+\mathrm{e}_{3}+\mathrm{e}_{5})/2,
\end{align*}
equipped with the multiplication of $\oct_{\Q}$.
This lattice contains the identity element of $\oct_{\Q}$ and is stable under the multiplication,
i.e. is an \emph{order} in $\oct_{\Q}$.
\end{defi}
\begin{rmk}\label{rmk quadratic structure of coxeter octonion}
    The underlying lattice of $\oct_{\Z}$ equipped with the quadratic form $\mathrm{N}|_{\oct_{\Z}}$
    is isometric to the even unimodular lattice
    \[\mathrm{E}_{8}=\left\{(x_{i})\in\Z^{8}\cup(\Z+\frac{1}{2})^{8}\,\middle\vert\,\sum_{i}x_{i}\equiv 0 \modulo 2\right\}.\]
\end{rmk}
Let $\mathrm{J}_{\Z}$ be the unimodular lattice
\[\left\{[a,b,c\,;x,y,z]\in \mathrm{J}_{\Q}\,\middle\vert\,a,b,c\in\Z,x,y,z\in \oct_{\Z}\right\}\]
of rank $27$ inside the $\Q$-vector space $\mathrm{J}_{\Q}$.
\begin{rmk}\label{rmk not a linear Z structure}
    This lattice is not stable under the Jordan multiplication $\circ$ defined on $\mathrm{J}_{\Q}$,
    since $[1,0,0\,;0,0,0]\circ [0,0,0\,;0,1,0]=\frac{1}{2}[0,0,0\,;0,1,0]\notin \jord_{\Z}$.
\end{rmk}

As in \Cref{rmk alternative description F4},
the $\Q$-group $\grpF$ coincides with the group $\aut_{(\mathrm{J}_{\Q},\det,\mathrm{I})/\Q}$.
The restriction of the cubic form $\det$ to $\jord_{\Z}$ has integral values,
and it is a \emph{polynomial law} in the sense of \cite{RobyPolyLaw}.
The triple $(\mathrm{J}_{\Q},\det,\mathrm{I})$ has a natural integral structure $(\mathrm{J}_{\Z},\det,\mathrm{I})$,
and the $\Z$-group scheme $\aut_{(\mathrm{J}_{\Z},\det,\mathrm{I})/\Z}$,
sending any commutative $\Z$-algebra $R$ to 
\[\set{g\in\GL(\jord_{\Z}\otimes_{\Z} R)}{g\mathrm{I}=\mathrm{I},\det(gX)=\det(X)\text{ for any }X\in\jord_{\Z}\otimes_{\Z} R},\]
is expected to be a reductive $\Z$-model of $\grpF$.
We are going to consider the $\Z$-group scheme $\aut_{(\mathrm{J}_\Z,\det,e)/\Z}$ 
for any $e\in \mathrm{J}_{\Z}$ satisfying certain conditions,
in order to produce several reductive $\Z$-models of $\grpF$ uniformly.
\begin{defi}\label{def positive definite jordan matrix}
    An element
    \[A=\left(\begin{matrix}
        a&z&\overline{y}\\
        \overline{z}&b&x\\
        y&\overline{x}&c
    \end{matrix}\right)\in \mathrm{J}_{\R}\]
    is said to be \emph{positive definite} if its seven ``minor determinants''
    \[a,b,c,ab-\mathrm{N}(z),bc-\mathrm{N}(x),ca-\mathrm{N}(y),\det(A)\in \R\]
    are all positive.
    A positive definite element $e$ in $\mathrm{J}_{\R}$ with $\det e=1$ is called a \emph{polarization}.
\end{defi}
Given a polarization $e$ contained in the lattice $\mathrm{J}_{\Z}$,
one constructs a $\Z$-group scheme $\mathcal{F}_{4,e}:=\aut_{(\mathrm{J}_{\Z},\det,e)/\Z}$
in the same way as $\aut_{(\mathrm{J}_{\Z},\det,\mathrm{I})/\Z}$.
The following result shows that this group scheme is a reductive $\Z$-model of $\grpF$.
\begin{prop}\label{prop automorphism of integral jordan algebra is reductive}
    \cite[Proposition 6.6, Example 6.7]{NonRed}
    For any choice of polarization $e\in \mathrm{J}_{\Z}$,
    the $\Z$-group $\mathcal{F}_{4,e}$ is smooth,
    its fiber $\mathcal{F}_{4,e}\otimes_{\Z}\Z/p\Z$ is semisimple for every prime $p$,
    and $\mathcal{F}_{4,e}(\R)$ is a compact Lie group of type $\lietype{F}{4}$.    
\end{prop}
Taking $e$ to be the identity element $\mathrm{I}$,
the generic fiber of $\mathcal{F}_{4,\mathrm{I}}$ is $\aut_{(\mathrm{J}_\Q,\det,\mathrm{I})/\Q}=\grpF$,
thus $\mathcal{F}_{4,\mathrm{I}}$ is a reductive $\Z$-model of $\grpF$.

If we take $e$ to be 
\begin{align*}
    \mathrm{E}:=
    [2,2,2\,;\beta,\beta,\beta],
    \beta=\frac{1}{2}(-1+\mathrm{e}_{1}+\mathrm{e}_{2}+\cdots+\mathrm{e}_{7})\in \mathrm{J}_{\Z},
\end{align*}
as in \cite[(5.4)]{Leech},
by \cite[Example 6.7]{NonRed} the generic fiber of $\mathcal{F}_{4,\mathrm{E}}$ is also isomorphic to $\grpF$.
We denote an isomorphism $\mathcal{F}_{4,\mathrm{E}}\otimes_{\Z}\Q\rightarrow \grpF$ by $\iota$.
Actually $\iota$ can be chosen to be the conjugation by an element in $\aut(\jord_{\Q},\det)$ that sends $\mathrm{E}$ to $\mathrm{I}$.

In \cite[Proposition 5.3]{G96},
Gross proves the following result:
\begin{prop}\label{prop exactly two reductive models}
    There are two $\grpF(\Q)$-orbits on the equivalence classes of reductive $\Z$-models of $\grpF$ in the genus of $\mathcal{F}_{4,\mathrm{I}}$,
    whose representatives are given by $(\mathcal{F}_{4,\mathrm{I}},\mathrm{id})$ and $(\mathcal{F}_{4,\mathrm{E}},\iota)$ respectively.
\end{prop}
\begin{rmk}\label{rmk comparing with Thm B in intro}
    In our case,
    as $\grpF$ is of adjoint type,
    the $\grpF(\Q)$-orbits in \Cref{prop exactly two reductive models} indeed parametrize the $\Z$-isomorphism classes considered in \Cref{introthm two reductive integral models of F4}.
\end{rmk}
Applying the mass formula \cite[Proposition 2.2]{G96} to $\grpF$,
we have 
\begin{align}\label{eqn mass formula for F4}
    \sum_{(\G,\iota)}\frac{1}{|\G(\Z)|}=\frac{1}{2^{4}}\zeta(-1)\zeta(-5)\zeta(-7)\zeta(-11)=\frac{691}{2^{15}\cdot 3^{6}\cdot 5^{2}\cdot 7^{2}\cdot 13},
\end{align}
where $(\G,\iota)$ varies over the $\grpF(\Q)$-conjugacy classes of reductive $\Z$-models of $\grpF$ in the genus of $\mathcal{F}_{4,\mathrm{I}}$.
As 
\begin{align}\label{eqn decompose the mass into sum}
    \frac{691}{2^{15}\cdot 3^{6}\cdot 5^{2}\cdot 7^{2}\cdot 13}=\frac{1}{2^{15}\cdot 3^{6}\cdot 5^{2}\cdot 7}+\frac{1}{2^{12}\cdot 3^{5}\cdot 7^{2}\cdot 13},
\end{align}
in order to prove \Cref{prop exactly two reductive models} it suffices to prove the following two things:
\begin{itemize}
    \item $\mathcal{F}_{4,\mathrm{I}}$ and $\mathcal{F}_{4,\mathrm{E}}$ are not $\grpF(\Q)$-conjugate.
    \item $|\mathcal{F}_{4,\mathrm{I}}(\Z)|\leq 2^{15}\cdot 3^{6}\cdot 5^{2}\cdot 7$ and $|\mathcal{F}_{4,\mathrm{E}}(\Z)|\leq 2^{12}\cdot 3^{5}\cdot 7^{2}\cdot 13$.
\end{itemize}
In his proof, Gross cites some results from \cite{ATLAS},
We are going to give another proof of \Cref{prop exactly two reductive models},
which avoids using results in \cite{ATLAS}.
\subsubsection{\texorpdfstring{$\mathcal{F}_{4,\mathrm{E}}(\Z)$}{}}\label{section second integral model of F4}
 Now we deal with the finite group $\mathcal{F}_{4,\mathrm{E}}(\Z)$.
Our goal is to prove:
\begin{prop}\label{prop upper bound order second model}
    $|\mathcal{F}_{4,\mathrm{E}}(\Z)|\leq 2^{12}\cdot 3^{5}\cdot 7^{2}\cdot 13$.
\end{prop}
With the choice of polarization $\mathrm{E}$, we can define a new bilinear form on $\mathrm{J}_{\Q}$:
\[\langle A,B\rangle_{\mathrm{E}}=(A,\mathrm{E},\mathrm{E})(B,\mathrm{E},\mathrm{E})-2(A,B,\mathrm{E}),\]
where the trilinear form $(~,~,~):\mathrm{J}_{\Q}^{3}\rightarrow \Q$ is defined by 
\begin{align*}
    (A,B,C)=&\frac{1}{2}[\det(A+B+C)-\det(A+B)-\det(B+C)-\det(C+A)\\
    &+\det(A)+\det(B)+\det(C)].
\end{align*} 
This bilinear form is positive definite and integral on $\mathrm{J}_{\Z}$ by \cite[Proposition 7.2]{Leech}.
\begin{notation}\label{notation elements of Jordan algebra}
    Here we give some notations for elements in $\mathrm{J}_{\R}$: 
    we write 
    \[\mathrm{E}_{1}:=[1,0,0\,;0,0,0],\mathrm{E}_{2}:=[0,1,0\,;0,0,0],\mathrm{E}_{3}:=[0,0,1\,;0,0,0]\] 
    and for any $x\in\oct_{\R}$,
    \[\mathrm{F}_{1}(x):=[0,0,0\,;x,0,0],\mathrm{F}_{2}(x):=[0,0,0\,;0,x,0],\mathrm{F}_{3}(x):=[0,0,0\,;0,0,x].\]
\end{notation}
Note that $1,\mathrm{e}_{1},\mathrm{e}_{2},\mathrm{e}_{3},\mathrm{h}_{1},\mathrm{h}_{2},\mathrm{h}_{3},\mathrm{h}_{4}$ is a basis of the lattice $\oct_{\Z}$,
thus we have the following basis of $\mathrm{J}_{\Z}$:

{\scriptsize\begin{align}\label{eqn basis of second Jordan algebra}
    \mathcal{B}:=\left(
        \begin{array}{c}
            \mathrm{E}_{1},\mathrm{E}_{2},\mathrm{E}_{3},\mathrm{F}_{1}(1),\mathrm{F}_{1}(\mathrm{e}_{1}),\mathrm{F}_{1}(\mathrm{e}_{2}),\mathrm{F}_{1}(\mathrm{e}_{3}),\mathrm{F}_{1}(\mathrm{h}_{1}),\mathrm{F}_{1}(\mathrm{h}_{2}),\mathrm{F}_{1}(\mathrm{h}_{3}),\mathrm{F}_{1}(\mathrm{h}_{4}),\mathrm{F}_{2}(1),\mathrm{F}_{2}(\mathrm{e}_{1}),\mathrm{F}_{2}(\mathrm{e}_{2}),\\
            \mathrm{F}_{2}(\mathrm{e}_{3}),\mathrm{F}_{2}(\mathrm{h}_{1}),\mathrm{F}_{2}(\mathrm{h}_{2}),\mathrm{F}_{2}(\mathrm{h}_{3}),\mathrm{F}_{2}(\mathrm{h}_{4}),\mathrm{F}_{3}(1),\mathrm{F}_{3}(\mathrm{e}_{1}),\mathrm{F}_{3}(\mathrm{e}_{2}),\mathrm{F}_{3}(\mathrm{e}_{3}),\mathrm{F}_{3}(\mathrm{h}_{1}),\mathrm{F}_{3}(\mathrm{h}_{2}),\mathrm{F}_{3}(\mathrm{h}_{3}),\mathrm{F}_{3}(\mathrm{h}_{4})
        \end{array}     
        \right).
\end{align}}
In the basis $\mathcal{B}$,
we give the Gram matrix of the quadratic lattice $(\mathrm{J}_{\Z},\langle~,~\rangle_{\mathrm{E}})$ in \Cref{Gram2ndModel}, \Cref{TableAppendix}.
\begin{proof}[Proof of \Cref{prop upper bound order second model}]
Each element in $\mathcal{F}_{4,\mathrm{E}}(\Z)=\aut(\mathrm{J}_{\Z},\det,\mathrm{E})$ preserves the bilinear form $\langle~,~\rangle_{\mathrm{E}}$ by the definition,
thus this finite group is a subgroup of the isometry group $\mathrm{O}(\mathrm{J}_{\Z},\langle~,~\rangle_{\mathrm{E}})$ of the quadratic lattice $(\mathrm{J}_{\Z},\langle~,~\rangle_{\mathrm{E}})$. 

The order of $\mathrm{O}(\mathrm{J}_{\Z},\langle~,~\rangle_{\mathrm{E}})$ can be determined with the help of the Plesken-Souvignier algorithm. 
Concretely, we can apply the \texttt{qfauto} function in \cite{PARI2} to the Gram matrix \Cref{Gram2ndModel} of $(\mathrm{J}_{\Z},\langle~,~\rangle_{\mathrm{E}})$,
and we find
\[|\mathrm{O}(\mathrm{J}_{\Z},\langle~,~\rangle_{\mathrm{E}})|=2^{13}\cdot 3^{5}\cdot 7^{2}\cdot 13.\]

Notice that the isometry group contains an involution $-\mathrm{id}$, 
which does not fix $\mathrm{E}$, 
thus we have
\[|\mathcal{F}_{4,\mathrm{E}}(\Z)|\leq \frac{1}{2}\left|\mathrm{O}(\mathrm{J}_{\Z},\langle~,~\rangle_{\mathrm{E}})\right|=2^{12}\cdot 3^{5}\cdot 7^{2}\cdot 13.\qedhere\]
\end{proof}
\begin{rmk}\label{rmk lattice structure of second model}
    The orthogonal complement of $\mathrm{E}$ in $(\mathrm{J}_{\Z},\langle~,~\rangle_{\mathrm{E}})$
    is a $26$-dimensional even lattice of determinant $3$ and with no roots \cite[Proposition 7.2]{Leech}.
    In Borcherds' thesis \cite[\S 5.7]{BorcherdsLattice}, he proves that a lattice satisfying these conditions is unique up to isomorphism 
    and calculates the order of its isometry group,
    giving another proof of \Cref{prop upper bound order second model}.
\end{rmk}
Furthermore,
the \texttt{qfauto} function
also give us a set of generators $\left\{-\mathrm{id},-\sigma_{1},\sigma_{2}\right\}$ of $\mathrm{O}(\mathrm{J}_{\Z},\langle~,~\rangle_{\mathrm{E}})$,
where the matrices of $\sigma_{1},\sigma_{2}$ in the basis $\mathcal{B}$ \eqref{eqn basis of second Jordan algebra}
are given in \Cref{Generators2ndModel}, \Cref{TableAppendix}.
Here we write $-\sigma_{1}$ instead of $\sigma_{1}$ 
because the second element in the result given by \cite{PARI2} sends $\mathrm{E}$ to $-\mathrm{E}$.
The isometry group $\mathrm{O}(\mathrm{J}_{\Z},\langle~,~\rangle_{\mathrm{E}})$ is the direct product of the subgroup generated by $\sigma_{1},\sigma_{2}$ and the order $2$ central subgroup $\pm \mathrm{id}$.
In the proof of \Cref{prop upper bound order second model},
we find that $\mathcal{F}_{4,\mathrm{E}}(\Z)$ is a subgroup of the group $\langle \sigma_{1},\sigma_{2}\rangle$.

In the basis $\mathcal{B}$,
the cubic form $\det$ on $\mathrm{J}_{\R}$ can be written down as a $27$-variable polynomial of degree $3$,
and we give this polynomial function as \texttt{MatDet} in our \cite{PARI2} program \cite{Codes}.
Using \cite{PARI2}, we verify that $\sigma_{1}$ and $\sigma_{2}$ both preserve the cubic form $\det$ and the element $\mathrm{E}$,
which implies the following result:
\begin{prop}\label{prop structure of second model}
    The finite groups $\mathcal{F}_{4,\mathrm{E}}(\Z)$ and $\langle \sigma_{1},\sigma_{2}\rangle$ coincide,
    and \[|\mathcal{F}_{4,\mathrm{E}}(\Z)|=2^{12}\cdot 3^{5}\cdot 7^{2}\cdot 13.\]
\end{prop}

\subsubsection{\texorpdfstring{$\mathcal{F}_{4,\mathrm{I}}(\Z)$}{}}\label{section first integral model of F4}
Now we look at the finite group $\mathcal{F}_{4,\mathrm{I}}(\Z)=\aut(\mathrm{J}_{\Z},\det,\mathrm{I})$,
and we want to prove the following proposition:
\begin{prop}\label{prop bound of the first model}
    The reductive $\Z$-model $\mathcal{F}_{4,\mathrm{I}}$ of $\grpF$ is not $\grpF(\Q)$-conjugate to $\mathcal{F}_{4,\mathrm{E}}$,
    and $|\mathcal{F}_{4,\mathrm{I}}(\Z)|\leq 2^{15}\cdot 3^{6}\cdot 5^{2}\cdot 7$.
\end{prop}
Denote the subset of $\mathrm{J}_{\Z}$ consisting of diagonal matrices by $\mathrm{D}$,
and the subset of elements whose diagonal entries are zero by $\mathrm{D}_{0}$.
The formula \eqref{eqn quadratic form on exceptional Jordan algebra} for the quadratic form $\mathrm{Q}$ on $\mathrm{J}_{\Z}$
shows that equipped with $\mathrm{Q}$ we have $\mathrm{J}_{\Z}=\mathrm{D}_{0}\oplus \mathrm{D}$ as quadratic lattices.
By \Cref{rmk quadratic structure of coxeter octonion},
the quadratic lattice $(\oct_{\Z},\mathrm{N})$ is isometric to $\lietype{E}{8}$,
thus $\mathrm{D}_{0}$ is isometric to $\lietype{E}{8}\oplus \lietype{E}{8}\oplus \lietype{E}{8}$.
On the other hand, the lattice $\mathrm{D}$ is isometric to 
\[\mathrm{I}_{3}=\Z^{3}, \mathrm{q}:(x_{1},x_{2},x_{3})\mapsto \frac{1}{2}\left(x_{1}^{2}+x_{2}^{2}+x_{3}^{2}\right).\]

Any element of $\mathcal{F}_{4,\mathrm{I}}(\Z)$ preserves the quadratic form $\mathrm{Q}$ on $\mathrm{J}_{\Z}$,
so $\mathcal{F}_{4,\mathrm{I}}(\Z)$ is a subgroup of the isometry group $\orth(\mathrm{J}_{\Z})$ of the quadratic lattice $\mathrm{J}_{\Z}$.
By the theory of root lattices,
we have 
\[\orth(\mathrm{J}_{\Z})\simeq \orth(\lietype{I}{3})\times\left(\orth(\oct_{\Z})\wr \mathrm{S}_{3}\right),\]
where $\mathrm{S}_{3}$ is the permutation group of three elements and $\wr$ stands for the wreath product.
Let $p$ be the restriction map $\mathcal{F}_{4,\mathrm{I}}(\Z)\hookrightarrow \orth(\mathrm{J}_{\Z})\twoheadrightarrow \orth(\mathrm{D}),g\mapsto g|_{\mathrm{D}}$,
where $\orth(\mathrm{D})\simeq\orth(\mathrm{I}_{3})$ is isomorphic to $\{\pm 1\}^{3}\rtimes \mathrm{S}_{3}$.

Let $\mathrm{O}(\mathrm{D}\,;\mathrm{I})$ be the group $\{\sigma\in \orth(\mathrm{D})\,|\,\sigma(\mathrm{I})=\mathrm{I}\}$,
which is isomorphic to the permutation group $\mathrm{S}_{3}$.
Since elements in $\mathcal{F}_{4,\mathrm{I}}(\Z)$ fix $\mathrm{I}$,
the image of $p$ is contained in $\orth(\mathrm{D}\,;\mathrm{I})$.
\begin{lemma}\label{lemma embed permutation group into F4}
    The image of $p$ is $\orth(\mathrm{D}\,;\mathrm{I})\simeq \mathrm{S}_{3}$.
\end{lemma}
\begin{proof}
    For an element $\sigma\in \mathrm{S}_{3}$,
    we denote by $g_{\sigma}$ the element
    \begin{align}\label{eqn realization of S3 in F4}
    [a_{1},a_{2},a_{3}\,;x_{1},x_{2},x_{3}]\mapsto [a_{\sigma^{-1}(1)},a_{\sigma^{-1}(2)},a_{\sigma^{-1}(3)}\,;\epsilon(\sigma)(x_{\sigma^{-1}(1)}),\epsilon(\sigma)(x_{\sigma^{-1}(2)}),\epsilon(\sigma)(x_{\sigma^{-1}(3)})]
    \end{align}
    in $\GL(\mathrm{J}_{\Z})$,
    where the map $\epsilon(\sigma):\oct_{\Z}\rightarrow\oct_{\Z}$ 
    is defined as identity when $\sigma$ is even, and as the conjugation when $\sigma$ is odd. 
    In this proof, we write $x^{*}:=\epsilon(\sigma)(x)$ for short.

    For any $A=[a_{1},a_{2},a_{3}\,;x_{1},x_{2},x_{3}]\in \mathrm{J}_{\Z}$,
    by the formula \eqref{eqn determinant on Jordan algebra} for the cubic form $\det$,
    we have 
    \begin{align*}
        \det\left(g_{\sigma}(A)\right)=&\prod_{i=1}^{3}a_{\sigma^{-1}(i)}+\mathrm{Tr}(x_{\sigma^{-1}(1)}^{*}x_{\sigma^{-1}(2)}^{*}x_{\sigma^{-1}(3)}^{*})-\sum_{i=1}^{3}a_{\sigma^{-1}(i)}\mathrm{N}(x_{\sigma^{-1}(i)}^{*})\\
        =&a_{1}a_{2}a_{3}+\mathrm{Tr}(x_{\sigma^{-1}(1)}^{*}x_{\sigma^{-1}(2)}^{*}x_{\sigma^{-1}(3)}^{*})-\sum_{i=1}^{3}a_{i}\mathrm{N}(x_{i}).
    \end{align*}
    The property \eqref{eqn trace-commutativity of octonion} of $\tr$ 
    implies that for any $x,y,z\in\oct_{\Z}$,
    \[\tr(xyz)=\tr(yzx)=\tr(zxy)=\tr(\overline{x}\cdot \overline{z}\cdot \overline{y})=\tr(\overline{z}\cdot\overline{y}\cdot\overline{x})=\tr(\overline{y}\cdot\overline{x}\cdot\overline{z}),\]
    which can also be stated as 
    $\mathrm{Tr}(x_{\sigma^{-1}(1)}^{*}x_{\sigma^{-1}(2)}^{*}x_{\sigma^{-1}(3)}^{*})=\mathrm{Tr}(x_{1}x_{2}x_{3})$ for any $\sigma\in \mathrm{S}_{3}$.
    Hence $\det(g_{\sigma}(A))=\det(A)$.
    Since $g_{\sigma}$ also fixes $\mathrm{I}$,
    it is an element in $\mathcal{F}_{4,\mathrm{I}}(\Z)$
    and its restriction $p(g_{\sigma})\in \orth(\mathrm{D}\,;\mathrm{I})\simeq \mathrm{S}_{3}$ is $\sigma$,
    thus $\mathrm{Im}(p)=\orth(\mathrm{D}\,;\mathrm{I})$.
\end{proof}
Let $\mathscr{D}$ be the kernel of $p$,
then we have a short exact sequence of finite groups:
\begin{align}\label{eqn short exact sequence for the first model}
    1\rightarrow \mathscr{D}\rightarrow \mathcal{F}_{4,\mathrm{I}}(\Z)\rightarrow \orth(\mathrm{D}\,;\mathrm{I})\simeq \mathrm{S}_{3}\rightarrow 1.
\end{align}
\begin{lemma}\label{lemma splitting of the exact sequence for the first model}
    The map $\kappa:\mathrm{S}_{3}\rightarrow \mathcal{F}_{4,\mathrm{I}}(\Z),\sigma\mapsto g_{\sigma}$ defined in \eqref{eqn realization of S3 in F4}
    gives a splitting of the short exact sequence \eqref{eqn short exact sequence for the first model}.
\end{lemma}
\begin{proof}
    It suffices to show that $\sigma\mapsto g_{\sigma}$ is a group homomorphism.
    For $\sigma,\tau\in \mathrm{S}_{3}$,
    we have 
    \begin{gather*}
    \begin{aligned}
        &g_{\tau}g_{\sigma}\left([a_{1},a_{2},a_{3}\,;x_{1},x_{2},x_{3}]\right)\\
        =&g_{\tau}\left([a_{\sigma^{-1}(1)},a_{\sigma^{-1}(2)},a_{\sigma^{-1}(3)}\,;\epsilon(\sigma)(x_{\sigma^{-1}(1)}),\epsilon(\sigma)(x_{\sigma^{-1}(2)}),\epsilon(\sigma)(x_{\sigma^{-1}(3)})]\right)\\
        =&\left[\begin{array}{c}
            a_{(\tau\sigma)^{-1}(1)},a_{(\tau\sigma)^{-1}(2)},a_{(\tau\sigma)^{-1}(3)}\,;\\
            \epsilon(\tau)\epsilon(\sigma)(x_{(\tau\sigma)^{-1}(1)}),\epsilon(\tau)\epsilon(\sigma)(x_{(\tau\sigma)^{-1}(2)}),\epsilon(\tau)\epsilon(\sigma)(x_{(\tau\sigma)^{-1}(3)})
        \end{array}\right].
    \end{aligned}
\end{gather*}
It can be easily seen that the map $\epsilon:\mathrm{S}_{3}\rightarrow \GL(\oct_{\Z})$ is a group homomorphism,
thus $g_{\tau}g_{\sigma}=g_{\tau\sigma}$ and $\sigma\mapsto g_{\sigma}$ is also a group homomorphism.
\end{proof}
This lemma tells us $\mathcal{F}_{4,\mathrm{I}}(\Z)=\mathscr{D}\rtimes \kappa(\mathrm{S}_{3})$
and $|\mathcal{F}_{4,\mathrm{I}}(\Z)|=3!\cdot |\mathscr{D}|$.
Now we study the structure of $\mathscr{D}$.
\begin{lemma}\label{lemma structure of group D}
    The group $\mathscr{D}$ is isomorphic to the group 
    \[\widetilde{\mathrm{SO}(\oct_{\Z})}:=\left\{(\alpha,\beta,\gamma)\in \sorth(\oct_{\Z})^{3}\,\middle\vert\,\overline{\alpha(x)\beta(y)}=\gamma(\overline{xy}),\forall x,y\in\oct_{\Z}\right\}.\]
\end{lemma}
\begin{proof}
    Fix $g\in \mathscr{D}$ and $x\in \oct_{\Z}$,
    we define $y,z,w\in \oct_{\Z}$ by the formula 
    \[g.\left(\begin{matrix}
        0&0&0\\
        0&0&x\\
        0&\overline{x}&0
    \end{matrix}\right)=
    \left(\begin{matrix}
        0&w&\overline{z}\\
        \overline{w}&0&y\\
        z&\overline{y}&0
    \end{matrix}\right).\]
    Since $g\in\mathcal{F}_{4,\mathrm{I}}(\Z)\subseteq \grpF(\Q)$ preserves the Jordan multiplication $\circ$, 
    we have 
    \begin{align*}
        \mathrm{N}(x)\left(\begin{matrix}
            0&0&0\\
            0&1&0\\
            0&0&1
        \end{matrix}\right)&=g.\left(\left(\begin{matrix}
            0&0&0\\
            0&0&x\\
            0&\overline{x}&0
        \end{matrix}\right)\circ\left(\begin{matrix}
            0&0&0\\
            0&0&x\\
            0&\overline{x}&0
        \end{matrix}\right)\right)\\
        &=\left(\begin{matrix}
            0&w&\overline{z}\\
            \overline{w}&0&y\\
            z&\overline{y}&0
        \end{matrix}\right)\circ \left(\begin{matrix}
            0&w&\overline{z}\\
            \overline{w}&0&y\\
            z&\overline{y}&0
        \end{matrix}\right)\\
        &=\left(\begin{matrix}
            \mathrm{N}(z)+\mathrm{N}(w)& \overline{yz}& wy\\
            yz & \mathrm{N}(w)+\mathrm{N}(y)&\overline{zw}\\
            \overline{wy}& zw & \mathrm{N}(y)+\mathrm{N}(z)
        \end{matrix}\right),
    \end{align*}
    which implies that $z=w=0$ and $\mathrm{N}(y)=\mathrm{N}(x)$.
    This gives us a homomorphism $g\mapsto \alpha_{g}$
    from $\mathscr{D}$ to $\orth(\oct_{\Z})$
    such that $g[0,0,0\,;x,0,0]=[0,0,0\,;\alpha_{g}(x),0,0]$ for $\in\oct_{\Z}$.
    
    Symmetrically, we also get $\beta_{g},\gamma_{g}\in\orth(\oct_{\Z})$
    such that 
    \[g[0,0,0\,;x,y,z]=[0,0,0\,;\alpha_{g}(x),\beta_{g}(x),\gamma_{g}(x)] \text{ for all }x,y,z\in \oct_{\Z}.\]
    Taking determinants of both sides,
    we get
    \[\tr(xyz)=\tr(\alpha_{g}(x)\beta_{g}(y)\gamma_{g}(z))\text{ for all }x,y,z\in\oct_{\Z}.\]
    This is equivalent to 
    $\langle \overline{\alpha_{g}(x)\beta_{g}(y)},\gamma_{g}(z)\rangle=\langle \overline{xy},z\rangle$.
    Since $\langle \overline{xy},z\rangle=\langle \gamma_{g}(\overline{xy}),\gamma_{g}(z)\rangle$,
    we have 
    \[\langle \overline{\alpha_{g}(x)\beta_{g}(y)}-\gamma_{g}(\overline{xy}),\gamma_{g}(z)\rangle=0\]
    for any $z\in \oct_{\Z}$.
    The bilinear form $\langle~,~\rangle$ is non-degenerate,
    so $\overline{\alpha_{g}(x)\beta_{g}(y)}=\gamma_{g}(\overline{xy})$ holds
    for any $x,y\in\oct_{\Z}$.
    By \cite[Lemma 1.14.4]{Yokota2009ExceptionalLG},
    we have $\alpha_{g},\beta_{g},\gamma_{g}\in \sorth(\oct_{\Z})$.
    
    Now we have obtained an injective homomorphism $\mathscr{D}\rightarrow \widetilde{\sorth(\oct_{\Z})}$.
    Conversely, 
    by the definition of the multiplication $\circ$ and the condition on $(\alpha,\beta,\gamma)\in\widetilde{\sorth(\oct_{\Z})}$,
    the morphism 
    \[[a,b,c\,;x,y,z]\mapsto [a,b,c\,;\alpha(x),\beta(y),\gamma(z)]\]
    lies in $\mathscr{D}$,
    thus $\mathscr{D}\simeq \widetilde{\sorth(\oct_{\Z})}$.
\end{proof}
Let $\varphi:\widetilde{\sorth(\oct_{\Z})}\rightarrow \sorth(\oct_{\Z})$ be the homomorphism sending a triple $(\alpha,\beta,\gamma)\in\widetilde{\sorth(\oct_{\Z})}$
to its third entry $\gamma\in \sorth(\oct_{\Z})$.
\begin{proof}[Proof of \Cref{prop bound of the first model}]
    For the bound on $|\mathcal{F}_{4,\mathrm{I}}(\Z)|$,
    it suffices to prove \[|\widetilde{\sorth(\oct_{\Z})}|\leq 2^{14}\cdot 3^{5}\cdot 5^{2}\cdot 7.\]
    By \cite[Theorem 1.14.2]{Yokota2009ExceptionalLG}, $\ker \varphi=\{(\mathrm{id},\mathrm{id},\mathrm{id}),(-\mathrm{id},-\mathrm{id},\mathrm{id})\}$.
    As a consequence, we have 
    \[|\widetilde{\sorth(\oct_{\Z})}|\leq 2\cdot|\sorth(\oct_{\Z})|=|\orth(\oct_{\Z})|=|\mathrm{W}(\lietype{E}{8})|=2^{14}\cdot 3^{5}\cdot 5^{2}\cdot 7,\]
    which gives us the desired upper bound for $|\mathcal{F}_{4,\mathrm{I}}(\Z)|$.

    Suppose that the reductive $\Z$-model $\mathcal{F}_{4,\mathrm{I}}$ of $\grpF$ is $\grpF(\Q)$-conjugate to $\mathcal{F}_{4,\mathrm{E}}$,
    then their $\Z$-points have the same order as finite groups.    By \Cref{prop structure of second model},
    \[|\mathcal{F}_{4,\mathrm{I}}(\Z)|=|\mathcal{F}_{4,\mathrm{E}}(\Z)|=2^{12}\cdot 3^{5}\cdot 7^{2}\cdot 13,\]
    thus by Cauchy's theorem $\mathcal{F}_{4,\mathrm{I}}(\Z)$ contains an element $x$ of order $13$.
    
    Using the short exact sequence \eqref{eqn short exact sequence for the first model},
    the element $x$ lies in the kernel $\ker(p)=\mathscr{D}$.
    Since the order of $\sorth(\oct_{\Z})$ is not divided by $13$,
    the element $\varphi(x)$ must be the trivial element.
    This implies that $x\in\ker(\varphi)\simeq \Z/2\Z$, which is a contradiction. 
\end{proof}
Now \Cref{prop structure of second model} and \Cref{prop bound of the first model} together imply \Cref{prop exactly two reductive models},
and as a corollary the equality in the upper bound in \Cref{prop bound of the first model} holds:
\begin{cor}\label{cor structure integral points of models}
        The finite group $\mathcal{F}_{4,\mathrm{I}}(\Z)$ has order $2^{15}\cdot 3^{6}\cdot 5^{2}\cdot 7$,
        and $\varphi$ is surjective.
\end{cor}

\section{Dimensions of spaces of invariants for \texorpdfstring{$\lietype{F}{4}$}{}}\label{section dimension of invariant subspaces of F4 representations}
 For a finite subgroup $\Gamma$ and an irreducible representation $U$ of the compact Lie group $\lietype{F}{4}$,
an interesting problem is to compute the dimension of the space of invariants $U^{\Gamma}$.
In this section we will give an algorithm to compute $\dim U^{\Gamma}$ 
for $\Gamma=\mathcal{F}_{4,\mathrm{I}}(\Z)$ or $\mathcal{F}_{4,\mathrm{E}}(\Z)$.
These dimensions will play an important role in our computation of spaces of automorphic forms in \Cref{section automorphic representations of anisotropic groups}.
The code of the computations in this section can be found in \cite{Codes}.

\subsection{Ideas and obstructions}\label{section ideas and obstructions of the computation}
 By the highest weight theory, 
the isomorphism classes of irreducible $\C$-representations of the compact Lie group $\lietype{F}{4}$ 
are in natural bijection with dominant weights of the irreducible root system $\mathrm{F}_{4}$.
Using notations in \cite[\S VI.4.9]{Lie}, we denote the weight $\lambda_{1}\varpi _{1}+\lambda_{2}\varpi_{2}+\lambda_{3}\varpi_{3}+\lambda_{4}\varpi_{4}$ by $\lambda=(\lambda_{1},\lambda_{2},\lambda_{3},\lambda_{4})$, 
where $\varpi_{1},\varpi_{2},\varpi_{3},\varpi_{4}$ are the four fundamental weights of $\mathrm{F}_{4}$.
Let $\mathrm{V}_{\lambda}$ be a representative of the isomorphism class of irreducible representations of $\lietype{F}{4}$ with highest weight $\lambda$. 
From now on we call $\mathrm{V}_{\lambda}$ \emph{the} irreducible representation of $\lietype{F}{4}$ with highest weight $\lambda$ for short.

The starting point of the computation of $\dim \mathrm{V}_{\lambda}^{\Gamma}$ for some finite subgroup $\Gamma$ of $\lietype{F}{4}$ is the following classic lemma:
\begin{lemma}\label{lemma dimension of invariant space in terms of conj classes}
    For a finite subgroup $\Gamma\subset \lietype{F}{4}$, we have
\[\dim \mathrm{V}_{\lambda}^{\Gamma}=\frac{1}{|\Gamma|}\sum_{\gamma\in\Gamma}\mathrm{Tr}|_{\mathrm{V}_{\lambda}}(\gamma)=\frac{1}{|\Gamma|}\sum_{c\in\mathrm{Conj}(\Gamma)}\mathrm{Tr}|_{\mathrm{V}_{\lambda}}(c)\cdot |c|,\]
where $\mathrm{Conj}(\Gamma)$ is the set of conjugacy classes of $\Gamma$ and $|c|$ denotes the cardinality of $c$.
\end{lemma}
%
Because of this lemma, it is enough to solve the following two problems to compute $\dim \mathrm{V}_{\lambda}^{\Gamma}$:
\begin{enumerate}[label= (\roman*)]
    \item Find all conjugacy classes of $\Gamma$, and choose a representative in a fixed maximal torus $T\subset \lietype{F}{4}$ for each conjugacy class;
    \item For an element $t\in T$, compute its trace $\mathrm{Tr}|_{\mathrm{V}_{\lambda}}(t)$.
\end{enumerate}
Problem (ii) can be dealt with using the following \emph{degenerate Weyl character formula}:
\begin{prop}\label{prop degenerate Weyl character formula}\cite[Proposition 2.1]{ChenevierRenard}
    Let $G$ be a connected compact Lie group, $T$ a maximal torus, $X=\mathrm{X}^{*}(T)$ the character group of $T$,
    and $\Phi$ the root system of $(G,T)$ with Weyl group $W$.
    Choose a system of positive roots $\Phi^{+}\subset \Phi$ with base $\Delta$ and also fix a $W$-invariant inner product $(~,~)$ on $X\otimes_{\Z} \R$.
    Let $\lambda$ be a dominant weight in $X$ and $t$ an element in $T$.
    Denote the connected component $\mathrm{C}_{G}(t)^{\circ}$ of the centralizer of $t$ by $M$.
    Set $\Phi_{M}^{+}=\Phi(M,T)\cap \Phi^{+}$ and $W^{M}=\{w\in W: w^{-1}\Phi_{M}^{+}\subset \Phi^{+}\}$.
    Let $\rho$ and $\rho_{M}$ be the half-sum of the elements of $\Phi^{+}$ and $\Phi_{M}^{+}$ respectively.
    We have:
    \begin{equation}\label{eqn degenerate Weyl}
        \mathrm{Tr}|_{\mathrm{V}_{\lambda}}(t)=\frac{\sum_{w\in W^{M}}\varepsilon(w)t^{w(\lambda+\rho)-\rho}\cdot \prod_{\alpha\in\Phi_{M}^{+}}\frac{(\alpha,w(\lambda+\rho))}{(\alpha,\rho_{M})}}{\prod_{\alpha\in \Phi^{+}\backslash \Phi_{M}^{+}}(1-t^{-\alpha})},
    \end{equation}
    where $\varepsilon:W\rightarrow\{\pm 1\}$ is the signature and $t^{x}$ denotes $x(t)$ for convenience.
\end{prop}
Using this approach, 
problem (i) is thus the main difficulty for our computation, 
and we will solve it in the following subsections.

\subsection{Generators of \texorpdfstring{$\mathcal{F}_{4,\mathrm{I}}(\Z)$}{} and \texorpdfstring{$\mathcal{F}_{4,\mathrm{E}}(\Z)$}{}}\label{section generators of integral points} 
The finite groups $\Gamma$ we are interested in 
are $\mathcal{F}_{4,\mathrm{I}}(\Z)$ and $\mathcal{F}_{4,\mathrm{E}}(\Z)$.
To find all their conjugacy classes,   
we first determine generators of these groups in this subsection.

In the end of \Cref{section second integral model of F4},
we have already showed that the group $\mathcal{F}_{4,\mathrm{E}}(\Z)$ is generated by two elements $\sigma_{1},\sigma_{2}$.
Their matrices in the basis $\mathcal{B}$,
given in \eqref{eqn basis of second Jordan algebra},
are written down in \Cref{Generators2ndModel}, \Cref{TableAppendix}.

Based on \Cref{cor structure integral points of models}, 
we have $\mathcal{F}_{4,\mathrm{I}}(\Z)=\mathscr{D}\rtimes \kappa(\mathrm{S}_{3})$, 
where $\kappa:\mathrm{S}_{3}\rightarrow \mathcal{F}_{4,\mathrm{I}}(\Z)$ is the morphism defined in \eqref{eqn realization of S3 in F4}. 
The group $\mathscr{D}$ is isomorphic to the group $\widetilde{\sorth(\oct_{\Z})}$,
which is a double cover of $\sorth(\oct_{\Z})$ by \Cref{cor structure integral points of models}.
Therefore it suffices to find generators of $\mathscr{D}$.

Since $\orth(\oct_{\Z})\simeq \orth(\lietype{E}{8})$ is equal to the Weyl group of $\lietype{E}{8}$, 
we can take the following set of generators for $\sorth(\oct_{\Z})$:
\[\left\{\mathrm{ref}(\alpha)\circ\mathrm{ref}(1)\,\middle\vert\,\alpha\in\oct_{\Z},\mathrm{N}(\alpha)=1\right\},\]
where for a root $\alpha$ in $\oct_{\Z}$ the reflection $\mathrm{ref}(\alpha)$ is defined as 
\[\mathrm{ref}(\alpha)(x):=x-\langle x,\alpha\rangle \alpha.\]
For a root $\alpha\in \oct_{\Z}$,
let $\mathrm{L}_{\alpha}$ (resp. $\mathrm{R}_{\alpha}$) be the left (resp. right) multiplication on $\oct_{\Z}$ by $\alpha$,
and define $\mathrm{B}_{\alpha}:=\mathrm{L}_{\alpha}\circ \mathrm{R}_{\alpha}=\mathrm{R}_{\alpha}\circ \mathrm{L}_{\alpha}$. 
These elements are contained in $\sorth(\oct_{\Z})$.
Notice that for any root $\alpha\in\oct_{\Z}$ and any vector $x\in \oct_{\R}$,
\begin{align*}
    \mathrm{ref}(\alpha)\circ \mathrm{ref}(1)(x)=\mathrm{ref}(\alpha)(-\overline{x})=-\overline{x}+\langle \overline{x},\alpha\rangle\alpha=-\overline{x}+(\overline{x}\overline{\alpha})\alpha+(\alpha x)\alpha=\mathrm{B}_{\alpha}(x),
\end{align*}
thus $\mathrm{ref}(\alpha)\circ\mathrm{ref}(1)=\mathrm{B}_{\alpha}$.

\begin{lemma}\label{lemma basic isotopy triple}
    For any root $\alpha\in\oct_{\Z}$,
    the triple $(\mathrm{L}_{\overline{\alpha}},\mathrm{R}_{\overline{\alpha}},\mathrm{B}_{\alpha})$ is an element in $\widetilde{\sorth(\oct_{\Z})}$.
\end{lemma}
\begin{proof}
    For any $x,y\in\oct_{\Z}$, $\overline{\mathrm{L}_{\overline{\alpha}}(x)\mathrm{R}_{\overline{\alpha}}(y)}=\overline{(\overline{\alpha} x)(y \overline{\alpha})}$.
    By \emph{Moufang laws} \cite[\S 6.5]{ConwayOct},
    \[(\overline{\alpha} x)(y\overline{\alpha})=(\overline{\alpha}(xy))\overline{\alpha}=\mathrm{B}_{\overline{\alpha}}(xy),\]
    thus $\overline{\mathrm{L}_{\overline{\alpha}}(x)\mathrm{R}_{\overline{\alpha}}(y)}=\overline{\mathrm{B}_{\overline{\alpha}}(xy)}=\mathrm{B}_{\alpha}(\overline{xy})$.
\end{proof}
By this lemma,
we can take 
\[\left\{(\mathrm{L}_{\overline{\alpha}},\mathrm{R}_{\overline{\alpha}},\mathrm{B}_{\alpha})\,\middle\vert\,\alpha\in\oct_{\Z},\mathrm{N}(\alpha)=1\right\}\cup \{(-\mathrm{id},-\mathrm{id},\mathrm{id})\}\]
as generators of $\mathscr{D}$.
Together with a set of generators of $\kappa(\mathrm{S}_{3})$
we have obtained generators of $\mathcal{F}_{4,\mathrm{I}}(\Z)$.

\subsection{Enumeration of conjugacy classes}\label{section enumeration conj classes}
 Now with generators of $\mathcal{F}_{4,\mathrm{I}}(\Z)$ and $\mathcal{F}_{4,\mathrm{E}}(\Z)$,
we can start to enumerate their conjugacy classes.
The \texttt{ConjugationClasses} function in \cite{GAP4} can assist us in enumerating the conjugacy classes of subgroups of permutation groups.
Therefore it is enough to realize these two finite groups as permutation groups.

For $\mathcal{F}_{4,\mathrm{I}}(\Z)$, 
we consider its action on the set of vectors $v\in\jord_{\Z}$ with $\mathrm{B}_{\mathrm{Q}}(v,v)\leq 2$.
The function \texttt{qfminim} in \cite{PARI2} can list all these vectors in the basis $\mathcal{B}$.
There are $738$ such vectors and they span the vector space $\mathrm{J}_{\R}$,
so the action of $\mathcal{F}_{4,\mathrm{I}}(\Z)$ on this set is faithful,
which gives us an embedding $\mathcal{F}_{4,\mathrm{I}}(\Z)\hookrightarrow \mathrm{S}_{738}$.
We can thus use this embedding to obtain a set of representatives of conjugacy classes of $\mathcal{F}_{4,\mathrm{I}}(\Z)$ via the help of \cite{GAP4}.

For the other group $\mathcal{F}_{4,\mathrm{E}}(\Z)$ we use a similar strategy.
As mentioned in \Cref{rmk lattice structure of second model}, the quadratic lattice $(\mathrm{J}_{\Z},\langle~,~\rangle_{\mathrm{E}})$ has no roots, 
so we consider the set of $v\in \mathrm{J}_{\Z}$ such that $\langle v,v\rangle_{\mathrm{E}}=3$,
which has cardinality $1640$ and generates $\mathrm{J}_{\R}$.
This gives an embedding $\mathcal{F}_{4,\mathrm{E}}(\Z)\hookrightarrow \mathrm{S}_{1640}$,
then we can use \cite{GAP4}.

Here we present the results, and all the codes are available in \cite{Codes}.
\begin{prop}\label{prop result conj classes}
    There are $113$ conjugacy classes in $\mathcal{F}_{4,\mathrm{I}}(\Z)$, 
    while $\mathcal{F}_{4,\mathrm{E}}(\Z)$ has $49$ conjugacy classes.
\end{prop}
Furthermore, \cite{GAP4} gives the size of each conjugacy class $c$,
and selects a representative for $c$ in the form of permutation.
We rewrite these representatives as matrices in the basis $\mathcal{B}$.

\subsection{Kac coordinates}\label{section Kac coordinates}
 In the previous subsection, 
for $\Gamma=\mathcal{F}_{4,\mathrm{I}}(\Z)$ or $\mathcal{F}_{4,\mathrm{E}}(\Z)$,
we obtained a list of its conjugacy classes and a representative element $g_{c}\in \Gamma$ for each conjugacy class $c$.

However, the representative $g_{c}$ may not be contained in the fixed maximal torus in \Cref{prop degenerate Weyl character formula}.
Notice that in the computation of the trace of $g_{c}$ for a $\Gamma$-conjugacy class $c$,
what really matters is the $\lietype{F}{4}$-conjugacy class containing $c$.
Furthermore, since $c$ is included in the finite group $\Gamma$, 
the $\lietype{F}{4}$-conjugacy class containing it must be torsion. 

In \cite{Reeder2010TorsionAO}, it is shown that we can choose a representative for a torsion $\lietype{F}{4}$-conjugacy class in a fixed maximal torus using its \emph{Kac coordinates}.
Here we provide a brief review, and more details can be found in Reeder's paper.

Let $G$ be a simply-connected simple compact Lie group,
$T$ a fixed maximal torus,
$X:=\mathrm{X}^{*}(T)$ and $Y:=\mathrm{X}_{*}(T)$ the groups of characters and cocharacters respectively,
and $\Phi$ the root system of $(G,T)$.
Denote the natural pairing $X\times Y\rightarrow\Z$ by $\langle~,~\rangle$.
Let $\Delta=\{\alpha_{1},\ldots,\alpha_{r}\}$ be a set of simple roots of $\Phi$,
and $\{\check{\varpi}_{1},\ldots,\check{\varpi}_{r}\}$ its dual basis in $Y$,
i.e. $\langle\alpha_{i},\check{\varpi}_{j}\rangle=\delta_{ij}$.

We have a surjective \emph{exponential map} $\exp:Y\otimes_{\Z}\R\rightarrow T$ 
determined uniquely by the property
\[\alpha\left(\exp(y)\right)=e^{2\pi i\langle\alpha,y\rangle},\forall\alpha\in X,y\in Y\otimes_{\Z}\R.\]
and $Y$ is the kernel of this exponential map. 
This induces an isomorphism $(Y\otimes_{\Z}\R)/Y\simeq T$.

Let $\widetilde{\alpha}_{0}=\sum\limits_{i=1}^{r}a_{i}\alpha_{i}$ be the highest root 
with respect to the choice of simple roots $\Delta$.
Set $\alpha_{0}$ to be the affine function $x\mapsto \langle \alpha_{0},x\rangle:=1-\langle\widetilde{\alpha}_{0},x\rangle$ for any $x\in Y\otimes_{\Z}\R$,
and $a_{0}=1, \check{\varpi}_{0}=0$.
Now we have $\sum\limits_{i=0}^{r}a_{i}\alpha_{i}=1$. 
The \emph{alcove} determined by $\Delta$ is the intersection of half-spaces:
\[C=\left\{x\in Y\otimes_{\Z}\R\,\middle\vert\,\langle \alpha_{i},x\rangle>0, \forall i=0,1,\ldots,r\right\},\]
or
\[\overline{C}=\left\{\sum_{i=0}^{r}x_{i}\check{\varpi}_{i}\,\middle\vert\, \sum_{i=0}^{r}a_{i}x_{i}=1,x_{i}\geq 0,\forall i=0,1,\ldots,r\right\}.\]

Each torsion element $s\in G$ is conjugate to $\exp(x)$ for a unique $x\in\overline{C}\cap (Y\otimes_{\Z}\Q)$
since the group $G$ is simply-connected.
Let $m$ be the order of $s$,
thus \[x=\frac{1}{m}\sum_{i=1}^{r}s_{i}\check{\varpi}_{i}\]
for some non-negative integers $s_{1},\ldots,s_{r}$ satisfying $\gcd\{m,s_{1},\ldots,s_{r}\}=1$.

Since $x\in\overline{C}$, we set $s_{0}:=m-\sum\limits_{i=1}^{r}a_{i}s_{i}\geq 0$.
Now the non-negative integers $s_{0},s_{1},\ldots,s_{r}$ satisfy $\gcd\{s_{0},\ldots,s_{r}\}=1$
and the equation
\[\sum_{i=0}^{r}a_{i}s_{i}=m\text{ with }a_{0}=1.\]
The coordinates $(s_{0},s_{1},\ldots,s_{r})$ are called the \emph{Kac coordinates of $s$},
which are uniquely determined by the $G$-conjugacy class of $s$.

In our case, the compact group $\lietype{F}{4}$ is simply-connected and the highest root $\widetilde{\alpha}_{0}=2\alpha_{1}+3\alpha_{2}+4\alpha_{3}+2\alpha_{4}$.
Here $\alpha_{1},\alpha_{2},\alpha_{3},\alpha_{4}$ are still chosen as in \cite[\S VI.4]{Lie}.
In conclusion, we have:
\begin{prop}\label{prop Kac coordinates determine conj classes}
    Let $T$ be a fixed maximal torus of $\lietype{F}{4}$.
    Any element of order $m$ in $\lietype{F}{4}$
    is conjugate to a unique element $\exp(\frac{\sum_{i=1}^{4}s_{i}\varpi_{i}}{m})$
    for some non-negative integers $s_{1},s_{2},s_{3},s_{4}$ arising from a $5$-tuple $(s_{0},s_{1},s_{2},s_{3},s_{4})$ in
    \begin{align}\label{eqn equation for Kac coordinates}
        \left\{(x_{0},\ldots,x_{4})\in\mathbb{N}^{5}\,\middle\vert\, x_{0}+2x_{1}+3x_{2}+4x_{3}+2x_{4}=m,\gcd\{x_{0},\ldots,x_{4}\}=1\right\}.
    \end{align}
\end{prop}
By solving the equation in \eqref{eqn equation for Kac coordinates},
we enumerate all the torsion $\lietype{F}{4}$-conjugacy classes of order $m$.

\subsection{Comparison of conjugacy classes}\label{section comparison of conj classes}
 Now we can enumerate $\lietype{F}{4}$-conjugacy classes of a given order,
but there are more constraints on the $\lietype{F}{4}$-conjugacy classes containing $\Gamma$-conjugacy classes obtained in \Cref{section enumeration conj classes}.
So we define the following class of $\lietype{F}{4}$-conjugacy classes:
\begin{defi}\label{def rational conjugacy classes}
    Let $c$ be an $\lietype{F}{4}$-conjugacy class,
    and we say that $c$ is a \emph{rational conjugacy class} if it satisfies:
    \begin{itemize}
        \item its trace $\tr(c)|_{\mathfrak{f}_{4}}$ on the adjoint representation $\mathfrak{f}_{4}$ of $\lietype{F}{4}$ is a rational number;
        \item its characteristic polynomial $\mathrm{P}_{c}(X):=\det(X\cdot \mathrm{id}-g|_{\mathrm{J}_{\C}})$ on $\mathrm{J}_{\C}:=\mathrm{J}_{\R}\otimes_{\R}\C$, $g\in\lietype{F}{4}$ being a representative of $c$, has rational coefficients.
    \end{itemize}
\end{defi}
\begin{rmk}\label{rmk equivalent def of rational conjugacy}
    This definition is equivalent to the trace $\tr(c)|_{V}$ being rational for any irreducible representation of $\lietype{F}{4}$.
\end{rmk}
For $\Gamma=\mathcal{F}_{4,\mathrm{I}}(\Z)$ or $\mathcal{F}_{4,\mathrm{E}}(\Z)$,
since $\Gamma$ is a subgroup of $\GL(\mathrm{J}_{\Z})$,
the $\lietype{F}{4}$-conjugacy class containing a $\Gamma$-conjugacy class of $\Gamma$
must be rational in the sense of \Cref{def rational conjugacy classes}.

Our strategy in this subsection is:
\begin{enumerate}[label=(\arabic*)]
    \item find all rational torsion $\lietype{F}{4}$-conjugacy classes, and for each of them choose a representative in the maximal torus $T$ fixed before in \Cref{section Kac coordinates};
    \item determine which $\lietype{F}{4}$-conjugacy class contains a given $\Gamma$-conjugacy class by comparing their traces and characteristic polynomials.
\end{enumerate}

Before explaining the algorithm for step $(1)$, we state the following lemma:
\begin{lemma}\label{lemma bounds on the order of rational conjugacy class}
    If $m$ is the order of an element in $\lietype{F}{4}$
    whose characteristic polynomial on $\mathrm{J}_{\C}$ has rational coefficients,
    then $m=66,70,72,78,84$ or $90$, or $m\leq 60$.
\end{lemma}
\begin{proof}
    As a representation of $\lietype{F}{4}$,
    $\mathrm{J}_{\C}$ is isomorphic to $\mathrm{V}_{\varpi_{4}}\oplus \C$,
    where $\C$ stands for the trivial representation.
    Since the zero weight appears twice in the weights of $\mathrm{V}_{\varpi_{4}}$,
    the characteristic polynomial is divisible by $(X-1)^{3}$.
    On the other hand, the roots of this polynomial contain a primitive $m$th root of unity,
    thus the polynomial is also divisible by the $m$th cyclotomic polynomial.
    Hence we have $\varphi(m)\leq 24$, where $\varphi$ denotes the Euler function.
    This implies $m\leq 60$,
    or $m=66,70,72,78,84$ or $90$.
\end{proof}
With the help of \cite{PARI2}, 
we enumerate all the Kac coordinates $s=(s_{0},s_{1},s_{2},s_{3},s_{4})$ satisfying the conditions in \eqref{eqn equation for Kac coordinates}
for each integer $m$ in 
\[\{n\leq 60\,|\,\varphi(n)\leq 24\}\cup\{66,70,72,78,84,90\}.\] 
For each such $s$,
we compute the trace on $\mathfrak{f}_{4}$ and the characteristic polynomial on $\mathrm{J}_{\C}$ 
of the corresponding element $t=\exp(\frac{\sum_{i=1}^{4}s_{i}\varpi_{i}}{m})\in T$.
Using this algorithm, 
we get the Kac coordinates of all rational torsion $\lietype{F}{4}$-conjugacy classes. 
\begin{prop}\label{prop enumeration of rational conjugacy classes}
    There are exactly $102$ rational torsion conjugacy classes in $\lietype{F}{4}$,
    whose Kac coordinates are listed in \Cref{Table all rational Kac}.
\end{prop} 
Our result coincides with \cite[Table 9.1]{PadowitzThesis}.
In \Cref{Table all rational Kac}, 
we also list the invariants defined below for all rational torsion $\lietype{F}{4}$-conjugacy class.

For a representative $g\in \lietype{F}{4}$ of a rational torsion conjugacy class $c$,
we can compute its characteristic polynomial on $\mathrm{J}_{\C}$:
\[\mathrm{P}_{g}(X)=\det\left(X\cdot \mathrm{id}-g|_{\mathrm{J}_{\C}}\right)=\sum_{i=0}^{27}(-1)^{i+1}a_{i}(g)X^{i}.\]
Now we assign to $g$ a quadruple
\[\mathrm{i}(g):=\left(a_{26}(g),a_{25}(g),a_{24}(g),\tr(\mathrm{Ad}(g)|_{\mathfrak{f}_{4}})\right),\] 
and set $\mathrm{i}(c):=\mathrm{i}(g)$.
\begin{cor}\label{cor criterion two conjugacy classes}
    Let $g_{1},g_{2}$ be two elements in either $\mathcal{F}_{4,\mathrm{I}}(\Z)$ or $\mathcal{F}_{4,\mathrm{E}}(\Z)$,
    then $g_{1}$ and $g_{2}$ are conjugate in $\lietype{F}{4}$ if and only if $\mathrm{i}(g_{1})=\mathrm{i}(g_{2})$.
\end{cor}
\begin{proof}
    This follows from \Cref{Table all rational Kac}.
    For each rational torsion conjugacy class $c$,
    we list its order $o(c)$ and the associated quadruple $i(c)$.
    We observe that two different classes $c$ have different $i(c)$.
\end{proof}
\begin{rmk}\label{rmk char poly not enough for conjugacy}
    There exist examples of two different rational torsion conjugacy classes in $\lietype{F}{4}$ whose characteristic polynomials on $\jord_{\C}$ are the same.
    For instance, the order $12$ conjugacy classes $c_{1}$ and $c_{2}$ represented by the Kac coordinates $(1,1,1,1,1)$ and $(2,1,0,1,2)$ respectively
    share the same characteristic polynomial on $\jord_{\C}$:
    \[X^{27}-X^{24}-2X^{15}+2X^{12}+X^{3}-1.\]
    However, the trace of $c_{1}$ on $\mathfrak{f}_{4}$ is $0$,
    while that of $c_{2}$ is $3$.
    This shows that the $26$-dimensional irreducible representation of $\lietype{F}{4}$ is not ``excellent'' in the sense of Padowitz.
    It is also observed in Padowitz's table \cite[Table 9.1]{PadowitzThesis}
    that the motives attached to the centralizers of these two conjugacy classes,
    in the sense of Gross, 
    are different.
\end{rmk}
Now we explain our algorithm for step $(2)$.
For each $\Gamma$-conjugacy class $c$ and its representative $g_{c}$ chosen in \Cref{section enumeration conj classes},
we compute the quadruple $\mathrm{i}(g_{c})$ and compare it with \Cref{Table all rational Kac}.
By \Cref{cor criterion two conjugacy classes} we can determine the $\lietype{F}{4}$-conjugacy class containing $c$.
In \Cref{KacTable} we list all the Kac coordinates $s$ whose corresponding rational conjugacy class $c_{s}$ in $\lietype{F}{4}$ 
satisfies that $c_{s}\cap \mathcal{F}_{4,\mathrm{I}}(\Z)$ or $c_{s}\cap \mathcal{F}_{4,\mathrm{E}}(\Z)$ is non-empty,
as well as the cardinalities of intersections $n_{1}(s)=|c_{s}\cap \mathcal{F}_{4,\mathrm{I}}(\Z)|$ and $n_{2}(s)=|c_{s}\cap \mathcal{F}_{4,\mathrm{E}}(\Z)|$.

\subsection{The formula for \texorpdfstring{$\dim \mathrm{V}_{\lambda}^{\Gamma}$}{}}\label{section final formula for dimensions}
 Now we can deduce the formula for $d_{i}(\lambda):=\dim \mathrm{V}_{\lambda}^{\Gamma_{i}},i=1,2$, where $\Gamma_{1}:=\mathcal{F}_{4,\mathrm{I}}(\Z)$ and $\Gamma_{2}:=\mathcal{F}_{4,\mathrm{E}}(\Z)$, 
for a given dominant weight $\lambda$:
\begin{align*}
   \dim \mathrm{V}_{\lambda}^{\Gamma_{i}}=\frac{1}{|\Gamma_{i}|}\sum_{c\in\mathrm{Conj}(\Gamma_{i})}\tr|_{\mathrm{V}_{\lambda}}(c)\cdot |c|=\frac{1}{|\Gamma_{i}|}\sum_{c\in\mathrm{Conj}(\lietype{F}{4})}\tr|_{\mathrm{V}_{\lambda}}(c)\cdot |c\cap \Gamma_{i}|.
\end{align*}
For each rational conjugacy class $c$ whose contribution to this formula is nonzero,
we have already given $|c\cap \Gamma_{i}|$ in \Cref{KacTable},
and according to \Cref{prop degenerate Weyl character formula} the trace $\tr|_{\mathrm{V}_{\lambda}}(c^{\prime})$ is an explicit function of $\lambda_{1},\lambda_{2},\lambda_{3},\lambda_{4}$.

This gives us the following theorem, which is the main computational result of this paper:
\begin{thm}\label{thm main computational result}
    For each dominant weight $\lambda$ of the compact Lie group $\mathrm{F}_{4}$,
    we have an explicit formula for
    \[d_{i}(\lambda)=\dim \mathrm{V}_{\lambda}^{\Gamma_{i}},i=1,2.\]
    For dominant weights $\lambda=(\lambda_{1},\lambda_{2},\lambda_{3},\lambda_{4})$
    with $2\lambda_{1}+3\lambda_{2}+2\lambda_{3}+\lambda_{4}\leq 13$,
    we list all the nonzero $d(\lambda):=d_{1}(\lambda)+d_{2}(\lambda)$ in \Cref{nonzeroAMF}.
\end{thm}
\begin{rmk}\label{rmk condition equivalent to condition on infinitesimal character}
    Later we will see the condition on $\lambda$ in \Cref{thm main computational result} 
    is equivalent to that the maximal eigenvalue (in the $26$-dimensional irreducible representation $\vrep{\varpi_{4}}$) of the infinitesimal character associated to $\mathrm{V}_{\lambda}$ is not larger than $21$.
\end{rmk}
In \cite{Codes}, we also provide a larger table of $[\lambda,d_{1}(\lambda),d_{2}(\lambda),d(\lambda)]$ for weights with $2\lambda_{1}+3\lambda_{2}+2\lambda_{3}+\lambda_{4}\leq 40$.

\section{Subgroups of \texorpdfstring{$\lietype{F}{4}$}{}}\label{section subgroups of F4}
 In this section we will classify subgroups of the compact Lie group $\lietype{F}{4}=\aut(\mathrm{J}_{\R},\circ)$ satisfying certain conditions 
and determine their centralizers in $\lietype{F}{4}$.
Our results will be used in \Cref{section Arthur classification for F4},
but this problem also has its own interest.
Our precise aim is to find all the conjugacy classes of closed subgroups $H$ of $\lietype{F}{4}$ such that:
\begin{enumerate}[label=(\arabic*)]
    \item $H$ is connected;
    \item The centralizer of $H$ in $\lietype{F}{4}$ is an elementary finite abelian $2$-groups, 
    i.e. it is a product of finitely many copies of $\Z/2\Z$.
    \item The multiplicity of zero weight in the restriction of the $26$-dimensional irreducible representation $\mathrm{V}_{\varpi_{4}}$ of $\lietype{F}{4}$ to $H$ is $2$.
\end{enumerate}

If we only consider the first condition,
the problem is equivalent to classifying connected semisimple Lie subalgebras of the complexified Lie algebra $\mathfrak{f}_{4}$,
up to the adjoint action of $\grpF(\C)$.
This has been studied by Dynkin in \cite{Dynkin1952} for all simple complex Lie algebras, without giving full details.
So we will give a detailed classification for $\lietype{F}{4}$ in this section,
following Dynkin's original idea and Losev's result \cite[Theorem 7.1]{Losev2005OnIO}.

Briefly, 
our strategy is to enumerate first all the connected simple subgroups of $\lietype{F}{4}$ inside maximal proper compact subgroups,
and to index them by the restrictions of $\mathrm{V}_{\varpi_{4}}$.
Then we compute their centralizers case by case,
and combine these results together to get all the connected subgroups satisfying our conditions.

\subsection{Element-conjugacy implies conjugacy}\label{section local global conjugacy}
 To be more precise,
what we want to classify, up to $\lietype{F}{4}$-conjugacy, are embeddings from connected compact Lie groups to $\lietype{F}{4}$ 
satisfying two additional conditions.
In this subsection we will explain why it is enough to consider their element-conjugacy classes,
where the notion of element-conjugacy is defined as follows:
\begin{defi}\label{def element conj global conj}\cite[\S 1]{ConjEleConj}
    Let $G$ and $H$ be two Lie groups and $\phi,\phi^{\prime}:H\rightarrow G$ be two Lie group homomorphisms.
We say that $\phi$ and $\phi^{\prime}$ are \emph{conjugate} if there is an element $g\in G$ such that 
\[g\phi(h)g^{-1}=\phi^{\prime}(h),\text{ for all }h\in H.\]
They are said to be \emph{element-conjugate} if for every $h\in H$, there is a $g\in G$ such that
\[g\phi(h)g^{-1}=\phi^{\prime}(h).\]
\end{defi}
The element-conjugacy can be rephrased in the following way:
\begin{lemma}\label{lemma element conj equivalent to equiv after composing linear reps}
    Let $\phi,\phi^{\prime}:H\rightarrow G$ be two homomorphisms between compact Lie groups,
    then they are element-conjugate if and only if for each linear representation $\pi:G\rightarrow \GL(V)$ 
    the compositions $\pi\circ \phi$ and $\pi\circ\phi^{\prime}$ are conjugate in $\GL(V)$.
\end{lemma}
\begin{proof}
   It is a consequence of the Peter-Weyl theorem for compact Lie groups,
   which says that two elements of $G$ are conjugate if and only if they have the same trace on all the irreducible representations of $G$.
\end{proof}
It is obvious that two conjugate homomorphisms are element-conjugate,
but the converse fails in general.
Fortunately, the converse holds when $G=\lietype{F}{4}$ and $H$ is connected,
due to the following result for Lie algebras:
\begin{thm}\label{thm Lie algebra F4 is acceptable}\cite[Proposition 6.2, Theorem 7.1]{Losev2005OnIO}
    Let $\mathfrak{f}_{4}$ be a simple complex Lie algebra of type $\lietype{F}{4}$ and $\lietype{F}{4,\C}$ the complexification of $\lietype{F}{4}$.
    Let $\mathfrak{h}$ be a reductive Lie algebra,
    and $\phi,\phi^{\prime}:\mathfrak{h}\rightarrow \mathfrak{f}_{4}$ two injective Lie algebra homomorphisms.
    If the restrictions of $\phi$ and $\phi^{\prime}$ to a Cartan subalgebra $\mathfrak{s}$ of $\mathfrak{h}$ are conjugate
    in the sense that $\varphi\circ\phi|_\mathfrak{s}=\phi^{\prime}|_{\mathfrak{s}}$ for an inner automorphism $\varphi$ of $\mathfrak{f}_{4}$,
    then $\phi$ and $\phi^{\prime}$ are conjugate. 
\end{thm}
\begin{rmk}
    Actually, in \cite{Losev2005OnIO} Losev uses the following equivalence relation on Lie algebra homomorphisms:
    two Lie algebra homomorphisms $\phi,\phi^{\prime}:\mathfrak{h}\rightarrow\mathfrak{g}$ are equivalent 
    if there exist liftings $H\rightarrow G$ of $\phi,\phi^{\prime}$ to reductive complex groups which are $G$-conjugate in the sense of \Cref{def element conj global conj}.
    By Lie group-Lie algebra correspondence this equivalence relation is the same as $\varphi\circ\phi=\phi^{\prime}$ for an inner automorphism $\varphi$ of $\mathfrak{f}_{4}$.
\end{rmk}
This theorem implies the result we need for $\lietype{F}{4}$:
\begin{prop}\label{prop compact Lie group F4 is acceptable}
    For any connected compact Lie group $H$, two element-conjugate homomorphisms from $H$ to $\lietype{F}{4}$ are conjugate. 
\end{prop}
\begin{proof}
    The argument that deduces this result from \Cref{thm Lie algebra F4 is acceptable} can be found in the proof of \cite[Proposition 3.5]{ConjEleConj}.  
\end{proof}

\subsection{A criterion for element-conjugacy}\label{section one representation is enough}
 According to \Cref{lemma element conj equivalent to equiv after composing linear reps} and \Cref{prop compact Lie group F4 is acceptable}, 
to check whether two homomorphism $\phi$ and $\phi^{\prime}$ from a connected compact Lie group $H$ to $\lietype{F}{4}$ are conjugate,
it suffices to verify that for every irreducible representation $\pi$ of $\lietype{F}{4}$, 
$\pi\circ\phi$ and $\pi\circ \phi^{\prime}$ are equivalent as $H$-representations.
Moreover, we have the following useful fact:
\begin{prop}\label{prop just one rep is enough for conjugacy}
    Let $(\pi_{0},\mathrm{J}_{0})$ be the $26$-dimensional irreducible representation of $\mathrm{F}_{4}$.
    Two homomorphisms $\phi,\phi^{\prime}$ from a connected compact subgroup $H$ to $\mathrm{F}_{4}$ are conjugate 
    if and only if two $H$-representations $\pi_{0}\circ \phi$ and $\pi_{0}\circ \phi^{\prime}$ are equivalent.
\end{prop}
This result is a part of \cite[Theorem 1.3]{Dynkin1952},
but Dynkin only gives a short sketch of the proof,
so in this subsection we will give the proof of \Cref{prop just one rep is enough for conjugacy}. 

We first give a preliminary discussion on orders.
Let $X$ be an abelian group and $\ell:X\rightarrow \R$ a $\Z$-linear map.
This map induces a total preorder $\leq$ on $X$ defined by $x\leq y$ if and only if $\ell(x)\leq \ell(y)$.
A preorder on $X$ of this form will be called an \emph{L-preorder}.
If the map $\ell$ is injective, the $L$-preorder it induces is an order and we call this order an \emph{$L$-order}.  
For instance, any free abelian group of finite rank admits $L$-orders.
\begin{lemma}\label{lemma linear orderings transfer}
    Let $f:X\rightarrow Y$ be a homomorphism between finitely generated free abelian groups $X$ and $Y$,
    with an $L$-order on $Y$,
    and $S$ a finite subset of $X-\{0\}$.
    There exists an $L$-preorder $\leq$ on $X$ such that 
    for any $s\in S$ we have either $s>0$ or $s<0$,
    and if $s>0$ then $f(s)\geq 0$ in $Y$.
\end{lemma}
\begin{proof}
    We choose $\ell:Y\hookrightarrow \R$ such that the $L$-order on $Y$ is defined by $\ell$.
    Write $S=S_{0}\sqcup S_{1}$, with $S_{0}=S\cap \ker f$.
    If $S_{0}$ is empty, then the $L$-preorder on $X$ defined by $\ell\circ f$ satisfies the conditions.

    If $S_{0}$ is not empty,
    we choose an arbitrary injective $\Z$-linear map $j:X\hookrightarrow \R$
    and set \[\varepsilon:=\frac{1}{2}\min_{s\in S_{1}}\frac{|\ell(f(s))|}{|j(s)|}.\]
    We claim that the $L$-preorder on $X$ defined by $j^{\prime}=\ell\circ f+\varepsilon j$ satisfies the desired conditions.
    Indeed, for $s\in S_{0}$, $j^{\prime}(s)=\varepsilon j(s)$ is nonzero.
    Also for $s\in S_{1}$, by our choice of $\varepsilon$,
    we have $|\varepsilon j(s)|<|\ell(f(s))|$,
    so $j^{\prime}(s)$ is nonzero and of the same sign as $\ell(f(s))$.
\end{proof}
The next lemma concerns the partial order $\preceq$ of the weights of the $26$-dimensional irreducible representation $\pi_{0}$ of $\lietype{F}{4}$.
Recall that for two weights $\lambda$ and $\mu$ of $\lietype{F}{4}$,
fixing a positive root system of $\lietype{F}{4}$,
we write $\lambda\succeq \mu$ if $\lambda-\mu$ is a finite sum of positive roots.
Note that every non-zero weight $\lambda$ of $\pi_{0}$ satisfies either $\lambda\succ 0$ or $\lambda\prec 0$.
\begin{lemma}\label{lemma ordering on weights of 26 dimensional representation}
    The $26$-dimensional irreducible representation $(\pi_{0},\mathrm{J}_{0})$ of $\lietype{F}{4}$ has four unique weights $\lambda_{1}\succ \lambda_{2}\succ \lambda_{3}\succ\lambda_{4}$
    satisfying that $\lambda\prec \lambda_{4}$ for all other weights $\lambda$.
    Moreover, those $4$ weights $\lambda_{1},\lambda_{2},\lambda_{3},\lambda_{4}$ form a $\Z$-basis of the weight lattice of $\lietype{F}{4}$. 
\end{lemma}
\begin{proof}
    Fix a maximal torus $T$ of $\lietype{F}{4}$,
    and let $X=\mathrm{X}^{*}(T)$ be its character lattice and $\Phi^{+}\subset X$ a positive root system with respect to $(\lietype{F}{4},T)$.
    We still use Bourbaki's notations \cite[\S VI.4.9]{Lie} for the root system $\lietype{F}{4}$.
    The simple roots with respect to $\Phi^{+}$ are given by 
    \[\alpha_{1}=\varepsilon_{2}-\varepsilon_{3},\alpha_{2}=\varepsilon_{3}-\varepsilon_{4},\alpha_{3}=\varepsilon_{4},\alpha_{4}=\frac{1}{2}(\varepsilon_{1}-\varepsilon_{2}-\varepsilon_{3}-\varepsilon_{4}),\]
    where $\varepsilon_{1},\varepsilon_{2},\varepsilon_{3},\varepsilon_{4}$ is the basis of $X\otimes_{\Z} \R\simeq \R^{4}$ chosen in \cite{Lie} satisfying 
    \[X=\Z\varepsilon_{1}+\Z\varepsilon_{2}+\Z\varepsilon_{3}+\Z\varepsilon_{4}+\Z\frac{\varepsilon_{1}+\varepsilon_{2}+\varepsilon_{3}+\varepsilon_{4}}{2}.\]

    The highest weight of $\pi_{0}$ is $\varpi_{4}=\alpha_{1}+2\alpha_{2}+3\alpha_{3}+2\alpha_{4}=\varepsilon_{1}$.
    The orbit of $\varpi_{4}$ under the Weyl group consists of 
    $\pm \varepsilon_{i}$ for $i=1,2,3,4$ and $\frac{1}{2}(\pm \varepsilon_{1}\pm \varepsilon_{2}\pm \varepsilon_{3}\pm \varepsilon_{4})$.
    These $24$ weights have multiplicity $1$,
    and the zero weight appears with multiplicity $2$.

    We claim that the weights 
    \begin{align*}
        \lambda_{1}=\varepsilon_{1}&,\lambda_{2}=\frac{1}{2}(\varepsilon_{1}+\varepsilon_{2}+\varepsilon_{3}+\varepsilon_{4}),\\
        \lambda_{3}=\frac{1}{2}(\varepsilon_{1}+\varepsilon_{2}+\varepsilon_{3}-\varepsilon_{4})&,\lambda_{4}=\frac{1}{2}(\varepsilon_{1}+\varepsilon_{2}-\varepsilon_{3}+\varepsilon_{4})
    \end{align*}
    satisfy the desired properties.
    Indeed, this follows from the following table
    (by \emph{positive} we mean $\lambda\succ 0$):
    \begin{table}[H]
        \centering
        \begin{tabular}{|l|l|}
        \hline 
        positive weight $\lambda$  & relation with $\lambda_{1},\lambda_{2},\lambda_{3},\lambda_{4}$ \\ \hline
        $\varepsilon_{1}$  & $\lambda_{1}$ \\ \hline
        $\varepsilon_{2}$  & $\lambda_{4}-\alpha_{3}-\alpha_{4}$\\ \hline
        $\varepsilon_{3}$  & $\lambda_{4}-\alpha_{1}-\alpha_{3}-\alpha_{4}$\\ \hline
        $\varepsilon_{4}$  & $\lambda_{4}-\alpha_{1}-\alpha_{2}-\alpha_{3}-\alpha_{4}$\\ \hline
        $(\varepsilon_{1}+\varepsilon_{2}+\varepsilon_{3}+\varepsilon_{4})/2$ & $\lambda_{2}=\lambda_{1}-\alpha_{4}$\\ \hline
        $(\varepsilon_{1}+\varepsilon_{2}+\varepsilon_{3}-\varepsilon_{4})/2$ & $\lambda_{3}=\lambda_{2}-\alpha_{3}$\\ \hline
        $(\varepsilon_{1}+\varepsilon_{2}-\varepsilon_{3}+\varepsilon_{4})/2$ & $\lambda_{4}=\lambda_{3}-\alpha_{2}$\\ \hline
        $(\varepsilon_{1}+\varepsilon_{2}-\varepsilon_{3}-\varepsilon_{4})/2$ & $\lambda_{4}-\alpha_{3}$\\ \hline
        $(\varepsilon_{1}-\varepsilon_{2}+\varepsilon_{3}+\varepsilon_{4})/2$ & $\lambda_{4}-\alpha_{1}$ \\ \hline
        $(\varepsilon_{1}-\varepsilon_{2}+\varepsilon_{3}-\varepsilon_{4})/2$ & $\lambda_{4}-\alpha_{1}-\alpha_{3}$\\ \hline
        $(\varepsilon_{1}-\varepsilon_{2}-\varepsilon_{3}+\varepsilon_{4})/2$ & $\lambda_{4}-\alpha_{1}-\alpha_{2}-\alpha_{3}$\\ \hline
        $(\varepsilon_{1}-\varepsilon_{2}-\varepsilon_{3}-\varepsilon_{4})/2$ & $\lambda_{4}-\alpha_{1}-\alpha_{2}-2\alpha_{3}$\\ \hline
        \end{tabular}
        \caption{Positive weights of the $26$-dimensional irreducible representation $\mathrm{V}_{\varpi_{4}}$ of $\lietype{F}{4}$}\label{table positive weights of 26 dimensional irreducible representation}
    \end{table}
    and the following identities:
    \[\varepsilon_{1}=\lambda_{1},\varepsilon_{2}=-\lambda_{1}+\lambda_{3}+\lambda_{4},\varepsilon_{3}=\lambda_{2}-\lambda_{4},\varepsilon_{4}=\lambda_{2}-\lambda_{3},\frac{\varepsilon_{1}+\varepsilon_{2}+\varepsilon_{3}+\varepsilon_{4}}{2}=\lambda_{2}.\qedhere\]
\end{proof}
\begin{proof}[Proof of \Cref{prop just one rep is enough for conjugacy}]
    By \Cref{prop compact Lie group F4 is acceptable} it suffices to show that 
    if $\pi_{0}\circ \phi$ and $\pi_{0}\circ \phi^{\prime}$ are equivalent as $H$-representations,
    then $\phi$ and $\phi^{\prime}$ are element-conjugate.
    Since any element of $H$ is included in some maximal torus,
    we may assume that $H$ is a torus.

    We fix a maximal torus $T$ of $\lietype{F}{4}$.
    As all maximal tori are conjugate in $\lietype{F}{4}$,
    up to replacing $\phi$ and $\phi^{\prime}$ by some $\lietype{F}{4}$-conjugate,
    we assume that both $\phi(H)$ and $\phi^{\prime}(H)$ are contained in $T$.
    Let $X=\mathrm{X}^{*}(T)$ and $Y=\mathrm{X}^{*}(H)$,
    then $\phi$ and $\phi^{\prime}$ induce $\Z$-linear maps $\phi^{*},\phi^{\prime,*}:X\rightarrow Y$ respectively.

    Choose an arbitrary $L$-order on $Y$,
    and denote by $\Phi\subset X$ the root system of $(\lietype{F}{4},T)$.
    By \Cref{lemma linear orderings transfer},
    there is an $L$-preorder $\leq$ (resp. $\leq^{\prime}$) on $X$ 
    such that for any $\alpha\in \Phi$ we have either $\alpha>0$ or $\alpha<0$ (resp. either $\alpha>^{\prime}0$ or $\alpha<^{\prime}0$),
    and the $\Z$-linear map $\phi^{*}$ (resp. $\phi^{\prime,*}$) preserves the preorders on $X,Y$.
    We denote the positive root system determined by the $L$-preorder $\leq$ (resp. $\leq^{\prime}$) by $\Phi^{+}$ (resp. $\Phi^{+,\prime}$).

    A general fact about root systems is that the Weyl group of $(\lietype{F}{4},T)$ acts transitively on the set of positive root systems of $(\lietype{F}{4},T)$.
    Up to conjugating $\phi^{\prime}$ by a suitable element in the normalizer $\mathrm{N}_{\lietype{F}{4}}(T)$,
    we may assume that $\Phi^{+,\prime}=\Phi^{+}$.
    Now our aim is to show $\phi=\phi^{\prime}$,
    which is equivalent to $\phi^{*}=\phi^{\prime,*}$.
    
    Let $\mathcal{W}$ be the multiset of $X$ consisting of the weights appearing in $\pi_{0}$.
    Let $\lambda_{1}\succ\lambda_{2}\succ\lambda_{3}\succ\lambda_{4}$ be the $4$ weights of $\pi_{0}$ defined in \Cref{lemma ordering on weights of 26 dimensional representation}
    and all of them have multiplicity $1$ in $\pi_{0}$.
    For the $\Z$-linear map $f=\phi^{*}$ or $\phi^{\prime,*}$,
    the preorder-preserving property of $f$ 
    and \Cref{table positive weights of 26 dimensional irreducible representation}
    imply that 
    $f(\lambda_{1})\geq f(\lambda_{2})\geq f(\lambda_{3})\geq f(\lambda_{4})$ 
    and $f(\lambda_{4})\geq f(\lambda)$ for all other weights $\lambda$ of $\pi_{0}$.
    In other words,
    $f(\lambda_{1})$ is the greatest element of $f(\mathcal{W})$,
    and for $i=2,3,4$,
    $f(\lambda_{i})$ is the greatest element of $f(\mathcal{W})\setminus \{f(\lambda_{1}),\ldots,f(\lambda_{i-1})\}$.
    By the assumption $\pi_{0}\circ \phi=\pi_{0}\circ\phi^{\prime}$,
    the multisets $\phi^{*}(\mathcal{W})$ and $\phi^{\prime,*}(\mathcal{W})$ of $Y$ coincide.
    It follows that we have $\phi^{*}(\lambda_{i})=\phi^{\prime,*}(\lambda_{i})$ for $i=1,2,3,4$,
    and as $\lambda_{1},\lambda_{2},\lambda_{3},\lambda_{4}$ form a basis of $X$ by \Cref{lemma ordering on weights of 26 dimensional representation},
    we deduce $\phi^{*}=\phi^{\prime,*}$.  
\end{proof}
Hence the conjugacy class of a homomorphism from a connected compact Lie group $H$ to $\mathrm{F}_{4}$ is determined by the restriction of the $26$-dimensional irreducible representation to $H$.

\subsection{Maximal proper connected subgroups}\label{section maximal subgroups of F4}
 Up to conjugacy,
the compact group $\mathrm{F}_{4}$ has five maximal proper connected subgroups by \cite[Theorem 5.5, Theorem 14.1]{Dynkin1952}.
We will recall these five subgroups in this subsection and show that there are no other maximal proper connected subgroups.

We first introduce the following notations, which will be used a lot of times in this section:
\begin{notation}\label{notation compact Lie groups}
    In this article, 
    we use the following notations of compact Lie groups:
    \begin{itemize}
        \item For $n\geq 2$, denote by $\su(n)$ the compact special unitary group with respect to the standard Hermitian form on $\C^{n}$.
        \item For $n\geq 3$, denote by $\sorth(n)$ the compact special orthogonal group with respect to the standard quadratic form on $\R^{n}$, and by $\spin(n)$ the compact spin group, which is a double cover of $\sorth(n)$.
        \item For $n\geq 1$, denote by $\symp(n)$ the \emph{compact} symplectic group:
        the group of invertible $n\times n$ quaternionic matrices that preserve the standard Hermitian form 
        \[\langle x,y\rangle=\overline{x_{1}}y_{1}+\cdots+\overline{x_{n}}y_{n}\]
        on $\mathbb{H}^{n}$, where $\mathbb{H}$ is Hamilton's quaternions. 
        \item The group $\lietype{G}{2}$ is defined as $\aut(\oct_{\R},\circ)$, the automorphism group of the real octonion division algebra,
        which is simply connected and has trivial center.
    \end{itemize}
\end{notation}
\begin{rmk}\label{rmk compact symplectic group}
    The complexification of the compact symplectic group $\symp(n)$ is the usual complex symplectic group $\symp(2n,\C)=\symp_{2n}(\C)$,
    which is defined as the group of linear transformations of $\C^{2n}$
    preserving the standard symplectic bilinear form.
\end{rmk}
\begin{notation}\label{notation quotient by diagonal}
    We denote by $\mu_{n}$ the group of $n$th roots of unity.
    If $m$ groups $G_{1},\ldots,G_{m}$ all have a unique central subgroup isomorphic to $\mu_{n}$ with an embedding $\iota_{i}:\mu_{n}\hookrightarrow G_{i}$,
    we denote by $\mu_{n}^{\Delta}$ the diagonal subgroup 
    \[\{\left(\iota_{1}(g),\ldots,\iota_{m}(g)\right)\,|\,g\in\mu_{n}\}\subset G_{1}\times\cdots\times G_{m}.\] 
    Note that when $n=2$ the embedding $\iota_{i}$ is unique, but when $n\geq 3$ we have to give $\iota_{1},\ldots,\iota_{m}$ for defining $\mu_{n}^{\Delta}$.
\end{notation}
Following Dynkin's definitions of $R$-subalgebras and $S$-subalgebras in \cite[\S 7]{Dynkin1952},
we give the following definition for subgroups:
\begin{defi}\label{def regular subgroups and other notations}
    Let $G$ be a connected compact Lie group and $H$ a connected closed subgroup. 
    We say that $H$ is a \emph{regular subgroup} if it is normalized by a maximal torus of $G$.
    If there is only one regular subgroup of $G$ containing $H$,
    namely $G$ itself,
    we call $H$ an \emph{$S$-subgroup},
    otherwise we call it an \emph{$R$-subgroup}.
\end{defi}
\begin{exs}\label{ex property of regular subgroups}
    \begin{enumerate}[label= (\arabic*)]
        \item Subgroups with maximal ranks are regular.
        \item A proper regular subgroup is an $R$-subgroup.
        \item The principal $3$-dimensional subgroups, in the sense of \cite[Theorem 4.1.6]{NilOrbit}, are $S$-subgroups by \cite[Theorem 9.1]{Dynkin1952}.
        \item A maximal proper regular subgroup has maximal rank.
    \end{enumerate}
\end{exs}

Let $H$ be a maximal proper regular subgroup of $G$,
i.e. if there is another regular subgroup $H^{\prime}$ of $G$ containing $H$,
then we have $H^{\prime}=G$.
The Borel-de Siebenthal theory 
classifies maximal proper regular subgroups of $G$, and can be phrased in terms of the extended Dynkin diagram of the Lie algebra of $G$. 
For our compact group $\lietype{F}{4}$,
any maximal proper regular subgroup $H$ is semisimple,
and the Dynkin diagram of its root system is obtained 
by deleting a vertex with prime label
from the extended Dynkin diagram of the root system of type $\lietype{F}{4}$:
\begin{align*}
    \dynkin[labels={0,1,2,3,4},label macro/.code={\alpha_{\drlap#1}},
    labels*={1,2,3,4,2},label macro*/.code={#1},
    extended]F4.
 \end{align*}
The vertex $\alpha_{1}$ corresponds to $\left(\symp(1)\times\symp(3)\right)/\mu_{2}^{\Delta}$,
$\alpha_{2}$ corresponds to $\left(\su(3)\times\su(3)\right)/\mu_{3}^{\Delta}$ (we will define this $\mu_{3}^{\Delta}$ in \Cref{section A2A2 subgroup}),
and $\alpha_{4}$ corresponds to $\spin(9)$.
The vertex $\alpha_{3}$ corresponds to $\left(\su(2)\times\su(4)\right)/\mu_{2}^{\Delta}$,
which is also regular but not maximal since we have the embedding:
\[\left(\su(2)\times\su(4)\right)/\mu_{2}^{\Delta}\simeq \left(\spin(3)\times\spin(6)\right)/\mu_{2}^{\Delta}\hookrightarrow \spin(9).\]
These three maximal proper regular subgroups are also maximal among proper connected subgroups of $\lietype{F}{4}$,
because any connected subgroup containing one of them has maximal rank and must be regular.

Besides these three regular subgroups,
$\lietype{F}{4}$ also admits other maximal proper connected subgroups that are not regular.
A non-regular maximal connected subgroup $H$ of $\lietype{F}{4}$ must be an $S$-subgroup.
As a subgroup of $\lietype{F}{4}$ containing an $S$-subgroup is also an $S$-subgroup,
it suffices to find all maximal $S$-subgroups of $\lietype{F}{4}$.
\begin{thm}\cite[Theorem 14.1]{Dynkin1952}\label{thm classification of maximal S-subgroups}
    Up to conjugacy, 
    there are two maximal $S$-subgroups in $\lietype{F}{4}$:
    the principal $\psu(2)$ and $\lietype{G}{2}\times \sorth(3)$,
    where $\psu(2):=\su(2)/\{\pm\mathrm{id}\}$ is the adjoint group of $\su(2)$.
\end{thm}
Putting the Borel-de Siebenthal theory and \Cref{thm classification of maximal S-subgroups} together,
we have:
\begin{thm}\label{thm classification of maximal subgroups}
    Up to conjugacy,
    there are five maximal proper connected subgroups of $\lietype{F}{4}$.
    They are respectively isomorphic to 
    \[\spin(9),\left(\symp(1)\times\symp(3)\right)/\mu_{2}^{\Delta},\left(\su(3)\times\su(3)\right)/\mu_{3}^{\Delta},\mathrm{G}_{2}\times\sorth(3),\text{(principal) }\psu(2).\]
\end{thm}
In the rest of this subsection,
we will give the explicit embeddings of these five maximal proper connected subgroups into $\lietype{F}{4}$
and compute their centralizers in $\lietype{F}{4}$.
   
\subsubsection{\texorpdfstring{$\spin(9)$}{}}\label{section B4 subgroup}
 There is an involution $\sigma\in\lietype{F}{4}$ on $\mathrm{J}_{\R}$ defined by:
\begin{align*}
    \sigma\,\left[a,b,c\,;x,y,z\right]=\left[a,b,c\,;x,-y,-z\right],\text{ for all }a,b,c\in\R,x,y,z\in\oct_{\R}.
\end{align*}
By \cite[Theorem 2.9.1]{Yokota2009ExceptionalLG},
the centralizer $\mathrm{C}_{\mathrm{F}_{4}}(\sigma)$ of $\sigma$ in $\lietype{F}{4}$ is also the stabilizer of $\mathrm{E}_{1}=\mathrm{diag}(1,0,0)\in \mathrm{J}_{\R}$.
\begin{lemma}\label{lemma spin(9) preserves subspaces}
    The group $\mathrm{C}_{\lietype{F}{4}}(\sigma)$ preserves respectively the subspaces 
    \[\mathrm{J}_{1}:=\left\{\left[0,b,-b\,;x,0,0\right]\,\middle\vert\,b\in \R,x\in \oct_{\R}\right\}\]
    and 
    \[\mathrm{J}_{2}:=\left\{\left[0,0,0\,;0,y,z\right]\,\middle\vert\,y,z\in\oct_{\R}\right\}\]
    of $\mathrm{J}_{\R}$.
\end{lemma}
\begin{proof}
    The first subspace $\mathrm{J}_{1}$ is exactly 
    $\{X\in \mathrm{J}_{\R}\,|\, E_{1}\circ X=0,\tr(X)=0\}$
    and the second subspace is 
    $\{X\in \mathrm{J}_{\R}\,|\, 2E_{1}\circ X=X\}$.
    The lemma follows from the fact that $\mathrm{C}_{\lietype{F}{4}}(\sigma)$ is the stabilizer of $\mathrm{E}_{1}$ in $\lietype{F}{4}$.
\end{proof}
This lemma gives the following group homomorphism:
\[\mathrm{C}_{\lietype{F}{4}}(\sigma)\rightarrow \sorth(\mathrm{J}_{1})\simeq \sorth(9),g\mapsto g|_{\mathrm{J}_{1}},\]
which induce an isomorphism $\mathrm{C}_{\lietype{F}{4}}(\sigma)\simeq \spin(9)$ by \cite[Theorem 16.7(ii)]{AdamsExceptionalGrp}.
Since the Borel-de Siebenthal theory shows that the regular connected subgroup of type $\lietype{B}{4}$ is unique up to $\lietype{F}{4}$-conjugacy,
so we shall thus refer to this group $\mathrm{C}_{\lietype{F}{4}}(\sigma)$ as $\spin(9)$ in the sequel,
by a slight abuse of language.

The restriction of the $26$-dimensional irreducible representation $(\pi_{0},\mathrm{J}_{0})$ to $\spin(9)$ is isomorphic to
\begin{align}\label{eqn restriction of st rep to B4}
    \mathbf{1}\oplus \mathrm{V}_{9}\oplus \mathrm{V}_{\spin},
\end{align}
where $\triv$ is the trivial representation, $\vrep{9}$ is the standard $9$-dimensional representation and $\vrep{\spin}$ is the $16$-dimensional spinor module.
These two representations $\mathrm{V}_{9}$ and $\mathrm{V}_{\spin}$ can be realized on $\mathrm{J}_{1}$ and $\mathrm{J}_{2}$ respectively.
\begin{notation}\label{notation direct sum of representations}
    To make the restriction of $\mathrm{J}_{0}$ not too messy when it involves both direct sums and tensor products,
    we will replace $\oplus$ by $+$ when writing down the decomposition.
    For example, we write $\jord_{0}|_{\spin(9)}$ as $\triv+\vrep{9}+\vrep{\spin}$. 
\end{notation}
The restriction of the adjoint representation $\mathfrak{f}_{4}$ of $\lietype{F}{4}$ to $\spin(9)$ is isomorphic to:
\begin{align}\label{eqn restriction of ad rep to B4}
    \wedge^{2}\vrep{9}+\vrep{\spin},   
\end{align}
where $\wedge^{2}\vrep{9}$ is the adjoint representation of $\spin(9)$.

Now we compute the centralizer of $\spin(9)$.
If an element $g$ centralizes $\spin(9)$, 
then it must commute with $\sigma\in\spin(9)$.
Hence $\mathrm{C}_{\mathrm{F}_{4}}(\spin(9))$ is contained in $\mathrm{C}_{\lietype{F}{4}}(\sigma)=\spin(9)$,
thus it is isomorphic to the center of $\spin(9)$, which is isomorphic to $\Z/2\Z$ and generated by $\sigma$. 
\begin{rmk}\label{rmk three symmetric ways to define Spin(9)}
    By symmetry, the stabilizer of $\mathrm{E}_{2}=\mathrm{diag}(0,1,0)$ (resp. $\mathrm{E}_{3}=\mathrm{diag}(0,0,1)$)
    is also the centralizer of the map $[a,b,c\,;x,y,z]\mapsto [a,b,c\,;-x,y,-z]$ 
    (resp. $[a,b,c\,;-x,-y,z]$) in $\lietype{F}{4}$,
    and is isomorphic to $\spin(9)$. 
\end{rmk}

\subsubsection{\texorpdfstring{$\left(\symp(1)\times \symp(3)\right)/\mu_{2}^{\Delta}$}{}}\label{section C3A1 subgroup}
 
The subalgebra of $\oct_{\R}$ generated by $1,\mathrm{e}_{1},\mathrm{e}_{2},\mathrm{e}_{4}$
is isomorphic to the quaternion division algebra $\mathbb{H}$,
and as a real vector space $\oct_{\R}$ can be decomposed as $\mathbb{H}\oplus \mathbb{H}\mathrm{e}_{5}$.
Using this decomposition,
the conjugation on $\oct_{\R}$ becomes
\[x+y\mathrm{e}_{5}\mapsto \overline{x}-y\mathrm{e}_{5},\text{ for all }x,y\in\mathbb{H}.\]
As $\mathrm{J}_{\R}=\mathrm{Herm}_{3}(\oct_{\R})$ is the space of ``Hermitian'' matrices in $\mathrm{M}_{3}(\oct_{\R})$,
we embed the space $\mathrm{Herm}_{3}(\mathbb{H})$ of ``Hermitian'' matrices in $\mathrm{M}_{3}(\mathbb{H})$ into $\mathrm{J}_{\R}$ via 
our identification of $\mathbb{H}$ as a subalgebra of $\oct_{\R}$.
Then we have the following isomorphism of vector spaces:
\begin{align*}
    \mathrm{Herm}_{3}(\mathbb{H})\oplus \mathbb{H}^{3}&\rightarrow\mathrm{J}_{\R},\\
    \left(M,a=(a_{1},a_{2},a_{3})\right)&\mapsto M+[0,0,0\,;a_{1}\mathrm{e}_{5},a_{2}\mathrm{e}_{5},a_{3}\mathrm{e}_{5}]. 
\end{align*}
With this identification, we have an involution $\gamma$ in $\lietype{F}{4}$ defined as 
\[\gamma(M,a)=(M,-a).\]
\begin{prop}\label{prop centralizer of an involution isomorphic to A1C3}\cite[Theorem 2.11.2]{Yokota2009ExceptionalLG}
    Let $\varphi:\symp(1)\times\symp(3)\rightarrow \GL(\mathrm{J}_{\R})$ be the morphism defined as 
    \[\varphi(p,A)\left(M,a\right)=\left(AMA^{-1},paA^{-1}\right),\text{ for }M\in \mathrm{Herm}_{3}(\mathbb{H}),a\in \mathbb{H}^{3}.\]
    Then the kernel of $\varphi$ is the diagonal subgroup $\mu_{2}^{\Delta}$,
    and the image of $\varphi$ is $\mathrm{C}_{\lietype{F}{4}}(\gamma)$.
    In particular, $\varphi$ induces an isomorphism:
    \[\left(\symp(1)\times\symp(3)\right)/\mu_{2}^{\Delta}\simeq \mathrm{C}_{\lietype{F}{4}}(\gamma).\]
\end{prop}
From now on we refer to the regular connected subgroup $\mathrm{C}_{\lietype{F}{4}}(\gamma)$ as $\left(\symp(1)\times\symp(3)\right)/\mu_{2}^{\Delta}$.

The restriction of the irreducible representation $\mathrm{J}_{0}$ of $\lietype{F}{4}$ to this subgroup is isomorphic to
\begin{align}\label{eqn restriction of st rep to A1C3}
    \st\otimes \mathrm{V}_{6}+ \triv\otimes \mathrm{V}_{14}, 
\end{align} 
where $\st$ is the $2$-dimensional standard representation of $\symp(1)\simeq \su(2)$, 
$\vrep{6}$ is the standard $6$-dimensional representation of $\symp(3)$ 
and $\vrep{14}$ is the $14$-dimensional irreducible representation of $\symp(3)$
which satisfies $\wedge^{2}\vrep{6}\simeq \vrep{14}\oplus \triv$.
The first component $\st\otimes\vrep{6}$ is realized on $\mathbb{H}^{3}$
and the second component $\triv\otimes\vrep{14}$ is realized on the trace-zero part of $\mathrm{Herm}_{3}(\mathbb{H})$.

The restriction of the adjoint representation $\mathfrak{f}_{4}$ of $\lietype{F}{4}$ to $\left(\symp(1)\times\symp(3)\right)/\mu_{2}^{\Delta}$ is isomorphic to 
\begin{align}\label{eqn restriction of Ad rep to A1C3}
    \sym^{2}\st\otimes\triv+\st\otimes\vrep{14}^{\prime}+\triv\otimes\sym^{2}\vrep{6},
\end{align}
where $\vrep{14}^{\prime}$ is another $14$-dimensional irreducible representation of $\symp(3)$.

By a similar argument in the case of $\spin(9)$, 
the centralizer of $\left(\symp(1)\times \symp(3)\right)/\mu_{2}^{\Delta}$ in $\mathrm{F}_{4}$ is isomorphic to 
$\mathrm{Z}(\left(\symp(1)\times \symp(3)\right)/\mu_{2}^{\Delta})\simeq \Z/2\Z$.
It is generated by the involution $\gamma$,
which corresponds to $(-1,1)$ in $\mathrm{Z}(\symp(1)\times\symp(3))\simeq \mu_{2}\times\mu_{2}$.
\begin{rmk}\label{rmk involutions with different centralizer}
    It may help to notice that there are exactly two conjugacy classes of involutions in $\lietype{F}{4}$,
    whose centralizers in $\lietype{F}{4}$ are $\spin(9)$ and $\left(\symp(1)\times\symp(3)\right)/\mu_{2}^{\Delta}$ respectively.
\end{rmk}

\subsubsection{\texorpdfstring{$\left(\su(3)\times \su(3)\right)/\mu_{3}^{\Delta}$}{}}\label{section A2A2 subgroup}
 Take $\omega=\frac{-1+\sqrt{-3}}{2}$
and identify the center of $\su(3)$ with $\mu_{3}$
by identifying $\omega$ with the scalar matrix $\omega\mathrm{I}_{3}$.
Then the diagonal subgroup $\mu_{3}^{\Delta}\subset \su(3)\times\su(3)$ is generated by $(\omega,\omega)$.

By \cite[Theorem 2.12.2]{Yokota2009ExceptionalLG},
the centralizer in $\lietype{F}{4}$ of a specific order $3$ element in $\lietype{F}{4}$ is isomorphic to $\left(\su(3)\times\su(3)\right)/\mu_{3}^{\Delta}$.
As before, by an abuse of language we will refer to this subgroup as $\left(\su(3)\times\su(3)\right)/\mu_{3}^{\Delta}$.
Notice that the roots of the first copy of $\su(3)$ are short roots of $\lietype{F}{4}$,
and those of the second copy are long roots of $\lietype{F}{4}$.

The group $\su(3)$ admits a non-trivial outer automorphism $\theta$.
By \cite[Remark 2.12.4]{Yokota2009ExceptionalLG}, 
there is an order $2$ element $\widetilde{\theta}$ in $\lietype{F}{4}$ 
which normalizes $\left(\su(3)\times\su(3)\right)/\mu_{3}^{\Delta}$ and acts on it by $(\theta,\theta)$.

The restriction of the irreducible representation $\jord_{0}$ to $\left(\su(3)\times\su(3)\right)/\mu_{3}^{\Delta}$ is isomorphic to 
\begin{align}\label{eqn restriction of st rep to A2A2}
\mathfrak{sl}_{3}\otimes \triv+ \vrep{3}\otimes \vrep{3}^{\prime}+ \vrep{3}^{\prime}\otimes \vrep{3},
\end{align}
where $\vrep{3}$ is the standard $3$-dimensional representation of $\su(3)$,
$\vrep{3}^{\prime}$ is the dual representation of $\vrep{3}$,
and $\mathfrak{sl}_{3}$ is the adjoint representation of $\su(3)$.

The restriction of the adjoint representation $\mathfrak{f}_{4}$ of $\lietype{F}{4}$ to $\left(\su(3)\times\su(3)\right)/\mu_{3}^{\Delta}$ is isomorphic to 
\begin{align}\label{eqn restriction of Ad rep to A2A2}
    \mathfrak{sl}_{3}\otimes\triv+\triv\otimes\mathfrak{sl}_{3}+\sym^{2}\vrep{3}\otimes \vrep{3}+\sym^{2}\vrep{3}^{\prime}\otimes\vrep{3}^{\prime}.
\end{align}
Again, we have an isomorphism
$\mathrm{C}_{\mathrm{F}_{4}}(\left(\su(3)\times \su(3)\right)/\mu_{3}^{\Delta})\simeq \Z/3\Z$.

\subsubsection{\texorpdfstring{$\mathrm{G}_{2}\times \sorth(3)$}{}}\label{section G2A1 subgroup}
 We define an injective morphism $\iota:\lietype{G}{2}\times\sorth(3)\hookrightarrow\GL(\jord_{\R})$ by 
\begin{align}\label{eqn embedding of G2+SO(3) into F4}
    \iota(g,O)[a,b,c\,;x,y,z]=O[a,b,c\,;g(x),g(y),g(z)]O^{-1},\text{ for all }a,b,c\in \R,x,y,z\in\oct_{\R},
\end{align}
by viewing $O\in\sorth(3)$ as an element in $\GL_{3}(\oct_{\R})$ with entries in $\R$.
This morphism is well-defined since real numbers $\R$ is the center of the octonion division algebra $\oct_{\R}$.
For any $g\in\lietype{G}{2}$ and $O\in\sorth(3)$,
the linear automorphism $\iota(g,O)$ preserves the cubic form $\det$ and the polarization $\mathrm{I}$,
thus $\iota$ induces an embedding of $\lietype{G}{2}\times\sorth(3)$ into $\lietype{F}{4}$.
In the sequel we will refer to the image of $\iota$ as $\lietype{G}{2}\times\sorth(3)$.

The restriction of the irreducible representation $\mathrm{J}_{0}$ to $\lietype{G}{2}\times\sorth(3)$ is isomorphic to 
\begin{align}\label{eqn restriction of the st rep to G2A1}
    \vrep{7}\otimes \sym^{2}\st + \triv\otimes \sym^{4}\st,
\end{align} 
where $\vrep{7}$ is the fundamental $7$-dimensional representation of $\mathrm{G}_{2}$ (the trace-zero part of $\oct_{\C}$) 
and $\st$ denotes the standard $2$-dimensional representation of $\su(2)$.
Here we use the exceptional isomorphism $\sorth(3)\simeq \psu(2)=\su(2)/\mu_{2}$ to 
view odd dimensional irreducible representations $\sym^{2n}\st, n\in\mathbb{N}$ of $\su(2)$
as irreducible representations of $\sorth(3)$.
The first component $\vrep{7}\otimes \sym^{2}\st$ is realized on the space 
\[\set{[0,0,0\,;x,y,z]}{x,y,z\in \oct_{\R},\tr(x)=\tr(y)=\tr(z)=0},\]
and the second component $\triv\otimes\sym^{4}\st$ is realized on the space 
\[\set{[a,b,c\,;x,y,z]}{a,b,c,x,y,z\in\R,a+b+c=0}.\]

The restriction of the adjoint representation $\mathfrak{f}_{4}$ of $\lietype{F}{4}$ to $\lietype{G}{2}\times\sorth(3)$ is isomorphic to 
\begin{align}\label{eqn restriction of the Ad rep to G2A1}
    \mathfrak{g}_{2}\otimes \triv+\vrep{7}\otimes\sym^{4}\st+\triv\otimes\sym^{2}\st,
\end{align}
where $\mathfrak{g}_{2}$ is the adjoint representation of $\lietype{G}{2}$.
\begin{prop}\label{prop centralizer of A1+G2 is trivial}
    The centralizer of $\lietype{G}{2}\times\sorth(3)$ in $\lietype{F}{4}$ is trivial.
\end{prop}
\begin{proof}
    Let $g$ be an element in $\mathrm{C}_{\lietype{F}{4}}(\lietype{G}{2}\times\sorth(3))$.
    Because the image of $\mathrm{diag}(1,-1,-1)\in\sorth(3)$ in $\lietype{F}{4}$ is the involution $\sigma$ defined in \Cref{section B4 subgroup},
    $g$ lies in $\mathrm{C}_{\lietype{F}{4}}(\sigma)$,
    thus it stabilizes $\lietype{E}{1}$.
    By \Cref{rmk three symmetric ways to define Spin(9)},
    we also have $g$ stabilizes $\mathrm{E}_{2}$ and $\mathrm{E}_{3}$ respectively.
    According to \cite[Theorem 16.7(iii), Lemma 15.15]{AdamsExceptionalGrp},
    $g$ is an element of the form
    \[[a,b,c\,;x,y,z]\mapsto [a,b,c\,;\alpha(x),\beta(y),\gamma(z)],\text{ for all }a,b,c\in \R,x,y,z\in\oct_{\R},\]
    where $\alpha,\beta,\gamma\in\sorth(\oct_{\R})$ satisfy
\begin{align}\label{eqn isotopy triple condition real octonion}
    \overline{\alpha(x)\beta(y)}=\gamma(\overline{xy}) \text{ for all }x,y\in\oct_{\R}.
\end{align}

    The image of $\left(\begin{smallmatrix}
        1&0&0\\
        0&0&1\\
        0&-1&0
    \end{smallmatrix}\right)\in\sorth(3)$ in $\lietype{F}{4}$ is the map 
    \[[a,b,c\,;x,y,z]\mapsto [a,c,b\,;-\overline{x},-\overline{z},\overline{y}].\]
    The fact that it commutes with $g$ implies that 
    $\alpha(\overline{x})=\overline{\alpha(x)}$ and $\beta(\overline{x})=\overline{\gamma(x)}$ for all $x\in \oct_{\R}$.
    By symmetry we get $\alpha=\beta=\gamma$ and \eqref{eqn isotopy triple condition real octonion} shows that 
    \[\alpha(x)\alpha(y)=\overline{\alpha(\overline{xy})}=\alpha(\overline{\overline{xy}})=\alpha(xy),\text{ for all }x,y\in\oct_{\R}.\]
    Hence $\alpha\in\lietype{G}{2}$ and we have proved that $\mathrm{C}_{\lietype{F}{4}}(\sorth(3))=\lietype{G}{2}$,
    thus the centralizer of $\lietype{G}{2}\times\sorth(3)$ in $\lietype{F}{4}$ is the center of $\lietype{G}{2}$,
    which is trivial.
\end{proof}

\subsubsection{The principal \texorpdfstring{$\psu(2)$}{}}\label{section principal SU2 inside F4}
The image of the principal embedding from $\su(2)$ into $\mathrm{F}_{4}$
is also a maximal proper connected subgroup of $\mathrm{F}_{4}$. 
The restriction of the irreducible representation $\mathrm{J}_{0}$ to this $\su(2)$ is isomorphic to 
\[\sym^{8}\st+\sym^{16}\st,\]
where $\st$ is the standard $2$-dimensional representation of $\su(2)$. 
This implies that the image is isomorphic to $\psu(2)$,
and we call it the principal $\psu(2)$ of $\lietype{F}{4}$.

By the general property of principal embeddings, 
its centralizer is the center of $\mathrm{F}_{4}$.
It is well-known that the center of $\lietype{F}{4}$ is trivial.

\subsection{Classification of \texorpdfstring{$\mathrm{A}_{1}$}{}-subgroups}\label{section A1 subgroups classification}
 In this subsection we will classify \emph{$\mathrm{A}_{1}$-subgroups} of $\mathrm{F}_{4}$,
i.e. subgroups that are isomorphic to $\su(2)$ or $\psu(2)$.
By \cite[Theorem 9.3]{Dynkin1952} every $\lietype{A}{1}$-subgroup $X$ of $\mathrm{F}_{4}$ is either the principal $\psu(2)$ or an $R$-subgroup, 
i.e. $X$ is contained in some proper regular subgroup of $\mathrm{F}_{4}$.
When $X$ is an $R$-subgroup, 
up to conjugacy it is contained in one of the three regular maximal proper connected subgroups of $\mathrm{F}_{4}$ we have found in \Cref{section maximal subgroups of F4}.
All these three regular subgroups arise from classical groups, 
thus their $\mathrm{A}_{1}$-subgroups are well-known. 

By \Cref{prop just one rep is enough for conjugacy},
a conjugacy class of $\lietype{A}{1}$-groups of $\lietype{F}{4}$ is determined uniquely by the restriction of the $26$-dimensional representation $\jord_{0}$ to it.
\begin{notation}
    An isomorphism class of $n$-dimensional representation of $\su(2)$ gives a partition of the integer $n$.
    We will use the notation $[N^{k_{N}},(N-1)^{k_{N-1}},\ldots,2^{k_{2}},1^{k_{1}}]$, where $k_{N}\neq 0$ and $\sum_{i=1}^{N}ik_{i}=n$,
    for a partition of $n$.
    For example, 
    the restriction of $\jord_{0}$ to the principal $\psu(2)$ is isomorphic to $\sym^{8}\st+\sym^{16}\st$,
    thus we index this $\lietype{A}{1}$-subgroup by the partition $[17,9]$ of $\dim \jord_{0}=26$.
\end{notation}

\subsubsection{\texorpdfstring{$\mathrm{A}_{1}$}{}-subgroups of \texorpdfstring{$\spin(9)$}{}}\label{section A1 subgroups of B4 subgroup}
 We start from $\lietype{A}{1}$-subgroups of $\sorth(9)$.
According to \cite[Theorem 5.1.2]{NilOrbit},
the conjugacy classes of morphisms $\su(2)\rightarrow \sorth(9)$
are in bijection with partitions of $9$ in which each even number appears even times.
\begin{lemma}\label{lemma A1 subgroups of Spin(9)}
    (1) There are $12$ different conjugacy classes of $\lietype{A}{1}$-subgroups of $\spin(9)$,
    which correspond to the following partitions of $9$:
    \[[9],[7,1^{2}],[5,3,1],[5,2^{2}],[5,1^{4}],[4^{2},1],[3^{3}],[3^{2},1^{3}],[3,2^{2},1^{2}],[3,1^{6}],[2^{4},1],[2^{2},1^{5}].\]
    (2) There are $10$ different conjugacy classes of $\lietype{A}{1}$-subgroups of $\lietype{F}{4}$ that are contained in the subgroup $\spin(9)$ given in \Cref{section B4 subgroup}.
    The restrictions of the $26$-dimensional irreducible representation $\jord_{0}$ of $\lietype{F}{4}$ to these $\lietype{A}{1}$-subgroups correspond to the following partitions of $26$:
    \begin{equation}\begin{aligned}\label{eqn A1 subgroups of B4}
        [11,9,5,1],[7^{3},1^{5}]&,[5^{3},3^{3},1^{2}],[3^{6},1^{8}],\\
        [5^{2},4^{2},3,2^{2},1],[5,4^{4},1^{5}],[4^{2},3^{3},2^{4},1]&,[3^{3},2^{6},1^{5}],[3,2^{8},1^{7}],[2^{6},1^{14}].
    \end{aligned}\end{equation}
\end{lemma}
\begin{proof}
    By the lifting property of covering maps and the fact that $\su(2)$ is simply connected,
    every $\lietype{A}{1}$-subgroup of $\sorth(9)$ is lifted uniquely to an $\lietype{A}{1}$-subgroup of $\spin(9)$. 
    The assertion (1) follows directly from \cite[Theorem 5.1.2]{NilOrbit},
    and the assertion (2) follows from the equivalence \eqref{eqn restriction of st rep to B4}.
\end{proof}
The $\mathrm{A}_{1}$-subgroups in the first row of \eqref{eqn A1 subgroups of B4} are isomorphic to $\psu(2)$
and the $\lietype{A}{1}$-subgroups in the second row are isomorphic to $\su(2)$.

\subsubsection{\texorpdfstring{$\mathrm{A}_{1}$}{}-subgroups of \texorpdfstring{$\left(\symp(1)\times \symp(3)\right)/\mu_{2}^{\Delta}$}{}}\label{section A1 subgroups of A1C3 subgroup}
 We apply the same argument for $\mathrm{A}_{1}$-subgroups of $\left(\symp(1)\times\symp(3)\right)/\mu_{2}^{\Delta}$.
By \cite[Theorem 5.1.3]{NilOrbit},
the set of conjugacy classes of morphisms $\su(2)\rightarrow \symp(3)$
are in bijection with partitions of $6$
in which each odd number appears even times.
\begin{lemma}\label{lemma A1 subgroups of A1C3}
    (1) There are $7$ different conjugacy classes of $\lietype{A}{1}$-subgroups of $\symp(3)$,
    which correspond to the following partitions of $6$:
    \[[6],[4,2],[4,1^{2}],[3^{2}],[2^{3}],[2^{2},1^{2}],[2,1^{4}].\]
    (2) There are $11$ different conjugacy classes of $\lietype{A}{1}$-subgroups of $\lietype{F}{4}$ that are contained in the subgroup $\left(\symp(1)\times\symp(3)\right)/\mu_{2}^{\Delta}$ given in \Cref{section A1 subgroups of A1C3 subgroup}.
    The restrictions of the $26$-dimensional irreducible representation $\jord_{0}$ of $\lietype{F}{4}$ to these $\lietype{A}{1}$-subgroups correspond to the following partitions of $26$:
    \begin{equation}
    \begin{aligned}\label{eqn A1 subgroups of A1C3}
        [9,7,5^{2}],[5^{3},3^{3},1^{2}]&,[5,3^{7}],[3^{6},1^{8}],\\
        [9,6^{2},5],[5^{2},4^{2},3,2^{2},1],[5,4^{4},1^{5}],[5,4^{2},3^{3},2^{2}]&,[3^{3},2^{6},1^{5}],[3,2^{8},1^{7}],[2^{6},1^{14}].
    \end{aligned}
    \end{equation}
\end{lemma}
\begin{proof}
    The assertion (1) follows directly from \cite[Theorem 5.1.3]{NilOrbit}.
    A morphism from $\su(2)$ to $\left(\symp(1)\times\symp(3)\right)/\mu_{2}^{\Delta}$ 
    arises from the product of two morphisms $\su(2)\rightarrow\symp(1)$ and $\su(2)\rightarrow\symp(3)$.
    The assertion (2) follows from the equivalence \eqref{eqn restriction of st rep to A1C3}.
\end{proof}
The $\mathrm{A}_{1}$-subgroups in the first row of \eqref{eqn A1 subgroups of A1C3} are isomorphic to $\psu(2)$ and the $\lietype{A}{1}$-subgroups in the second row are isomorphic to $\su(2)$. 


\subsubsection{\texorpdfstring{$\mathrm{A}_{1}$}{} subgroups of \texorpdfstring{$\left(\su(3)\times\su(3)\right)/\mu_{3}^{\Delta}$}{}}\label{section A1 subgroups of A2A2 subgroup}
 The restriction of the standard representation $\vrep{3}$ of $\su(3)$ to an $\mathrm{A}_{1}$-subgroup of $\su(3)$ can only be $[3]$ or $[2,1]$.
By the equivalences \eqref{eqn restriction of st rep to A2A2} and \eqref{eqn restriction of Ad rep to A2A2},
we have the following result:
\begin{lemma}\label{lemma A1 subgroups of A2A2}
    There are $8$ different conjugacy classes of $\lietype{A}{1}$-subgroups of $\lietype{F}{4}$ that are contained in the subgroup 
    $\left(\su(3)\times\su(3)\right)/\mu_{3}^{\Delta}$ given in \Cref{section A2A2 subgroup}.
    The restrictions of the $26$-dimensional irreducible representation $\jord_{0}$ of $\lietype{F}{4}$ to these $\lietype{A}{1}$-subgroups correspond to the following partitions of $26$:
    \begin{equation}
        \begin{aligned}\label{eqn A1 subgroups of A2A2}
            [5^{3},3^{3},1^{2}]&,[5,3^{7}],[3^{6},1^{8}]\\
            [5,4^{2},3^{3},2^{2}],[4^{2},3^{3},2^{4},1]&,[3^{3},2^{6},1^{5}],[3,2^{8},1^{7}],[2^{6},1^{14}].
        \end{aligned}
    \end{equation}
\end{lemma}
The $\mathrm{A}_{1}$-subgroups in the first row of \eqref{eqn A1 subgroups of A2A2} are isomorphic to $\psu(2)$ 
and subgroups in the second row are isomorphic to $\su(2)$. 

\subsubsection{Conclusion}\label{section conclusion of classification of A1 subgroups}
 Now we have enumerated (up to conjugacy) all $\mathrm{A}_{1}$-subgroups of $\mathrm{F}_{4}$ and indexed them by the restriction of the $26$-dimensional irreducible representation $\mathrm{J}_{0}$ of $\lietype{F}{4}$.
\begin{prop}\label{prop A1 subgroups of F4}
    (1) There are $7$ conjugacy classes of subgroups of $\lietype{F}{4}$ that are isomorphic to $\psu(2)$,
    corresponding to the following partitions of $26$:
    \[[17,9],[11,9,5,1],[9,7,5^{2}],[7^{3},1^{5}],[5^{3},3^{3},1^{2}],[5,3^{7}],[3^{6},1^{8}].\]
    (2) There are $7$ conjugacy classes of subgroups of $\lietype{F}{4}$ that are isomorphic to $\su(2)$,
    corresponding to the following partitions of $26$:
    \[[9,6^{2},5],[5^{2},4^{2},3,2^{2},1],[5,4^{4},1^{5}],[5,4^{2},3^{3},2^{2}],[4^{2},3^{3},2^{4},1],[3^{3},2^{6},1^{5}],[3,2^{8},1^{7}],[2^{6},1^{14}].\]
\end{prop}
The theory of Jacobson-Morozov shows that 
the set of conjugacy classes of morphisms $\su(2)\rightarrow \lietype{F}{4}$
is in bijection with  
the set of nilpotent orbits of the semisimple Lie algebra $\mathfrak{f}_{4}$.
The nilpotent orbits of $\mathfrak{f}_{4}$ are labeled in \cite[\S 8.4]{NilOrbit},
and we will use the same labelings for $\lietype{A}{1}$-subgroups of $\lietype{F}{4}$:
\begin{table}[H]
    \centering
    \renewcommand{\arraystretch}{1.5}
    \begin{tabular}{|c|c||c|c||c|c|}
    \hline 
    Label & Restriction of $\jord_{0}$ & Label & Restriction of $\jord_{0}$& Label & Restriction of $\jord_{0}$ \\ \hline
    $\lietype{A}{1}$ & $[2^{6},1^{14}]$ & $\lietype{A}{2}+\widetilde{\lietype{A}{1}}$ & $[4^{2},3^{3},2^{4},1]$ & $\lietype{B}{3}$ & $[7^{3},1^{5}]$\\ \hline
    $\widetilde{\lietype{A}{1}}$ & $[3,2^{8},1^{7}]$ & $\lietype{B}{2}$ & $[5,4^{4},1^{5}]$ & $\lietype{C}{3}$ & $[9,6^{2},5]$ \\ \hline
    $\lietype{A}{1}+\widetilde{\lietype{A}{1}}$ & $[3^{3},2^{6},1^{5}]$ & $\widetilde{\lietype{A}{2}}+\lietype{A}{1}$ & $[5,4^{2},3^{3},2^{2}]$ & $\lietype{F}{4}(a_{2})$ & $[9,7,5^{2}]$\\ \hline 
    $\lietype{A}{2}$ & $[3^{6},1^{8}]$ & $\lietype{C}{3}(a_{1})$ & $[5^{2},4^{2},3,2^{2},1]$ & $\lietype{F}{4}(a_{1})$ & $[11,9,5,1]$\\ \hline 
    $\widetilde{\lietype{A}{2}}$ & $[5,3^{7}]$ & $\lietype{F}{4}(a_{3})$ & $[5^{3},3^{3},1^{2}]$ & $\lietype{F}{4}$ & $[17,9]$ \\ \hline 
    \end{tabular}
    \caption{Labels of $\lietype{A}{1}$-subgroups of $\lietype{F}{4}$}\label{table labels of A1 subgroups}
\end{table}
\begin{notation}\label{notation A1 subgroup}
    With \Cref{table labels of A1 subgroups},
    for a conjugacy class of $\lietype{A}{1}$-subgroups of $\lietype{F}{4}$,
    we have two ways to refer to it.
    For example, for the conjugacy class of principal $\psu(2)$,
    we call it the class $[17,9]$ or the class with label $\lietype{F}{4}$.
\end{notation}

\subsubsection{Centralizers}\label{section centralizers of A1 subgroups}
 The next thing we are going to do is to compute the centralizer, 
or the neutral component of the centralizer, 
of each $\mathrm{A}_{1}$-subgroup of $\mathrm{F}_{4}$.
In the following paragraphs, 
we choose a representative $\su(2)\rightarrow \mathrm{F}_{4}$ for each conjugacy class of $\lietype{A}{1}$-subgroups,
whose image is denoted by $X$,
and then determine $\mathrm{C}_{\lietype{F}{4}}(X)$ or $\mathrm{C}_{\lietype{F}{4}}(X)^{\circ}$. 

The following lemma will be used when computing the centralizer of a subgroup in $\lietype{F}{4}$:
\begin{lemma}\label{lemma compute the centralizer}
    Let $G$ be the quotient of a Lie group $G_{0}$ by a finite central subgroup $\Gamma$.
    If $H_{0}$ is a connected subgroup of $G_{0}$, whose image in $G$ is denoted by $H$,
    then the inverse image of $\mathrm{C}_{G}(H)$ in $G_{0}$ is $\mathrm{C}_{G_{0}}( H_{0})$ and $\mathrm{C}_{G}(H)\simeq \mathrm{C}_{G_{0}}( H_{0})/\Gamma$.
\end{lemma}
\begin{proof}
    It suffices to prove that any $ g_{0}\in G_{0}$ whose image $g$ lies in $\mathrm{C}_{G}(H)$ centralizes $ H_{0}$. 
    For any $ h_{0}\in H_{0}$ with image $h$ in $H$, we have $ghg^{-1}h^{-1}=1$ in $G$,
    thus $ g_{0} h_{0} g_{0}^{-1} h_{0}^{-1}\in \Gamma$.
    The continuous map
    $\varphi: H_{0}\rightarrow \Gamma,  h_{0}\mapsto  g_{0} h_{0} g_{0}^{-1} h_{0}^{-1}\text{ for } h_{0}\in H_{0}$
    must be constant,
    because $ H_{0}$ is connected and $\Gamma$ is discrete as a finite group.
    The map $\varphi$ sends $1\in H_{0}$ to $1\in\Gamma$, thus $\varphi\equiv 1$, which implies that $ g_{0}$ centralizes $ H_{0}$ in $G_{0}$.
\end{proof}
In some cases we can not compute the centralizer $\mathrm{C}_{\lietype{F}{4}}(X)$ easily, 
then we use the following lemma to determine its neutral component $\mathrm{C}_{\lietype{F}{4}}(X)^{\circ}$:
\begin{lemma}\label{lemma determine the neutral centralizer}
    Let $H$ be a connected subgroup of a compact Lie group $G$,
    and $d$ the multiplicity of $\triv$ in the restriction of the adjoint representation $\mathfrak{g}$ of $G$ to $H$,
    then $\mathrm{C}_{G}(H)^{\circ}$ is the unique connected $d$-dimensional subgroup of $C_{G}(H)$.    
    In particular, the centralizer $\mathrm{C}_{G}(H)$ is a finite group when $d=0$. 
\end{lemma}
\begin{proof}
    This follows directly from the fact $d=\dim \mathrm{C}_{G}(H)$.
    %
\end{proof}

\paragraph{\texorpdfstring{$[17,9]$}{}}
We choose $X$ to be the principal $\psu(2)$ in $\mathrm{F}_{4}$, whose centralizer in $\lietype{F}{4}$ is trivial.

\paragraph{\texorpdfstring{$[11,9,5,1]$}{}}
We choose $X$ to be the principal $\psu(2)$ of the $\spin(9)$ given in \Cref{section B4 subgroup}.
The restriction of the adjoint representation $\mathfrak{f}_{4}$ of $\mathrm{F}_{4}$ to $X$ corresponds to the partition $[15,11^{2},7,5,3]$ of $52$,
which implies that $\mathrm{C}_{\mathrm{F}_{4}}(X)$ is a finite group by \Cref{lemma determine the neutral centralizer}. 

\paragraph{\texorpdfstring{$[9,7,5^{2}]$}{}}
We choose $X$ to be the principal $\psu(2)$ of the $\left(\symp(1)\times\symp(3)\right)/\mu_{2}^{\Delta}$ given in \Cref{section C3A1 subgroup}.
The restriction of the adjoint representation $\mathfrak{f}_{4}$ to $X$ corresponds to the partition $[11^{2},9,7,5,3^{3}]$ of $52$,
thus $\mathrm{C}_{\lietype{F}{4}}(X)$ is a finite group by \Cref{lemma determine the neutral centralizer}.

\paragraph{\texorpdfstring{$[7^{3},1^{5}]$}{}}
We choose $X$ to be the principal $\psu(2)$ of the factor $\lietype{G}{2}$ in the subgroup $\lietype{G}{2}\times\sorth(3)$ given in \Cref{section G2A1 subgroup}.
The other factor $\sorth(3)$ of $\lietype{G}{2}\times \sorth(3)$ centralizes this $\lietype{A}{1}$-subgroup $X$.
The restriction of the adjoint representation $\mathfrak{f}_{4}$ of $\lietype{F}{4}$ to $X$ corresponds to the partition $[11,7^{5},3,1^{3}]$ of $52$,
thus $\mathrm{C}_{\lietype{F}{4}}(X)^{\circ}$ is the $\sorth(3)$ in $\lietype{G}{2}\times\sorth(3)$ by \Cref{lemma determine the neutral centralizer},
which is in the class $[5,3^{7}]$ and labeled by $\widetilde{\lietype{A}{2}}$.

\paragraph{\texorpdfstring{$[5^{3},3^{3},1^{2}]$}{}}\label{section A1 subgroup principal in A2A2}
We choose $X$ to be the principal $\psu(2)$ of the $\left(\su(3)\times\su(3)\right)/\mu_{3}^{\Delta}$ given in \Cref{section A2A2 subgroup}.
The restriction of the adjoint representation $\mathfrak{f}_{4}$ of $\lietype{F}{4}$ to $X$ corresponds to the partition $[7^{2},5^{4},3^{6}]$ of $52$,
thus $\mathrm{C}_{\mathrm{F}_{4}}(X)$ is a finite group by \Cref{lemma determine the neutral centralizer}. 
The center of $\left(\su(3)\times\su(3)\right)/\mu_{3}^{\Delta}$,
which is a cyclic group of order $3$,
is contained in $\mathrm{C}_{\lietype{F}{4}}(X)$.

\paragraph{\texorpdfstring{$[5,3^{7}]$}{}}\label{section A1 subgroup centralizing G2}
We choose $X$ to be the factor $\sorth(3)$ in the subgroup $\lietype{G}{2}\times\sorth(3)$ of $\lietype{F}{4}$ given in \Cref{section G2A1 subgroup}.
In the proof of \Cref{prop centralizer of A1+G2 is trivial},
we have shown that the centralizer $\mathrm{C}_{\lietype{F}{4}}(X)$ is the other factor $\lietype{G}{2}$.

\paragraph{\texorpdfstring{$[3^{6},1^{8}]$}{}}
We choose $X$ to be the principal $\psu(2)$ of the second copy of $\su(3)$ in the subgroup $\left(\su(3)\times\su(3)\right)/\mu_{3}^{\Delta}$ given in \Cref{section A2A2 subgroup}.
The first copy of $\su(3)$ centralizes $X$ and has dimension $8$.
The restriction of the adjoint representation $\mathfrak{f}_{4}$ of $\lietype{F}{4}$ to $X$ corresponds to the partition $[5,3^{13},1^{8}]$ of $52$,
thus $\mathrm{C}_{\mathrm{F}_{4}}(X)^{\circ}$ is the first copy of $\su(3)$ in $\left(\su(3)\times\su(3)\right)/\mu_{3}^{\Delta}$ by \Cref{lemma determine the neutral centralizer},
whose roots are short roots of $\lietype{F}{4}$.

\paragraph{\texorpdfstring{$[9,6^{2},5]$}{}}
We choose $X_{0}$ to be the principal $\su(2)$ of $\symp(3)$,
and $X$ to be the image of $X_{0}$ in the subgroup $\left(\symp(1)\times\symp(3)\right)/\mu_{2}^{\Delta}$ given in \Cref{section C3A1 subgroup}.
The group $\left(\symp(1)\times\symp(3)\right)/\mu_{2}^{\Delta}$ is defined as $\mathrm{C}_{\lietype{F}{4}}(\gamma)$,
where $\gamma$ is an involution in $\lietype{F}{4}$ and is the image of $(1,-\mathrm{I}_{3})\in \symp(1)\times\symp(3)$ in the quotient group.

Since $X$ contains the element $\gamma$,
the centralizer of $X$ in $\lietype{F}{4}$ is contained in $\mathrm{C}_{\lietype{F}{4}}(\gamma)=\left(\symp(1)\times\symp(3)\right)/\mu_{2}^{\Delta}$,
thus $\mathrm{C}_{\lietype{F}{4}}(X)=\mathrm{C}_{\left(\symp(1)\times\symp(3)\right)/\mu_{2}^{\Delta}}(X)$.
By \Cref{lemma compute the centralizer}, we have: 
\[\mathrm{C}_{\left(\symp(1)\times\symp(3)\right)/\mu_{2}^{\Delta}}(X)=\mathrm{C}_{\symp(1)\times\symp(3)}(1\times X_{0})/\mu_{2}^{\Delta}=(\symp(1)\times \mathrm{Z}(\symp(3)))/\mu_{2}^{\Delta}\simeq \symp(1).\]
Hence $\mathrm{C}_{\lietype{F}{4}}(X)$ is an $\lietype{A}{1}$-subgroup in the class $[2^{6},1^{14}]$ and labeled by $\lietype{A}{1}$.

\paragraph{\texorpdfstring{$[5^{2},4^{2},3,2^{2},1]$}{}}
We choose $X_{0}$ to be the image of
\[\su(2)\hookrightarrow \symp(1)\times \symp(2)\hookrightarrow \symp(3),\] 
where the first arrow is the principal morphism of $\symp(1)\times\symp(2)$, 
and the second is defined as $(x,A)\mapsto \left(\begin{smallmatrix}
    x&0\\
    0&A
\end{smallmatrix}\right)$, for any $x\in\symp(1),A\in\symp(2)$.
Let $X$ be the image of $X_{0}$ in $\left(\symp(1)\times\symp(3)\right)/\mu_{2}^{\Delta}=\mathrm{C}_{\lietype{F}{4}}(\gamma)$.

The element $\gamma$ corresponds to $(1,-\mathrm{I}_{3})$ in $\symp(1)\times\symp(3)$,
thus it is contained in $X$,
so $\mathrm{C}_{\lietype{F}{4}}(X)\subset \mathrm{C}_{\lietype{F}{4}}(\gamma)$ and $\mathrm{C}_{\lietype{F}{4}}(X)=\mathrm{C}_{\left(\symp(1)\times\symp(3)\right)/\mu_{2}^{\Delta}}(X)$.
Again by \Cref{lemma compute the centralizer}, we have:
\[\mathrm{C}_{\left(\symp(1)\times\symp(3)\right)/\mu_{2}^{\Delta}}(X)=\mathrm{C}_{\symp(1)\times \symp(3)}\left(1\times X_{0}\right)/\mu_{2}^{\Delta}=\left(\symp(1)\times \langle\gamma_{1}\rangle\times \langle\gamma_{2}\rangle\right)/\mu_{2}^{\Delta},\]
where $\gamma_{1}=\left(\begin{smallmatrix}
    -1&0&0\\
    0&1&0\\
    0&0&1
\end{smallmatrix}\right)$ and $\gamma_{2}=\left(\begin{smallmatrix}
    1&0&0\\
    0&-1&0\\
    0&0&-1
\end{smallmatrix}\right)$ are two order $2$ elements in $\symp(3)$. 
Hence $\mathrm{C}_{\mathrm{F}_{4}}(X)$ is the product of $\symp(1)$ and an order $2$ group,
and this $\lietype{A}{1}$-subgroup $\symp(1)$ is in the class $[2^{6},1^{14}]$ and labeled by $\lietype{A}{1}$.

\paragraph{\texorpdfstring{$[5,4^{4},1^{5}]$}{}}
We choose a morphism:
\[\su(2)\hookrightarrow \spin(5)\hookrightarrow \spin(5)\times \spin(4)\rightarrow \spin(9)\hookrightarrow \mathrm{F}_{4},\] 
where the first arrow is the principal morphism of $\spin(5)$,
and the subgroup $\spin(9)$ of $\lietype{F}{4}$ is defined as $\mathrm{C}_{\lietype{F}{4}}(\sigma)$ in \Cref{section B4 subgroup}.
This morphism is injective since the factor $\spin(5)$ has zero intersection with the kernel of $\spin(5)\times\spin(4)\rightarrow \spin(9)$,
and we denote its image by $X$.

The element $\sigma$ defined in \Cref{section B4 subgroup} is contained in $X$, 
hence the centralizer of $X$ in $\lietype{F}{4}$ is contained in $\spin(9)$,
thus $\mathrm{C}_{\lietype{F}{4}}(X)=\mathrm{C}_{\spin(9)}(X)$.
Denote the natural projection $\spin(9)\rightarrow \sorth(9)$ by $p$.
The centralizer of $p(X)$ in $\sorth(9)$ is $\orth(4)$,
the image of $\mathrm{Pin}(4)$ under $p$.
By \Cref{lemma compute the centralizer}, we have 
\[\mathrm{C}_{\lietype{F}{4}}(X)=\mathrm{C}_{\spin(9)}(X)= p^{-1}(\orth(4))= \mathrm{Pin}(4),\]
and $\mathrm{C}_{\lietype{F}{4}}(X)^{\circ}\simeq \spin(4)$ is the product of two $\lietype{A}{1}$-subgroups in the class $[2^{6},1^{14}]$.

\paragraph{\texorpdfstring{$[5,4^{2},3^{3},2^{2}]$}{}}
We choose an embedding:
\[\su(2)\hookrightarrow \symp(1)\times \sorth(3)\hookrightarrow \symp(1)\times\symp(3),\]
where the first arrow is the principal morphism of $\symp(1)\times \sorth(3)$,
and the embedding $\sorth(3)\rightarrow \symp(3)$ is given by viewing an orthogonal $3\times 3$ matrix as an matrix in $\GL(3,\mathbb{H})$ preserving the standard Hermitian form on $\mathbb{H}^{3}$.
Let $X_{0}$ be the image of this embedding,
and $X$ the image of $X$ in the subgroup $\left(\symp(1)\times\symp(3)\right)/\mu_{2}^{\Delta}=\mathrm{C}_{\lietype{F}{4}}(\gamma)$ of $\lietype{F}{4}$ given in \Cref{section C3A1 subgroup}.

The group $X_{0}$ contains $(-1,\mathrm{I}_{3})$,
thus the element $\gamma$ is contained in $X$.
So the centralizer $\mathrm{C}_{\lietype{F}{4}}(X)$ is contained in $\mathrm{C}_{\lietype{F}{4}}(\gamma)$ and 
$\mathrm{C}_{\lietype{F}{4}}(X)=\mathrm{C}_{\left(\symp(1)\times\symp(3)\right)/\mu_{2}^{\Delta}}(X)$.
By \Cref{lemma compute the centralizer}, we have 
\[ \mathrm{C}_{\left(\symp(1)\times\symp(3)\right)/\mu_{2}^{\Delta}}(X)=\left(\mathrm{Z}(\symp(1))\times \mathrm{C}_{\symp(3)}(\sorth(3))\right)/\mu_{2}^{\Delta}\simeq \mathrm{C}_{\symp(3)}\left(\sorth(3)\right).\]
A $3\times 3$ matrix in $\symp(3)$ commutes with all elements in $\sorth(3)$ if and only if it is a scalar matrix,
thus it must be of the form $h\cdot\mathrm{I}_{3}$ for some norm $1$ element $h\in \mathbb{H}$.
Hence $C_{\mathrm{F}_{4}}(X)\simeq \symp(1)$ is an $\lietype{A}{1}$-subgroup in the class $[3^{3},2^{6},1^{5}]$ and labeled by $\lietype{A}{1}+\widetilde{\lietype{A}{1}}$.

\paragraph{\texorpdfstring{$[4^{2},3^{3},2^{4},1]$}{}}\label{section A1 subgroup Spin(3)+Spin(3) into Spin(9)}
We choose a morphism:
\[\spin(3)\hookrightarrow \spin(3)\times \spin(3)\times \spin(3)\rightarrow\spin(9)=\mathrm{C}_{\lietype{F}{4}}(\sigma)\hookrightarrow \mathrm{F}_{4},\] 
where the first arrow is the diagonal embedding.
This is also an embedding and we denote its image in $\lietype{F}{4}$ by $X$.

Again we have $\mathrm{C}_{\lietype{F}{4}}(X)=\mathrm{C}_{\spin(9)}(X)$,
and by \Cref{lemma compute the centralizer}, 
the centralizer of $X$ in $\spin(9)$ 
is the inverse image in $\spin(9)$ of 
the subgroup
\begin{align*}
    \left\{\left(\begin{matrix}
        a_{11}\mathrm{I}_{3}&a_{12}\mathrm{I}_{3}&a_{13}\mathrm{I}_{3}\\
        a_{21}\mathrm{I}_{3}&a_{22}\mathrm{I}_{3}&a_{23}\mathrm{I}_{3}\\
        a_{31}\mathrm{I}_{3}&a_{32}\mathrm{I}_{3}&a_{33}\mathrm{I}_{3}
    \end{matrix}\right)\,\middle\vert\, \left(\begin{matrix}
        a_{11} &a_{12} &a_{13}\\
        a_{21} & a_{22} & a_{23}\\
        a_{31} & a_{32} & a_{33}
    \end{matrix}\right)\in\sorth(3)\right\}
\end{align*}
of $\sorth(9)$.
Hence $\mathrm{C}_{\lietype{F}{4}}(X)\simeq \spin(3)$ is also an $\lietype{A}{1}$-subgroup in the class $[4^{2},3^{3},2^{4},1]$.

\paragraph{\texorpdfstring{$[3^{3},2^{6},1^{5}]$}{}}\label{section A1 subgroup which is the diagonal quaternionic matrix}
We denote by $X_{0}$ the image of $\symp(1)\hookrightarrow \symp(3)$ given by $h\mapsto h\mathrm{I}_{3}$,
and by $X$ the image of $X_{0}$ under the embedding of $\symp(3)$ into the group $\left(\symp(1)\times\symp(3)\right)/\mu_{2}^{\Delta}=\mathrm{C}_{\lietype{F}{4}}(\gamma)$ given in \Cref{section C3A1 subgroup}.

The element $\gamma=(1,-\mathrm{I}_{3})$ (modulo $\mu_{2}^{\Delta}$) is contained in $X$, 
so the centralizer $\mathrm{C}_{\lietype{F}{4}}(X)$ equals $\mathrm{C}_{\left(\symp(1)\times\symp(3)\right)/\mu_{2}^{\Delta}}(X)$. 
By \Cref{lemma compute the centralizer}, we have 
\[\mathrm{C}_{\left(\symp(1)\times\symp(3)\right)/\mu_{2}^{\Delta}}(X)=\mathrm{C}_{\symp(1)\times\symp(3)}(1\times X_{0})/\mu_{2}^{\Delta}=\left(\symp(1)\times\mathrm{C}_{\symp(3)}(X_{0})\right)/\mu_{2}^{\Delta}.\]
A $3\times 3$ matrix $A\in\symp(3)$ commutes with $h\mathrm{I}_{3}$ for all norm $1$ quaternions $h$, 
if and only if all entries of $A$ are real.
Hence $\mathrm{C}_{\symp(3)}(X_{0})=\GL(3,\R)\cap \symp(3)=\orth(3)$,
and as a result 
$\mathrm{C}_{\lietype{F}{4}}(X)\simeq \symp(1)\times \sorth(3)$
is the product of two $\lietype{A}{1}$-subgroups in the classes $[2^{6},1^{14}]$ and $[5,3^{7}]$ respectively. 
These two $\lietype{A}{1}$-subgroups are labeled by $\lietype{A}{1}$ and $\widetilde{\lietype{A}{2}}$ respectively.

\paragraph{\texorpdfstring{$[3,2^{8},1^{7}]$}{}}
We choose a morphism:
\[\spin(3)\hookrightarrow \spin(3)\times \spin(6)\rightarrow \spin(9)=\mathrm{C}_{\lietype{F}{4}}(\sigma)\hookrightarrow \mathrm{F}_{4},\]
which is injective,
and denote by $X$ its image in $\lietype{F}{4}$.

The element $\sigma$ is contained in $X$, 
thus $\mathrm{C}_{\mathrm{F}_{4}}(X)=\mathrm{C}_{\spin(9)}(X)$.
Again by \Cref{lemma compute the centralizer}, 
this centralizer is the group $\mathrm{Pin}(6)$ in the morphism we choose. 

\paragraph{\texorpdfstring{$[2^{6},1^{14}]$}{}}
We choose $X$ to be the factor $\symp(1)$ in the $\left(\symp(1)\times\symp(3)\right)/\mu_{2}^{\Delta}$ given in \Cref{section C3A1 subgroup}.
Using \Cref{lemma compute the centralizer},
we obtain that the centralizer $\mathrm{C}_{\lietype{F}{4}}(X)$ is the other factor $\symp(3)$.

\subsection{Connected simple subgroups}\label{section classification simple subgroups}
 In this subsection, we will classify connected simple subgroups of $\mathrm{F}_{4}$ whose ranks are larger than $1$,
and then determine their centralizers in $\mathrm{F}_{4}$.

Let $H$ be a proper connected simple subgroup of $\mathrm{F}_{4}$ whose rank is larger than $1$.
It is (up to conjugacy) contained in one of 
the following four maximal proper connected subgroups classified in \Cref{section maximal subgroups of F4}:
\[\spin(9),\left(\symp(1)\times\symp(3)\right)/\mu_{2}^{\Delta},\left(\su(3)\times\su(3)\right)/\mu_{3}^{\Delta},\mathrm{G}_{2}\times\sorth(3).\]
Moreover, by \cite[Theorem 14.2]{Dynkin1952} the group $\mathrm{F}_{4}$ has no simple $S$-subgroup except the principal $\psu(2)$,
so we have:
\begin{lemma}\label{lemma possible embeddings}
    Let $H$ be a proper connected simple subgroup of $\lietype{F}{4}$ with $\operatorname{rank}H\geq 2$,
    then up to conjugacy $H$ is contained in one of the following fixed subgroups of $\lietype{F}{4}$:
    \[\spin(9),\symp(3),\left(\su(3)\times\su(3)\right)/\mu_{3}^{\Delta}.\]
\end{lemma}
\begin{rmk}\label{rmk G2 factor contained in B4}
    The simple subgroup $\lietype{G}{2}$ is not in this list,
    because it is conjugate to a subgroup of $\spin(7)\hookrightarrow \spin(9)$ \cite[Theorem 1.16.4]{Yokota2009ExceptionalLG}.    
\end{rmk}
The possible Lie types for $H$ are:
\[\mathrm{A}_{2},\mathrm{A}_{3},\mathrm{A}_{4},\mathrm{B}_{2},\mathrm{B}_{3},\mathrm{B}_{4},\mathrm{C}_{3},\mathrm{C}_{4},\mathrm{D}_{4},\mathrm{G}_{2}.\]
\begin{prop}\label{prop exclude two maximal rank Lie types}
    There are no connected subgroups of $\mathrm{F}_{4}$ whose Lie type is $\mathrm{A}_{4}$ or $\mathrm{C}_{4}$.
\end{prop}
\begin{proof}
    Suppose that $\lietype{F}{4}$ admits a connected subgroup $H$ with type $\lietype{A}{4}$ or $\lietype{C}{4}$.
    Since $\mathrm{rank}(H)=4$,
    by \Cref{lemma possible embeddings}
    there exists an embedding of $H$ into $\spin(9)$.
    
    The case that $H$ is of type $\lietype{C}{4}$ is impossible,
    because $\dim H=36=\dim \spin(9)$ but $H$ and $\spin(9)$ have different Lie types.
    Hence $H$ has type $\lietype{A}{4}$.
    The morphism $H\hookrightarrow \spin(9)\rightarrow\sorth(9)$ gives $H$ a self-dual $9$-dimensional representation of $H$,
    which leads to contradiction 
    since the $\lietype{A}{4}$-type group $H$ does not admit such a representation.
\end{proof}
\subsubsection{Cases except \texorpdfstring{$\lietype{A}{2}$}{PDFstring}}\label{section classification of simple subgroups except A2}
 In the remaining possible Lie types for connected simple subgroups of $\lietype{F}{4}$,
the type $\lietype{A}{2}$ is more complicated.
So we first look at the other types:
\begin{prop}\label{prop classification of simple subgroups except A2}
        (1) For each type among
        \[\mathrm{A}_{3},\mathrm{B}_{2},\mathrm{B}_{3},\mathrm{B}_{4},\mathrm{C}_{3},\mathrm{D}_{4},\mathrm{G}_{2},\]
        there exists a simply-connected subgroup of $\lietype{F}{4}$ with this type.
        \\
        (2) Let $H$ be a connected compact Lie group such that 
        it admits an embedding into $\lietype{F}{4}$
        and its Lie type is among
        \[\mathrm{A}_{3},\mathrm{B}_{2},\mathrm{B}_{3},\mathrm{B}_{4},\mathrm{C}_{3},\mathrm{D}_{4},\mathrm{G}_{2}.\]
        Then $H$ is simply-connected and the embedding $H\hookrightarrow \lietype{F}{4}$ is unique up to conjugacy.                 
\end{prop}
Before proving this proposition case by case,
we explain our strategy.
Fixing a Lie type, we first construct an embedding $\phi_{0}$ from the simply-connected compact Lie group $H_{0}$ of the given type into $\lietype{F}{4}$. 
We claim that to prove \Cref{prop classification of simple subgroups except A2}(2) for this Lie type, 
it suffices to show that for any connected simple compact Lie group $H$ of the same type with $H_{0}$,
i.e. $H$ is isomorphic to the quotient of $H_{0}$ by a finite central subgroup,
and any embedding $\phi:H\rightarrow\lietype{F}{4}$,
the restriction of the $26$-dimensional irreducible representation $\jord_{0}$ along $\phi$ is unique, 
up to equivalence of $H_{0}$-representations. 
Here we view the restriction of $\jord_{0}$ along $\phi:H\rightarrow \lietype{F}{4}$ as a representation of $H_{0}$
by the composition with a central isogeny $H_{0}\rightarrow H$.
\begin{proof}[Proof of the claim]
    For a connected compact Lie group $H$ of the same Lie type as $H_{0}$ and an embedding $\phi:H\hookrightarrow \lietype{F}{4}$,
    we can lift $\phi$ to a morphism $\phi\circ i:H_{0}\rightarrow \lietype{F}{4}$ via a central isogeny $i:H_{0}\rightarrow H$.
    This morphism $\phi\circ i$ is conjugate to $\phi_{0}$ by the uniqueness of $\jord_{0}|_{H_{0}}$ and \Cref{prop just one rep is enough for conjugacy},
    thus $i$ is injective, which implies that $H$ is also simply-connected.
    For any two embeddings $\phi,\phi^{\prime}:H\hookrightarrow \lietype{F}{4}$,
    applying \Cref{prop just one rep is enough for conjugacy} to $\phi\circ i$ and $\phi^{\prime}\circ i$,
    we have $\phi\circ i$ and $\phi^{\prime}\circ i$ are conjugate in $\lietype{F}{4}$,
    thus $\phi$ and $\phi^{\prime}$ are conjugate.
\end{proof}

\paragraph{\texorpdfstring{$\mathrm{B}_{4}$}{}}\label{section simple factor B4}
In this case $H_{0}\simeq\spin(9)$ and we take $\phi_{0}$ to be $H_{0}\simeq \spin(9)\hookrightarrow \lietype{F}{4}$, 
where $\spin(9)\hookrightarrow \lietype{F}{4}$ is constructed in \Cref{section B4 subgroup}.

For any embedding $\phi$ from a $\lietype{B}{4}$-type connected compact Lie group $H$ into $\lietype{F}{4}$,
by \Cref{lemma possible embeddings}
the image $\mathrm{Im}(\phi)$ (up to conjugate) is a subgroup of the $\spin(9)$ in $\lietype{F}{4}$,
thus $\phi$ factors through an embedding $H\rightarrow \spin(9)$.
This embedding must be an isomorphism,
so the restrictions of $\jord_{0}$ along $\phi_{0}$ and $\phi$ are equivalent as $H_{0}$-representations.

\paragraph{\texorpdfstring{$\mathrm{D}_{4}$}{}}\label{section simple factor D4}
In this case $H_{0}\simeq \spin(8)$ 
and we take $\phi_{0}$ to be the composition of the natural embedding $\spin_{8}\hookrightarrow \spin(9)$ with $\spin(9)\hookrightarrow \lietype{F}{4}$.

For any embedding $\phi$ from a $\lietype{D}{4}$-type connected compact Lie group $H$ into $\lietype{F}{4}$,
by \Cref{lemma possible embeddings},
$\phi$ (up to conjugacy) factors through an embedding $H\rightarrow \spin(9)$.
The restriction of the $9$-dimensional irreducible representation $\vrep{9}$ to $H$ 
is isomorphic to either $\triv+\vrep{8}$ or $\triv+\vrep{\spin}^{+}$ or $\triv+\vrep{\spin}^{-}$,
where $\vrep{8}$ is the standard $8$-dimensional representation of $\spin(8)$,
and $\vrep{\spin}^{\pm}$ are two $8$-dimensional spinor representations of $\spin(8)$.
For those three possibilities,
we obtain the same equivalence class of $\jord_{0}|_{H}$,
which is equivalent to $\triv^{\oplus 2}+\vrep{8}+\vrep{\spin}^{+}+\vrep{\spin}^{-}$ as $H_{0}$-representations, 
thus the restriction of $\jord_{0}$ along $\phi$
is unique, up to equivalence of $H_{0}$-representations.

\paragraph{\texorpdfstring{$\mathrm{A}_{3}$}{}}\label{section simple factor A3}
In this case $H_{0}\simeq \su(4)$,
and we take $\phi_{0}$ to be the composition of the natural embedding $\su(4)\simeq \spin(6)\hookrightarrow \spin(9)$ with $\spin(9)\hookrightarrow \lietype{F}{4}$.

For any embedding $\phi$ from a $\lietype{A}{3}$-type connected compact Lie group $H$ into $\lietype{F}{4}$,
by \Cref{lemma possible embeddings},
$\phi$ (up to conjugacy) factors through an embedding from $H$ to $\symp(3)$ or $\spin(9)$.

If $\phi$ factors through $\symp(3)$,
then the image of $\phi$ gives a $\lietype{A}{3}$-type subgroup of $\symp(3)$.
This subgroup of $\symp(3)$ must be regular,
but this contradicts with the Borel-de Siebenthal theory.

If $\phi$ factors through $\spin(9)$,
the standard representation $\vrep{9}$ of $\spin(9)$ gives a self-dual $9$-dimensional representation of $H$.
Up to equivalence,
there are two possibilities for the restriction of $\vrep{9}$ to $H$:
\[\triv^{\oplus 3}+\wedge^{2}\vrep{4} \text{ or }\triv+\vrep{4}+\vrep{4}^{\prime},\]
where $\vrep{4}$ is the standard $4$-dimensional representation of $\su(4)$ and $\vrep{4}^{\prime}$ is its dual.
For both cases,
the restriction of the irreducible representation $\jord_{0}$ of $\lietype{F}{4}$ along $\phi$ is isomorphic to 
\[\triv^{\oplus 4}+\vrep{4}^{\oplus 2}+(\vrep{4}^{\prime})^{\oplus 2}+\wedge^{2}\vrep{4}.\]
Hence the restriction of $\jord_{0}$ along $\phi$ is unique,
up to equivalence of $H_{0}$-representations.

\paragraph{\texorpdfstring{$\mathrm{B}_{3}$}{}}\label{section simple factor B3} 
In this case $H_{0}\simeq \spin(7)$,
and we take $\phi_{0}$ to be the composition of the natural embedding $\spin(7)\hookrightarrow \spin(9)$ with $\spin(9)\hookrightarrow \lietype{F}{4}$.

For any embedding $\phi$ from a $\lietype{B}{3}$-type connected compact Lie group $H$ into $\lietype{F}{4}$,
by \Cref{lemma possible embeddings} and the Borel-de Siebenthal theory,
$\phi$ (up to conjugacy) factors through an embedding from $H$ to $\spin(9)$.
The restriction of the standard representation $\vrep{9}$ of $\spin(9)$ to $H$ must be isomorphic to 
either $\triv^{\oplus 2}+\vrep{7}$ or $\triv+\vrep{\spin}$,
where $\vrep{7}$ is the standard $7$-dimensional representation of $\spin(7)$,
and $\vrep{\spin}$ is the $8$-dimensional spinor representation of $\spin(7)$.
For both cases,
the restriction of the irreducible representation $\jord_{0}$ of $\lietype{F}{4}$ along $\phi$ is isomorphic to 
\[\triv^{\oplus 3}+\vrep{7}+\vrep{\spin}^{\oplus 2}.\]
Hence the restriction of $\jord_{0}$ along $\phi$ is unique,
up to equivalence of $H_{0}$-representations.

\paragraph{\texorpdfstring{$\mathrm{C}_{3}$}{}}\label{section simple factor C3}
In this case $H_{0}\simeq \symp(3)$,
and we take $\phi_{0}$ to be $\symp(3)\hookrightarrow\left(\symp(1)\times\symp(3)\right)/\mu_{2}^{\Delta}\hookrightarrow \lietype{F}{4}$,
where the subgroup $\left(\symp(1)\times\symp(3)\right)/\mu_{2}^{\Delta}$ is given in \Cref{section C3A1 subgroup}.

For any embedding $\phi$ from a $\lietype{C}{3}$-type connected compact Lie group $H$ into $\lietype{F}{4}$,
by \Cref{lemma possible embeddings},
$\phi$ (up to conjugacy) factors through a central-kernel morphism from $H_{0}$ to $\symp(3)$ or $\spin(9)$.

If $\phi$ factors through $\spin(9)$,
then the standard representation $\vrep{9}$ of $\spin(9)$
induces an orthogonal $9$-dimensional representation of $\symp(3)$.
However, each non-trivial irreducible orthogonal representation of $\symp(3)$ has dimension larger than $9$,
which leads to a contradiction.

So $\phi$ factors through $\symp(3)$,
then the embedding $H\rightarrow \symp(3)$ must be an isomorphism.
Since every automorphism of $\symp(3)$ is inner,
the restriction of $\jord_{0}$ along $\phi$ is unique,
up to equivalence of $H_{0}$-representations.

\paragraph{\texorpdfstring{$\mathrm{B}_{2}$}{}}\label{section simple factor B2}
In this case $H_{0}\simeq \symp(2)\simeq\spin(5)$,
and we take $\phi_{0}$ to be the composition of the natural embedding 
$\symp(2)\hookrightarrow \symp(3)\hookrightarrow \left(\symp(1)\times\symp(3)\right)/\mu_{2}^{\Delta}$
with the embedding $\left(\symp(1)\times\symp(3)\right)/\mu_{2}^{\Delta}\hookrightarrow \lietype{F}{4}$ given in \Cref{section C3A1 subgroup}.

For any embedding $\phi$ from a $\lietype{B}{2}$-type connected compact Lie group $H$ into $\lietype{F}{4}$,
by \Cref{lemma possible embeddings} and the Borel-de Siebenthal theory,
$\phi$ (up to conjugacy) factors through an embedding from $H$ to $\symp(3)$ or $\spin(9)$.

If $\phi$ factors through $\symp(3)$,
then the restriction of the standard representation $\vrep{6}$ of $\symp(3)$ to $H$
must be isomorphic to $\triv^{\oplus 2}+\vrep{4}$,
where $\vrep{4}$ is the standard $4$-dimensional symplectic representation of $\symp(2)$.
The restriction of the irreducible representation $\jord_{0}$ along $\phi$ is isomorphic to 
$\triv^{\oplus 5}+\vrep{4}^{\oplus 4}+\vrep{5}$,
where $\vrep{5}$ is the standard $5$-dimensional orthogonal representation of $\spin(5)$.

If $\phi$ factors through $\spin(9)$,
then the restriction of the standard representation $\vrep{9}$ to $H$ must be isomorphic to 
$\triv^{\oplus 4}+\vrep{5}$ or $\triv+\vrep{4}^{\oplus 2}$.
For these two possibilities,
the restriction of $\jord_{0}$ along $\phi$ is isomorphic to $\triv^{\oplus 5}+\vrep{4}^{\oplus 4}+\vrep{5}$.
Hence the restriction of $\jord_{0}$ along $\phi$ is unique,
up to equivalence of $H_{0}$-representations.

\paragraph{\texorpdfstring{$\mathrm{G}_{2}$}{}}\label{section simple factor G2}
In this case $H_{0}\simeq \lietype{G}{2}$,
and we take $\phi_{0}$ to be the embedding $\lietype{G}{2}\hookrightarrow\lietype{G}{2}\times\sorth(3)\hookrightarrow \lietype{F}{4}$,
as given in \Cref{section G2A1 subgroup}.

Combining \Cref{lemma possible embeddings} and the fact that all non-trivial representations of $\lietype{G}{2}$ have dimension larger than $6$,
any embedding $\phi$ from a $\lietype{G}{2}$-type connected compact Lie group $H$ into $\lietype{F}{4}$
(up to conjugacy) factors through an embedding from $H$ to $\spin(9)$.
The restriction of the standard representation $\vrep{9}$ of $\spin(9)$ to $H$ must be isomorphic to $\triv^{\oplus 2}+ \vrep{7}$,
where $\vrep{7}$ is the same as in \eqref{eqn restriction of the st rep to G2A1}.
So the restriction of the representation $\jord_{0}$ of $\lietype{F}{4}$ along $\phi$
must be isomorphic to $\triv^{\oplus 5}+\vrep{7}^{\oplus 3}$.
Hence the restriction of $\jord_{0}$ along $\phi$ is unique,
up to equivalence of $H_{0}$-representations.

\subsubsection{The case \texorpdfstring{$\mathrm{A}_{2}$}{}}\label{section simple factor A2}
 For the Lie type $\lietype{A}{2}$,
our idea is the same with the proof of \Cref{prop classification of simple subgroups except A2},
but this time we have several conjugacy classes of embeddings from a $\lietype{A}{2}$-type group to $\lietype{F}{4}$.
\begin{prop}\label{prop classification of A2 subgroups}
    (1) There are $3$ conjugacy classes of embeddings from $\su(3)$ to $\lietype{F}{4}$,
    \\
    (2) There is a unique conjugacy class of embeddings from $\psu(3)=\su(3)/\mathrm{Z}(\su(3))$ to $\lietype{F}{4}$.
\end{prop}
\begin{proof}
    By \Cref{lemma possible embeddings},
    any embedding $\phi$ from a connected $\lietype{A}{2}$-type compact Lie group $H$ to $\lietype{F}{4}$
    (up to conjugacy) factors through $\spin(9)$ or $\symp(3)$ or $\left(\su(3)\times\su(3)\right)/\mu_{2}^{\Delta}$.

    We start from the case that $\phi$ factors through $\left(\su(3)\times\su(3)\right)/\mu_{3}^{\Delta}$,
    and fix an embedding $\iota:\left(\su(3)\times\su(3)\right)/\mu_{3}^{\Delta}\hookrightarrow\lietype{F}{4}$.
    It is easy to classify the conjugacy classes of embeddings $\psi:H\hookrightarrow \left(\su(3)\times\su(3)\right)/\mu_{3}^{\Delta}$,
    where $H$ is a connected $\lietype{A}{2}$-type compact Lie group, i.e. $H\simeq \su(3)$ or $\psu(3)$.
    Recall that $\su(3)$ has a non-trivial outer automorphism $\theta$.
    Set $N$ to be the normalizer of $\left(\su(3)\times\su(3)\right)/\mu_{3}^{\Delta}$ in $\lietype{F}{4}$.
    The group $N$ contains the element $\widetilde{\theta}$, whose conjugation on $\left(\su(3)\times\su(3)\right)/\mu_{3}^{\Delta}$ induces the outer automorphism $(\theta,\theta)$.
    We list the $N$-conjugacy classes of embeddings $\psi:H\hookrightarrow \left(\su(3)\times\su(3)\right)/\mu_{3}^{\Delta}$,
    together with the restriction of $\jord_{0}$ along $\iota\circ\psi$:
    \begin{table}[H]
        \centering
        \renewcommand{\arraystretch}{1.3}
        \begin{tabular}{|c|c|c|c|}
        \hline 
        Index &$H$ & $\psi$ & The restriction of $\jord_{0}$ along $\phi=\iota\circ\psi$ \\ \hline
        $1$&$\su(3)$ & $g\mapsto (g,1)$ & $(\vrep{3}+\vrep{3}^{\prime})^{\oplus 3}+\mathfrak{sl}_{3}$ \\ \hline 
        $2$&$\su(3)$ & $g\mapsto (1,g)$ & $\triv^{\oplus 8}+(\vrep{3}+\vrep{3}^{\prime})^{\oplus 3}$\\ \hline
        $3$&$\psu(3)$& $g\mapsto (g,g)$ & $\triv^{\oplus 2}+\mathfrak{sl}_{3}^{\oplus 3}$\\ \hline
        $4$&$\su(3)$ & $g\mapsto (g,\theta(g))$ & $\vrep{3}+\vrep{3}^{\prime}+\sym^{2}\vrep{3}+\sym^{2}\vrep{3}^{\prime}+\mathfrak{sl}_{3}$ \\ \hline
    \end{tabular}
    \caption{$N$-conjugacy classes of embeddings from $\lietype{A}{2}$-type connected compact Lie groups to $(\su(3)\times\su(3))/\mu_{3}^{\Delta}$}\label{table A2 subgroups in A2+A2}
    \end{table}
    
    If $\phi$ factors through $\symp(3)$,
    the standard representation $\vrep{6}$ of $\symp(3)$ gives a  self-dual $6$-dimensional representation of $H$,
    thus the restriction of $\vrep{6}$ to $H$ must be isomorphic to $\vrep{3}+\vrep{3}^{\prime}$.
    So the restriction of $\jord_{0}$ to $H$
    is isomorphic to $(\vrep{3}+\vrep{3}^{\prime})^{\oplus 3}+\mathfrak{sl}_{3}$.

    If $\phi$ factors through $\spin(9)$,
    the standard representation $\vrep{9}$ of $\spin(9)$ gives a self-dual $9$-dimensional representation of $H$,
    thus the restriction of $\vrep{9}$ to $H$ must be isomorphic to $\triv^{\oplus 3}+\vrep{3}+\vrep{3}^{\prime}$ or $\triv+\mathfrak{sl}_{3}$.
    For the first case, the restriction of $\jord_{0}$ to $H$ is isomorphic to $\triv^{\oplus 8}+(\vrep{3}+\vrep{3}^{\prime})^{\oplus 3}$,
    and for the second case, the restriction of $\jord_{0}$ to $H$ is isomorphic to $\triv^{\oplus 2}+\mathfrak{sl}_{3}^{\oplus 3}$.

    In conclusion, combining \Cref{prop just one rep is enough for conjugacy} with our analysis on the restriction of $\jord_{0}$,
    we get that every embedding from a connected $\lietype{A}{2}$-type compact Lie group to $\lietype{F}{4}$ 
    is conjugate to one of the embeddings $\phi=\iota\circ \psi$ in \Cref{table A2 subgroups in A2+A2}.    
\end{proof}
\subsubsection{Centralizers}\label{section centralizers of simple subgroups}
 Similarly with the arguments in \Cref{section centralizers of A1 subgroups},
using \Cref{lemma compute the centralizer} and \Cref{lemma determine the neutral centralizer},
for each conjugacy class of embeddings from a connected simple compact Lie group to $\lietype{F}{4}$,
we can determine its centralizer in $\lietype{F}{4}$:
\begin{itemize}
    \item Type $\lietype{B}{4}$: the centralizer is a cyclic group of order $2$.
    \item Type $\lietype{D}{4}$: the centralizer is isomorphic to $\Z/2\Z\times\Z/2\Z$.
    \item Type $\lietype{A}{3}$: the centralizer is an $\lietype{A}{1}$-subgroup in the class $[3,2^{8},1^{7}]$, which is labeled by $\widetilde{\lietype{A}{1}}$.
    \item Type $\lietype{B}{3}$: the centralizer is the semi-direct product of a rank $1$ torus with a cyclic group of order $2$.
    \item Type $\lietype{C}{3}$: the centralizer is an $\lietype{A}{1}$-subgroup in the class $[2^{6},1^{14}]$, which is labeled by $\lietype{A}{1}$.
    \item Type $\lietype{B}{2}$: the centralizer is isomorphic to $\mathrm{Pin}(4)$, and its connected component is the direct product of two $\lietype{A}{1}$-subgroups in the class $[2^{6},1^{14}]$.
    \item Type $\lietype{G}{2}$: the centralizer is an $\lietype{A}{1}$-subgroup in the class $[5,3^{7}]$, which is labeled by $\widetilde{\lietype{A}{2}}$.
    \item Type $\lietype{A}{2}$: Let $\phi:H\hookrightarrow \lietype{F}{4}$ be a representative of a conjugacy class of embeddings listed in \Cref{table A2 subgroups in A2+A2},
    which is indexed by a number from $1$ to $4$.
    \begin{enumerate}[label=(\arabic*)]
        \item If $\phi$ is indexed by $1$, then its centralizer is conjugate to the $\su(3)$ indexed by $2$.
        \item If $\phi$ is indexed by $2$, then its centralizer is conjugate to the $\su(3)$ indexed by $1$.
        \item If $\phi$ is indexed by $3$, then its centralizer is finite and contains an order $3$ element.
        \item If $\phi$ is indexed by $4$, then its centralizer is a cyclic group of order $3$.
    \end{enumerate}
\end{itemize}

\subsection{Connected subgroups satisfying certain conditions}\label{section connected subgroups satisfying our conditions}
 After a long journey of classifying conjugacy classes of connected simple subgroups of $\mathrm{F}_{4}$ 
and computing their centralizers in $\lietype{F}{4}$,
we are finally able to enumerate all the connected subgroups $H$ of $\mathrm{F}_{4}$ satisfying our three conditions listed in the beginning of \Cref{section subgroups of F4}.

We first classify all the connected subgroups $H$ of $\mathrm{F}_{4}$ such that 
$\mathrm{C}_{\mathrm{F}_{4}}(H)$ is an elementary finite abelian $2$-group,
via our classifications in \Cref{section A1 subgroups classification} and \Cref{section classification simple subgroups}.
\begin{notation}
    From now on, for an $\lietype{A}{1}$-subgroup of $\lietype{F}{4}$,
    if its conjugacy class corresponds to the partition $p$ of $26$,
    we will simply denote this $\lietype{A}{1}$-subgroup by $\lietype{A}{1}^{p}$.
    For example, we will denote the principal $\psu(2)$ of $\lietype{F}{4}$ by $\lietype{A}{1}^{[17,9]}$.
    For an $\lietype{A}{2}$-type subgroup of $\lietype{F}{4}$,
    if its conjugacy class is indexed by $n\in\{1,2,3,4\}$ in \Cref{section simple factor A2} \Cref{table A2 subgroups in A2+A2},
    then we denote it simply by $\lietype{A}{2}^{(n)}$.
\end{notation}
Now let $H$ be a connected subgroup of $\lietype{F}{4}$ whose centralizer in $\lietype{F}{4}$ is an elementary finite abelian $2$-group.
Let $\Phi$ be the root system of $H$,
and we can write it as a disjoint union of irreducible root systems:
\[\Phi=\Phi_{1}\sqcup \cdots\sqcup \Phi_{s}.\] 
We denote by $m$ the number of $i\in\{1,2,\ldots,s\}$ such that $\Phi_{i}\simeq \lietype{A}{1}$.
\begin{lemma}\label{lemma subgroup one irreducible root system}
    If $s=1$, i.e. $H$ is simple,
    then $H$ is conjugate to one of the following subgroups of $\lietype{F}{4}$:
    \[\mathrm{F}_{4},\spin(9),\spin(8),\mathrm{A}_{1}^{[17,9]},\mathrm{A}_{1}^{[11,9,5,1]},\mathrm{A}_{1}^{[9,7,5^{2}]}.\]
\end{lemma}
\begin{proof}
    By our computations in \Cref{section centralizers of A1 subgroups} and \Cref{section centralizers of simple subgroups},
    we have if the centralizer of $H$ in $\lietype{F}{4}$ is finite,
    then it must be conjugate to one of the following subgroups of $\lietype{F}{4}$:
    \[\mathrm{F}_{4},\spin(9),\spin(8),\mathrm{A}_{2}^{(3)},\mathrm{A}_{2}^{(4)},\mathrm{A}_{1}^{[17,9]},\mathrm{A}_{1}^{[11,9,5,1]},\mathrm{A}_{1}^{[9,7,5^{2}]},\mathrm{A}_{1}^{[5^{3},3^{3},1^{2}]}.\]
    According to \Cref{section A1 subgroup principal in A2A2} and \Cref{section centralizers of simple subgroups},
    if $H$ is in the conjugacy class of $\lietype{A}{2}^{(3)},\lietype{A}{2}^{(4)}$ or $\lietype{A}{1}^{[5^{3},3^{3},1^{2}]}$,
    then the centralizer of $H$ in $\lietype{F}{4}$ contains an element of order $3$.
\end{proof}
\begin{lemma}\label{lemma subgroup more than one irreducible root system no A1 facter}
    If $s>1$ and $m=0$,
    then there is no such $H$ satisfying $\mathrm{C}_{\lietype{F}{4}}(H)$ is an elementary finite abelian $2$-group.
\end{lemma}
\begin{proof}
    Since $s>1$ and $m=0$,
    the irreducible root systems $\Phi_{1}$ and $\Phi_{2}$ both have rank $2$ and $s=2$.
    Hence $H$ must be isomorphic to the quotient of $\su(3)\times\su(3)$ by a finite central subgroup.
    By our classification in \Cref{section simple factor A2},
    $H$ is conjugate to the subgroup $\left(\su(3)\times\su(3)\right)/\mu_{3}^{\Delta}$ introduced in \Cref{section A2A2 subgroup}.
    However, the centralizer of this $H$ contains an order $3$ element.
\end{proof}
\begin{lemma}\label{lemma subgroup two irreducible root system}
    If $s=2$ and $m\geq 1$, then $H$ is conjugate to one of the following subgroups of $\lietype{F}{4}$,
    and the embedding from $H$ to $\lietype{F}{4}$ is unique up to conjugacy:
    \begin{gather*}
        \left(\lietype{A}{1}^{[2^{6},1^{14}]}\times\symp(3)\right)/\mu_{2}^{\Delta},\lietype{A}{1}^{[5,3^{7}]}\times\lietype{G}{2}, \left(\mathrm{A}_{1}^{[5,4^{2},3^{3},2^{2}]}\times \mathrm{A}_{1}^{[3^{3},2^{6},1^{5}]}\right)/\mu_{2}^{\Delta},\mathrm{A}_{1}^{[7^{3},1^{5}]}\times \mathrm{A}_{1}^{[5,3^{7}]},\\
         \left(\lietype{A}{1}^{[9,6^{2},5]}\times \lietype{A}{1}^{[2^{6},1^{14}]}\right)/\mu_{2}^{\Delta}, \left(\lietype{A}{1}^{[5^{2},4^{2},3,2^{2},1]}\times\lietype{A}{1}^{[2^{6},1^{14}]}\right)/\mu_{2}^{\Delta},\left(\mathrm{A}_{1}^{[4^{2},3^{3},2^{4},1]}\times \mathrm{A}_{1}^{[4^{2},3^{3},2^{4},1]}\right)/\mu_{2}^{\Delta}.
    \end{gather*}
\end{lemma}
\begin{proof}
    Since $s=2$ and $m\geq 1$,
    up to conjugacy $H$ is of the form $(X\times H_{0})/\Gamma$,
    where $X$ is an $\lietype{A}{1}$-subgroup of $\lietype{F}{4}$,
    $H_0$ is a connected simple subgroup of $\lietype{F}{4}$,
    and $\Gamma$ is either trivial or the subgroup $\mu_{2}^{\Delta}$ of $X\times H_{0}$.
    Since the centralizer of $H$ in $\lietype{F}{4}$ is an elementary finite abelian $2$-groups,
    the centralizer of $H_{0}$ in $\mathrm{C}_{\lietype{F}{4}}(X)$ and 
    the centralizer of $X$ in $\mathrm{C}_{\lietype{F}{4}}(X)$ are both elementary finite abelian $2$-groups.
    Hence $\mathrm{C}_{\lietype{F}{4}}(X)$ has positive rank,
    and has no central element with order larger than $2$,
    thus we can exclude the following classes for $X$:
    \[[17,9],\,[11,9,5,1],\,[9,7,5^{2}],\,[5^{3},3^{3},1^2],\,[3^{6},1^{8}],\,[3,2^{8},1^{7}].\]

    Now we use \Cref{section centralizers of A1 subgroups} to do the classification case by case. 
    For any class of $\mathrm{A}_{1}$ subgroup,
    we fix an embedding $X\hookrightarrow \lietype{F}{4}$,
    then the uniqueness (up to conjugacy) of embeddings $H\hookrightarrow \lietype{F}{4}$ is equivalent to that of embeddings $H_{0}\hookrightarrow \mathrm{C}_{\lietype{F}{4}}(X)$.
    \begin{enumerate}[label=(\roman*)]
        \item If $X=\mathrm{A}_{1}^{[7^{3},1^{5}]}$, then $\mathrm{C}_{\lietype{F}{4}}(X)^{\circ}$ is an $\lietype{A}{1}$-subgroup in the class $[5,3^{7}]$. So $H_{0}=\mathrm{C}_{\lietype{F}{4}}(X)$ and $\Gamma$ is trivial.
        \item If $X=\mathrm{A}_{1}^{[5,3^{7}]}$, then $\mathrm{C}_{\lietype{F}{4}}(X)\simeq\lietype{G}{2}$. 
        We can take $H_{0}$ to be either $\lietype{G}{2}$, 
        or the principal $\sorth(3)$ of $\lietype{G}{2}$ (in the class $[7^{3},1^{5}]$), which is the unique $\mathrm{A}_{1}$-subgroup of $\lietype{G}{2}$ with elementary finite abelian $2$-group centralizer. 
        In both cases, $\Gamma$ is trivial.
        \item If $X=\lietype{A}{1}^{[9,6^{2},5]}$, then $\mathrm{C}_{\lietype{F}{4}}(X)$ is an $\lietype{A}{1}$-subgroup in the class $[2^{6},1^{14}]$. So $H_{0}=\mathrm{C}_{\lietype{F}{4}}(X)$ and $\Gamma=\mu_{2}^{\Delta}$.
        \item If $X=\lietype{A}{1}^{[5^{2},4^{2},3,2^{2},1]}$, then $\mathrm{C}_{\lietype{F}{4}}(X)^{\circ}$ is an $\lietype{A}{1}$-subgroup in the class $[2^{6},1^{14}]$.
        So $H_{0}=\mathrm{C}_{\lietype{F}{4}}(X)^{\circ}$ and $\Gamma=\mu_{2}^{\Delta}$.
        \item If $X=\lietype{A}{1}^{[5,4^{4},1^{5}]}$, then $\mathrm{C}_{\lietype{F}{4}}(X)^{\circ}$ is the product of two $\lietype{A}{1}$-subgroups in the class $[2^{6},1^{14}]$.
        So $H_{0}$ can only be the diagonal $\su(2)$ in $\mathrm{C}_{\lietype{F}{4}}(X)^{\circ}$,
        which is in the class $[3,2^{8},1^{7}]$,
        and $\Gamma=\mu_{2}^{\Delta}$.
        However, 
        the centralizer of $H_{0}$ in $\lietype{F}{4}$ contains an order $4$ central element.
        \item If $X=\lietype{A}{1}^{[5,4^{2},3^{3},2^{2}]}$, then $\mathrm{C}_{\lietype{F}{4}}(X)$ is an $\lietype{A}{1}$-subgroup in the class $[3^{3},2^{6},1^{5}]$.
        So $H_{0}=\mathrm{C}_{\lietype{F}{4}}(X)$ and $\Gamma=\mu_{2}^{\Delta}$. 
        \item If $X=\lietype{A}{1}^{[4^{2},3^{3},2^{4},1]}$, then $\mathrm{C}_{\lietype{F}{4}}(X)$ is an $\lietype{A}{1}$-subgroup in the same class. 
        So $H_{0}=\mathrm{C}_{\lietype{F}{4}}(X)$ and $\Gamma=\mu_{2}^{\Delta}$.
        \item If $X=\lietype{A}{1}^{[3^{3},2^{6},1^{5}]}$, then $\mathrm{C}_{\lietype{F}{4}}(X)$ is the product of two $\lietype{A}{1}$-subgroups in the classes $[2^{6},1^{14}]$ and $[5,3^{7}]$ respectively.
        So $H_{0}$ is the diagonal $\su(2)$ in $\mathrm{C}_{\lietype{F}{4}}(X)$, in the class $[5^{2},4^{2},3,2^{2},1]$, and $\Gamma=\mu_{2}^{\Delta}$.
        \item If $X=\lietype{A}{1}^{[2^{6},1^{14}]}$, then $\mathrm{C}_{\lietype{F}{4}}(X)\simeq \symp(3)$.
        We can take $H_{0}$ to be either $\symp(3)$, the principal $\su(2)$ in $\symp(3)$ (in the class $[9,6^{2},5]$),
        or the principal $\su(2)$ of $\symp(1)\times \symp(2)\subseteq \symp(3)$ (in the class $[5^{2},4^{2},3,2^{2},1]$),
        and in all the cases $\Gamma=\mu_{2}^{\Delta}$. \qedhere
    \end{enumerate}
\end{proof}
\begin{lemma}\label{lemma subgroup more than two irreducible root system}
    If $s>2$, then $H$ is conjugate to one of the following subgroups of $\lietype{F}{4}$,
    and the embedding from $H$ to $\lietype{F}{4}$ is unique up to conjugacy:
    \begin{gather*}
         \left(\mathrm{A}_{1}^{[2^{6},1^{14}]}\times \mathrm{A}_{1}^{[2^{6},1^{14}]}\times \symp(2)\right)/\mu_{2}^{\Delta},\\
         \mathrm{A}_{1}^{[5,3^{7}]}\times \left(\mathrm{A}_{1}^{[3^{3},2^{6},1^{5}]}\times \mathrm{A}_{1}^{[2^{6},1^{14}]}\right)/\mu_{2}^{\Delta},\\
         \left(\mathrm{A}_{1}^{[5,4^{4},1^{5}]}\times \mathrm{A}_{1}^{[2^{6},1^{14}]}\times \mathrm{A}_{1}^{[2^{6},1^{14}]}\right)/\mu_{2}^{\Delta},\\
         \prod_{i=1}^{4}\lietype{A}{1}^{[2^{6},1^{14}]}/\mu_{2}^{\Delta}:=\left(\mathrm{A}_{1}^{[2^{6},1^{14}]}\times \mathrm{A}_{1}^{[2^{6},1^{14}]}\times \mathrm{A}_{1}^{[2^{6},1^{14}]}\times \mathrm{A}_{1}^{[2^{6},1^{14}]}\right)/\mu_{2}^{\Delta}.
    \end{gather*}
\end{lemma}
\begin{proof}
    This follows from a similar argument as in the proof of \Cref{lemma subgroup two irreducible root system}
    and the results in \Cref{section centralizers of A1 subgroups} and \Cref{section centralizers of simple subgroups}.
\end{proof}
In \Cref{lemma subgroup one irreducible root system}, \Cref{lemma subgroup more than one irreducible root system no A1 facter}, \Cref{lemma subgroup two irreducible root system} and \Cref{lemma subgroup more than two irreducible root system},
we have enumerated $17$ candidates for the conjugacy classes of connected subgroups $H$ of $\lietype{F}{4}$ such that the $\mathrm{C}_{\lietype{F}{4}}(H)$ is an elementary finite abelian $2$-group.
Some of them do not satisfy the third condition given in the beginning of \Cref{section subgroups of F4}:
\begin{lemma}\label{lemma subgroups not satisfying the third condition}
    If a subgroup $H$ of $\lietype{F}{4}$ is conjugate to one of the following subgroups:
    \begin{gather*}
    \mathrm{A}_{1}^{[11,9,5,1]}, 
    \mathrm{A}_{1}^{[9,7,5^{2}]},
    \left(\mathrm{A}_{1}^{[5^{2},4^{2},3,2^{2},1]}\times \mathrm{A}_{1}^{[2^{6},1^{14}]}\right)/\mu_{2}^{\Delta},
    \end{gather*}
    then the zero weight appears $4$ times in the restriction of the $26$-dimensional irreducible representation $\jord_{0}$ of $\lietype{F}{4}$ to $H$.
\end{lemma}
\begin{proof}
    The restrictions of the representation $\mathrm{J}_{0}$ of $\lietype{F}{4}$ to the two $\lietype{A}{1}$-subgroups in the list above
    can be read from their corresponding partitions.
    In both cases, the multiplicity of the zero weight in $\mathrm{J}_{0}|_{H}$ is $4$.
%

    If $H$ is conjugate to $\left(\mathrm{A}_{1}^{[5^{2},4^{2},3,2^{2},1]}\times \mathrm{A}_{1}^{[2^{6},1^{14}]}\right)/\mu_{2}^{\Delta}$,
    then the restriction $\mathrm{J}_{0}|_{H}$ is isomorphic to 
    \[\left((\sym^{4}\st)^{\oplus 2}+\sym^{2}\st+\triv\right)\otimes\triv+\left(\sym^{3}\st+\st\right)\otimes\st,\]
    in which the zero weight appears $4$ times.
%
%
\end{proof}
For each of the rest $13$ conjugacy classes of proper connected subgroups $H$ of $\lietype{F}{4}$,
we list the following information in \Cref{section centralizers and restrictions}, 
which will be used in \Cref{section Arthur classification for F4}:
\begin{itemize}
    \item the centralizer $\mathrm{C}_{\lietype{F}{4}}(H)$ of $H$ in $\lietype{F}{4}$,
    \item the restriction of the $26$-dimensional irreducible representation $\mathrm{J}_{0}$ to $H$, 
    \item and the restriction of the adjoint representation $\mathfrak{f}_{4}$ of $\lietype{F}{4}$ to $H$.
\end{itemize}
A direct consequence is that all these $13$ conjugacy classes satisfy the three conditions given in the beginning of \Cref{section subgroups of F4}.
In conclusion, 
we have proved the following theorem:
\begin{thm}\label{thm classification result of subgroups satisfying three conditions}
    There are $13$ conjugacy classes of proper connected subgroups $H$ of $\lietype{F}{4}$ satisfying the following conditions:
    \begin{enumerate}[label=(\arabic*)]
        \item The centralizer of $H$ in $\lietype{F}{4}$ is an elementary finite abelian $2$-group.
        \item The zero weight appears twice in the restriction of the $26$-dimensional irreducible representation $\mathrm{J}_{0}$ of $\lietype{F}{4}$ to $H$.
    \end{enumerate}
    These $13$ subgroups are:
    \begin{gather*}
        \lietype{A}{1}^{[17,9]},
        \spin(9),
        \spin(8),
        \lietype{A}{1}^{[5,3^{7}]}\times \lietype{G}{2},
        \lietype{A}{1}^{[7^{3},1^{5}]}\times\lietype{A}{1}^{[5,3^{7}]},
        \left(\lietype{A}{1}^{[2^{6},1^{14}]}\times\symp(3)\right)/\mu_{2}^{\Delta},\\
        \left(\lietype{A}{1}^{[2^{6},1^{14}]}\times \lietype{A}{1}^{[2^{6},1^{14}]}\times\symp(2)\right)/\mu_{2}^{\Delta},
        \left(\lietype{A}{1}^{9,6^{2},5}\times \lietype{A}{1}^{[2^{6},1^{14}]}\right)/\mu_{2}^{\Delta},
        \left(\lietype{A}{1}^{[5,4^{2},3^{3},2^{2}]}\times \lietype{A}{1}^{[3^{3},2^{6},1^{5}]}\right)/\mu_{2}^{\Delta},\\
        \left(\lietype{A}{1}^{[4^{2},3^{3},2^{4},1]}\times \lietype{A}{1}^{[4^{2},3^{3},2^{4},1]}\right)/\mu_{2}^{\Delta},
        \lietype{A}{1}^{[5,3^{7}]}\times\left(\lietype{A}{1}^{[3^{3},2^{6},1^{5}]}\times \lietype{A}{1}^{[2^{6},1^{14}]}\right)/\mu_{2}^{\Delta},\\
        \left(\lietype{A}{1}^{[5,4^{4},1^{5}]}\times \lietype{A}{1}^{[2^{6},1^{14}]}\times \lietype{A}{1}^{[2^{6},1^{14}]}\right)/\mu_{2}^{\Delta},
        \prod_{i=1}^{4}\lietype{A}{1}^{[2^{6},1^{14}]}/\mu_{2}^{\Delta}.
    \end{gather*}
\end{thm}
\subsection{Centralizers and restrictions}
\label{section centralizers and restrictions}
\subsubsection{\texorpdfstring{$\mathrm{A}_{1}^{[17,9]}$}{}}\label{section info principal PSU(2)}
 This is the principal $\psu(2)$ of $\mathrm{F}_{4}$, 
whose centralizer in $\lietype{F}{4}$ is trivial.
The restriction of $\mathrm{J}_{0}$ to $H$ corresponds to the partition $[17,9]$ of $26$,
and the restriction of $\mathfrak{f}_{4}$ to $H$ corresponds to the partition $[23,15,11,3]$ of $52$. 

\subsubsection{\texorpdfstring{$\spin(9)$}{}}\label{section info Spin(9)}
 The centralizer of $H$ in $\lietype{F}{4}$ is the center of $H$,
which is isomorphic to $\Z/2\Z$.

The restriction of $\mathrm{J}_{0}$ to $H$ is isomorphic to 
\[\triv+\vrep{9}+\vrep{\spin},\]
and the restriction of $\mathfrak{f}_{4}$ to $H$ is isomorphic to 
\[\wedge^{2}\vrep{9}+\vrep{\spin},\]
where $\vrep{9}$ is the standard representation of $\spin(9)$ 
and $\vrep{\spin}$ is the $16$-dimensional spinor representation.

\subsubsection{\texorpdfstring{$\left(\mathrm{A}_{1}^{[2^{6},1^{14}]}\times\symp(3)\right)/\mu_{2}^{\Delta}$}{}}\label{section info A1+Sp(3)}
 The centralizer of $H$ in $\lietype{F}{4}$ is the center of $H$,
which is isomorphic to $\Z/2\Z$.

The restriction of $\mathrm{J}_{0}$ to $H$ is isomorphic to
\[\st\otimes \vrep{6}+\triv\otimes \vrep{14},\] 
and the restriction of $\mathfrak{f}_{4}$ to $H$ is isomorphic to 
\[\sym^{2}\st\otimes\triv + \st\otimes \vrep{14}^{\prime}+ \triv\otimes\sym^{2}\vrep{6},\] 
where $\vrep{6}$ is the standard $6$-dimensional representation of $\symp(3)$,
$\vrep{14}$ is the $14$-dimensional irreducible representation of $\symp(3)$ that is a sub-representation of $\wedge^{2}\vrep{6}$,
and $\vrep{14}^{\prime}$ is another $14$-dimensional irreducible representation of $\symp(3)$ that is not equivalent to $\vrep{14}$.
From now on, 
we will denote $\vrep{14}$ by $\wedge^{*}\vrep{6}$, 
and similarly for the $5$-dimensional irreducible representation of $\symp(2)$.
 
\subsubsection{\texorpdfstring{$\mathrm{A}_{1}^{[5,3^{7}]}\times \mathrm{G}_{2}$}{}}\label{section info A1+G2}
 The centralizer of $H$ in $\lietype{F}{4}$ is trivial.

The restriction of $\mathrm{J}_{0}$ to $H$ is isomorphic to 
\[\sym^{2}\st\otimes \vrep{7}+ \sym^{4}\st\otimes\triv,\]
and the restriction of $\mathfrak{f}_{4}$ to this subgroup is isomorphic to 
\[\triv\otimes\mathfrak{g}_{2}+ \sym^{2}\st\otimes\triv+ \sym^{4}\st \otimes \vrep{7},\]
where $\vrep{7}$ is the $7$-dimensional irreducible representation of $\lietype{G}{2}$,
and $\mathfrak{g}_{2}$ is the adjoint representation of $\mathrm{G}_{2}$.

\subsubsection{\texorpdfstring{$\spin(8)$}{}}\label{section info Spin(8)}
 The centralizer of $H$ in $\lietype{F}{4}$ is the center of $H$,
which is isomorphic to $\mathrm{Z}(\spin(8))\simeq \Z_{2}\times\Z_{2}$.

The restriction of $\mathrm{J}_{0}$ to $H$ is isomorphic to 
\[\triv^{\oplus 2}+\vrep{8}+\vrep{\spin}^{+}+\vrep{\spin}^{-},\]
and the restriction of $\mathfrak{f}_{4}$ to $H$ is isomorphic to 
\[\wedge^{2}\vrep{8}+\vrep{8}+\vrep{\spin}^{+}+\vrep{\spin}^{-},\]
where $\vrep{8}$ is the $8$-dimensional vector representation of $\spin(8)$,
i.e. the composition of $\spin(8)\rightarrow\sorth(8)$ with the standard $8$-dimensional representation of $\sorth(8)$,
and $\vrep{\spin}^{\pm}$ are two $8$-dimensional spinor representations.

\subsubsection{\texorpdfstring{$\left(\mathrm{A}_{1}^{[2^{6},1^{14}]}\times \mathrm{A}_{1}^{[2^{6},1^{14}]}\times \symp(2)\right)/\mu_{2}^{\Delta}$}{}}\label{section info A1+A1+Sp(2)}
 The centralizer of $H$ in $\lietype{F}{4}$ is the center of $H$, 
which is isomorphic to 
$\Z/2\Z\times\Z/2\Z$.

The restriction of $\mathrm{J}_{0}$ to $H$ is isomorphic to 
\[\triv+
\st\otimes\st\otimes \triv+
\st\otimes\triv\otimes \vrep{4}+
\triv\otimes\st\otimes \vrep{4}+
\triv\otimes\triv\otimes \wedge^{*}\vrep{4},
\]
and the restriction of $\mathfrak{f}_{4}$ to $H$ is isomorphic to 
\begin{gather*}
    \left(\sym^{2}\st\otimes\triv+\triv\otimes\sym^{2}\st\right)\otimes \triv 
    +\left(\st\otimes\triv+\triv\otimes\st\right)\otimes\vrep{4}\\
    +\st\otimes\st\otimes\wedge^{*}\vrep{4}
    +\triv\otimes\triv\otimes\sym^{2}\vrep{4},
\end{gather*}
where $\vrep{4}$ is the standard representation of $\symp(2)$ and $\wedge^{*}\vrep{4}$ is the $5$-dimensional irreducible representation of $\symp(2)$.

\subsubsection{\texorpdfstring{$\mathrm{A}_{1}^{[7^{3},1^{5}]}\times \mathrm{A}_{1}^{[5,3^{7}]}$}{}}\label{section info principal PSU(2) in G2 + SO(3)}
 The centralizer of $H$ in $\lietype{F}{4}$ is trivial.

The restriction of $\mathrm{J}_{0}$ to $H$ is isomorphic to 
\[\sym^{6}\st\otimes\sym^{2}\st+
\triv\otimes\sym^{4}\st,\]
and the restriction of $\mathfrak{f}_{4}$ to $H$ is isomorphic to 
\[\left(\sym^{10}\st+\sym^{2}\st\right)\otimes\triv+
\triv\otimes\sym^{2}\st+
\sym^{6}\st\otimes\sym^{4}\st.\]

\subsubsection{\texorpdfstring{$\mathrm{A}_{1}^{[5,3^{7}]}\times \left(\mathrm{A}_{1}^{[3^{3},2^{6},1^{5}]}\times \mathrm{A}_{1}^{[2^{6},1^{14}]}\right)/\mu_{2}^{\Delta}$}{}}\label{section info a bizzare  A1+A1+A1 via A1+Sp(3)}
 The centralizer of $H$ in $\lietype{F}{4}$ is the center of $H$, which is a cyclic group of order $2$.

The restriction of $\mathrm{J}_{0}$ to $H$ is isomorphic to 
\[\sym^{4}\st\otimes\triv\otimes\triv+
\sym^{2}\st\otimes \left(\st\otimes \st+\sym^{2}\st\otimes\triv\right),\]
and the restriction of $\mathfrak{f}_{4}$ to $H$ is isomorphic to 
\begin{gather*}
    \sym^{4}\st\otimes \left(\st\otimes \st+\sym^{2}\st\otimes\triv\right)+ 
    \sym^{2}\st\otimes\triv\otimes\triv\\
    +\triv\otimes\left(\sym^{2}\st\otimes\triv+\triv\otimes\sym^{2}\st+\sym^{3}\st\otimes\st\right).
\end{gather*}

\subsubsection{\texorpdfstring{$\left(\mathrm{A}_{1}^{[5,4^{4},1^{5}]}\times \mathrm{A}_{1}^{[2^{6},1^{14}]}\times \mathrm{A}_{1}^{[2^{6},1^{14}]}\right)/\mu_{2}^{\Delta}$}{}}\label{section info principal A1 in Spin(5) with Spin(4)}
 The centralizer of $H$ in $\lietype{F}{4}$ is the center of $H$,
which is isomorphic to $\Z/2\Z\times\Z/2\Z$. 

The restriction of $\mathrm{J}_{0}$ to $H$ is isomorphic to 
\[\triv+\triv\otimes\st\otimes\st+
\sym^{3}\st\otimes\left(\st\otimes\triv+\triv\otimes\st\right)+
\sym^{4}\st\otimes\triv\otimes\triv 
,\]
and the restriction of $\mathfrak{f}_{4}$ to $H$ is isomorphic to
\begin{gather*}
    \triv\otimes\left(\sym^{2}\st\otimes\triv+\triv\otimes\sym^{2}\st\right)+
    \sym^{2}\st\otimes\triv\otimes\triv+
    \sym^{3}\st\otimes\left(\st\otimes\triv+\triv\otimes\st\right)\\
    +\sym^{4}\st\otimes\st\otimes\st+ 
    \sym^{6}\st\otimes\triv\otimes\triv.
\end{gather*}

\subsubsection{\texorpdfstring{$\left(\mathrm{A}_{1}^{[9,6^{2},5]}\times \mathrm{A}_{1}^{[2^{6},1^{14}]}\right)/\mu_{2}^{\Delta}$}{}}\label{section info principal A1 in Sp(3) with Sp(1)}
 The centralizer of $H$ in $\lietype{F}{4}$ is the center of $H$,
which is a cyclic group of order $2$.

The restriction of $\mathrm{J}_{0}$ to $H$ is isomorphic to 
\[\sym^{5}\st\otimes\st
+\left(\sym^{8}\st+\sym^{4}\st\right)\otimes\triv,\]
and the restriction of $\mathfrak{f}_{4}$ to $H$ is isomorphic to 
\begin{align*}
    \triv\otimes\sym^{2}\st+
    \left(\sym^{9}\st+\sym^{3}\st\right)\otimes\st+
    \left(\sym^{10}\st+\sym^{6}\st+\sym^{2}\st\right)\otimes\triv.
\end{align*}

\subsubsection{\texorpdfstring{$\left(\mathrm{A}_{1}^{[5,4^{2},3^{3},2^{2}]}\times \mathrm{A}_{1}^{[3^{3},2^{6},1^{5}]}\right)/\mu_{2}^{\Delta}$}{}}\label{section info principal A1 of Sp(1)+SO(3) with diagonal Sp(1)}
 The centralizer of $H$ in $\mathrm{F}_{4}$ is the center of $H$,
which is a cyclic group of order $2$.

The restriction of $\mathrm{J}_{0}$ to $H$ is isomorphic to 
\[\sym^{4}\st\otimes\triv+
\left(\sym^{3}\st+\st\right)\otimes\st+
\sym^{2}\st\otimes \sym^{2}\st,\]
and the restriction of $\mathfrak{f}_{4}$ to $H$ is isomorphic to 
\begin{align*}
\st\otimes\sym^{3}\st+ 
\left(\sym^{4}\st+\triv\right)\otimes\sym^{2}\st+
\left(\sym^{5}\st+\sym^{3}\st\right)\otimes\st+
\left(\sym^{2}\st\right)^{\oplus 2}\otimes\triv.
\end{align*}

\subsubsection{\texorpdfstring{$\left(\mathrm{A}_{1}^{[4^{2},3^{3},2^{4},1]}\times \mathrm{A}_{1}^{[4^{2},3^{3},2^{4},1]}\right)/\mu_{2}^{\Delta}$}{}}\label{section info Spin(3)+Spin(3)}
 The centralizer of $H$ in $\lietype{F}{4}$ is the center of $H$,
which is a cyclic group of order $2$.

The restriction of $\mathrm{J}_{0}$ to $H$ is isomorphic to 
\[\triv+ 
\sym^{3}\st\otimes\st+
\sym^{2}\st\otimes\sym^{2}\st+
\st\otimes\sym^{3}\st,\]
and the restriction of $\mathfrak{f}_{4}$ to $H$ is isomorphic to 
\begin{align*}
    \left(\sym^{4}\st+\triv\right)\otimes\sym^{2}\st+ 
    \sym^{2}\st\otimes\left(\sym^{4}\st+\triv\right)+
    \sym^{3}\st\otimes\st+ 
    \st\otimes\sym^{3}\st.
\end{align*}

\subsubsection{\texorpdfstring{$\prod_{i=1}^{4}\mathrm{A}_{1}^{[2^{6},1^{14}]}/\mu_{2}^{\Delta}$}{}}\label{section info four copies of A1}
 The centralizer of $H$ in $\lietype{F}{4}$ is the center of $H$,
which is isomorphic to $\Z/2\Z\times\Z/2\Z\times\Z/2\Z$.

The restriction of $\mathrm{J}_{0}$ to $H$ is isomorphic to 
\[\triv^{\oplus 2}+\sum_{\sym}\st\otimes\st\otimes \triv\otimes \triv,\]
where the second term stands for the direct sum of tensor products of standard representations at every two copies of $\lietype{A}{1}^{[2^{6},1^{14}]}$ in $H$.
The restriction of $\mathfrak{f}_{4}$ to $H$ is isomorphic to 
\[\sum_{\sym}\sym^{2}\st\otimes\triv\otimes\triv\otimes\triv+
\sum_{\sym}\st\otimes\st\otimes\triv\otimes\triv+
\st\otimes\st\otimes\st\otimes\st.\]

\section{Arthur's conjectures on automorphic representations}\label{section automorphic representations and Arthur's conjectures}
 In this section,
we are going to review the theory of automorphic representations
and Arthur's conjectures on discrete automorphic representations.
For our purposes,
it is enough to restrict to the special case of level $1$ algebraic automorphic forms of 
a reductive group $G$ over $\Q$ admitting a reductive $\Z$-model,
as in \cites{ChenevierRenard}{ChenevierLannes}.
We mainly follow these two references.

\subsection{A brief review of automorphic representations}\label{section review of automorphic representations}
 
In this subsection we give a quick review on automorphic representations, 
following \cite[\S 4.3]{ChenevierLannes}.
Let $G$ be a connected reductive group over $\Q$ with a reductive $\Z$-model $(\G,\mathrm{id})$,
and $A_{G}$ be the maximal $\Q$-split torus of the center $\mathrm{Z}(G)$ of $G$.
Denote by $G(\A)^{1}$ the quotient of $G(\A)$ by the neutral component of $A_{G}(\R)$,
and consider the adelic quotient 
\[[G]:=G(\Q)\backslash G(\A)^{1}=G(\Q)A_{G}(\R)^{\circ}\backslash G(\A).\]
We have a left $G(\Q)$-invariant right Haar measure $\mu$ on $G(\A)$ by \cite[\S II.9]{WeilGroupeTopo},
and the volume of $[G]$ is finite with respect to this measure.
The topological group $G(\A)$ acts on the space $\mathcal{L}(G):=\mathrm{L}^{2}([G])$ of square-integrable functions on $[G]$ by right translations.
Equipped with the \emph{Petersson inner product} defined as 
\[\langle f,f^{\prime}\rangle:=\int \overline{f}f^{\prime}d\mu,\]
the space $\mathcal{L}(G)$ becomes a unitary representation of $G(\A)$.
We denote the closure of the sum of all closed and topologically irreducible subrepresentations of $\mathcal{L}(G)$ by $\mathcal{L}_{\disc}(G)$.

Denote by $\Pi(G)$ the set of equivalence classes of irreducible unitary complex representations $\pi$ of $G(\A)$
such that $\pi=\pi_{\infty}\otimes\pi_{f}$,
where $\pi_{\infty}$ is an irreducible unitary representation of $G(\R)$,
and $\pi_{f}$ is a smooth irreducible representation of $G(\A_{f})$ satisfying $\pi_{f}^{\G(\widehat{\Z})}\neq 0$.
We have the following decomposition:
\begin{align}\label{eqn decomposition of L2 spectrum}
\mathcal{L}_{\disc}(G)^{\G(\widehat{\Z})}=\overline{\bigoplus_{\pi\in \Pi(G)}}\mathrm{m}(\pi)\,\pi^{\G(\widehat{\Z})}=\overline{\bigoplus_{\pi\in\Pi(G)}}\mathrm{m}(\pi)\,\pi_{\infty}\otimes\pi_{f}^{\G(\widehat{\Z})},
\end{align}
where the integers $\mathrm{m}(\pi)\geq 0$ are finite due to a fundamental result of Harish-Chandra \cite[\S I.2, Theorem 1]{chandra1968automorphic}.
We call the integer $\mathrm{m}(\pi)$ the \emph{multiplicity} of $\pi$ in $\mathcal{L}_{\disc}(G)$.

Now we give the definition of level one discrete automorphic representations,
and refer to \cite[\S 4]{borel1979automorphic} for the general definition of automorphic representations.
\begin{defi}\label{def automorphic representation L2 sense}
    A \emph{level one discrete automorphic representation} is a representation $\pi$ of $G(\A)$ in $\Pi(G)$
    such that its multiplicity $\mathrm{m}(\pi)$ in \eqref{eqn decomposition of L2 spectrum} is nonzero.
    We denote the subset of $\Pi(G)$ consisting of level one discrete automorphic representations by $\Pi_{\disc}(G)$.
\end{defi}
\begin{notation}\label{notation level one assumption}
    Since in this paper we only deal with level one automorphic representations,
    so we will always omit ``level one'' from now on.
\end{notation}
\begin{defi}\label{def cuspidal automorphic representation}
    A square-integrable Borel function $f:[G]\rightarrow \C$ is a \emph{cusp form} if for the unipotent radical $U$ of each proper parabolic subgroup of $G$,
    we have 
    \[\int_{U(\Q)\backslash U(\A)}f(ug)du=0\]
    for almost all $g\in G(\A)$.
    We denote the subspace of $\mathcal{L}(G)$ consisting of the classes of cusp forms by $\mathcal{L}_{\cusp}(G)$.
    A discrete automorphic representation is \emph{cuspidal} if it is a subrepresentation of $\mathcal{L}_{\cusp}(G)$,
    and we denote by $\Pi_{\cusp}(G)$ the subset of $\Pi(G)$ consisting of cuspidal representations.
\end{defi}
\begin{rmk}\label{rmk relations between automorphic spectrums}
    A result of Gelfand, Graev and Piatetski-Shapiro \cite{GGPS66} asserts that 
    \[\mathcal{L}_{\cusp}(G)\subset\mathcal{L}_{\disc}(G)\text{ and }\Pi_{\cusp}(G)\subset\Pi_{\disc}(G).\]
    When $G(\R)$ is compact, every automorphic representation of $G$ is discrete by the Peter-Weyl theorem.
\end{rmk}
Denote by $\mathrm{H}(G)=\bigotimes_{p}\mathrm{H}_{p}(G)$ the \emph{spherical Hecke algebra} of the pair $\left(G(\A_{f}),\G(\widehat{\Z})\right)$.
For any representation $\pi=\pi_{\infty}\otimes\pi_{f}\in\Pi(G)$, 
the space $\pi_{f}^{\G(\widehat{\Z})}$ is an irreducible representation of the spherical Hecke algebra $\mathrm{H}(G)$.
Since $\mathrm{H}(G)$ is commutative \cite[Proposition 2.10]{Satake},
the dimension of $\pi_{f}^{\G(\widehat{\Z})}$ is $1$.
Hence the $\G(\widehat{\Z})$-invariant space of the $\pi$-isotypic subspace $\mathcal{L}_{\disc}(G)_{\pi}$ of $\mathcal{L}_{\disc}(G)$,
as a $G(\R)$-representation,
is the direct sum of $\mathrm{m}(\pi)$ copies of $\pi_{\infty}$.
This implies the following result:
\begin{lemma}\label{lemma dimension of isotypic spaces}
    Let $V$ be an irreducible unitary representation of the Lie group $G(\R)$,
    and $\mathcal{A}_{V}(G)$ the space of $G(\R)$-equivariant linear maps from $V$ to $\mathcal{L}_{\disc}(G)^{\G(\widehat{\Z})}$.
    Then we have the following equality:
    \begin{align}\label{eqn dimension of isotypic spaces}
        \dim \mathcal{A}_{V}(G)=\sum_{\pi\in\Pi(G),\,\pi_{\infty}\simeq V}\mathrm{m}(\pi).
    \end{align}
\end{lemma}
\begin{rmk}\label{rmk view the space equivariant maps as multiplicity space}
    The space $\mathcal{A}_{V}(G)=\Hom_{G(\R)}(V,\mathcal{L}_{\disc}(G)^{\G(\widehat{\Z})})$ can be viewed as the multiplicity space of $V$ in \eqref{eqn decomposition of L2 spectrum}.
\end{rmk}
\subsubsection{Automorphic representations for \texorpdfstring{$\grpF$}{PDFstring}}\label{section automorphic representations of anisotropic groups}
 When the reductive group $G$ has compact real points,
due to \cite{AlgModForm} we can describe the multiplicity space $\mathcal{A}_{V}(G)$ of $V$ in $\mathcal{L}_{\disc}(G)^{\G(\widehat{\Z})}$ in a more computable manner,
which is explained in \cite[\S 4.4.1]{ChenevierLannes}.
Applying \cite[Lemma 4.4.2]{ChenevierLannes} to $\grpF$ and using the fact that every irreducible representation of $\lietype{F}{4}$ is self-dual,
we get:
\begin{prop}\label{prop isomorphism between isotypic spaces and algebraic modular forms}
    Let $(\rho,V)$ be an irreducible representation of $\lietype{F}{4}=\grpF(\R)$.
    The vector space $\mathcal{A}_{V}(\grpF)$ is canonically isomorphic to the following space:
    \[\mathrm{M}_{V}(\grpF):=\set{f:\grpF(\A_{f})/\mathcal{F}_{4,\mathrm{I}}(\widehat{\Z})\rightarrow V}{f(\gamma g)=\rho(\gamma)f(g)\text{ for all }\gamma\in \grpF(\Q),g\in\grpF(\A_{f})}.\]
\end{prop}
We choose a set of representatives $\{1,g_{\mathrm{E}}\}$ of $\grpF(\Q)\backslash\grpF(\A_{f})/\mathcal{F}_{4,\mathrm{I}}(\widehat{\Z})$ corresponding to the two reductive $\Z$-models 
$(\mathcal{F}_{4,\mathrm{I}},\mathrm{id})$ and $(\mathcal{F}_{4,\mathrm{E}},\iota)$ of $\grpF$ in \Cref{prop exactly two reductive models}.
By \cite[Equation (4.4.1)]{ChenevierLannes} the evaluation map $f\mapsto \left(f(1),f(g_{\mathrm{E}})\right)$ induces a bijection:
\[\mathrm{M}_{V}(\grpF)\simeq V^{\mathcal{F}_{4,\mathrm{I}}(\Z)}\oplus V^{\mathcal{F}_{4,\mathrm{E}}(\Z)}.\] 

Combining the results in this section with \Cref{thm main computational result},
we have the following computational result:
\begin{cor}\label{cor dimension formula for the multiplicity space}
    For any dominant weight $\lambda$ of $\lietype{F}{4}$,
    we have an explicit formula for the dimension of $\mathcal{A}_{\mathrm{V}_{\lambda}}(\grpF)$,
    where $\mathrm{V}_{\lambda}$ is the irreducible representation of $\lietype{F}{4}=\grpF(\R)$ with highest weight $\lambda$.
    For $\lambda=(\lambda_{1},\lambda_{2},\lambda_{3},\lambda_{4})$ with $2\lambda_{1}+3\lambda_{2}+2\lambda_{3}+\lambda_{4}\leq 13$,
    the dimension $\dim \mathcal{A}_{\mathrm{V}_{\lambda}}(\grpF)$ equals the $d(\lambda)$ in \Cref{nonzeroAMF}.
\end{cor}

\subsection{Local parametrization of \texorpdfstring{$\Pi(G)$}{PDFstring}}\label{section local parametrization of automorphic rep}
 Let $G$ be a connected reductive group over $\Q$ with a fixed reductive $\Z$-model $(\G,\mathrm{id})$.
Let $\widehat{G}$ be its complex Langlands dual group,
i.e. the root datum of $\widehat{G}$ is the dual root datum of $G$.
A representation $\pi\in\Pi(G)$ can be decomposed as $\pi=\pi_{\infty}\otimes\left(\bigotimes_{p}\pi_{p}\right)$, 
where $\pi_{p}$ is a \emph{spherical} irreducible smooth representation of $G(\Q_{p})$ for each $p$,
i.e. $\pi_{p}^{\G(\Z_{p})}\neq 0$,
and $\pi_{\infty}$ is an irreducible unitary representation of the Lie group $G(\R)$.

In this subsection, we will recall the parametrizations for spherical irreducible smooth representations of $G(\Q_{p})$
and for irreducible unitary representations of $G(\R)$.
Our main reference is \cite[\S 6.2, \S 6.3]{ChenevierLannes}.

\subsubsection{Satake parameter}\label{section Satake parameter}
 For each prime number $p$, 
a spherical irreducible smooth representation $\pi$ of $G(\Q_{p})$
is determined by the action of the spherical Hecke algebra $\mathrm{H}_{p}(G)$ for the pair $\left(G(\Q_{p}),\G(\Z_{p})\right)$
on the subspace of invariants $\pi^{\G(\Z_{p})}$.
Since $\dim \pi^{\G(\Z_{p})}=1$,
the equivalence class of $\pi$ is determined uniquely by the ring homomorphism $\mathrm{H}_{p}(G)\rightarrow\C$
given by the $\mathrm{H}_{p}(G)$-action on $\pi^{\G(\Z_{p})}$.

By \cite[Scholium 6.2.2]{ChenevierLannes},
the \emph{Satake isomorphism} gives a canonical bijection between the set of ring homomorphisms $\mathrm{H}_{p}(G)\rightarrow \C$ 
and the set $\widehat{G}(\C)_{\mathrm{ss}}$ of semisimple conjugacy classes in $\widehat{G}(\C)$.
This induces a bijection $\pi\mapsto \mathrm{c}_{p}(\pi)$ between 
the set of equivalence classes of spherical irreducible smooth representations of $G(\Q_{p})$
and the set $\widehat{G}(\C)_{\mathrm{ss}}$.
The conjugacy class $\mathrm{c}_{p}(\pi)$ is called the \emph{Satake parameter} of $\pi_{p}$.

\subsubsection{Infinitesimal character}\label{section infinitesimal character}
 Let $\mathfrak{g}$ be the Lie algebra of $G(\C)$,
and $\widehat{\mathfrak{g}}$ the Lie algebra of $\widehat{G}(\C)$.
We fix a Cartan subalgebra $\mathfrak{t}$ of $\mathfrak{g}$ and a Borel subalgebra $\mathfrak{b}\subset\mathfrak{g}$ containing $\mathfrak{t}$,
and denote the Weyl group of $\mathfrak{g}$ with respect to $\mathfrak{t}$ by $W$. 

As explained in \cite[\S 6.3.4]{ChenevierLannes},
we can associate a character $\mathrm{Z}(\mathrm{U}(\mathfrak{g}))\rightarrow \C$ to an irreducible unitary representation $(\pi,V)$ of $G(\R)$,
where $\mathrm{Z}(\mathrm{U}(\mathfrak{g}))$ is the center of the universal enveloping algebra of $\mathfrak{g}$.
By \cite[Scholium 6.3.2 and Equation (6.3.1)]{ChenevierLannes},
the \emph{Harish-Chandra isomorphism} induces the following canonical bijections:
\begin{align}\label{eqn Harish-Chandra isomorphism}
    \mathrm{Hom}_{\C\text{-}\mathrm{alg}}(\mathrm{Z}(\mathrm{U}(\mathfrak{g})),\C)\simeq \widehat{\mathfrak{g}}_{\mathrm{ss}}\simeq \left(X^{*}(\mathfrak{t})\otimes_{\Z}\C\right)/W,
\end{align}
where $\widehat{\mathfrak{g}}_{\mathrm{ss}}$ is the set of semisimple conjugacy classes in $\widehat{\mathfrak{g}}$. 
Hence we associate to $(\pi,V)$ a semisimple conjugacy class $\mathrm{c}_{\infty}(\pi)\in\widehat{\mathfrak{g}}_{\mathrm{ss}}$,
called the \emph{infinitesimal character} of $\pi$. 

As proved by Harish-Chandra \cite[Corollary 10.37]{Knapp86},
up to isomorphism there are only a finite number of irreducible unitary representations of $G(\R)$ 
with a given infinitesimal character.
When $G(\R)$ is compact, the situation is much simpler due to the following result:
\begin{prop}\label{prop for compact group inf char is bijection}\cite[\S 7.4.6]{Enveloping}
    Let $G(\R)$ be a compact group,
    and $\rho\in X^{*}(\mathfrak{t})\otimes\C$ the half-sum of positive roots with respect to $(\mathfrak{g},\mathfrak{b},\mathfrak{t})$.
    For a dominant weight $\lambda$ of $G(\R)$,
    the infinitesimal character of the highest weight representation $\mathrm{V}_{\lambda}$ of $G(\R)$
    is $\lambda+\rho$,
    viewed as an element in $\widehat{\mathfrak{g}}_{\mathrm{ss}}$ via \eqref{eqn Harish-Chandra isomorphism}.
    In particular, the infinitesimal character $\lambda+\rho$ determines $\mathrm{V}_{\lambda}$ uniquely.
\end{prop}

\subsubsection{Langlands parametrization}\label{section Langlands parametrization}
 Now we recall Langlands parametrization of $\Pi(G)$,
following \cite[\S 6.4.2]{ChenevierLannes}.
\begin{defi}\label{def set of conjugacy classes}
    Let $H$ be a connected reductive $\C$-group with complex Lie algebra $\mathfrak{h}$.
    We denote by $H(\C)_{\mathrm{ss}}$ (resp. $\mathfrak{h}_{\mathrm{ss}}$)
    the set of $H(\C)$-conjugacy classes of semisimple elements of $H(\C)$ (resp. $\mathfrak{h}$).
    We denote by $\mathcal{X}(H)$ the set of families $(\mathrm{c}_{\infty},\mathrm{c}_{2},\mathrm{c}_{3},\mathrm{c}_{5},\ldots)$,
    where $\mathrm{c}_{\infty}\in\mathfrak{h}_{\mathrm{ss}}$ and $\mathrm{c}_{p}\in H(\C)_{\mathrm{ss}}$ for all primes $p$.
\end{defi}

By results in \Cref{section Satake parameter} and \Cref{section infinitesimal character},
we associate to a representation $\pi=\pi_{\infty}\otimes\left(\bigotimes_{p}\pi_{p}\right)\in\Pi(G)$
a conjugacy class $\mathrm{c}_{p}(\pi):=\mathrm{c}_{p}(\pi_{p})$ in $\widehat{G}(\C)_{\mathrm{ss}}$ for each $p$,
and a conjugacy class $\mathrm{c}_{\infty}(\pi):=\mathrm{c}_{\infty}(\pi_{\infty})$ in $\widehat{\mathfrak{g}}_{\mathrm{ss}}$.
Hence we have a canonical map $\Pi(G)\rightarrow \mathcal{X}(\widehat{G})$
defined as
\[\pi=\pi_{\infty}\otimes\left(\bigotimes_{p}\pi_{p}\right)\mapsto \mathrm{c}(\pi)=(\mathrm{c}_{\infty}(\pi),\mathrm{c}_{2}(\pi),\mathrm{c}_{3}(\pi),\cdots)\in\mathcal{X}(\widehat{G}).\] 
The family of conjugacy classes $\mathrm{c}(\pi)$ determines $\pi_{f}$ and the infinitesimal character of $\pi_{\infty}$,
and the map $\mathrm{c}$ has finite fibers.
When $G(\R)$ is compact,
the fiber of $\mathrm{c}$ is either empty or a singleton.
\begin{defi}\label{def Langlands parameter of a pair}
    Let $G$ be a semisimple $\Q$-group admitting a reductive $\Z$-model,
    and $r:\widehat{G}\rightarrow \SL_{n}$ an algebraic representation of its dual group,
    which induces a map $\mathcal{X}(\widehat{G})\rightarrow\mathcal{X}(\SL_{n})$.
    For any $\pi\in\Pi(\G)$,
    we define the following family of conjugacy classes: 
    \[\psi(\pi,r):=r\left(\mathrm{c}(\pi)\right)\in\mathcal{X}(\SL_{n}),\]
    and refer to it as the \emph{Langlands parameter of the pair $(\pi,r)$}.
\end{defi}

\subsection{Global parametrization and the Langlands group}\label{section global parametrization and Langlands group}
 For the global parametrization of level one discrete automorphic representations,
now we need to use a \emph{conjectural} group $\mathcal{L}_{\Z}$, 
the so-called \emph{Langlands group of $\Z$}, 
to formulate the global Arthur-Langlands conjecture.
In Arthur's work \cite{ArthurUnipRep}, 
he uses another group $\mathcal{L}_{\Q}$.
However, since we only consider level one discrete automorphic representations in this paper, 
it is more convenient to use the group $\mathcal{L}_{\Z}$ that we are going to recall,
following \cites[Appendix B]{ChenevierRenard}[Preface]{ChenevierLannes}.

We assume that $\mathcal{L}_{\Z}$ is a compact Hausdorff topological group equipped with
\begin{itemize}
    \item A conjugacy class $\frob_{p}$ in $\mathcal{L}_{\Z}$, for each prime $p$,
    \item A conjugacy class of continuous homomorphisms $\mathrm{h}:\mathrm{W}_{\R}\rightarrow \mathcal{L}_{\Z}$,
    called the \emph{Hodge morphism}.
    Here $\mathrm{W}_{\R}$ is the \emph{Weil group of $\R$}, 
    which is a non-split extension of $\gal(\C/\R)=\{1,j\}$ by $\mathrm{W}_{\C}=\C^{\times}$,
    for the natural action of $\gal(\C/\R)$ on $\C^{\times}$.
    It is generated by its open subgroup $\C^{\times}$ together with an element $j$,
    with relations $j^{2}=-1$ and $jzj^{-1}=\overline{z}$ for every $z\in\C^{\times}$.
\end{itemize}
This group $\mathcal{L}_{\Z}$ satisfies three axioms that we will introduce one by one. 
\begin{axiom}\label{axiom 1} (\emph{Cebotarev property})
    The union of conjugacy classes $\frob_{p}$ is dense in $\mathcal{L}_{\Z}$.
\end{axiom}
\begin{rmk}\label{rmk axiom using Sato-Tate}
    In \cite[Appendix B]{ChenevierRenard},
    the axiom they use is the \emph{general Sato-Tate conjecture}:
    the conjugacy classes $\mathrm{Frob}_{p}$ are equidistributed in the compact group $\mathcal{L}_{\Z}$
    equipped with its Haar measure of mass $1$.
    This is a universal form of the Sato-Tate conjecture for automorphic representations
    and it implies the Cebotarev property,
    but \Cref{axiom 1} is enough for us in this article.
\end{rmk}
This axiom tells us for two homomorphisms $\psi,\psi^{\prime}$ from $\mathcal{L}_{\Z}$ to some topological group $H$,
if $\psi(\frob_{p})$ and $\psi^{\prime}(\frob_{p})$ are conjugate in $H$ for each prime $p$,
then $\psi$ and $\psi^{\prime}$ are element-conjugate.
An important type of homomorphisms involving $\mathcal{L}_{\Z}$ is:
\begin{defi}\label{def global discrete Arthur parameter}
    Let $G$ be a reductive $\Q$-group admitting a reductive $\Z$-model.
    A \emph{discrete global Arthur parameter (of level one)} of $G$ is 
    a $\widehat{G}(\C)$-conjugacy class of continuous group homomorphisms
    \[\psi:\mathcal{L}_{\Z}\times\mathrm{SL}_{2}(\C)\rightarrow \widehat{G}(\C)\]
    such that $\psi|_{\SL_{2}(\C)}$ is algebraic and the centralizer $\mathrm{C}_{\psi}$ of $\mathrm{Im}(\psi)$ in $\widehat{G}(\C)$ is finite modulo the center of $\widehat{G}(\C)$.
    We call $\mathrm{C}_{\psi}$ the \emph{(global) component group} of $\psi$,
    and denote the set of discrete global Arthur parameters of $G$ by $\Psi_{\mathrm{disc}}(G)$.
\end{defi}
\begin{rmk}\label{rmk discreteness vs irreducibility}
    The condition on $\mathrm{C}_{\psi}$ in \Cref{def global discrete Arthur parameter}
    implies that a discrete global Arthur parameter for $G=\GL_{n}$ 
    is an equivalence class of $n$-dimensional irreducible representations of $\mathcal{L}_{\Z}\times\SL_{2}(\C)$.
\end{rmk}
In parallel with Langlands parametrization in \Cref{section Langlands parametrization}, 
we can also associate to any $\psi\in\Psi_{\mathrm{disc}}(G)$ a collection of conjugacy classes 
$\mathrm{c}(\psi)=\left(\mathrm{c}_{\infty}(\psi),\mathrm{c}_{2}(\psi),\mathrm{c}_{3}(\psi),\cdots\right)\in\mathcal{X}(\widehat{G})$.
For each prime $p$, the conjugacy class $\mathrm{c}_{p}(\psi)$ is defined by:
\[\ \mathrm{c}_{p}(\psi):=\psi(\frob_{p},e_{p}),\,e_{p}=\left(\begin{matrix}
    p^{-1/2}&0\\
    0&p^{1/2}
\end{matrix}\right)\in\SL_{2}(\C).
\]
The infinitesimal character $\mathrm{c}_{\infty}(\psi)$ of $\psi$ is defined 
to be the infinitesimal character of the archimedean Arthur parameter $\psi\circ(\mathrm{h}\times\mathrm{id}):\mathrm{W}_{\R}\times\SL_{2}(\C)\rightarrow \widehat{G}(\C)$,
which is explained in \cite[\S A.2]{ChenevierRenard}.
We also set $e_{\infty}$ to be $\begin{pmatrix}
    1/2 & 0\\
    0 &-1/2
\end{pmatrix}\in \mathfrak{sl}_{2,\C}$. 

The following axiom connects the collection of conjugacy classes 
attached to a discrete automorphic representation and
that attached to a discrete global Arthur parameter.
\begin{axiom}\label{axiom 2}(Arthur-Langlands conjecture for $\GL_{n}$)
    For every integer $n\geq 1$,
    there is a unique bijection 
    \[\Pi_{\disc}(\GL_{n})\xrightarrow{\simeq}\Psi_{\disc}(\GL_{n}),\,\pi\mapsto \psi_{\pi}\]
    such that $c(\pi)=c(\psi_{\pi})$ for all discrete automorphic representations $\pi$ of $\GL_{n}$. 
    Moreover, the discrete global Arthur parameter $\psi_{\pi}$ is trivial on $\SL_{2}(\C)$ if and only if we have $\pi\in\Pi_{\cusp}(\GL_{n})$.
\end{axiom}

\begin{rmk}\label{rmk generalized Ramanujan conjecture}
    This axiom and the compactness of $\mathcal{L}_{\Z}$ imply the so-called \emph{generalized Ramanujan conjecture}:
    for any $\pi\in\Pi_{\cusp}(\GL_{n})$ and any prime $p$, 
    the eigenvalues of $\mathrm{c}_{p}(\pi)$ all have absolute value $1$.
\end{rmk}
For general reductive groups,
we have the following third axiom:
\begin{axiom}\label{axiom 3}
    Let $G$ be a reductive group admitting a reductive $\Z$-model $(\G,\mathrm{id})$,
    then there exists a decomposition 
    \begin{align}
        \mathcal{L}_{\disc}(G)^{\G(\widehat{\Z})}=\overset{\perp}{\bigoplus}_{\psi\in\Psi_{\disc}(G)}\mathcal{A}_{\psi}(G),
    \end{align}
    stable under the actions of $G(\R)$ and $\mathrm{H}(G)$,
    and satisfying the following property:
    for $\pi\in\Pi(G)$,
    if $\pi^{\G(\widehat{\Z})}$ appears in $\mathcal{A}_{\psi}(G)$,
    then we have $\mathrm{c}(\pi)=\mathrm{c}(\psi)$.
\end{axiom}
This axiom tells us for any level one discrete automorphic representation $\pi\in\Pi_{\disc}(G)$,
there exists a discrete global Arthur parameter $\psi$ of $G$ such that $\mathrm{c}(\psi)=\mathrm{c}(\pi)$.
In general, this discrete global Arthur parameter is not unique
since two element-conjugate embeddings into $\widehat{G}(\C)$ may not be conjugate.
Conversely,
given a discrete global Arthur parameter $\psi$ of $G$,
there are finitely many (possibly zero) adelic representations $\pi\in \Pi(G)$ satisfying $\mathrm{c}(\pi)=\mathrm{c}(\psi)$,
and we denote the subset of $\Pi(G)$ consisting of such representations by $\Pi_{\psi}(G)$.

In other words,
discrete global Arthur parameters are the objects parametrizing discrete automorphic representations. 
A natural problem is that 
for a given Arthur parameter $\psi$, 
which representations in $\Pi_{\psi}(G)$ 
appear in the discrete spectrum $\mathcal{L}(G)_{\disc}$.
We will see a (conjectural) answer in \Cref{section Arthur multiplicity formula}.

Another property about $\mathcal{L}_{\Z}$ that we will use is that it is connected:
\begin{prop}\cite[Proposition 9.3.4]{ChenevierLannes}\label{prop Langlands group is connected}
    Suppose that $\mathcal{L}_{\Z}$ is a compact topological group satisfying the axioms above,
    then it is connected.
\end{prop}
\subsubsection{Sato-Tate group}\label{section Sato Tate groups of Arthur parameters}
 For a discrete global Arthur parameter $\psi\in\Psi_{\disc}(G)$,
we pick a representative $\mathcal{L}_{\Z}\times\SL_{2}(\C)\rightarrow\widehat{G}(\C)$ 
and consider its restriction to a maximal compact subgroup:
\[\psi_{\mathrm{c}}:\mathcal{L}_{\Z}\times\su(2)\rightarrow \widehat{G}(\C).\]
The image of this morphism is contained in some maximal compact subgroup of $\widehat{G}(\C)$.
Fix a maximal connected compact subgroup $K$ of $\widehat{G}(\C)$,
and without loss of generality we assume that $\psi_{\mathrm{c}}$ is a morphism from $\mathcal{L}_{\Z}\times\su(2)\rightarrow K$.
\begin{defi}\label{def image group of Arthur parameter}
    For any $\psi\in\Psi_{\disc}(G)$,
    we define $\mathrm{H}(\psi)$ to be the $K$-conjugacy class of the image of its associated morphism $\mathcal{L}_{\Z}\times\su(2)\rightarrow K$.
    For any $\pi\in\Pi_{\disc}(G)$,
    if there exists a unique global Arthur parameter $\psi_{\pi}\in\Psi_{\disc}(G)$ such that $\mathrm{c}(\pi)=\mathrm{c}(\psi_{\pi})$, 
    we define $\mathrm{H}(\pi)$ to be $\mathrm{H}(\psi_{\pi})$.
\end{defi}
\begin{rmk}\label{rmk image well defined}
    Since maximal connected compact subgroups of $\SL_{2}(\C)$ are unique up to conjugacy,
    the $\widehat{G}(\C)$-conjugacy class of the image of $\mathcal{L}_{\Z}\times\su(2)\rightarrow K$ is well-defined.
    Combining with \cite[Lemma 2.4]{ConjEleConj},
    the $K$-conjugacy class $\mathrm{H}(\psi)$ is well-defined.
\end{rmk}
\begin{rmk}\label{rmk Sato-Tate in the introduction}
    The conjugacy class $\mathrm{H}(\psi)$, or $\mathrm{H}(\pi)$, of subgroups of $K$
    is called the \emph{``Sato-Tate group''} in the introduction \Cref{section introduction},
    although it coincides with the usual Sato-Tate group (see \cite[Proposition-Definition B.1]{ChenevierRenard}) if and only if the restriction of $\psi$ to $\SL_{2}(\C)$ is trivial.
\end{rmk}
A cuspidal automorphic representation $\pi$ of $\pgl_{n}$
can be viewed as an element of $\Pi_{\cusp}(\GL_{n})$ with trivial central character,
and the global Arthur parameter $\psi_{\pi}$ associated to $\pi$ via \Cref{axiom 2} takes value in $\SL_{n}(\C)=\widehat{\pgl_{n}}(\C)$.
In this case,
the global Arthur parameter $\psi_{\pi}$ is trivial on $\SL_{2}(\C)$,
and the conjugacy class $\mathrm{H}(\pi)$ of subgroups of $\su(n)$ 
coincides with the usual Sato-Tate group of $\pi$.

\subsection{Cuspidal representations of \texorpdfstring{$\GL_{n}$}{}}\label{section cuspidal representations of GL(n)}
 
Arthur's classification of automorphic representations involves self-dual cuspidal representations of $\GL_{n},n\geq 1$.
Moreover, these representations of $\GL_{n}$ are trivial on the center of $\GL_{n}$ when they have level one,
thus we can replace $\GL_{n}$ by $\pgl_{n}$.
In this subsection we will say more about this class of automorphic representations.
\begin{defi}\label{def selfdual cuspidal representation}
    A representation $\pi\in \Pi_{\cusp}(\pgl_{n})$ is \emph{self-dual} if it is isomorphic to its dual representation $\pi^{\vee}$,
    and we denote the subset of $\Pi_{\mathrm{cusp}}(\mathrm{PGL}_{n})$ consisting of self-dual representations by $\Pi_{\cusp}^{\perp}(\pgl_{n})$.
\end{defi}
\begin{rmk}\label{rmk equivalent condition of self-duality}
    By the multiplicity one theorem of Jacquet-Shalika,
    this self-dual condition is equivalent to that 
    $\mathrm{c}_{p}(\pi)=\mathrm{c}_{p}(\pi)^{-1}$ for each prime $p$ and $\mathrm{c}_{\infty}(\pi)=-\mathrm{c}_{\infty}(\pi)$.
\end{rmk}
For a representation $\pi\in\Pi_{\cusp}(\pgl_{n})$,
its infinitesimal character $\mathrm{c}_{\infty}(\pi)$ is a conjugacy class in $\mathfrak{sl}_{n}$.
Denote by $\mathrm{Weights}(\pi)$ the multiset of eigenvalues of $\mathrm{c}_{\infty}(\pi)$.
\begin{defi}\label{def algebraic and regular representations}
    A cuspidal automorphic representation $\pi$ of $\pgl_{n}$ is 
    \begin{itemize}
        \item \emph{algebraic}
        \footnote{The term \emph{algebraic} is in the sense of Borel \cite[\S 18.2]{BorelAutomorphicLFunction}.}
         if $\weights(\pi)\subset\frac{1}{2}\Z$ and for any $w,w^{\prime}\in\mathrm{Weights}(\pi)$ we have $w-w^{\prime}\in\Z$.
        \item \emph{regular} if $\mathrm{Weights}(\pi)$ is a set of $n$ elements.
    \end{itemize} 
    We denote  by $\Pi_{\alg}^{\perp}(\pgl_{n})$ the subset of $\Pi_{\cusp}^{\perp}(\pgl_{n})$ consisting of algebraic representations,
and by $\Pi_{\alg,\reg}^{\perp}(\pgl_{n})$ the subset consisting of algebraic regular representations.
\end{defi}
For an algebraic self-dual cuspidal representation $\pi$ of $\pgl_{n}$, 
let $k_{1}\geq k_{2}\geq\cdots\geq k_{n}$
be the weights of $\pi$ (counted with multiplicity).
Since $\pi$ is self-dual, we have $k_{i}=-k_{n+1-i}$ for $i=1,2,\ldots,n$.
Following \cite[\S 1.5]{ChenevierRenard}, 
we call the integers
\[w_{i}=2k_{i},\,i=1,2,\ldots,[n/2]\]
the \emph{Hodge weights} of $\pi$ and
call the maximal Hodge weight $\mathrm{w}(\pi):=w_{1}$ the \emph{motivic weight} of $\pi$.

\subsubsection{Arthur's orthogonal-symplectic alternative}\label{section orth-sympl alternative}
 We can divide the set of self-dual cuspidal representations of $\pgl_{n}$ into two parts,
by Arthur's \emph{symplectic-orthogonal alternative}.
Our reference is \cite[\S 8.3.1]{ChenevierLannes}.

The classical groups over $\Z$ that are Chevalley groups are therefore 
$\symp_{2g}$ for $g\geq 1$, 
$\sorth_{r,r}$ for $r\geq 2$,
and $\sorth_{r+1,r}$ for $r\geq 1$.
For one of these groups $G$, 
we denote the standard representation of $\widehat{G}(\C)$ by $\st:\widehat{G}(\C)\hookrightarrow \SL_{\mathrm{n}(G)}(\C)$.
For instance, $\mathrm{n}(\mathrm{Sp}_{2g})=2g+1$, 
$\mathrm{n}(\sorth_{r,r})=2r$ and $\mathrm{n}(\sorth_{r+1,r})=2r$. 
This map $\st$ also induces a natural map from $\mathcal{X}(\widehat{G})$ to $\mathcal{X}(\SL_{\mathrm{n}(G)})$.
We have the following theorem by Arthur:
\begin{thm}\cite[Theorem 1.4.1]{arthur2013endoscopic}\label{thm symplectic-orthogonal alternative of selfdual representations}
    For any $n\geq 1$ and a self-dual cuspidal representation $\pi$ of $\pgl_{n}$, 
    there exists a classical Chevalley group $G^{\pi}$,
    unique up to isomorphism,
    with the following properties:
    \begin{enumerate}[label=(\roman*)]
        \item We have $\mathrm{n}(G^{\pi})=n$.
        \item There exists a representation $\pi^{\prime}\in\Pi_{\mathrm{disc}}(G^{\pi})$ such that $\psi(\pi^{\prime},\st)=\mathrm{c}(\pi)$.
    \end{enumerate}
\end{thm}
\begin{defi}\label{defi symplectic orthogonal representation}
    A representation $\pi\in\Pi_{\cusp}^{\perp}(\pgl_{n})$
is called \emph{orthogonal} if $\widehat{G^{\pi}}(\C)\simeq\sorth_{n}(\C)$
and \emph{symplectic} otherwise.
We denote the subset of $\Pi_{\cusp}^{\perp}(\pgl_{n})$ consisting of orthogonal representations by $\Pi_{\cusp}^{\mathrm{o}}(\pgl_{n})$,
and the subset consisting of symplectic representations by $\Pi_{\cusp}^{\mathrm{s}}(\mathrm{PGL}_{n})$.
\end{defi}
For $*=\alg$ or $\alg,\reg$,
we define $\Pi_{*}^{\mathrm{o}}(\pgl_{n})=\Pi_{\cusp}^{\mathrm{o}}(\pgl_{n})\cap\Pi_{*}^{\perp}(\pgl_{n})$
and $\Pi_{*}^{\mathrm{s}}(\pgl_{n})=\Pi_{\cusp}^{\mathrm{s}}(\pgl_{n})\cap\Pi_{*}^{\perp}(\pgl_{n})$.
We define the subset $\Pi_{\alg}^{\symp_{2n}}(\pgl_{2n})\subset \Pi_{\alg,\reg}^{\mathrm{s}}(\pgl_{2n})$ as:
\[\set{\pi\in\Pi_{\alg,\reg}^{\mathrm{s}}(\pgl_{2n})}{\mathrm{Im}(\psi_{\pi})\simeq\symp(n)},\]
and similarly define 
\[\Pi_{\alg}^{\sorth_{n}}(\pgl_{n})=\set{\pi\in\Pi_{\alg,\reg}^{\mathrm{o}}(\pgl_{n})}{\mathrm{Im}(\psi_{\pi})\simeq\sorth(n)}.\]
\begin{ex}\label{ex algebraic representation of GL(2)}
    A representation $\pi\in\Pi_{\cusp}(\pgl_{2})$ is necessarily self-dual and symplectic,
    thus $\Pi_{\cusp}(\pgl_{2})=\Pi_{\cusp}^{\perp}(\pgl_{2})=\Pi_{\cusp}^{\mathrm{s}}(\pgl_{2})$.
    Moreover, for each positive integer $w$ 
    we have a bijection between the set of level one normalized Hecke eigenforms of weight $w+1$
    and the set of $\pi\in\Pi_{\alg}^{\perp}(\pgl_{2})$ with Hodge weight $w$.
    In particular, level one algebraic cuspidal representations with Hodge weight $w$ exist only when $w\geq 11$.
\end{ex}

\subsubsection{Global \texorpdfstring{$\varepsilon$}{}-factor}\label{section root number of cusp reps}
 An important factor related to a cuspidal representation $\pi$ is its \emph{global $\varepsilon$-factor} $\varepsilon(\pi)$.
We briefly give its definition as follows:
for two level one cuspidal representations $\pi\in\Pi_{\cusp}(\pgl_{n})$ and $\pi^{\prime}\in\Pi_{\cusp}(\pgl_{n^{\prime}})$,
Jacquet, Shalika and Piatetski-Shapiro 
define a factor $\varepsilon(\pi\times\pi^{\prime})$
when studying the meromorphic continuation and functional equation of the Rankin-Selberg $L$-function $\mathrm{L}(s,\pi\times\pi^{\prime})$ \cite[\S 9]{cogdelllectures}.
\begin{defi}\label{def global epsilon factor}
    The \emph{global $\varepsilon$-factor} of $\pi\in\Pi_{\cusp}(\pgl_{n})$ is defined as $\varepsilon(\pi):=\varepsilon(\pi\times\triv)$.
\end{defi}
For orthogonal algebraic representations, we have the following result by Arthur:
\begin{thm}\cite[Theorem 1.5.3]{arthur2013endoscopic}\label{thm Arthur epsilon factor of orthogonal rep is trivial}
    If $\pi\in\Pi_{\alg}^{\mathrm{o}}(\pgl_{n})$, then $\varepsilon(\pi)=1$.
\end{thm}
In \cite[\S 8.2.21]{ChenevierLannes},
a method to compute $\varepsilon(\pi)$ for $\pi\in\Pi_{\alg}^{\mathrm{s}}(\pgl_{n})$ is explained.
To recall that method,
we review first the archimedean Local Langlands correspondence \cite{Langlandsclassification}.
We can associate with each irreducible unitary representation $U$ of $\GL_{n}(\R)$
a unique (up to conjugacy) semisimple representation $\mathrm{L}(U):\mathrm{W}_{\R}\rightarrow\GL_{n}(\C)$.
By Clozel's purity lemma \cite[Lemma 4.9]{PurityClozel},
for a representation $\pi\in\Pi_{\alg}^{\perp}(\pgl_{n})$,
the associated representation $\mathrm{L}(\pi_{\infty})$ is a direct sum of the following types of irreducible representations:
\begin{itemize}
    \item the trivial representation $\triv$,
    \item the sign character $\epsilon_{\C/\R}=\eta/|\eta|$, where $\eta$ is the character $\mathrm{W}_{\R}\rightarrow \mathrm{W}_{\R}^{\mathrm{ab}}\simeq \R^{\times}$,
    \item and the $2$-dimensional induced representation $\mathbf{I}_{w}:=\mathrm{Ind}_{\mathrm{W}_{\C}}^{\mathrm{W}_{\R}}\left(z\mapsto z^{w/2}\overline{z}^{-w/2}\right)$ for some positive integer $w$,
        where $z\mapsto z^{w/2}\overline{z}^{-w/2}$ stands for the character $z\mapsto \left(z/|z|\right)^{w}$ by an abuse of notation.
\end{itemize}
There is a unique way to associate a fourth root of unity $\varepsilon(\rho)$ with each $\rho:\mathrm{W}_{\R}\rightarrow\GL_{n}(\C)$ of the above forms 
such that $\varepsilon(\rho\oplus\rho^{\prime})=\varepsilon(\rho)\varepsilon(\rho^{\prime})$
and
\[\varepsilon(\triv)=1,\,\varepsilon(\epsilon_{\C/\R})=i,\,\varepsilon(\mathbf{I}_{w})=i^{w+1}\text{ for any integer }w>0.\]
There is a connection between this factor $\varepsilon\left(\mathrm{L}(\pi_{\infty})\right)$ and the global $\varepsilon$-factor of $\pi$:
\begin{prop}\cite[Proposition 8.2.22]{ChenevierLannes}\label{prop relation between epsilon factors}
    For $\pi\in\Pi_{\alg}^{\perp}(\pgl_{n})$,
    we have
    \[\varepsilon(\pi)=\varepsilon(\mathrm{L}(\pi_{\infty})).\]
\end{prop}
As a consequence,
we can calculate the global $\varepsilon$-factor of $\pi$ provided we know the representation $\mathrm{L}(\pi_{\infty})$ of $\mathrm{W}_{\R}$ corresponding to $\pi_{\infty}$.
Actually, one has the following result:
\begin{prop}\cite[Proposition 8.2.13]{ChenevierLannes}\label{prop Langlands parameter of symplectic representation}
    Let $\pi\in\Pi_{\alg}^{\mathrm{s}}(\pgl_{n})$ and $w_{1}\geq w_{2}\geq\cdots\geq w_{n/2}$ its Hodge weights,
    then \[\mathrm{L}(\pi_{\infty})\simeq \mathbf{I}_{w_{1}}\oplus \mathbf{I}_{w_{2}}\oplus\cdots\oplus \mathbf{I}_{w_{n/2}}.\]
\end{prop}

\subsection{Arthur-Langlands conjecture}\label{section Arthur Langlands conjecture}
Assuming the existence of the Langlands group $\mathcal{L}_{\Z}$ described in \Cref{section global parametrization and Langlands group},
\Cref{axiom 3} says that for any reductive group $G$ admitting a reductive $\Z$-model
and any discrete automorphic representation $\pi$ of $G$,
there exists a discrete global Arthur parameter $\psi:\mathcal{L}_{\Z}\times \SL_{2}(\C)\rightarrow \widehat{G}(\C)$
such that $\mathrm{c}(\pi)=\mathrm{c}(\psi)$.
\begin{rmk}\label{rmk uniqueness of Arthur parameter}
When the group $\widehat{G}(\C)$ satisfies the ``element-conjugacy implies conjugacy'' property as in \Cref{prop compact Lie group F4 is acceptable},
the discrete global Arthur parameter $\psi$ satisfying $\mathrm{c}(\psi)=\mathrm{c}(\pi)$,
as a conjugacy class of homomorphisms $\mathcal{L}_{\Z}\times\SL_{2}(\C)\rightarrow\widehat{G}(\C)$,
is unique. 
\end{rmk}
Let $G$ be semisimple,
and fix an irreducible algebraic representation $r:\widehat{G}\rightarrow \SL_{n,\C}$.
Following \cite[\S 6.4.4]{ChenevierLannes},
we are going to see what the Langlands parameter $\psi(\pi,r)$ defined in \Cref{def Langlands parameter of a pair} looks like for a discrete automorphic representation $\pi$ of $G$.

Composing $r$ with a discrete global Arthur parameter $\psi:\mathcal{L}_{\Z}\times\SL_{2}(\C)\rightarrow \widehat{G}(\C)$ corresponding to $\pi$,
we get an $n$-dimensional representation $r\circ\psi$ of $\mathcal{L}_{\Z}\times\SL_{2}(\C)$.
This representation can be decomposed as 
\[\bigoplus_{i=1}^{k}r_{i}\otimes\sym^{d_{i}-1}\st\]
for some irreducible representations $r_{i}:\mathcal{L}_{\Z}\rightarrow\SL_{n_{i}}$ and certain integers $d_{i}\geq 1$, where $\st$ denotes the standard $2$-dimensional representation of $\SL_{2}(\C)$.

By Arthur-Langlands conjecture for general linear groups, 
i.e. \Cref{axiom 2} in \Cref{section global parametrization and Langlands group}, 
every irreducible representation $r_{i}:\mathcal{L}_{\Z}\rightarrow \GL_{n_{i}}(\C)$ corresponds to a unique cuspidal representation $\pi_{i}$ of $\pgl_{n_{i}}$.
For $v=p$ or $\infty$, we have an identity between conjugacy classes:
\[r(\mathrm{c}_{v}(\pi))=\bigoplus_{i}^{k}\mathrm{c}_{v}(\pi_{i})\otimes\sym^{d_{i}-1}(e_{v}).\]
To formulate a global identity,
we introduce the following notations:
\begin{itemize}
    \item Define $e\in\mathcal{X}(\SL_{2})$ to be $(e_{\infty},e_{2},e_{3},\cdots)$ and denote $\sym^{d-1}(e)\in\mathcal{X}(\SL_{d})$ by $[d]$.
    \item Denote by $(c,c^{\prime})\mapsto c\oplus c^{\prime}$ the map $\mathcal{X}(\SL_{a})\times\mathcal{X}(\SL_{b})\rightarrow \mathcal{X}(\SL_{a+b})$ induced by the direct sum,
    and by $(c,c^{\prime})\mapsto c\otimes c^{\prime}$ the map $\mathcal{X}(\SL_{a})\times\mathcal{X}(\SL_{b})\rightarrow\mathcal{X}(\SL_{ab})$ induced by the tensor product.
    We write $c\otimes[d]$ as $c[d]$ for short.
    \item For $\pi\in\Pi_{\cusp}(\pgl_{m})$,
    the element $\mathrm{c}(\pi)\in\mathcal{X}(\SL_{m})$ will simply be denoted by $\pi$.
\end{itemize}
With these notations,
we can combine the identities for $r(\mathrm{c}_{v}(\pi))$ together into one:
\[\psi(\pi,r)=r(\mathrm{c}(\pi))=\bigoplus_{i=1}^{k}\pi_{i}[d_{i}],\,\pi_{i}\in\Pi_{\cusp}(\pgl_{n_{i}}).\]
Now we state Arthur-Langlands conjecture for semisimple groups: 
\begin{conj}\label{conj general Arthur-Langlands conjecture}
    (\emph{Arthur-Langlands conjecture})
    Let $G$ be a semisimple $\Q$-group admitting a reductive $\Z$-model.
    For any $\pi\in\Pi_{\disc}(G)$ and every algebraic representation $r:\widehat{G}\rightarrow \SL_{n,\C}$, 
    there exists a collection of triples $(n_{i},\pi_{i},d_{i})_{i=1,\ldots,k}$ with $d_{i},n_{i}\geq 1$ integers satisfying $n=\sum_{i}n_{i}d_{i}$ and
    $\pi_{i}\in\Pi_{\mathrm{cusp}}(\mathrm{PGL}_{n_{i}})$ such that    
     \begin{align*}
        \psi(\pi,r)=\pi_{1}[d_{1}]\oplus\cdots \oplus\pi_{k}[d_{k}].
     \end{align*}
\end{conj}
This conjecture was proved by Arthur in \cite{arthur2013endoscopic} when $G$ is a split classical group
and $r$ is the standard representation of $\widehat{G}$.
Moreover, the collection of triples $(n_{i},\pi_{i},d_{i})$ in the conjecture is necessarily unique up to permutation by a result of Jacquet and Shalika \cite{UniqueALblock}:
\begin{prop}\label{prop uniqueness of blocks in Arthur-Langlands conjecture}\cite[Proposition 6.4.5]{ChenevierLannes}
    Let $k,l\geq 1$ be integers.
    For $1\leq i\leq k$ (resp. $1\leq j\leq l$), consider integers $n_{i},d_{i}\geq 1$ (resp. $n_{j}^{\prime},d_{j}^{\prime}\geq 1$) and a representation $\pi_{i}$ (resp. $\pi_{j}^{\prime}$) in $\Pi_{\cusp}(\pgl_{n_{i}})$ (resp. $\Pi_{\cusp}(\pgl_{n_{j}^{\prime}})$).
    Suppose that we have $n:=\sum_{i}n_{i}d_{i}=\sum_{j}n_{j}^{\prime}d_{j}^{\prime}$ and 
    \[\pi_{1}[d_{1}]\oplus\cdots\oplus\pi_{k}[d_{k}]=\pi_{1}^{\prime}[d_{1}^{\prime}]\oplus\cdots\oplus\pi_{l}^{\prime}[d_{l}^{\prime}].\]
    Then $k=l$ and there exists a permutation $\sigma\in \mathrm{S}_{k}$ such that for every $1\leq i\leq k$ we have $(n_{i}^{\prime},\pi_{i}^{\prime},d_{i}^{\prime})=(n_{\sigma(i)},\pi_{\sigma(i)},d_{\sigma(i)})$.
\end{prop}    
We call the triple $(k,(n_{i},d_{i})_{1\leq i\leq k})$,
up to permutations of the $(n_{i},d_{i})$,
the \emph{endoscopic type} of $\psi(\pi,r)$.
The parameter is called \emph{stable} if $k=1$ 
and \emph{endoscopic} otherwise.
It is called \emph{tempered} if $d_{i}=1$ for all $i$ 
and \emph{non-tempered} otherwise. 

In \Cref{conj general Arthur-Langlands conjecture},
cuspidal representations of $\pgl_{n},n\geq 1$ are building blocks of Langlands parameters $\psi(\pi,r)$.
Furthermore, the following result shows that under some conditions,
for example when $G(\R)$ is compact,
we only need algebraic cuspidal representations:
\begin{prop}\label{prop algebraic representations as building blocks}\cite[Proposition 8.2.8]{ChenevierLannes}
    Let $G$ be a semisimple $\Q$-group admitting a reductive $\Z$-model, $\pi\in\Pi_{\disc}(G)$ 
    and $r:\widehat{G}\rightarrow \SL_{n,\C}$ an $n$-dimensional algebraic representation of $\widehat{G}$.
    Suppose that 
    \begin{enumerate}[label= (\roman*)]
        \item $\mathrm{c}_{\infty}(\pi)\in\widehat{\mathfrak{g}}_{\mathrm{ss}}$ is the infinitesimal character of a finite-dimensional irreducible complex representation of $G_{\C}$,
        \item and $\psi(\pi,r)=\oplus_{i=1}^{k}\pi_{i}[d_{i}]$ with $\pi_{i}\in\Pi_{\cusp}(\pgl_{n_{i}}),i=1,\ldots,k$.
    \end{enumerate} 
    Then $\pi_{i}$ is algebraic for $i=1,\ldots,k$.
    Moreover, the class of $\mathrm{w}(\pi_{i})+d_{i}-1$ in $\Z/2\Z$ depends only on $r$ and not on the integer $i$ or even on $\pi$.
\end{prop}

\subsection{Arthur's multiplicity formula}\label{section Arthur multiplicity formula}
  
Arthur gives a \emph{conjectural} formula for the multiplicity of 
an adelic representation $\pi\in\Pi(G)$ in the discrete spectrum $\mathcal{L}_{\disc}(G)$.
In this section, we will state this
for a simply-connected anisotropic $\Q$-group $G$ admitting a reductive $\Z$-model,
following \cite[\S 8]{ArthurUnipRep}.

For a representation $\pi\in\Pi(G)$,
there are finitely many discrete global Arthur parameters $\psi$ of $G$ such that $\mathrm{c}(\pi)=\mathrm{c}(\psi)$.
According to \cite{ArthurUnipRep},
the multiplicity $\mathrm{m}(\pi)$ of $\pi$ in $\mathcal{L}_{\disc}(G)$
should be the sum of $m_{\psi}$ over the set of all such $\psi$,
where $m_{\psi}$ is some integer that we are going to introduce.
We note that these $\psi$ all belong to the following subset of $\Psi_{\disc}(G)$:
\begin{defi}\label{def Adams Johnson parameters}
    We define $\Psi_{\mathrm{AJ}}(G)$ to be the subset of $\Psi_{\disc}(G)$
    consisting of $\psi\in\Psi_{\disc}(G)$ satisfying that
    $\mathrm{c}_{\infty}(\psi)$ is the infinitesimal character of a finite dimensional irreducible representation of $G_{\C}$.
\end{defi}
\begin{rmk}\label{rmk relation with Adams Johnson parameter}
    The subscript $\mathrm{AJ}$ stands for \emph{Adams-Johnson}.
    This means the archimedean Arthur parameter $\mathrm{W}_{\R}\times\SL_{2}(\C)\rightarrow \widehat{G}(\C)$ for $\psi\in\Psi_{\disc}(G)$ 
    is an \emph{Adams-Johnson parameter} in the sense of \cite[\S 8.4.14]{ChenevierLannes}
    if and only if $\psi\in\Psi_{\mathrm{AJ}}(G)$.
    The condition that $\mathrm{c}_{\infty}(\psi)$ is the infinitesimal character of a finite-dimensional irreducible representation
    is the condition (AJ1) in \cite[\S 8.4.14]{ChenevierLannes},
    and the second condition (AJ2) for Adams-Johnson parameters 
    is automatically satisfied in our case by \Cites[\S 4.2.2]{TaibiDimension}[Proposition 6]{CohRepReal}.
\end{rmk}
Now we let $\psi\in\Psi_{\mathrm{AJ}}(G)$.
In \Cref{def global discrete Arthur parameter},
the global component group $\mathrm{C}_{\psi}$ of $\psi$ is defined to be the centralizer of $\mathrm{Im}(\psi)$ in $\widehat{G}(\C)$.
When $G$ is semisimple,
this group is finite since the center of $\widehat{G}$ is finite.
Moreover, as explained in \cite[\S 8.4.14]{ChenevierLannes}, 
$\mathrm{C}_{\psi}$ is  an elementary finite abelian $2$-group, 
i.e. a product of finitely many copies of $\Z/2\Z$.
For any $\psi\in\Psi_{\mathrm{AJ}}(G)$,
Arthur's formula for $m_{\psi}$ involves two quadratic characters of $\mathrm{C}_{\psi}$.

\subsubsection{The character \texorpdfstring{$\rho_{\psi}^{\vee}$}{PDFstring}}\label{section rho character in general}
 The first character of $\mathrm{C}_{\psi}$ is defined as follows.

By \Cref{prop for compact group inf char is bijection},
the conjugacy class $\mathrm{c}_{\infty}(\psi)$ for $\psi\in\Psi_{\mathrm{AJ}}(G)$ is regular,
viewed as a cocharacter of a maximal torus $\widehat{T}$ of $\widehat{G}$ chosen as in \cite[\S 8.4.14]{ChenevierLannes}.
Hence there is a unique Borel subgroup $\widehat{B}\supset \widehat{T}$ of $\widehat{G}$
with respect to whom the infinitesimal character $\mathrm{c}_{\infty}(\psi)$ is strictly dominant. 
Let $\rho^{\vee}_{\psi}$ be the half-sum of positive roots with respect to $(\widehat{G},\widehat{B},\widehat{T})$.
Since $G$ is simply-connected, 
$\rho^{\vee}_{\psi}\in \frac{1}{2}X^{*}(\widehat{T})$ is a character of $\widehat{T}$.
Its restriction to the component group $\mathrm{C}_{\psi}$ is the first character we need,
and we denote $\rho^{\vee}|_{\mathrm{C}_{\psi}}$ by $\rho^{\vee}_{\psi}$ for short.

\subsubsection{Arthur's character \texorpdfstring{$\varepsilon_{\psi}$}{PDFstring}}\label{section Arthur epsilon character}
 A discrete global Arthur parameter $\psi\in\Psi_{\mathrm{AJ}}(G)$
induces a morphism 
\[\mathrm{C}_{\psi}\times\mathcal{L}_{\Z}\times\SL_{2}(\C)\rightarrow \widehat{G}(\C).\]
Restricting the adjoint representation $\widehat{\mathfrak{g}}$ of $\widehat{G}(\C)$
along this morphism,
it can be decomposed into a direct sum
\begin{align}\label{eqn decomposition of adjoint representation}
    \widehat{\mathfrak{g}}|_{\mathrm{C}_{\psi}\times\mathcal{L}_{\Z}\times\SL_{2}(\C)}=\bigoplus_{i=1}^{l}\chi_{i}\otimes \pi_{i}[d_{i}],
\end{align}
where $\chi_{i}$ is a quadratic character of $\mathrm{C}_{\psi}$,
and $\pi_{i}$ is an $n_{i}$-dimensional irreducible representation of $\mathcal{L}_{\psi}$ which is identified as an element in $\Pi_{\cusp}^{\perp}(\pgl_{n_{i}})$.
Moreover, since $\psi$ belongs to $\Psi_{\mathrm{AJ}}(G)$,
according to \Cref{prop algebraic representations as building blocks}
these cuspidal representations $\pi_{i}$ are algebraic.
\begin{defi}\label{def Arthur character epsilon psi}\cite[Equation 8.4]{ArthurUnipRep}
    Let $\psi\in\Psi_{\mathrm{AJ}}(G)$,
    and $I$ be the subset of $\{1,\ldots,l\}$
    consisting of $i$ satisfying that in \eqref{eqn decomposition of adjoint representation} the cuspidal representation $\pi_{i}$ is self-dual and $\varepsilon(\pi_{i})=-1$.
    Arthur's character $\varepsilon_{\psi}:\mathrm{C}_{\psi}\rightarrow \mu_{2}$
    is defined by 
    \[\varepsilon_{\psi}(s):=\prod_{i\in I}\chi_{i}(s),\text{ for every }s\in\mathrm{C}_{\psi}.\]
\end{defi}
The following result shows that it is sufficient to calculate the global epsilon factors $\varepsilon(\pi_{i})$
for $i$ in a subset of $\{1,\ldots,l\}$:
\begin{prop}\label{prop epsilon character of Arthur}
    Let $\psi\in\Psi_{\mathrm{AJ}}(G)$.
    For any $s\in\mathrm{C}_{\psi}$, 
    let $I_{s}$ be the subset of $\{1,\ldots,l\}$ consisting of $i$
    satisfying that in \eqref{eqn decomposition of adjoint representation} the representation $\pi_{i}$ is self-dual, $d_{i}$ is even, and $\chi_{i}(s)=-1$.
    Then we have:
    \[\varepsilon_{\psi}(s)=\prod_{i\in I_{s}}\varepsilon(\pi_{i}).\]
\end{prop}
\begin{proof}
    When $d_{i}$ is odd, the $d_{i}$-dimensional irreducible representation of $\SL_{2}(\C)$ is orthogonal. 
    Since the adjoint representation is an orthogonal representation,
    we have two possibilities for
    the self-dual representation $\pi_{i}$ of $\mathcal{L}_{\Z}$:
    \begin{itemize}
        \item if $\pi_{i}$ is orthogonal, then by \Cref{thm Arthur epsilon factor of orthogonal rep is trivial}, $\varepsilon(\pi_{i})=1$;
        \item if $\pi_{i}$ is symplectic, then the $C_{\psi}\times\mathcal{L}_{\Z}\times\SL_{2}(\C)$-representation $\chi_{i}\otimes \pi_{i}[d_{i}]$ appears in $\widehat{g}$ with even multiplicity, thus its contribution to $\varepsilon_{\psi}(s)$ is trivial.
    \end{itemize}
    Hence only $i\in I$ with $d_{i}$ even can non-trivially contribute to $\varepsilon_{\psi}(s)$.
    For any $s\in \mathrm{C}_{\psi}$, we have
    \[\varepsilon_{\psi}(s)=\prod_{2|d_{i},\,\pi_{i}=\pi_{i}^{\vee},\,\varepsilon(\pi_{i})=-1}\chi_{i}(s)=\prod_{2|d_{i},\,\pi_{i}=\pi_{i}^{\vee},\,\chi_{i}(s)=-1}\varepsilon(\pi_{i})=\prod_{i\in I_{s}}\varepsilon(\pi_{i}).\qedhere\]
\end{proof}

\subsubsection{The multiplicity formula}\label{section multiplicity formula}
 With two characters $\rho^{\vee}_{\psi}$ and $\varepsilon_{\psi}$ in hand, 
we can state Arthur's following conjecture:
\begin{conj}\label{conj Arthur multiplicity formula}
    (\emph{Arthur's multiplicity formula}) 
    Let $G$ be a simply-connected anisotropic $\Q$-group with a reductive $\Z$-model,
    and $\pi$ a level one adelic representation in $\Pi(G)$. 
    We have the following formula for the multiplicity $\mathrm{m}(\pi)$ of $\pi$ in the discrete spectrum $\mathcal{L}_{\disc}(G)$:
    \begin{align}\label{eqn Arthur multiplicity formula}
        \mathrm{m}(\pi)=\sum_{\psi\in\Psi_{\disc}(G),\,c(\psi)=c(\pi)}m_{\psi},\text{ where }m_{\psi}=\left\{\begin{array}{@{}cl}
            1, & \text{if }\rho^{\vee}_{\psi}=\varepsilon_{\psi},\\
            0, & \text{otherwise.}
        \end{array}\right.
    \end{align}
\end{conj}

\section{Classification of global Arthur parameters for \texorpdfstring{$\grpF$}{PDFstring}}\label{section Arthur classification for F4}
 In this section,
we are going to apply Arthur's conjectures recalled in \Cref{section Arthur Langlands conjecture} and \Cref{section Arthur multiplicity formula} 
to the simply-connected anisotropic $\Q$-group $\grpF$ defined in \Cref{def algebraic group F4}.
The dual group $\widehat{\grpF}$ is isomorphic to the extension $\mathbf{F}_{4,\C}$ of $\grpF$ to $\C$.
In other words, 
the complex Lie group $\widehat{\grpF}(\C)$ is isomorphic to the complexification $\lietype{F}{4,\C}$ of the real compact Lie group $\lietype{F}{4}$.

\subsection{Arthur parameters of \texorpdfstring{$\grpF$}{PDFstring}}\label{section Arthur parameters of F4}
 The real points $\lietype{F}{4}=\grpF(\R)$ is compact,
so an adelic representation $\pi\in\Pi(\grpF)$ is determined uniquely by $\mathrm{c}(\pi)$.
On the other hand,
by \Cref{prop compact Lie group F4 is acceptable} and \Cref{axiom 1},
a discrete global Arthur parameter $\psi$ of $\grpF$
is also determined uniquely by $\mathrm{c}(\psi)\in\mathcal{X}(\widehat{\grpF})$.
Moreover, we have the following criterion, which is a direct corollary of \Cref{prop just one rep is enough for conjugacy}:
\begin{prop}
    Let $\psi_{1}$ and $\psi_{2}$ be two discrete global Arthur parameters of $\grpF$,
    and $\mathrm{r}_{0}:\widehat{\grpF}\rightarrow\SL_{26,\C}$ the $26$-dimensional irreducible representation of $\grpF(\C)$.
    Then $\psi_{1}=\psi_{2}$ if and only if $\mathrm{r}_{0}(\mathrm{c}(\psi_{1}))=\mathrm{r}_{0}(\mathrm{c}(\psi_{2}))$.
\end{prop}
By this result,
we will identify a discrete global Arthur parameter $\psi\in\Psi_{\disc}(\grpF)$
with the corresponding family of conjugacy classes $\mathrm{r}_{0}(\mathrm{c}(\psi))\in \mathcal{X}(\SL_{26})$.

For a level one discrete automorphic representation $\pi\in\Pi_{\disc}(\grpF)$,
the discrete global Arthur parameter $\psi\in\Psi_{\mathrm{AJ}}(\grpF)$ such that $\mathrm{c}(\psi)=\mathrm{c}(\pi)$ predicted by \Cref{axiom 1} is unique.
We denote this parameter by $\psi_{\pi}$,
which is identified with $\psi(\pi,\mathrm{r}_{0})\in\mathcal{X}(\SL_{26})$.
Conversely,
for $\psi\in\Psi_{\mathrm{AJ}}(\grpF)$,
we denote the unique representation $\pi\in\Pi(\pi)$ such that $\mathrm{c}(\pi)=\mathrm{c}(\psi)$ by $\pi_{\psi}$.

The following lemma gives us some constraint on the infinitesimal character $\mathrm{c}_{\infty}(\psi)$ of $\psi\in\Psi_{\mathrm{AJ}}(\grpF)$:
\begin{lemma}\label{lemma eigenvalues of infinitesimal character of F4 representation}
    Let $\mathrm{c}_{\infty}\in(\mathfrak{f}_{4})_{\mathrm{ss}}$ be the infinitesimal character of an irreducible representation of the compact group $\lietype{F}{4}$,
    then there exists four non-negative integers $a,b,c,d$ such that the eigenvalues (counted with multiplicity) of $\mathrm{r}_{0}(\mathrm{c}_{\infty})\in(\mathfrak{sl}_{26})_{\mathrm{ss}}$ are:
    \begin{gather*}
        0,0,\pm(a+1),\pm(b+1),\pm(a+b+2),\pm(b+c+2),\pm(a+b+c+3),\pm(b+c+d+3),\\
        \pm(a+b+c+d+4),\pm(a+2b+c+4),\pm(a+2b+c+d+5),\pm(a+2b+2c+d+6),\\
        \pm(a+3b+2c+d+7),\pm(2a+3b+2c+d+8).
    \end{gather*}
\end{lemma}
\begin{proof}
    If we write the highest weight $\lambda$ of this irreducible representation of $\lietype{F}{4}$ as $a\varpi_{1}+b\varpi_{2}+c\varpi_{3}+d\varpi_{4}$,
    then by \Cref{prop for compact group inf char is bijection}
    the infinitesimal character $\mathrm{c}_{\infty}$ is $\lambda+\rho= (a+1)\varpi_{1}+(b+1)\varpi_{2}+(c+1)\varpi_{3}+(d+1)\varpi_{4}$.
    The eigenvalues of $\mathrm{r}_{0}(\mathrm{c}_{\infty})$ are of the form $\langle \lambda+\rho,\alpha^{\vee}\rangle$,
    where $\alpha^{\vee}$ runs over the $26$ weights of $\widehat{\grpF}(\C)$ appearing in the representation $\mathrm{r}_{0}$.
    By an easy calculation,
    we get the eigenvalues in the lemma.
\end{proof}
As recalled in \Cref{section Sato Tate groups of Arthur parameters},
we associate to $\psi\in\Psi_{\mathrm{AJ}}(\grpF)$ 
a morphism $\psi_{\mathrm{c}}:\mathcal{L}_{\Z}\times\su(2)\rightarrow\lietype{F}{4}$ between compact Lie groups.
This homomorphism inherits the following properties from $\psi$:
\begin{itemize}
    \item the image $\mathrm{Im}(\psi_{c})$ is connected due to \Cref{prop Langlands group is connected},
    \item the centralizer of $\mathrm{Im}(\psi_{\mathrm{c}})$ in $\grpF$ coincides with 
    the global component group $\mathrm{C}_{\psi}$ of $\psi$,
    which is an elementary finite abelian $2$-group by \cite[\S 8.4.14]{ChenevierLannes},
    \item and the zero weight appears exactly twice in the restriction of the $26$-dimensional irreducible representation $\mathrm{J}_{0}$ of $\lietype{F}{4}$ along $\psi_{\mathrm{c}}$ by \Cref{lemma eigenvalues of infinitesimal character of F4 representation}.
\end{itemize}
Hence $\mathrm{Im}(\psi_{\mathrm{c}})$ is a subgroup of $\lietype{F}{4}$ satisfying the three conditions in the beginning of \Cref{section subgroups of F4},
thus the class $\mathrm{H}(\psi)$ defined in \Cref{def image group of Arthur parameter} is the conjugacy class of one of the subgroups of $\lietype{F}{4}$ listed in 
\Cref{thm classification result of subgroups satisfying three conditions}.

According to \Cref{conj general Arthur-Langlands conjecture},
the discrete global Arthur parameter $\psi_{\pi}=\psi(\pi,\mathrm{r}_{0})$ corresponding to a discrete automorphic representation $\pi\in\Pi_{\disc}(\grpF)$ should be of the form:
\[\pi_{1}[d_{1}]\oplus\cdots\oplus\pi_{k}[d_{k}],\]
where $\pi_{i}\in\Pi_{\cusp}(\pgl_{n_{i}})$ and $\sum_{i=1}^{k}n_{i}d_{i}=26$.
By \Cref{prop algebraic representations as building blocks},
every $\pi_{i}$ is algebraic,
and it is also self-dual by the following lemma:
\begin{splemma}\label{lemma building block is selfdual for F4}
    Let $\pi\in\Pi_{\disc}(\grpF)$
    and $\psi_{\pi}=\pi_{1}[d_{1}]\oplus\cdots\oplus\pi_{k}[d_{k}]$
    be its corresponding discrete global Arthur parameter,
    then for each $i=1,\ldots,k$, 
    the representation $\pi_{i}\in\Pi_{\cusp}(\pgl_{n_{i}})$ is self-dual.
\end{splemma}
\begin{proof}
    By our classification result in \Cref{section connected subgroups satisfying our conditions},
    identifying $\pi_{i}\in\Pi_{\cusp}(\pgl_{n_{i}})$ as an irreducible representation of $\mathcal{L}_{\Z}$,
    it must be of the form
    $\mathcal{L}_{\Z}\twoheadrightarrow H\xrightarrow{r}\SL_{n_{i}}(\C)$,
    where $H$ is a connected compact subgroup of $\lietype{F}{4}$
    and $r$ is a self-dual irreducible representation of $H$,
    thus $\pi_{i}$ itself is self-dual.
\end{proof}
So a discrete global Arthur parameter $\psi\in\Psi_{\mathrm{AJ}}(\grpF)$ corresponding to some $\pi\in\Pi_{\disc}(\grpF)$ must be of the form
\begin{align}\label{eqn Arthur Langlands decomposition of parameter}
    \psi=\pi_{1}[d_{1}]\oplus \cdots\oplus \pi_{k}[d_{k}],\text{ where }\pi_{i}\in\Pi_{\alg}^{\perp}(\pgl_{n_{i}}),\,\sum_{i=1}^{k}n_{i}d_{i}=26.
\end{align}
We now briefly explain our strategy to enumerate all the discrete global Arthur parameters $\psi\in \Psi_{\mathrm{AJ}}(\grpF)$.
Any possible candidate $H$ for $\mathrm{H}(\psi)$ listed in \Cref{thm classification result of subgroups satisfying three conditions}
is a finite quotient of $\prod_{i=1}^{r}H_{i}$, where each $H_{i}$ is simply-connected and simple.
Then we construct the following two types of $\psi_{\mathrm{c}}:\mathcal{L}_{\Z}\times\su(2)\rightarrow\lietype{F}{4}$ with $\mathrm{Im}(\psi_{c})=H$:
\begin{itemize}
    \item Given a family of surjective morphisms $\psi_{i}:\mathcal{L}_{\Z}\rightarrow H_{i}$, $1\leq i\leq r$ satisfying some regularity condition, we define $\psi_{\mathrm{c}}$ to be 
    \[\mathcal{L}_{\Z}\times\su(2)\xrightarrow{\mathrm{proj}}\mathcal{L}_{\Z}\xrightarrow{\prod_{i}\psi_{i}}\prod_{i}H_{i}\rightarrow H.\]
    \item If $H_{j}\simeq \su(2)$ for some $j$, given a family of surjective morphisms $\psi_{i}:\mathcal{L}_{\Z}\rightarrow H_{i},\,i\neq j$ satisfying some regular condition, we define $\psi_{\mathrm{c}}$ to be
    \[\mathcal{L}_{\Z}\times\su(2)\xrightarrow{\prod\limits_{i\neq j}\psi_{i} \times \mathrm{id}}\prod_{i\neq j}H_{i} \times \su(2)\rightarrow H.\]
\end{itemize}
In this way, we obtain all the Arthur parameters $\psi$ with $\mathrm{H}(\psi)=H$,
and by composing $\psi$ with the irreducible $26$-dimensional representation $\mathrm{r}_{0}$ of $\grpF(\C)$,
we can classify all the endoscopic types $(k,(n_{i},d_{i})_{1\leq i\leq k})$.
We will see this more precisely for each $H$ in \Cref{section classification of Arthur parameters of F4}.
\begin{ex}\label{ex endoscopic type from image of parameter}
    If the class $\mathrm{H}(\psi)$ associated to $\psi\in\Psi_{\mathrm{AJ}}(\grpF)$ is the conjugacy class of 
    \[H=\left(\lietype{A}{1}^{[9,6^{2},5]}\times\lietype{A}{1}^{[2^{6},1^{14}]}\right)/\mu_{2}^{\Delta},\]
    by \Cref{section info principal A1 in Sp(3) with Sp(1)} the restriction of the $26$-dimensional irreducible representation $(r_{0},\mathrm{J}_{0})$ along $\psi$ is isomorphic to 
    \[\sym^{5}\st\otimes\st+
    \sym^{8}\st\otimes\triv+
    \sym^{4}\st\otimes\triv.\]
Depending on how $\mathcal{L}_{\Z}$ and $\su(2)$ are mapped to this subgroup $H\subset\lietype{F}{4}$,
we have the following three possible endoscopic types for $\psi$:
\begin{itemize}
    \item $\left(3,(2,6),(1,5),(1,9)\right)$, $\psi=\pi[6]\oplus[5]\oplus[9],\pi\in\Pi_{\alg}^{\perp}(\pgl_{2})$;
    \item $\left(3,(9,1),(5,1),(6,2)\right)$, $\psi=\sym^{8}\pi\oplus\sym^{4}\pi\oplus\sym^{5}\pi[2],\pi\in\Pi_{\alg}^{\perp}(\pgl_{2})$;
    \item $\left(3,(9,1),(5,1),(12,1)\right)$, $\psi=\sym^{8}\pi_{1}\oplus\sym^{4}\pi_{2}\oplus(\sym^{5}\pi_{1}\otimes\pi_{2}),\pi_{1},\pi_{2}\in\Pi_{\alg}^{\perp}(\pgl_{2})$.
\end{itemize}
\end{ex}

\subsection{The multiplicity formula for \texorpdfstring{$\grpF$}{}}\label{section multiplicity formula for F4}
 
For a discrete global Arthur parameter $\psi\in\Psi_{\mathrm{AJ}}(\grpF)$, 
Arthur's multiplicity formula \Cref{conj Arthur multiplicity formula} predicts that the multiplicity $\mathrm{m}(\pi_{\psi})$ of $\pi_{\psi}$ in $\mathcal{L}_{\disc}(\grpF)$
equals to $m_{\psi}$,
the formula for which is given in \eqref{eqn Arthur multiplicity formula}.
To calculate $m_{\psi}$,
it suffices to know two characters of $\mathrm{C}_{\psi}$: Arthur's character $\varepsilon_{\psi}$,
and $\rho^{\vee}_{\psi}$.
We have given the formula of $\varepsilon_{\psi}$ in \Cref{prop epsilon character of Arthur},
and in this subsection we will give a recipe for the character $\rho^{\vee}_{\psi}$
for our $\Q$-group $\grpF$.

We fix a maximal torus $\widehat{T}$ of $\widehat{\grpF}$ and a Borel subgroup $\widehat{B}\supset \widehat{T}$ as in \Cref{section rho character in general}
such that the infinitesimal character $\mathrm{c}_{\infty}(\psi)$, 
as a cocharacter of $\widehat{T}$
is strictly dominant with respect to $(\widehat{\grpF},\widehat{B},\widehat{T})$.
We denote the four simple roots of the root system with respect to $(\widehat{\grpF},\widehat{B},\widehat{T})$ by $\alpha_{i}^{\vee},i=1,2,3,4$
\footnote{Here we still follow Bourbaki's notation, but since we are considering the root system of the dual group $\widehat{G}$, the simple root $\alpha_{i}^{\vee},1\leq i\leq 4$ corresponds to $\alpha_{5-i}$ in Bourbaki.}.

By \Cref{lemma eigenvalues of infinitesimal character of F4 representation},
we can order the eigenvalues (counted with multiplicity) of $\mathrm{c}_{\infty}(\psi)$
as $\mu_{1}>\mu_{2}>\mu_{3}>\mu_{4}>\mu_{5}\geq \cdots> \mu_{26}$.
The partial order relation of the positive weights of $\mathrm{r}_{0}$ in \Cref{table positive weights of 26 dimensional irreducible representation}
implies that 
\[\mu_{1}=\langle \mathrm{c}_{\infty}(\psi),2\alpha_{1}^{\vee}+3\alpha_{2}^{\vee}+2\alpha_{3}^{\vee}+\alpha_{4}^{\vee}\rangle,\,\mu_{4}=\langle \mathrm{c}_{\infty}(\psi),\alpha_{1}^{\vee}+2\alpha_{2}^{\vee}+\alpha_{3}^{\vee}+\alpha_{4}^{\vee}\rangle.\]
Notice that 
\[(2\alpha_{1}^{\vee}+3\alpha_{2}^{\vee}+2\alpha_{3}^{\vee}+\alpha_{4}^{\vee})+(\alpha_{1}^{\vee}+2\alpha_{2}^{\vee}+\alpha_{3}^{\vee}+\alpha_{4}^{\vee})\equiv \alpha_{1}^{\vee}+\alpha_{2}^{\vee}+\alpha_{3}^{\vee}\equiv \rho^{\vee}_{\psi}\modulo 2X^{*}(\widehat{T}),\]
thus the character $\rho^{\vee}_{\psi}$ of $\mathrm{C}_{\psi}\subset \widehat{T}[2]$
is the product of
$(2\alpha_{1}^{\vee}+3\alpha_{2}^{\vee}+2\alpha_{3}^{\vee}+\alpha_{4}^{\vee})|_{\mathrm{C}_{\psi}}$ 
and $(\alpha_{1}^{\vee}+2\alpha_{2}^{\vee}+\alpha_{3}^{\vee}+\alpha_{4}^{\vee})|_{\mathrm{C}_{\psi}}$.
Hence it suffices to determine these two characters.

If $\psi=\pi_{1}[d_{1}]\oplus \cdots \oplus\pi_{k}[d_{k}]$ as in \eqref{eqn Arthur Langlands decomposition of parameter},
the eigenvalues of $\mathrm{r}_{0}(\mathrm{c}_{\infty}(\psi))\in(\mathfrak{sl}_{26})_{\mathrm{ss}}$ are of the form $w+\frac{j}{2}$,
where $w$ is a weight of $\pi_{i}$ 
and $j\in\{  d_{i}-1,d_{i}-3,\ldots,-d_{i}+3,-d_{i}+1\}$.
For each $i=1,\ldots,k$, we define a multiset 
\[\mathcal{W}_{i}:=\set{w+\frac{j}{2}}{w\in\weights(\pi_{i})\text{ and }j=d_{i}-1,d_{i}-3,\ldots,-(d_{i}-3),-(d_{i}-1)}.\]
\begin{prop}\label{prop calculation of rho character for F4}
    There exists a unique index $i$ (resp. $j$) in $\{1,\ldots,k\}$
    such that $\mu_{1}\in\mathcal{W}_{i}$ (resp. $\mu_{4}\in\mathcal{W}_{j}$).
    If we denote respectively by $\epsilon_{i}$ and $\epsilon_{j}$ 
    the characters of $\mathrm{C}_{\psi}$ induced by the $\mathrm{C}_{\psi}$-actions on $\pi_{i}[d_{i}]$ and $\pi_{j}[d_{j}]$,
    then $\rho^{\vee}_{\psi}=\epsilon_{i}\cdot\epsilon_{j}$.
\end{prop}
\begin{proof}
    The uniqueness of $i$ and $j$ follows from the fact that 
    $\mu_{1}$ and $\mu_{4}$ are different from other eigenvalues of $\mathrm{r}_{0}(\mathrm{c}_{\infty}(\psi))$.

    For any $s\in \mathrm{C}_{\psi}$,
    we have 
    \[\rho^{\vee}_{\psi}(s)=(2\alpha_{1}^{\vee}+3\alpha_{2}^{\vee}+2\alpha_{3}^{\vee}+\alpha_{4}^{\vee})(s)\cdot(\alpha_{1}^{\vee}+2\alpha_{2}^{\vee}+\alpha_{3}^{\vee}+\alpha_{4}^{\vee})(s).\] 
    Since $\mu_{1}\in\mathcal{W}_{i}$,
    the value $(2\alpha_{1}^{\vee}+3\alpha_{2}^{\vee}+2\alpha_{3}^{\vee}+\alpha_{4}^{\vee})(s)$
    is the scalar given by the action of $s$ on the irreducible summand $\pi_{i}[d_{i}]$,
    which equals $\epsilon_{i}(s)$ by definition.
    Similarly, we have $(\alpha_{1}^{\vee}+2\alpha_{2}^{\vee}+\alpha_{3}^{\vee}+\alpha_{4}^{\vee})(s)=\epsilon_{j}(s)$    
    and the identity $\rho^{\vee}_{\psi}=\epsilon_{i}\cdot\epsilon_{j}$.
\end{proof}

\subsection{Classification of Arthur parameters}\label{section classification of Arthur parameters of F4}
 Now we can do (\emph{conjectural}) classification of global Arthur parameters for $\grpF$:
\begin{spthm}\label{thm arthur parameter with nonzero multiplicity}
    Admitting the existence of the Langlands group $\mathcal{L}_{\Z}$ defined in \Cref{section global parametrization and Langlands group} and Arthur's multiplicity formula \Cref{conj Arthur multiplicity formula},
    a (level one) discrete global Arthur parameter $\psi\in\Psi_{\mathrm{AJ}}(\grpF)$
    satisfies $\mathrm{m}(\pi_{\psi})=1$
    if and only if 
    it belongs to the parameters described in the following propositions (from \Cref{prop multiplicity principal A1} to \Cref{prop multiplicity F4 tempered}).  
\end{spthm}
In this subsection,
we will prove \Cref{thm arthur parameter with nonzero multiplicity} case by case,
depending on the conjugacy class $\mathrm{H}(\psi)$ associated to the discrete global Arthur parameter $\psi$.
For each subgroup $H$ of $\lietype{F}{4}=\grpF(\R)$ listed in \Cref{section connected subgroups satisfying our conditions},
we classify all the endoscopic types of $\psi\in\Psi_{\mathrm{AJ}}(\grpF)$ such that $\mathrm{H}(\psi)$ is the conjugacy class of $H$ like what we have done in \Cref{ex endoscopic type from image of parameter},
then apply Arthur's multiplicity formula \Cref{conj Arthur multiplicity formula}, \Cref{prop epsilon character of Arthur} and \Cref{prop calculation of rho character for F4} to $\psi$
and get those with $\mathrm{m}(\pi_{\psi})=1$.
\begin{notation}\label{notation abuse of notations of subgroups}
    From now on, when $\mathrm{H}(\psi)$ is the $\lietype{F}{4}$-conjugacy class of $H$,
    we say $\mathrm{H}(\psi)=H$ by an abuse of notation. 
\end{notation}
\begin{rmk}\label{rmk skip this section at first}
    Since the proof of \Cref{thm arthur parameter with nonzero multiplicity} is long,
    readers can read first the proof of \Cref{prop smallest motivic weight 36 tempered stable representation} in \Cref{section classify representations in the multiplicity space}
    to see how Arthur's conjectures are used.
\end{rmk}

\subsubsection{\texorpdfstring{$H=\mathrm{A}_{1}^{[17,9]}$}{}}\label{section Arthur principal A1}
 The restriction of the $26$-dimensional irreducible representation $\jord_{0}$ to $H$ is isomorphic to 
\[\sym^{16}\st+\sym^{8}\st.\]

For $\psi\in\Psi_{\mathrm{AJ}}(\grpF)$ satisfying $\mathrm{H}(\psi)=H$ and $\mathrm{m}(\pi_{\psi})=1$,
there are two possible endoscopic types:
\begin{enumerate}[label=(\roman*)]
    \item $\left(2,(1,17),(1,9)\right)$, which corresponds to the parameter $[17]\oplus[9]$ of the trivial representation of $\grpF(\A)$.
    \item $\left(2,(17,1),(9,1)\right)$.
        The discrete global Arthur parameters $\psi$ with this type are constructed as follows:
        for a representation $\pi\in\Pi_{\alg}^{\perp}(\pgl_{2})$ and a positive integer $k$,
        we denote by $\sym^{k}\pi$ the representation in $\Pi_{\alg,\reg}^{\perp}(\pgl_{k+1})$ corresponding to the irreducible representation given by  
        \[\mathcal{L}_{\Z}\xrightarrow{\psi_{\pi}}\SL_{2}(\C)\rightarrow \SL(\sym^{k}\st)\simeq \SL_{k+1}(\C).\]
        A global Arthur parameter of this type is of the form: 
        \[\sym^{16}\pi\oplus\sym^{8}\pi,\,\pi\in\Pi_{\alg}^{\perp}(\pgl_{2}).\]
\end{enumerate}

\begin{spprop}\label{prop multiplicity principal A1}
   For a discrete global Arthur parameter $\psi\in\Psi_{\mathrm{AJ}}(\grpF)$ satisfying $\mathrm{H}(\psi)=H$,
    the multiplicity $\mathrm{m}(\pi_{\psi})=1$ if and only if $\psi$ is one of the following parameters: 
   \begin{itemize}
    \item $[17]\oplus[9]$, which corresponds to the trivial representation of $\grpF(\A)$.
    \item $\sym^{16}\pi\oplus\sym^{8}\pi,\,\pi\in\Pi_{\alg}^{\perp}(\pgl_{2})$.
    \end{itemize}
\end{spprop}
\begin{proof}
    This is because $\mathrm{C}_{\psi}$ is trivial.
\end{proof}

\subsubsection{\texorpdfstring{$H=\left(\mathrm{A}_{1}^{[9,6^{2},5]}\times \mathrm{A}_{1}^{[2^{6},1^{14}]}\right)/\mu_{2}^{\Delta}$}{}}\label{section Arthur principal A1 in Sp(3) with Sp(1)}
 By \Cref{section info principal A1 in Sp(3) with Sp(1)} the restriction of the $26$-dimensional irreducible representation $\jord_{0}$ of $\lietype{F}{4}$ to $H$ is isomorphic to 
\[\sym^{5}\st\otimes\st+(\sym^{8}\st+\sym^{4}\st)\otimes\triv,\]
and the centralizer of $H$ in $\lietype{F}{4}$ is $\mathrm{Z}(H)\simeq \Z/2\Z$.

For $\psi\in\Psi_{\mathrm{AJ}}(\grpF)$ satisfying $\mathrm{H}(\psi)=H$ and $\mathrm{m}(\pi_{\psi})=1$,
there are three possible endoscopic types:
\begin{enumerate}[label=(\roman*)]
    \item $(3,(2,6),(1,5),(1,9))$.
    A global Arthur parameter of this type is of the form:
    \[\pi[6]\oplus [5]\oplus [9],\,\pi\in\Pi_{\alg}^{\perp}(\pgl_{2}).\]
    \item $(3,(9,1),(5,1),(6,2))$.
    A global Arthur parameter of this type is of the form:
    \[\sym^{8}\pi\oplus\sym^{4}\pi\oplus \sym^{5}\pi[2],\,\pi\in\Pi_{\alg}^{\perp}(\pgl_{2}).\]
    \item $(3,(12,1),(9,1),(5,1))$.
    For two distinct representations $\pi_{1},\pi_{2}\in\Pi_{\alg}^{\perp}(\pgl_{2})$,
    we can construct the following $12$-dimensional irreducible representation of $\mathcal{L}_{\Z}$:
    \[\mathcal{L}_{\Z}\xrightarrow{(\psi_{\pi_{1}},\psi_{\pi_{2}})} \SL_{2}(\C)\times\SL_{2}(\C)\xrightarrow{\sym^{5}\otimes \mathrm{id}} \SL_{12}(\C),\]
    which induces a cuspidal representation of $\pgl_{12}$, denoted by $\sym^{5}\pi_{1}\otimes \pi_{2}$.
    A global Arthur parameter of this type is of the form:
    \[\sym^{8}\pi_{1}\oplus\sym^{4}\pi_{1}\oplus\left(\sym^{5}\pi_{1}\otimes\pi_{2}\right),\,\pi_{1}\not\simeq\pi_{2}\in\Pi_{\alg}^{\perp}(\pgl_{2}).\]
\end{enumerate}
\begin{rmk}\label{rmk condition of parameter on archimedean place}
    In fact, for a $(3,(12,1),(9,1),(5,1))$-type parameter 
    \[\psi=\sym^{8}\pi_{1}\oplus\sym^{4}\pi_{1}\oplus\left(\sym^{5}\pi_{1}\otimes\pi_{2}\right),\,\pi_{1},\pi_{2}\in\Pi_{\alg}^{\perp}(\pgl_{2}),\]
    there are some conditions on the motivic weights $\mathrm{w}(\pi_{1}),\mathrm{w}(\pi_{2})$ to make $\psi$ a parameter in $\Psi_{\mathrm{AJ}}(\grpF)$.
    We will add these conditions 
    for global Arthur parameters $\psi$ with $m_{\psi}=1$ when necessary. 
    For example, when $\mathrm{w}(\pi_{2})>9\mathrm{w}(\pi_{1})$
    the condition for $\psi\in\Psi(\grpF)$ is that $\mathrm{w}(\pi_{2})\geq 9\mathrm{w}(\pi_{1})+2$,
    which is satisfied automatically since $\mathrm{w}(\pi_{2})$ and $9\mathrm{w}(\pi_{1})$ are two distinct odd numbers.
\end{rmk}

For this subgroup $H$ of $\lietype{F}{4}$,
the restriction of the adjoint representation $\mathfrak{f}_{4}$ of $\lietype{F}{4}$ to $H$ is isomorphic to 
\begin{align*}
    \triv\otimes\sym^{2}\st+
    \left(\sym^{9}\st+\sym^{3}\st\right)\otimes\st+
    \left(\sym^{10}\st+\sym^{6}\st+\sym^{2}\st\right)\otimes\triv.
\end{align*}

\begin{spprop}\label{prop multiplicity principal A1 in Sp(3) with Sp(1)}
    For a discrete global Arthur parameter $\psi\in\Psi_{\mathrm{AJ}}(\grpF)$ satisfying $\mathrm{H}(\psi)=H$,
    the multiplicity $\mathrm{m}(\pi_{\psi})=1$ if and only if $\psi$ is one of the following parameters: 
    \begin{itemize}
    \item $\pi[6]\oplus [5]\oplus [9]$, where $\pi\in\Pi_{\alg}^{\perp}(\pgl_{2})$.
    \item $\sym^{8}\pi\oplus\sym^{4}\pi\oplus \sym^{5}\pi[2]$, where $\pi\in\Pi_{\alg}^{\perp}(\pgl_{2})$ satisfies $\mathrm{w}(\pi)\equiv 3\modulo 4$.
    \item $\sym^{8}\pi_{1}\oplus\sym^{4}\pi_{1}\oplus\left(\sym^{5}\pi_{1}\otimes\pi_{2}\right)$, where $\pi_{1},\pi_{2}\in\Pi_{\alg}^{\perp}(\pgl_{2})$ have motivic weights $w_{1},w_{2}$ respectively such that $w_{2}>9w_{1}$ or $5w_{1}<w_{2}<7w_{1}$.
    \end{itemize}
\end{spprop}
\begin{proof}
    We denote the generator of $\mathrm{C}_{\psi}=\mathrm{Z}(H)$ by $\gamma$.
    
    \textbf{Case (i)}: $\psi=\pi[6]\oplus [5]\oplus [9]$, where $\pi\in\Pi_{\alg}^{\perp}(\pgl_{2})$ has motivic weight $w$. 
    In this case the restriction of $\mathfrak{f}_{4}$ along $\psi$ is isomorphic to 
    \[\sym^{2}\pi\oplus \pi[10]\oplus \pi[4]\oplus [11]\oplus [7]\oplus [3].\]
    By \Cref{prop epsilon character of Arthur}, we have:
    \[\varepsilon_{\psi}(\gamma)=\varepsilon(\pi)\cdot \varepsilon(\pi)=\varepsilon(\mathbf{I}_{w})^{2}=1.\]
    On the other side, since $w\geq 11$ we have $\mu_{1}=\frac{w+5}{2}$ and $\mu_{4}=\frac{w-1}{2}$.
    Both of them come from the irreducible summand $\pi[6]$ in $\psi$, 
    so $\rho^{\vee}_{\psi}$ must be the trivial character by \Cref{prop calculation of rho character for F4}.
    By Arthur's multiplicity formula, 
    $\mathrm{m}(\pi_{\psi})=1$ for any $\pi\in\Pi_{\alg}^{\perp}(\mathrm{PGL}_{2})$.

    \textbf{Case (ii)}: $\psi=\sym^{8}\pi\oplus\sym^{4}\pi\oplus \sym^{5}\pi[2]$, 
    where $\pi\in \Pi_{\alg}^{\perp}(\pgl_{2})$ has motivic weight $w$. 
    In this case the restriction of $\mathfrak{f}_{4}$ along $\psi$ is isomorphic to 
    \[\sym^{10}\pi\oplus\sym^{9}\pi[2]\oplus\sym^{6}\pi\oplus\sym^{3}\pi[2]\oplus\sym^{2}\pi\oplus[3].\]
    By \Cref{prop epsilon character of Arthur}, we have:
    \begin{align*}
    \varepsilon_{\psi}(\gamma)&=\varepsilon(\sym^{3}\pi)\cdot\varepsilon(\sym^{9}\pi)\\
    &=\varepsilon(\mathbf{I}_{3w}+\mathbf{I}_{w})\cdot\varepsilon(\mathbf{I}_{9w}+\mathbf{I}_{7w}+\mathbf{I}_{5w}+\mathbf{I}_{3w}+\mathbf{I}_{w})\\
    &=(-1)^{(w+1)/2+(3w+1)/2}\cdot (-1)^{(w+1)/2+(3w+1)/2+(5w+1)/2+(7w+1)/2+(9w+1)/2}\\
    &=(-1)^{(w+3)/2}.
    \end{align*}
    On the other side, 
    $\mu_{1}=4w$ comes from $\sym^{8}\pi$ 
    and $\mu_{4}=\frac{5w-1}{2}$ comes from $\sym^{5}\pi[2]$.
    So $\rho^{\vee}_{\psi}(\gamma)=-1$ by \Cref{prop calculation of rho character for F4}.
    By Arthur's multiplicity formula, 
    $\mathrm{m}(\pi_{\psi})=1$ if and only if $w\equiv 3\modulo 4$.

    \textbf{Case (iii)}: $\psi=\sym^{8}\pi_{1}\oplus\sym^{4}\pi_{1}\oplus\left(\sym^{5}\pi_{1}\otimes\pi_{2}\right)$,
    where $\pi_{1},\pi_{2}\in\Pi_{\alg}^{\perp}(\mathrm{PGL_{2}})$ have motivic weight $w_{1},w_{2}$ respectively.
    Since this parameter is tempered,
    the character $\varepsilon_{\psi}$ is always trivial.
    We only need to find what condition $w_{1},w_{2}$ should satisfy to make $\rho^{\vee}_{\psi}(\gamma)=1$.
    In this case, $\gamma$ acts on $\sym^{8}\pi_{1}$ and $\sym^{4}\pi_{1}$ by $1$ and on $\sym^{5}\pi_{1}\otimes\pi_{2}$ by $-1$.
    We can see that $\mu_{1}=4w_{1}$ or $\frac{5w_{1}+w_{2}}{2}$, depending on the values of $w_{1},w_{2}$.
    \begin{enumerate}[label= (\arabic*)]
        \item If $\mu_{1}=4w_{1}$, which is equivalent to $w_{2}<3w_{1}$. 
              Now $\rho_{\psi}^{\vee}(\gamma)=1$ if and only if $\mu_{4}=3w_{1}$
              since the other positive weights $w_{1},2w_{1}$ in $\sym^{4}\pi_{1}\oplus\sym^{8}\pi_{1}$ both have multiplicity $2$.
              However, $3w_{1}$ is larger than all the Hodge weights of $\psi$ 
              except $4w_{1}$ and $\frac{5w_{1}+w_{2}}{2}$, which shows that it can only be $\mu_{2}$ or $\mu_{3}$. 
              So in this case $\rho_{\psi}^{\vee}(\gamma)=-1$.
        \item If $\mu_{1}=\frac{5w_{1}+w_{2}}{2}$, which is equivalent to $w_{2}>3w_{1}$. 
              Now $\rho^{\vee}_{\psi}(\gamma)=1$ if and only if $\mu_{4}=\frac{w_{1}+w_{2}}{2}$ or $\frac{-w_{1}+w_{2}}{2}$.
              \begin{enumerate}
                  \item $\mu_{4}=\frac{w_{1}+w_{2}}{2}$ is equivalent to $4w_{1}>\frac{w_{1}+w_{2}}{2}>3w_{1}$, thus $5w_{1}<w_{2}<7w_{1}$.
                  \item $\mu_{4}=\frac{-w_{1}+w_{2}}{2}$ is equivalent to $\frac{-w_{1}+w_{2}}{2}>4w_{1}$, thus $w_{2}>9w_{1}$.
              \end{enumerate} 
    \end{enumerate} 
    By Arthur's multiplicity formula $\mathrm{m}(\pi_{\psi})=1$ if and only if $w_{2}>9w_{1}$ or $5w_{1}<w_{2}<7w_{1}$.
\end{proof}

\subsubsection{\texorpdfstring{$H=\left(\mathrm{A}_{1}^{[5,4^{2},3^{3},2^{2}]}\times \mathrm{A}_{1}^{[3^{3},2^{6},1^{5}]}\right)/\mu_{2}^{\Delta}$}{}}
\label{section Arthur principal A1 of Sp(1)+SO(3) with diagonal Sp(1)}
 By \Cref{section info principal A1 of Sp(1)+SO(3) with diagonal Sp(1)} the restriction of the $26$-dimensional irreducible representation $\jord_{0}$ of $\lietype{F}{4}$ to $H$ is isomorphic to 
\[\sym^{4}\st\otimes\triv+
\left(\sym^{3}\st+\st\right)\otimes\st+
\sym^{2}\st\otimes \sym^{2}\st,\]
and the centralizer of $H$ in $\lietype{F}{4}$ is $\mathrm{Z}(H)\simeq \Z/2\Z$.

For $\psi\in\Psi_{\mathrm{AJ}}(\grpF)$ satisfying $\mathrm{H}(\psi)=H$ and $\mathrm{m}(\pi_{\psi})=1$,
there are three possible endoscopic types:
\begin{enumerate}[label=(\roman*)]
    \item $(4,(3,3),(2,4),(2,2),(1,5))$.
    A global Arthur parameter of this type is of the form:
        \[\sym^{2}\pi[3]\oplus \pi[4]\oplus\pi[2]\oplus[5],\,\pi\in\Pi_{\alg}^{\perp}(\pgl_{2}).\]
    \item $(4,(5,1),(4,2),(3,3),(2,2))$.
    A global Arthur parameter of this type is of the form:
        \[\sym^{4}\pi\oplus \sym^{3}\pi[2]\oplus \sym^{2}\pi[3]\oplus\pi[2],\,\pi\in\Pi_{\alg}^{\perp}(\pgl_{2}).\]
    \item $(4,(9,1),(8,1),(5,1),(4,1))$.
    A global Arthur parameter of this type is of the form:
        \[\sym^{4}\pi_{1}\oplus(\sym^{3}\pi_{1}\otimes\pi_{2})\oplus(\sym^{2}\pi_{1}\otimes\sym^{2}\pi_{2})\oplus(\pi_{1}\otimes\pi_{2}),\,\pi_{1},\pi_{2}\in\Pi_{\alg}^{\perp}(\pgl_{2}),\]
        where the representations $\sym^{k}\pi_{1}\otimes\sym^{l}\pi_{2}$ are defined similarly as the representation $\sym^{5}\pi_{1}\otimes\pi_{2}$ appearing in $[(12,1),(9,1),(5,1)]$-type parameters
        introduced in \Cref{section Arthur principal A1 in Sp(3) with Sp(1)}.
\end{enumerate}

For this subgroup $H$ of $\lietype{F}{4}$,
the restriction of the adjoint representation $\mathfrak{f}_{4}$ of $\lietype{F}{4}$ to $H$ is isomorphic to 
\begin{align*}
\st\otimes\sym^{3}\st+ 
\left(\sym^{4}\st+\triv\right)\otimes\sym^{2}\st+
\left(\sym^{5}\st+\sym^{3}\st\right)\otimes\st+
\left(\sym^{2}\st\right)^{\oplus 2}\otimes\triv.
\end{align*}
\begin{spprop}\label{prop multiplicity principal A1 of Sp(1)+SO(3) with diagonal Sp(1)}
    For a discrete global Arthur parameter $\psi\in\Psi_{\mathrm{AJ}}(\grpF)$ satisfying $\mathrm{H}(\psi)=H$,
    the multiplicity $\mathrm{m}(\pi_{\psi})=1$ if and only if $\psi$ is one of the following parameters: 
   \begin{itemize}
    \item $\sym^{2}\pi[3]\oplus \pi[4]\oplus\pi[2]\oplus[5]$, where $\pi\in\Pi_{\alg}^{\perp}(\pgl_{2})$.
    \item $\sym^{4}\pi\oplus \sym^{3}\pi[2]\oplus \sym^{2}\pi[3]\oplus\pi[2]$, where $\pi\in\Pi_{\alg}^{\perp}(\pgl_{2})$.
    \item $\sym^{4}\pi_{1}\oplus(\sym^{3}\pi_{1}\otimes\pi_{2})\oplus(\sym^{2}\pi_{1}\otimes\sym^{2}\pi_{2})\oplus(\pi_{1}\otimes\pi_{2})$, 
    where $\pi_{1},\pi_{2}\in\Pi_{\alg}^{\perp}(\pgl_{2})$ have motivic weights $w_{1},w_{2}$ respectively 
    such that 
    \[w_{1}>3w_{2}\text{ or }w_{1}<w_{2}<3w_{1}\text{ or }3w_{1}<w_{2}<5w_{1}.\]
    \end{itemize} 
\end{spprop}
\begin{proof}
    We denote the generator of $\mathrm{C}_{\psi}=\mathrm{Z}(H)$ by $\gamma$.
    
    \textbf{Case (i)}: $\psi=\sym^{2}\pi[3]\oplus \pi[4]\oplus\pi[2]\oplus[5]$,
    where $\pi\in \Pi^{\perp}_{\alg}(\mathrm{PGL}_{2})$ has motivic weight $w$.
    In this case the restriction of $\mathfrak{f}_{4}$ along $\psi$ is isomorphic to 
    \[\sym^{3}\pi[2]\oplus\sym^{2}\pi[5]\oplus\sym^{2}\pi\oplus\pi[6]\oplus\pi[4]\oplus[3]\oplus [3].\]
    By \Cref{prop epsilon character of Arthur}, we have:
    \begin{align*}
        \varepsilon_{\psi}(\gamma)&=\varepsilon(\sym^{3}\pi)\cdot\varepsilon(\pi)\cdot\varepsilon(\pi)
        =\varepsilon(\mathbf{I}_{3w}+\mathbf{I}_{w})\cdot \varepsilon(\mathbf{I}_{w})^{2}
        =(-1)^{2w+1}=-1.
    \end{align*}
    On the other side, $\mu_{1}=w+1$ comes from $\sym^{2}\pi[3]$ and $\mu_{4}=\frac{w+3}{2}$ comes from $\pi[4]$.
    Since $\gamma$ acts on $\sym^{2}\pi[3]$ by $1$ and on $\pi[4]$ by $-1$, 
    we have $\rho_{\psi}^{\vee}(\gamma)=-1$ by \Cref{prop calculation of rho character for F4}.
    By Arthur's multiplicity formula, 
    $\mathrm{m}(\pi_{\psi})=1$ for any $\pi\in\Pi_{\alg}^{\perp}(\pgl_{2})$.
    
    \textbf{Case (ii)}: $\psi=\sym^{4}\pi\oplus \sym^{3}\pi[2]\oplus \sym^{2}\pi[3]\oplus\pi[2]$,
    where $\pi\in \Pi_{\alg}^{\perp}(\mathrm{PGL}_{2})$ has motivic weight $w$.
    In this case
    the restriction of $\mathfrak{f}_{4}$ along $\psi$ is isomorphic to 
    \[\sym^{5}\pi[2]\oplus\sym^{4}\pi[3]\oplus\sym^{3}\pi[2]\oplus(\sym^{2}\pi)^{\oplus 2}\oplus\pi[4]\oplus[3].\]
    By \Cref{prop epsilon character of Arthur}, we have:
    \begin{align*}
        \varepsilon_{\psi}(\gamma)=\varepsilon(\pi)\cdot\varepsilon(\sym^{3}\pi)\cdot\varepsilon(\sym^{5}\pi)
        =\varepsilon(\mathbf{I}_{w})\varepsilon(\mathbf{I}_{3w}+\mathbf{I}_{w})\varepsilon(\mathbf{I}_{5w}+\mathbf{I}_{3w}+\mathbf{I}_{w})
        =(-1)^{3w+1}=1.
    \end{align*}
    On the other side, $\mu_{1}=2w$ comes from $\sym^{4}\pi$ and $\mu_{4}=w+1$ comes from $\sym^{2}\pi[3]$.
    Since $\gamma$ acts on $\sym^{4}\pi$ and $\sym^{2}\pi[3]$ both by $1$, 
    we have $\rho_{\psi}^{\vee}(\gamma)=1$ by \Cref{prop calculation of rho character for F4}.
    Arthur's multiplicity formula shows that $\mathrm{m}(\pi_{\psi})=1$ for any $\pi\in\Pi_{\alg}^{\perp}(\pgl_{2})$.

    \textbf{Case (iii)}: 
    $\psi=\sym^{4}\pi_{1}\oplus(\sym^{3}\pi_{1}\otimes\pi_{2})\oplus(\sym^{2}\pi_{1}\otimes\sym^{2}\pi_{2})\oplus(\pi_{1}\otimes\pi_{2})$,
    where $\pi_{1},\pi_{2}\in\Pi_{\alg}^{\perp}(\mathrm{PGL_{2}})$ have motivic weights $w_{1},w_{2}$ respectively.
    The motivic weights satisfy $w_{2}\neq w_{1},w_{2}\neq 3w_{1}$,
    otherwise the zero weight appears more than twice and $\psi$ fails to be in $\Psi_{\mathrm{AJ}}(\grpF)$.
    In this case $\varepsilon_{\psi}$ is trivial.
    The element $\gamma$ acts on $\sym^{4}\pi_{1}$ and $\sym^{2}\pi_{1}\otimes\sym^{2}\pi_{2}$ by $1$, 
    and on $\sym^{3}\pi_{1}\otimes\pi_{2},\pi_{1}\otimes\pi_{2}$ by $-1$.
    The largest weight $\mu_{1}$ is $2w_{1}$ or $w_{1}+w_{2}$.
    \begin{enumerate}[label= (\arabic*)]
        \item If $w_{1}>w_{2}$, then $\mu_{1}=2w_{1}$. 
            Now $\mu_{4}$ equals to $\frac{3w_{1}-w_{2}}{2}$ or $w_{1}+w_{2}$.
            The character $\rho_{\psi}^{\vee}$ is trivial if and only if $\mu_{4}=w_{1}+w_{2}$, 
            which is equivalent to $w_{1}>3w_{2}$.
        \item If $w_{1}<w_{2}$, then $\mu_{1}=w_{1}+w_{2}$.
            \begin{enumerate}
                \item If $w_{2}>3w_{1}$, then 
                \[w_{1}+w_{2}>w_{2}>\max(-w_{1}+w_{2},\frac{3w_{1}+w_{2}}{2})>\min(-w_{1}+w_{2},\frac{3w_{1}+w_{2}}{2})\]
                and they are larger than other weights,
                thus $\mu_{4}=-w_{1}+w_{2}$ or $\frac{3w_{1}+w_{2}}{2}$. 
                So $\rho_{\psi}^{\vee}(\gamma)=1$ if and only if $\mu_{4}=-w_{1}+w_{2}$, 
                thus if and only if $\frac{3w_{1}+w_{2}}{2}>w_{2}-w_{1}$, 
                which is equivalent to that $3w_{1}<w_{2}<5w_{1}$.
                \item If $w_{2}<3w_{1}$, then 
                \[w_{1}+w_{2}>\frac{3w_{1}+w_{2}}{2}>\max(2w_{1},w_{2})>\min(2w_{1},w_{2})\]
                and they are larger than other weights. 
                So we always have $\rho_{\psi}^{\vee}(\gamma)=1$.
            \end{enumerate}
    \end{enumerate} 
    By Arthur's multiplicity formula, $\mathrm{m}(\pi_{\psi})=1$ if and only if $w_{1}>3w_{2}$ or $w_{1}<w_{2}<5w_{1}$ and $w_{2}\neq 3w_{1}$.
\end{proof}

\subsubsection{\texorpdfstring{$H=\left(\mathrm{A}_{1}^{[4^{2},3^{3},2^{4},1]}\times \mathrm{A}_{1}^{[4^{2},3^{3},2^{4},1]}\right)/\mu_{2}^{\Delta}$}{}}
\label{section Arthur Spin(3)+Spin(3)}
 By \Cref{section info Spin(3)+Spin(3)},
the restriction of the $26$-dimensional irreducible representation $\jord_{0}$ of $\lietype{F}{4}$ to $H$ is isomorphic to 
\[\triv+ 
\sym^{3}\st\otimes\st+
\sym^{2}\st\otimes\sym^{2}\st+
\st\otimes\sym^{3}\st,\]
and the centralizer of $H$ in $\lietype{F}{4}$ is $\mathrm{Z}(H)\simeq \Z/2\Z$.

For $\psi\in\Psi_{\mathrm{AJ}}(\grpF)$ satisfying $\mathrm{H}(\psi)=H$ and $\mathrm{m}(\pi_{\psi})=1$,
there are two possible endoscopic types:
\begin{enumerate}[label=(\roman*)]
    \item $(4,(4,2),(3,3),(2,4),(1,1))$.
    A global Arthur parameter of this type is of the form:
        \[\sym^{3}\pi[2]\oplus\sym^{2}\pi[3]\oplus\pi[4]\oplus[1],\,\pi\in\Pi_{\alg}^{\perp}(\pgl_{2}).\]
    \item $(4,(9,1),(8,1),(8,1),(1,1))$.
        A global Arthur parameter of this type is of the form:
        \[(\sym^{3}\pi_{1}\otimes\pi_{2})\oplus(\sym^{2}\pi_{1}\otimes\sym^{2}\pi_{2})\oplus(\pi_{1}\otimes\sym^{3}\pi_{2})\oplus[1],\,\pi_{1}\not\simeq\pi_{2}\in\Pi_{\alg}^{\perp}(\pgl_{2}).\]
\end{enumerate}

For this subgroup $H$ of $\lietype{F}{4}$,
the restriction of the adjoint representation $\mathfrak{f}_{4}$ of $\lietype{F}{4}$ to $H$ is isomorphic to 
\begin{align*}
    \left(\sym^{4}\st+\triv\right)\otimes\sym^{2}\st+ 
    \sym^{2}\st\otimes\left(\sym^{4}\st+\triv\right)+
    \sym^{3}\st\otimes\st+ 
    \st\otimes\sym^{3}\st.
\end{align*}
\begin{spprop}\label{prop multiplicity Spin(3)+Spin(3)}
    A discrete global Arthur parameter $\psi\in\Psi_{\mathrm{AJ}}(\grpF)$ satisfying $\mathrm{H}(\psi)=H$ and $\mathrm{m}(\pi_{\psi})=1$ 
    if and only if $\psi$ is one of the following parameters: 
   \begin{itemize}
    \item $\sym^{3}\pi[2]\oplus\sym^{2}\pi[3]\oplus\pi[4]\oplus[1]$, where $\pi\in\Pi_{\alg}^{\perp}(\pgl_{2})$ satisfies $\mathrm{w}(\pi)\equiv 3\modulo 4$.
    \item $(\sym^{3}\pi_{1}\otimes\pi_{2})\oplus(\sym^{2}\pi_{1}\otimes\sym^{2}\pi_{2})\oplus(\pi_{1}\otimes\sym^{3}\pi_{2})\oplus[1]$, where $\pi_{1},\pi_{2}$ have motivic weights $w_{1},w_{2}$ respectively 
    such that $w_{2}<w_{1}<3w_{2}$.
    \end{itemize}
\end{spprop}
\begin{proof}
    We denote the generator of $\mathrm{C}_{\psi}=\mathrm{Z}(H)$ by $\sigma$.

    \textbf{Case (i)}: $\psi=\sym^{3}\pi[2]\oplus\sym^{2}\pi[3]\oplus\pi[4]\oplus[1]$, 
    where $\pi\in \Pi_{\alg}^{\perp}(\mathrm{PGL}_{2})$ has motivic weight $w$.
    In this case
    the restriction of $\mathfrak{f}_{4}$ along $\psi$ is isomorphic to
    \[\sym^{4}\pi[3]\oplus\sym^{3}\pi[2]\oplus\sym^{2}\pi[5]\oplus\sym^{2}\pi\oplus\pi[4]\oplus[3].\]
    By \Cref{prop epsilon character of Arthur}, we have:
    \begin{align*}
        \varepsilon_{\psi}(\sigma)=\varepsilon(\sym^{3}\pi)\cdot\varepsilon(\pi)=\varepsilon(\mathbf{I}_{3w}+\mathbf{I}_{w})\cdot\varepsilon(\mathbf{I}_{w})=(-1)^{(3w+1)/2}.
    \end{align*}
    On the other side, $\mu_{1}=\frac{3w+1}{2}$ comes from $\sym^{3}\pi[2]$ and $\mu_{4}=w$ comes from $\sym^{2}\pi[3]$.
    Since $\sigma$ acts on $\sym^{3}\pi[2]$ by $-1$ and on $\sym^{2}\pi[3]$ by $1$,
    we have $\rho^{\vee}_{\psi}(\sigma)=-1$ by \Cref{prop calculation of rho character for F4}.
    By Arthur's multiplicity formula, $\mathrm{m}(\pi_{\psi})=1$ if and only if $w\equiv 3\modulo 4$.

    \textbf{Case (ii)}: $\psi=(\sym^{3}\pi_{1}\otimes\pi_{2})\oplus(\sym^{2}\pi_{1}\otimes\sym^{2}\pi_{2})\oplus(\pi_{1}\otimes\sym^{3}\pi_{2})\oplus[1]$,
    where $\pi_{1},\pi_{2}\in\Pi_{\alg}^{\perp}(\mathrm{PGL_{2}})$ have motivic weights $w_{1}>w_{2}$ respectively.
    In this case, $\varepsilon_{\psi}$ is trivial.
    On the other side, $\mu_{1}=\frac{3w_{1}+w_{2}}{2}$ and 
    $\mu_{4}=w_{1}$ or $\frac{w_{1}+3w_{2}}{2}$ or $\frac{3w_{1}-w_{2}}{2}$.
    By \Cref{prop calculation of rho character for F4},
    $\rho_{\psi}^{\vee}$ is trivial if and only if $\mu_{4}=\frac{w_{1}+3w_{2}}{2}$ or $\frac{3w_{1}-w_{2}}{2}$.
    \begin{enumerate}[label= (\arabic*)]
        \item $\mu_{4}=\frac{w_{1}+3w_{2}}{2}$ if and only if $\frac{3w_{1}-w_{2}}{2}>\frac{w_{1}+3w_{2}}{2}>w_{1}$, which is equivalent to $2w_{2}<w_{1}<3w_{2}$.
        \item $\mu_{4}=\frac{3w_{1}-w_{2}}{2}$ if and only if $\frac{w_{1}+3w_{2}}{2}>\frac{3w_{1}-w_{2}}{2}$, which is equivalent to $w_{1}<2w_{2}$.
    \end{enumerate}
    By Arthur's multiplicity formula, $\mathrm{m}(\pi_{\psi})=1$ if and only if $w_{2}<w_{1}<3w_{2}$ and $w_{1}\neq 2w_{2}$.
    Notice that $w_{1}\neq 2w_{2}$ holds automatically since $w_{1}$ is odd.
\end{proof}

\subsubsection{\texorpdfstring{$H=\mathrm{A}_{1}^{[7^{3},1^{5}]}\times \mathrm{A}_{1}^{[5,3^{7}]}$}{}}
\label{section Arthur principal PSU(2) in G2 + SO(3)}
 By \Cref{section info principal PSU(2) in G2 + SO(3)},
the restriction of the $26$-dimensional irreducible representation $\jord_{0}$ of $\lietype{F}{4}$ to $H$ is isomorphic to 
\[\sym^{6}\st\otimes\sym^{2}\st+
\triv\otimes\sym^{4}\st,\]
and the centralizer of $H$ in $\lietype{F}{4}$ is trivial.

For $\psi\in\Psi_{\mathrm{AJ}}(\grpF)$ satisfying $\mathrm{H}(\psi)=H$ and $\mathrm{m}(\pi_{\psi})=1$,
there are three possible endoscopic types:
\begin{enumerate}[label=(\roman*)]
    \item $(2,(7,3),(1,5))$.
    A global Arthur parameter of this type is of the form:
        \[\sym^{6}\pi[3]\oplus[5],\,\pi\in\Pi_{\alg}^{\perp}(\pgl_{2}).\]
    \item $(2,(5,1),(3,7))$.
    A global Arthur parameter of this type is of the form:
        \[\sym^{4}\pi\oplus \sym^{2}\pi[7],\,\pi\in\Pi_{\alg}^{\perp}(\pgl_{2}).\]
    \item $(2,(21,1),(5,1))$.
    A global Arthur parameter of this type is of the form:
        \[\left(\sym^{6}\pi_{1}\otimes\sym^{2}\pi_{2}\right)\oplus \sym^{4}\pi_{2},\,\pi_{1},\pi_{2}\in\Pi_{\alg}^{\perp}(\pgl_{2}).\]
\end{enumerate}
\begin{spprop}\label{prop multiplicity principal PSU(2) in G2 + SO(3)}
    A discrete global Arthur parameter $\psi\in\Psi_{\mathrm{AJ}}(\grpF)$ satisfying $\mathrm{H}(\psi)=H$ and $\mathrm{m}(\pi_{\psi})=1$ 
    if and only if $\psi$ is one of the following parameters: 
   \begin{itemize}
    \item $\sym^{6}\pi[3]\oplus[5]$, where $\pi\in\Pi_{\alg}^{\perp}(\pgl_{2})$.
    \item $\sym^{4}\pi\oplus \sym^{2}\pi[7]$, where $\pi\in\Pi_{\alg}^{\perp}(\pgl_{2})$. 
    \item $\left(\sym^{6}\pi_{1}\otimes\sym^{2}\pi_{2}\right)\oplus \sym^{4}\pi_{2}$,
    where $\pi_{1},\pi_{2}\in\Pi_{\alg}^{\perp}(\mathrm{PGL_{2}})$ have motivic weights $w_{1},w_{2}$ respectively such that $w_{2}\neq w_{1}$ and $w_{2}\neq 3w_{1}$.
\end{itemize}
\end{spprop}
\begin{proof}
    This follows from the fact that $\mathrm{C}_{\psi}$ is trivial.
    The conditions $w_{2}\neq w_{1}$ and $w_{2}\neq 3w_{1}$ in the third case 
    are equivalent to that $\psi=\left(\sym^{6}\pi_{1}\otimes\sym^{2}\pi_{2}\right)\oplus\sym^{4}\pi_{2}\in\Psi_{\mathrm{AJ}}(\grpF)$.
\end{proof}

\subsubsection{\texorpdfstring{$H=\mathrm{A}_{1}^{[5,3^{7}]}\times \left(\mathrm{A}_{1}^{[3^{3},2^{6},1^{5}]}\times \mathrm{A}_{1}^{[2^{6},1^{14}]}\right)/\mu_{2}^{\Delta}$}{}}
\label{section Arthur a bizzare  A1+A1+A1 via A1+Sp(3)}
 By \Cref{section info a bizzare A1+A1+A1 via A1+Sp(3)},
the restriction of the $26$-dimensional irreducible representation $\jord_{0}$ of $\lietype{F}{4}$ to $H$ is isomorphic to 
\[\sym^{4}\st\otimes\triv\otimes\triv+
\sym^{2}\st\otimes \left(\st\otimes \st+\sym^{2}\st\otimes\triv\right),\]
and the centralizer of $H$ in $\lietype{F}{4}$ is $\mathrm{Z}(H)\simeq \Z/2\Z$.

For $\psi\in\Psi_{\mathrm{AJ}}(\grpF)$ satisfying $\mathrm{H}(\psi)=H$ and $\mathrm{m}(\pi_{\psi})=1$,
there are four possible endoscopic types:
\begin{enumerate}[label=(\roman*)]
    \item $(3,(6,2),(5,1),(3,3))$.
    A global Arthur parameter of this type is of the form:
        \[\sym^{4}\pi_{1}\oplus(\sym^{2}\pi_{1}\otimes\pi_{2}[2])\oplus\sym^{2}\pi_{1}[3],\,\pi_{1},\pi_{2}\in\Pi_{\alg}^{\perp}(\pgl_{2}).\]
    \item $(3,(9,1),(6,2),(5,1))$.
    A global Arthur parameter of this type is of the form:
        \[\sym^{4}\pi_{1}\oplus(\sym^{2}\pi_{1}\otimes\pi_{2}[2])\oplus(\sym^{2}\pi_{1}\otimes\sym^{2}\pi_{2}),\,\pi_{1},\pi_{2}\in\Pi_{\alg}^{\perp}(\pgl_{2}).\]
    \item $(3,(4,3),(3,3),(1,5))$.
    A global Arthur parameter of this type is of the form:
        \[\sym^{2}\pi_{1}[3]\oplus (\pi_{1}\otimes\pi_{2}[3])\oplus[5],\,\pi_{1},\pi_{2}\in\Pi_{\alg}^{\perp}(\pgl_{2}).\]
    \item $(3,(12,1),(9,1),(5,1))$.
    A global Arthur parameter of this type is of the form:
        \[\sym^{4}\pi_{1}\oplus(\sym^{2}\pi_{1}\otimes\pi_{2}\otimes\pi_{3})\oplus(\sym^{2}\pi_{1}\otimes\sym^{2}\pi_{3}),\,\pi_{1},\pi_{2},\pi_{3}\in\Pi_{\alg}^{\perp}(\pgl_{2}).\]
\end{enumerate}

For this subgroup $H$ of $\lietype{F}{4}$,
the restriction of the adjoint representation $\mathfrak{f}_{4}$ of $\lietype{F}{4}$ to $H$ is isomorphic to 
\begin{gather*}
    \sym^{4}\st\otimes \left(\st\otimes \st+\sym^{2}\st\otimes\triv\right)+ 
    \sym^{2}\st\otimes\triv\otimes\triv\\
    +\triv\otimes\left(\sym^{2}\st\otimes\triv+\triv\otimes\sym^{2}\st+\sym^{3}\st\otimes\st\right).
\end{gather*}
\begin{spprop}\label{prop multiplicity a bizzare  A1+A1+A1 via A1+Sp(3)}
    For a discrete global Arthur parameter $\psi\in\Psi_{\mathrm{AJ}}(\grpF)$ satisfying $\mathrm{H}(\psi)=H$,
    the multiplicity $\mathrm{m}(\pi_{\psi})=1$ if and only if $\psi$ is one of the following parameters: 
   \begin{itemize}
    \item $\sym^{4}\pi_{1}\oplus(\sym^{2}\pi_{1}\otimes\pi_{2}[2])\oplus\sym^{2}\pi_{1}[3]$, 
    where $\pi_{1},\pi_{2}\in\Pi_{\alg}^{\perp}(\pgl_{2})$ have motivic weights $w_{1},w_{2}$ respectively such that $w_{2}<2w_{1}-1$ or $w_{2}>4w_{1}+1$.
    \item $\sym^{4}\pi_{1}\oplus(\sym^{2}\pi_{1}\otimes\pi_{2}[2])\oplus(\sym^{2}\pi_{1}\otimes\sym^{2}\pi_{2})$, 
    where $\pi_{1},\pi_{2}\in\Pi_{\alg}^{\perp}(\pgl_{2})$ have motivic weights $w_{1},w_{2}$ respectively and satisfy one of the following conditions:
    \begin{itemize}
        \item $2w_{1}+1<w_{2}<4w_{1}-1,\ w_{2}\equiv 1\modulo 4$;
        \item $w_{2}<2w_{1}-1$ or $w_{2}>4w_{1}+1$, and $w_{2}\equiv 3\modulo 4,\ w_{1}\neq w_{2}$.
    \end{itemize}
    \item $\sym^{2}\pi_{1}[3]\oplus (\pi_{1}\otimes\pi_{2}[3])\oplus[5]$, 
    where $\pi_{1},\pi_{2}\in\Pi_{\alg}^{\perp}(\pgl_{2})$ have motivic weights $w_{1},w_{2}$ respectively such that $w_{2}>3w_{1}$.
    \item $\sym^{4}\pi_{1}\oplus(\sym^{2}\pi_{1}\otimes\pi_{2}\otimes\pi_{3})\oplus(\sym^{2}\pi_{1}\otimes\sym^{2}\pi_{3})$, 
    where $\pi_{1},\pi_{2},\pi_{3}\in\Pi_{\alg}^{\perp}(\pgl_{2})$ have motivic weights $w_{1},w_{2},w_{3}$ respectively such that 
    one of the following conditions holds:
    \begin{itemize}
        \item $w_{2}>\max(3w_{3},4w_{1}+w_{3})$;
        \item $2w_{1}+w_{3}<w_{2}<4w_{1}-w_{3}$;
        \item $3w_{3}<w_{2}<2w_{1}-w_{3}$;
        \item $2w_{1}+w_{3}<w_{2}<\min(4w_{1}+w_{3},3w_{3})$;
        \item $|4w_{1}-w_{3}|<w_{2}<w_{3}-2w_{1}$;
        \item $|2w_{1}-w_{3}|<w_{2}<\min(4w_{1}-w_{3},3w_{3})$ and $w_{3}\neq w_{1}$, $w_{3}\neq w_{2}$.
    \end{itemize} 
\end{itemize}
\end{spprop}
\begin{proof}
    We denote the generator of $\mathrm{C}_{\psi}$ by $\gamma=(1,-1,1)\in\mathrm{Z}(H)$.

    \textbf{Case (i)}: $\psi=\sym^{4}\pi_{1}\oplus\sym^{2}\pi_{1}\otimes\pi_{2}[2]\oplus\sym^{2}\pi_{1}[3]$, 
    where $\pi_{1},\pi_{2}\in\Pi_{\alg}^{\perp}(\mathrm{PGL}_{2})$ have motivic weights $w_{1},w_{2}$ respectively.
    In this case
    the restriction of $\mathfrak{f}_{4}$ along $\psi$ is isomorphic to 
    \begin{gather*}
        \left(\sym^{4}\pi_{1}\otimes\pi_{2}[2]\right)\oplus\sym^{4}\pi_{1}[3]\oplus\sym^{2}\pi_{1}\oplus\sym^{2}\pi_{2}\oplus\pi_{2}[4]\oplus[3].
    \end{gather*}
    By \Cref{prop epsilon character of Arthur} we have $\varepsilon_{\psi}(\gamma)=\varepsilon(\sym^{4}\pi_{1}\otimes\pi_{2})\cdot \varepsilon(\pi_{2})$.
    Notice that 
    \[\varepsilon(\mathbf{I}_{w}\otimes\mathbf{I}_{w^{\prime}})=\varepsilon(\mathbf{I}_{w+w^{\prime}}+\mathbf{I}_{|w-w^{\prime}|})=i^{w+w^{\prime}+|w-w^{\prime}|+2}=(-1)^{\max(w,w^{\prime})+1},\]
    thus 
    \begin{align*}
        \varepsilon_{\psi}(\gamma)=\varepsilon\left((\mathbf{I}_{4w_{1}}+\mathbf{I}_{3w_{1}}+\mathbf{I}_{2w_{1}}+\mathbf{I}_{w_{1}})\otimes\mathbf{I}_{w_{2}}\right)=(-1)^{\max(4w_{1},w_{2})+\max(2w_{1},w_{2})}.
    \end{align*}
    Hence $\varepsilon_{\psi}(\gamma)=1$ if and only if $w_{2}<2w_{1}$ or $w_{2}>4w_{1}$.
    On the other side, $\mu_{1}=2w_{1}$ or $w_{1}+\frac{w_{2}+1}{2}$.
    The generator $\gamma$ of $\mathrm{C}_{\psi}$ acts on $\sym^{4}\pi_{1}$ and $\sym^{2}\pi_{1}[3]$ by $1$ and on $\sym^{2}\pi_{1}\otimes\pi_{2}[2]$ by $-1$.
    We also notice that $\psi\in\Psi_{\mathrm{AJ}}(\grpF)$ implies that $w_{2}\notin\{2w_{1}\pm 1,4w_{1}\pm1\}$.
    \begin{enumerate}[label= (\arabic*)]
        \item If $w_{2}<2w_{1}-1$, then $\mu_{1}=2w_{1}$. 
        Now we have $2w_{1}>w_{1}+\frac{w_{2}+1}{2}>w_{1}+\frac{w_{2}-1}{2}>w_{1}+1$ and they are larger than other Hodge weights, 
        thus $\mu_{4}=w_{1}+1$. Hence $\rho_{\psi}^{\vee}(\gamma)=1$.
        \item If $w_{2}>2w_{1}+1$, then $\mu_{1}=w_{1}+\frac{w_{2}+1}{2}$.
                Now \[w_{1}+\frac{w_{2}+1}{2}>w_{1}+\frac{w_{2}-1}{2}>\max(2w_{1},\frac{w_{2}+1}{2})>\min(2w_{1},\frac{w_{2}-1}{2})\geq w_{1}+1\]
                and they are larger than other weights. 
                So $\mu_{4}=2w_{1}$ or $\frac{w_{2}+1}{2},\frac{w_{2}-1}{2}$.
                However, if $\mu_{4}=2w_{1}$, then we must have $\frac{w_{2}-1}{2}<2w_{1}<\frac{w_{2}+1}{2}$, which is absurd because there is no integer between $\frac{w_{2}-1}{2}$ and $\frac{w_{2}+1}{2}$.
                Hence $\mu_{4}=\frac{w_{2}\pm 1}{2}$ and $\rho_{\psi}^{\vee}(\gamma)=1$.
    \end{enumerate}
    In conclusion, $\rho_{\psi}^{\vee}(\gamma)=1$ for any $\pi_{1},\pi_{2}$. 
    By Arthur's multiplicity formula, 
    $\mathrm{m}(\pi_{\psi})=1$ if and only if $w_{2}<2w_{1}-1$ or $w_{2}>4w_{1}+1$.

    \textbf{Case (ii)}: $\psi=\sym^{4}\pi_{1}\oplus(\sym^{2}\pi_{1}\otimes\pi_{2}[2])\oplus(\sym^{2}\pi_{1}\otimes\sym^{2}\pi_{2})$, 
    where $\pi_{1},\pi_{2}\in\Pi_{\alg}^{\perp}(\pgl_{2})$ have motivic weights $w_{1},w_{2}$ respectively.
    In this case
    the restriction of $\mathfrak{f}_{4}$ along $\psi$ is isomorphic to 
    \begin{gather*}
        \left(\sym^{4}\pi_{1}\otimes\sym^{2}\pi_{2}\right)\oplus\left(\sym^{4}\pi_{1}\otimes\pi_{2}[2]\right)\oplus\sym^{3}\pi_{2}[2]\oplus\sym^{2}\pi_{1}\oplus\sym^{2}\pi_{2}\oplus[3].
    \end{gather*}
    By \Cref{prop epsilon character of Arthur} we have:
    \begin{align*}
        \varepsilon_{\psi}(\gamma)&=\varepsilon(\sym^{4}\pi_{1}\otimes\pi_{2})\cdot\varepsilon(\sym^{3}\pi_{2})
        =(-1)^{\max(4w_{1},w_{2})+\max(2w_{1},w_{2})+(w_{2}-1)/2}.
    \end{align*}
    On the other side, $\gamma$ acts on $\sym^{4}\pi_{1}, \sym^{2}\pi_{1}\otimes\sym^{2}\pi_{2}$ by $1$ and on $\sym^{2}\pi_{1}\otimes\pi_{2}[2]$ by $-1$.
    \begin{enumerate}[label= (\arabic*)]
        \item If $w_{1}>w_{2}$,then $\mu_{1}=2w_{1}$. Now $\mu_{4}$ must be $w_{1}+\frac{w_{2}-1}{2}$ and we have $\rho_{\psi}^{\vee}(\gamma)=-1$.
        \item If $w_{1}<w_{2}$, then $\mu_{1}=w_{1}+w_{2}$.
                Now $\rho_{\psi}^{\vee}(\gamma)=1$ if and only if $\mu_{4}$ comes from $\sym^{4}\pi_{1}$ or $\sym^{2}\pi_{1}\otimes\sym^{2}\pi_{2}$.
                We can easily verify that none of the weights of these two irreducible summands is possible to be $\mu_{4}$.
            \end{enumerate}
    In conclusion, $\rho_{\psi}^{\vee}(\gamma)=-1$.
    By Arthur's multiplicity formula, for $\psi\in\Psi_{\mathrm{AJ}}(\grpF)$ the multiplicity $\mathrm{m}(\pi_{\psi})=1$ if and only if one of the following conditions holds:
    \begin{itemize}
        \item $2w_{1}+1<w_{2}<4w_{1}-1,\,w_{2}\equiv 1\modulo 4$;
        \item $w_{2}<2w_{1}-1$ or $w_{2}>4w_{1}+1$, and $w_{2}\equiv 3\modulo 4,\,w_{1}\neq w_{2}$.
    \end{itemize}

    \textbf{Case (iii)}: $\psi=\sym^{2}\pi_{1}[3]\oplus(\pi_{1}\otimes\pi_{2}[3])\oplus[5]$, 
    where $\pi_{1},\pi_{2}\in\Pi_{\alg}^{\perp}(\pgl_{2})$ have motivic weights $w_{1},w_{2}$ respectively.
    In this case, 
    the representations of $\SL_{2}(\C)$ in the restriction of $\mathfrak{f}_{4}$ along $\psi$ are all odd dimensional,
    thus $\varepsilon_{\psi}(\gamma)=1$ by \Cref{prop epsilon character of Arthur}.
    On the other side,
    $\gamma$ acts on $\sym^{2}\pi_{1}[3]$ by $1$ and on $\pi_{1}\otimes\pi_{2}[3]$ by $-1$.
    We have $\mu_{1}=w_{1}+1$ or $\frac{w_{1}+w_{2}}{2}+1$.
    \begin{enumerate}[label= (\arabic*)]
        \item If $w_{1}>w_{2}$, then $\mu_{1}=w_{1}+1$. 
        The condition that $\psi\in\Psi_{\mathrm{AJ}}(\grpF)$ implies that $w_{1}>w_{2}+4$,
        thus $w_{1}+1>w_{1}>w_{1}-1>\frac{w_{1}+w_{2}}{2}+1$, 
        which are larger than other weights. 
        So $\mu_{4}=\frac{w_{1}+w_{2}}{2}+1$ and $\rho_{\psi}^{\vee}(\gamma)=-1$.
        \item If $w_{1}<w_{2}$, then $\mu_{1}=\frac{w_{1}+w_{2}}{2}+1$. 
        Similarly, we have $w_{1}<w_{2}-4$. 
        Now $\mu_{4}$ must be $w_{1}+1$ or $\frac{w_{2}-w_{1}}{2}+1$, 
        so $\rho_{\psi}^{\vee}(\gamma)=1$ if and only if $\mu_{4}=\frac{w_{2}-w_{1}}{2}+1$.
        This is equivalent to $w_{2}>3w_{1}$.
    \end{enumerate}
    By Arthur's multiplicity formula, $\mathrm{m}(\pi_{\psi})=1$ if and only if $w_{2}>3w_{1}$.

    \textbf{Case (iv)}: $\psi=\sym^{4}\pi_{1}\oplus(\sym^{2}\pi_{1}\otimes\pi_{2}\otimes\pi_{3}) \oplus(\sym^{2}\pi_{1}\otimes\sym^{2}\pi_{3})$,
    where $\pi_{1},\pi_{2},\pi_{3}\in\Pi_{\alg}^{\perp}(\pgl_{2})$ have motivic weights $w_{1},w_{2},w_{3}$ respectively.
    In this case, $\varepsilon_{\psi}(\gamma)=1$ since the parameter is tempered.
    On the other side, $\gamma$ acts on $\sym^{4}\pi_{1}$ and $\sym^{2}\pi_{1}\otimes\sym^{2}\pi_{3}$ by $1$ and on $\sym^{2}\pi_{1}\otimes\pi_{2}\otimes\pi_{3}$ by $-1$.
    We denote the ratios $w_{1}/w_{3},w_{2}/w_{3}$ by $r_{1},r_{2}$ respectively,
    and denote the multiset of elements $\mu/w_{3}$,
    $\mu$ running over the eigenvalues of $\mathrm{c}_{\infty}(\psi)$,
    by $\widetilde{\mathcal{W}}$.
    We still order the elements of $\widetilde{\mathcal{W}}$ by $\mu_{1}>\mu_{2}>\cdots>\mu_{26}$.
    The largest one $\mu_{1}$ must be $r_{1}+1$ or $2r_{1}$ or $r_{1}+\frac{r_{2}+1}{2}$.
    \begin{enumerate}[label= (\arabic*)]
        \item If $r_{1}<1,r_{2}<1$, then $\mu_{1}=r_{1}+1$. Now $\mu_{2}=2r_{1}$ or $1$ or $r_{1}+\frac{r_{2}+1}{2}$.
        \begin{enumerate}
            \item If $r_{1}>1/2$ and $r_{2}<2r_{1}-1$, then $\mu_{2}=2r_{1}$. Now $r_{1}+1>2r_{1}>r_{1}+\frac{r_{2}+1}{2}>r_{1}+\frac{1-r_{2}}{2}$, which are larger than other $22$ elements, thus $\mu_{4}=r_{1}+\frac{1-r_{2}}{2}$ and $\rho_{\psi}^{\vee}(\gamma)=-1$.
            \item If $r_{1}<1/2$ and $r_{2}<1-2r_{1}$, then $\mu_{2}=1$. Now $\rho_{\psi}^{\vee}(\gamma)=1$ if and only if $\mu_{4}=1-r_{1}$, which is equivalent to $|4r_{1}-1|<r_{2}$.
            \item If $r_{2}>|2r_{1}-1|$, then $\mu_{2}=r_{1}+\frac{r_{2}+1}{2}$. Now $\rho_{\psi}^{\vee}(\gamma)=1$ if and only if $\mu_{4}=2r_{1}$ or $1$, which is equivalent to $r_{2}<4r_{1}-1$.
        \end{enumerate}
        \item If $r_{1}>1,r_{2}<2r_{1}-1$, then $\mu_{1}=2r_{1}$. Now $\rho_{\psi}^{\vee}(\gamma)=1$ if and only if $\mu_{4}=r_{1}+1$, which is equivalent to $r_{2}>3$.
        \item If $r_{2}>1,r_{2}>2r_{1}-1$, then $\mu_{1}=r_{1}+\frac{r_{2}+1}{2}$. Now $\mu_{2}$ belongs to the (multi)set $\{r_{1}+1,2r_{1},r_{1}+\frac{r_{2}-1}{2},\frac{r_{2}+1}{2}\}$.
            \begin{enumerate}
                \item If $r_{1}<1$ and $r_{2}<2r_{1}+1$, then $\mu_{2}=r_{1}+1$. Now $\rho_{\psi}^{\vee}(\gamma)=1$ if and only if $\mu_{4}=\frac{r_{2}+1}{2}$, which is equivalent to $r_{2}<4r_{1}-1$.
                \item If $r_{1}>1$ and $r_{2}<2r_{1}+1$, then $\mu_{2}=2r_{1}$. Now $\mu_{4}=\min(r_{1}+1,r_{1}+\frac{r_{2}-1}{2})$, thus $\rho^{\vee}_{\psi}(\gamma)=1$ if and only if $r_{2}<3$.
                \item If $r_{1}>1$ and $r_{2}>2r_{1}+1$, then $\mu_{2}=r_{1}+\frac{r_{2}-1}{2}$. Now $\rho_{\psi}^{\vee}(\gamma)=1$ if and only if $\mu_{4}=\frac{r_{2}\pm1}{2}$, which is equivalent to $r_{2}<4r_{1}-1$ or $r_{2}>4r_{1}+1$. 
                \item If $r_{1}<1$ and $r_{2}>2r_{1}+1$, then $\mu_{2}=\frac{r_{2}+1}{2}$. Now $\rho_{\psi}^{\vee}(\gamma)=1$ if and only if $\mu_{4}=r_{1}+\frac{r_{2}-1}{2}$ or $\frac{r_{2}+1}{2}-r_{1}$, which is equivalent to that $r_{2}<\min(3,4r_{1}+1)$ or $r_{2}>\max(3,4r_{1}+1)$.
            \end{enumerate}
    \end{enumerate}
    In conclusion, by Arthur's multiplicity formula,
    $\mathrm{m}(\pi_{\psi})=1$ if and only if $w_{1},w_{2},w_{3}$ satisfy one of the conditions listed in the proposition.
\end{proof}

\subsubsection{\texorpdfstring{$H=\left(\mathrm{A}_{1}^{[5,4^{4},1^{5}]}\times \mathrm{A}_{1}^{[2^{6},1^{14}]}\times \mathrm{A}_{1}^{[2^{6},1^{14}]}\right)/\mu_{2}^{\Delta}$}{}}
\label{section Arthur principal A1 in Spin(5) with Spin(4)}
 By \Cref{section info principal A1 in Spin(5) with Spin(4)},
the restriction of the $26$-dimensional irreducible representation $\jord_{0}$ of $\lietype{F}{4}$ to $H$ is isomorphic to 
\[\triv+\triv\otimes\st\otimes\st+
\sym^{3}\st\otimes\left(\st\otimes\triv+\triv\otimes\st\right)+
\sym^{4}\st\otimes\triv\otimes\triv 
,\]
and the centralizer of $H$ in $\lietype{F}{4}$ is $\mathrm{Z}(H)\simeq \Z/2\Z\times\Z/2\Z$.

For $\psi\in\Psi_{\mathrm{AJ}}(\grpF)$ satisfying $\mathrm{H}(\psi)=H$ and $\mathrm{m}(\pi_{\psi})=1$,
there are three possible endoscopic types:
\begin{enumerate}[label=(\roman*)]
    \item $(5,(8,1),(5,1),(4,2),(2,2),(1,1))$.
    A global Arthur parameter of this type is of the form:
        \[\sym^{4}\pi_{1}\oplus(\sym^{3}\pi_{1}\otimes\pi_{2})\oplus\sym^{3}\pi_{1}[2]\oplus \pi_{2}[2]\oplus[1],\,\pi_{1},\pi_{2}\in\Pi_{\alg}^{\perp}(\pgl_{2}).\]
    \item $(5,(4,1),(2,4),(2,4),(1,5),(1,1))$.
    A global Arthur parameter of this type is of the form:
        \[(\pi_{1}\otimes\pi_{2})\oplus\pi_{1}[4]\oplus\pi_{2}[4]\oplus[5]\oplus[1],\,\pi_{1},\pi_{2}\in\Pi_{\alg}^{\perp}(\pgl_{2}).\]
    \item $(5,(8,1),(8,1),(5,1),(4,1),(1,1))$.
    A global Arthur parameter of this type is of the form:
        \[\sym^{4}\pi_{1}\oplus(\sym^{3}\pi_{1}\otimes\pi_{2})\oplus(\sym^{3}\pi_{1}\otimes\pi_{3})\oplus(\pi_{2}\otimes\pi_{3})\oplus[1],\,\pi_{1},\pi_{2},\pi_{3}\in\Pi_{\alg}^{\perp}(\pgl_{2}).\]
\end{enumerate}

For this subgroup $H$ of $\lietype{F}{4}$,
the restriction of the adjoint representation $\mathfrak{f}_{4}$ of $\lietype{F}{4}$ to $H$ is isomorphic to 
\begin{gather*}
    \triv\otimes\left(\sym^{2}\st\otimes\triv+\triv\otimes\sym^{2}\st\right)+
    \sym^{2}\st\otimes\triv\otimes\triv+
    \sym^{3}\st\otimes\left(\st\otimes\triv+\triv\otimes\st\right)\\
    +\sym^{4}\st\otimes\st\otimes\st+ 
    \sym^{6}\st\otimes\triv\otimes\triv 
\end{gather*}
\begin{spprop}\label{prop multiplicity principal A1 in Spin(5) with Spin(4)}
    For a discrete global Arthur parameter $\psi\in\Psi_{\mathrm{AJ}}(\grpF)$ satisfying $\mathrm{H}(\psi)=H$,
    the multiplicity $\mathrm{m}(\pi_{\psi})=1$ if and only if $\psi$ is one of the following parameters: 
   \begin{itemize}
    \item $\sym^{4}\pi_{1}\oplus(\sym^{3}\pi_{1}\otimes\pi_{2})\oplus\sym^{3}\pi_{1}[2]\oplus \pi_{2}[2]\oplus[1]$, 
    where $\pi_{1},\pi_{2}\in\Pi_{\alg}^{\perp}(\pgl_{2})$ have motivic weights $w_{1},w_{2}$ respectively and satisfy one of the following conditions
    \begin{itemize}
        \item $w_{2}<w_{1}$ or $w_{2}>4w_{1}+1$, and $w_{2}\equiv 3\modulo4$;
        \item $3w_{1}<w_{2}<4w_{1}-1$ and $w_{2}\equiv 1\modulo4$.
    \end{itemize}
    \item $(\pi_{1}\otimes\pi_{2})\oplus\pi_{1}[4]\oplus\pi_{2}[4]\oplus[5]\oplus[1]$, 
    where $\pi_{1},\pi_{2}\in\Pi_{\alg}^{\perp}(\pgl_{2})$ have motivic weights $w_{1}>w_{2}$ respectively and $w_{1}\equiv 3\modulo 4,\ w_{2}\equiv 1\modulo 4,\ w_{2}<w_{1}-4$.
    \item $\sym^{4}\pi_{1}\oplus(\sym^{3}\pi_{1}\otimes\pi_{2})\oplus(\sym^{3}\pi_{1}\otimes\pi_{3})\oplus(\pi_{2}\otimes\pi_{3})\oplus[1]$, 
    where $\pi_{1},\pi_{2},\pi_{3}\in\Pi_{\alg}^{\perp}(\pgl_{2})$ have motivic weights $w_{1}$ and $w_{2}>w_{3}$ respectively 
    satisfying one of the following conditions:
    \begin{itemize}
       \item $w_{1}>w_{3}$ and $2w_{1}-w_{3}<w_{2}<2w_{1}+w_{3}$;
       \item $w_{3}<3w_{1}<w_{2}<2w_{1}+w_{3}$;
       \item $w_{1}<w_{3}<3w_{1},\,w_{2}>4w_{1}+w_{3}$. 
    \end{itemize}                  
\end{itemize}
\end{spprop}
\begin{proof}
    We take a set of generators $\{\sigma=(-1,1,1),\sigma_{1}=(1,1,-1)\}$ of $\mathrm{C}_{\psi}=\mathrm{Z}(H)\simeq \Z/2\Z\times\Z/2\Z$.
    Let $\chi_{1},\chi_{2}$ be two generators of the character group of $\mathrm{C}_{\psi}$ such that $\chi_{1}(\sigma)=\chi_{2}(\sigma_{1})=-1,\chi_{1}(\sigma_{1})=\chi_{2}(\sigma)=1$.
    
    \textbf{Case (i)}: $\psi=\sym^{4}\pi_{1}\oplus(\sym^{3}\pi_{1}\otimes\pi_{2})\oplus\sym^{3}\pi_{1}[2]\oplus \pi_{2}[2]\oplus[1]$, 
    where $\pi_{1},\pi_{2}\in\Pi_{\alg}^{\perp}(\pgl_{2})$ have motivic weights $w_{1},w_{2}$ respectively.
    In this case,
    the restriction of $\mathfrak{f}_{4}$ along $\psi$ is isomorphic to:
    \begin{gather*}
        \sym^{6}\pi_{1}\oplus\left(\sym^{4}\pi_{1}\otimes\pi_{2}[2]\right)\oplus \left(\sym^{3}\pi_{1}\otimes\pi_{2}\right)\oplus\sym^{3}\pi_{1}[2]\oplus\sym^{2}\pi_{1}\oplus\sym^{2}\pi_{2}\oplus [3].
    \end{gather*}
    By \Cref{prop epsilon character of Arthur} we have:
    \begin{gather*}
        \varepsilon_{\psi}(\sigma)=\varepsilon(\sym^{3}\pi_{1})=\varepsilon(\mathbf{I}_{3w_{1}}+\mathbf{I}_{w_{1}})=(-1)^{(3w_{1}+1)/2+(w_{1}+1)/2}=-1,\\
        \varepsilon_{\psi}(\sigma_{1})=\varepsilon(\sym^{4}\pi_{1}\otimes\pi_{2})\cdot\varepsilon(\sym^{3}\pi_{1})
        =(-1)^{\max(4w_{1},w_{2})+\max(2w_{1},w_{2})+(w_{2}-1)/2}.
    \end{gather*}
    So $\varepsilon_{\psi}=\chi_{1}$ or $\chi_{1}\chi_{2}$.
    On the other side,
    the largest weight $\mu_{1}$ is $2w_{1}$ or $\frac{3w_{1}+w_{2}}{2}$.
    \begin{enumerate}[label= (\arabic*)]
        \item If $w_{1}>w_{2}$, then $\mu_{1}=2w_{1}$. 
        Now $2w_{1}>\frac{3w_{1}+w_{2}}{2}>\frac{3w_{1}+1}{2}>\frac{3w_{1}-1}{2}$ and they are larger than other weights, 
        thus $\mu_{4}=\frac{3w_{1}-1}{2}$ and $\rho^{\vee}_{\psi}=\chi_{1}\chi_{2}$.
        \item If $w_{1}<w_{2}$, then $\mu_{1}=\frac{3w_{1}+w_{2}}{2}$. Now $\mu_{2}=2w_{1}$ or $\frac{w_{1}+w_{2}}{2}$.
        \begin{enumerate}
            \item If $w_{2}<3w_{1}$, then $\mu_{2}=2w_{1}$. Now $\mu_{4}=\frac{w_{1}+w_{2}}{2}$ or $\frac{3w_{1}\pm 1}{2}$, thus $\rho^{\vee}_{\psi}=1$ or $\chi_{2}$.
            \item If $w_{2}>3w_{1}$, then $\mu_{2}=\frac{w_{1}+w_{2}}{2}$. Now $\mu_{4}=2w_{1}$ or $\frac{w_{2}\pm 1}{2}$, thus $\rho^{\vee}=\chi_{1}$ or $\chi_{1}\chi_{2}$. 
            Notice that $\mu_{4}=2w_{1}$ if and only if $2w_{1}$ lies between $\frac{w_{2}+1}{2}$ and $\frac{w_{2}-1}{2}$, which can not happen. So $\rho^{\vee}_{\psi}=\chi_{1}\chi_{2}$ for any $w_{2}>3w_{1}$ and $w_{2}\neq 4w_{1}\pm 1$.
        \end{enumerate}
    \end{enumerate}
    Hence by Arthur's multiplicity formula, $\mathrm{m}(\pi_{\psi})=1$ if and only if one of the following conditions holds:
    \begin{itemize}
        \item $w_{2}<w_{1}$ or $w_{2}>4w_{1}+1$, and $w_{2}\equiv 3\modulo4$;
        \item $3w_{1}<w_{2}<4w_{1}-1$, and $w_{2}\equiv 1\modulo4$.
    \end{itemize}

    \textbf{Case (ii)}: $\psi=(\pi_{1}\otimes\pi_{2})\oplus\pi_{1}[4]\oplus\pi_{2}[4]\oplus[5]\oplus[1]$, 
    where $\pi_{1},\pi_{2}\in\Pi_{\alg}^{\perp}(\pgl_{2})$ have motivic weights $w_{1}>w_{2}$ respectively.
    In this case,
    the restriction of $\mathfrak{f}_{4}$ along $\psi$ is isomorphic to 
    \[\sym^{2}\pi_{1}\oplus\sym^{2}\pi_{2}\oplus \left(\pi_{1}\otimes\pi_{2}[5]\right)\oplus \pi_{1}[4]\oplus\pi_{2}[4]\oplus [7]\oplus [3].\]
    By \Cref{prop epsilon character of Arthur} we have:
    \begin{gather*}
        \varepsilon_{\psi}(\sigma)=\varepsilon(\pi_{1})\cdot\varepsilon(\pi_{2})=\varepsilon(\mathbf{I}_{w_{1}})\cdot\varepsilon(\mathbf{I}_{w_{2}})=(-1)^{(w_{1}+w_{2})/2+1}\\
        \varepsilon_{\psi}(\sigma_{1})=\varepsilon(\pi_{2})=\varepsilon(\mathbf{I}_{w_{2}})=(-1)^{(w_{2}+1)/2}.
    \end{gather*}
    On the other side,
    the condition $\psi\in\Psi_{\mathrm{AJ}}(\grpF)$ implies that $w_{2}<w_{1}-4$.
    Since 
    \[\frac{w_{1}+w_{2}}{2}>\frac{w_{1}+3}{2}>\frac{w_{1}+1}{2}>\frac{w_{1}-1}{2}\] 
    and they are larger than other weights,
    we have $\mu_{1}=\frac{w_{1}+w_{2}}{2}$ and $\mu_{4}=\frac{w_{1}-1}{2}$. 
    The global component group $\mathrm{C}_{\psi}$ acts on $\pi_{1}\otimes\pi_{2}$ and $\pi_{1}[4]$ by $\chi_{2}$ and $\chi_{1}$ respectively, 
    thus by \Cref{prop calculation of rho character for F4} the character $\rho_{\psi}^{\vee}=\chi_{1}\chi_{2}$.
    By Arthur's multiplicity formula, $\mathrm{m}(\pi_{\psi})=1$ if and only if $w_{1}\equiv 3\modulo4, w_{2}\equiv 1\modulo4$ and $w_{2}<w_{1}-4$.

    \textbf{Case (iii)}: $\psi=\sym^{4}\pi_{1}\oplus(\sym^{3}\pi_{1}\otimes\pi_{2})\oplus(\sym^{3}\pi_{1}\otimes\pi_{3})\oplus(\pi_{2}\otimes\pi_{3})\oplus[1]$, 
    where $\pi_{1},\pi_{2},\pi_{3}\in\Pi_{\alg}^{\perp}(\pgl_{2})$ have motivic weights $w_{1},w_{2},w_{3}$ respectively and we assume that $w_{2}>w_{3}$.
    In this case $\varepsilon_{\psi}$ is trivial since $\psi$ is tempered.
    On the other side, $\mathrm{C}_{\psi}$ acts on the four summands
    $\sym^{4}\pi_{1},\sym^{3}\pi_{1}\otimes\pi_{2},\sym^{3}\pi_{1}\otimes\pi_{3}$ and $\pi_{2}\otimes\pi_{3}$ by $1,\chi_{1},\chi_{1}\chi_{2}$ and $\chi_{2}$ respectively.
    Denote the ratios $w_{1}/w_{3},w_{2}/w_{3}$ by $r_{1},r_{2}$ respectively 
    and the corresponding multiset by $\widetilde{\mathcal{W}}$ as in the proof of \Cref{prop multiplicity a bizzare A1+A1+A1 via A1+Sp(3)}.
    We still order the elements of $\widetilde{\mathcal{W}}$ by $\mu_{1}>\mu_{2}>\cdots>\mu_{26}$,
    then by \Cref{prop calculation of rho character for F4} the character $\rho_{\psi}^{\vee}=1$ if and only if $\mu_{1}$ and $\mu_{4}$ come from the same irreducible summand of $\psi$.
    The largest element $\mu_{1}$ is $2r_{1}$ or $\frac{3r_{1}+r_{2}}{2}$ or $\frac{r_{2}+1}{2}$.
    \begin{enumerate}[label= (\arabic*)]
        \item If $r_{2}<r_{1}$, then $\mu_{1}=2r_{1}$. 
        Now $2r_{1}>\frac{3r_{1}+r_{2}}{2}>\frac{3r_{1}+1}{2}>\frac{3r_{1}-r_{2}}{2}>r_{1}$,
        thus $\rho^{\vee}_{\psi}$ is not trivial.
        \item If $r_{2}>r_{1}$ and $r_{1}>1/3$, then $\mu_{1}=\frac{3r_{1}+r_{2}}{2}$. 
        \begin{enumerate}
            \item If $r_{1}>1$, then $\rho^{\vee}_{\psi}=1$ if and only if $\mu_{4}=\frac{r_{1}+r_{2}}{2}$,
            which is equivalent to $2r_{1}-1<r_{2}<2r_{1}+1$.
            \item If $r_{1}<1$, then $\rho^{\vee}_{\psi}=1$ if and only if $\mu_{4}=\frac{r_{2}\pm r_{1}}{2}$. 
            \begin{enumerate}[label= (\Roman*)]
                \item $\mu_{4}=\frac{r_{2}+r_{1}}{2}$ if and only if $2r_{1}<\frac{r_{2}+r_{1}}{2}<\frac{3r_{1}+1}{2}\Leftrightarrow 3r_{1}<r_{2}<2r_{1}+1$.
                \item $\mu_{4}=\frac{r_{2}-r_{1}}{2}$ if and only if $\frac{r_{2}-r_{1}}{2}>\frac{3r_{1}+1}{2}\Leftrightarrow r_{2}>4r_{1}+1$.
            \end{enumerate}
        \end{enumerate}
        \item If $r_{1}<1/3$, then $\mu_{1}=\frac{r_{2}+1}{2}$. 
        Now $\frac{r_{2}+1}{2},\frac{r_{2}\pm r_{1}}{2},\frac{3r_{1}+r_{2}}{2}$ are larger than $\frac{r_{2}-1}{2}$, so $\frac{r_{2}-1}{2}$ can not be $\mu_{4}$ and thus $\rho^{\vee}_{\psi}\neq 1$.
    \end{enumerate}
    In conclusion, 
    by Arthur's multiplicity formula, 
    $\mathrm{m}(\pi_{\psi})=1$ if and only if $w_{1},w_{2},w_{3}$ satisfy one of the three conditions in \Cref{prop multiplicity principal A1 in Spin(5) with Spin(4)}.
\end{proof}

\subsubsection{\texorpdfstring{$H=\prod_{i=1}^{4}\mathrm{A}_{1}^{[2^{6},1^{14}]}/\mu_{2}^{\Delta}$}{}}
\label{section Arthur four copies of A1}
 By \Cref{section info four copies of A1},
the restriction of the $26$-dimensional irreducible representation $\jord_{0}$ of $\lietype{F}{4}$ to $H$ is isomorphic to 
\[\triv^{\oplus 2}+\sum_{\sym}\st\otimes\st\otimes \triv\otimes \triv,\]
and the centralizer of $H$ in $\lietype{F}{4}$ is $\mathrm{Z}(H)\simeq \Z/2\Z\times\Z/2\Z\times\Z/2\Z$.

For $\psi\in\Psi_{\mathrm{AJ}}(\grpF)$ satisfying $\mathrm{H}(\psi)=H$ and $\mathrm{m}(\pi_{\psi})=1$,
there are two possible endoscopic types:
\begin{enumerate}[label=(\roman*)]
    \item $(8,(4,1),(4,1),(4,1),(2,2),(2,2),(2,2),(1,1),(1,1))$.
    A global Arthur parameter of this type is of the form:
        \[\left(\bigoplus_{1\leq i<j\leq 3}\pi_{i}\otimes\pi_{j}\right)\oplus\left(\bigoplus_{1\leq i\leq 3}\pi_{i}[2]\right)\oplus[1]\oplus[1],\,\pi_{1},\pi_{2},\pi_{3}\in\Pi_{\alg}^{\perp}(\pgl_{2}).\]
    \item $(8,(4,1),(4,1),(4,1),(4,1),(4,1),(4,1),(1,1),(1,1))$.
    A global Arthur parameter of this type is of the form:
        \[\left(\bigoplus_{1\leq i<j\leq 4}\pi_{i}\otimes\pi_{j}\right)\oplus[1]\oplus[1],\,\pi_{1},\pi_{2},\pi_{3},\pi_{4}\in\Pi_{\alg}^{\perp}(\pgl_{2}).\]
\end{enumerate}

For this subgroup $H$ of $\lietype{F}{4}$,
the restriction of the adjoint representation $\mathfrak{f}_{4}$ of $\lietype{F}{4}$ to $H$ is isomorphic to 
\[\sum_{\sym}\sym^{2}\st\otimes\triv\otimes\triv\otimes\triv+
\sum_{\sym}\st\otimes\st\otimes\triv\otimes\triv+
\st\otimes\st\otimes\st\otimes\st.\]
\begin{spprop}\label{prop multiplicity four copies of A1}
    For a discrete global Arthur parameter $\psi\in\Psi_{\mathrm{AJ}}(\grpF)$ satisfying $\mathrm{H}(\psi)=H$,
    the multiplicity $\mathrm{m}(\pi_{\psi})=1$ if and only if $\psi$ has the form: 
    \[\psi=\left(\bigoplus_{1\leq i<j\leq 3}\pi_{i}\otimes\pi_{j}\right)\oplus\left(\bigoplus_{1\leq i\leq 3}\pi_{i}[2]\right)\oplus[1]\oplus[1],\] 
                where $\pi_{1},\pi_{2},\pi_{3}\in\Pi_{\alg}^{\perp}(\pgl_{2})$ have motivic weights $w_{1}>w_{2}>w_{3}$ respectively such that one of the following conditions holds:
                    \begin{itemize}
                        \item $w_{1}>w_{2}+w_{3}+1$, and $w_{1}\equiv w_{3}\equiv 3\modulo 4,\ w_{2}\equiv 1\modulo 4$;
                        \item $w_{1}<w_{2}+w_{3}-1$, and $w_{1}\equiv w_{3}\equiv 1\modulo 4,\ w_{2}\equiv 3\modulo 4$.
                    \end{itemize}
\end{spprop}
\begin{proof}
    We take a set of generators $\{\gamma=(-1,1,1,1),\gamma_{1}=(1,-1,1,1),\gamma_{2}=(1,1,-1,1)\}$ of $\mathrm{C}_{\psi}=\mathrm{Z}(H)\simeq \Z/2\Z\times\Z/2\Z\times\Z/2\Z$.
    
    \textbf{Case (i)}: $\psi=(\bigoplus_{1\leq i<j\leq 3}\pi_{i}\otimes\pi_{j})\oplus(\bigoplus_{1\leq i\leq 3}\pi_{i}[2])\oplus[1]\oplus[1]$, 
    where $\pi_{1},\pi_{2},\pi_{3}\in\Pi_{\alg}^{\perp}(\pgl_{2})$ have motivic weights $w_{1}>w_{2}>w_{3}$ respectively.
    In this case,
    the restriction of $\mathfrak{f}_{4}$ along $\psi$ is isomorphic to 
    \[\left(\pi_{1}\otimes\pi_{2}\otimes\pi_{3}[2]\right)\oplus\left(\bigoplus_{1\leq i<j\leq 3}\pi_{i}\otimes\pi_{j}\right)\oplus\left(\bigoplus_{1\leq i\leq 3}\sym^{2}\pi_{i}\right)\oplus\left(\bigoplus_{1\leq i\leq 3}\pi_{i}[2]\right)\oplus [3].\]
    By \Cref{prop epsilon character of Arthur} we have:
    \begin{gather*}
        \varepsilon_{\psi}(\gamma)=\varepsilon(\pi_{1})\cdot\varepsilon(\pi_{1}\otimes\pi_{2}\otimes\pi_{3})=(-1)^{\max(w_{1},w_{2}+w_{3})+(w_{1}-1)/2},\\
        \varepsilon_{\psi}(\gamma_{1})=\varepsilon(\pi_{2})\cdot\varepsilon(\pi_{1}\otimes\pi_{2}\otimes\pi_{3})=(-1)^{\max(w_{1},w_{2}+w_{3})+(w_{2}-1)/2},\\
        \varepsilon_{\psi}(\gamma_{2})=\varepsilon(\pi_{3})\cdot\varepsilon(\pi_{1}\otimes\pi_{2}\otimes\pi_{3})=(-1)^{\max(w_{1},w_{2}+w_{3})+(w_{3}-1)/2}.
    \end{gather*}
    On the other side, the largest element $\mu_{1}$ must be $\frac{w_{1}+w_{2}}{2}$ 
    and $\mu_{4}$ is the middle one of $\{\frac{w_{1}+1}{2},\frac{w_{1}-1}{2},\frac{w_{2}+w_{3}}{2}\}$.
    Since there is no integer between $\frac{w_{1}+1}{2}$ and $\frac{w_{1}-1}{2}$,
    we have $\mu_{4}\neq\frac{w_{2}+w_{3}}{2}$.
    So $\rho_{\psi}^{\vee}$ is the product of two characters of $\mathrm{C}_{\psi}$ coming from $\pi_{1}\otimes\pi_{2}$ and $\pi_{1}[2]$ respectively,
    thus $\rho_{\psi}^{\vee}(\gamma)=\rho_{\psi}^{\vee}(\gamma_{2})=1$ and $\rho_{\psi}^{\vee}(\gamma_{1})=-1$.

    By Arthur's multiplicity formula, $\mathrm{m}(\pi_{\psi})=1$ if and only if one of the following conditions holds:
    \begin{itemize}
        \item $w_{1}>w_{2}+w_{3}+1$, and $w_{1}\equiv w_{3}\equiv 3\modulo 4, w_{2}\equiv 1\modulo 4$;
        \item $w_{1}<w_{2}+w_{3}-1$, and $w_{1}\equiv w_{3}\equiv 1\modulo 4, w_{2}\equiv 3\modulo 4$.
    \end{itemize}

    \textbf{Case (ii)}: $\psi=(\bigoplus_{1\leq i<j\leq 4}\pi_{i}\otimes\pi_{j})\oplus[1]\oplus[1]$,
    where $\pi_{1},\pi_{2},\pi_{3},\pi_{4}\in\Pi_{\alg}^{\perp}(\pgl_{2})$ have motivic weights $w_{1}>w_{2}>w_{3}>w_{4}$ respectively.
    In this case, $\varepsilon_{\psi}$ is trivial.
    On the other side, $\mu_{1}$ must be $\frac{w_{1}+w_{2}}{2}$.
    Notice that $\mathrm{C}_{\psi}$ acts on $6$ components $\pi_{i}\otimes\pi_{j}$ via $6$ different characters, so $\rho_{\psi}^{\vee}$ is trivial if and only if $\mu_{4}=\frac{w_{1}-w_{2}}{2}$.
    However, 
    \[\frac{w_{1}-w_{2}}{2}<\frac{w_{1}-w_{3}}{2}<\frac{w_{1}-w_{4}}{2}<\frac{w_{1}+w_{4}}{2}<\frac{w_{1}+w_{3}}{2}<\frac{w_{1}+w_{2}}{2},\]
    thus $\rho_{\psi}^{\vee}\neq 1$ and $\mathrm{m}(\pi_{\psi})=0$.
\end{proof}

\subsubsection{\texorpdfstring{$H=\mathrm{A}_{1}^{[5,3^{7}]}\times \mathrm{G}_{2}$}{}}
\label{section Arthur A1+G2}
 In this case,
we need to consider cuspidal representations $\pi\in\Pi_{\alg,\reg}^{\mathrm{o}}(\pgl_{7})$
such that the image of the corresponding irreducible representation $\mathcal{L}_{\Z}\rightarrow \SL_{7}(\C)$
is a compact Lie group of type $\lietype{G}{2}$.
This kind of representations correspond to 
discrete automorphic representations of the unique semisimple anisotropic $\Z$-group of type $\lietype{G}{2}$
with stable tempered type,
which have been studied in \cite[\S 8]{ChenevierRenard},
conditional to the existence of $\mathcal{L}_{\Z}$ and Arthur's multiplicity formula.
We denote by $\Pi_{\alg}^{\lietype{G}{2}}(\pgl_{7})\subset\Pi_{\alg,\reg}^{\mathrm{o}}(\pgl_{7})$ the subset of these representations.
The Hodge weights of a representation $\pi\in\Pi_{\alg}^{\lietype{G}{2}}(\pgl_{7})$
have the form $w+v>w>v$,
where $w,v$ are even integers.

By \Cref{section info A1+G2},
the restriction of the $26$-dimensional irreducible representation $\jord_{0}$ of $\lietype{F}{4}$ to $H$ is isomorphic to 
\[\sym^{2}\st\otimes \vrep{7}+ \sym^{4}\st\otimes\triv,\]
where $\vrep{7}$ is the $7$-dimensional irreducible representation of $\lietype{G}{2}$,
and the centralizer of $H$ in $\lietype{F}{4}$ is trivial.

For $\psi\in\Psi_{\mathrm{AJ}}(\grpF)$ satisfying $\mathrm{H}(\psi)=H$ and $\mathrm{m}(\pi_{\psi})=1$,
there are two possible endoscopic types:
\begin{enumerate}[label=(\roman*)]
    \item $(2,(7,3),(1,5))$.
    A global Arthur parameter of this type is of the form:
        \[\pi[3]\oplus [5],\,\pi\in\Pi_{\alg}^{\lietype{G}{2}}(\pgl_{7}).\]
    \item $(2,(21,1),(5,1))$.
    A global Arthur parameter of this type is of the form:
        \[(\pi\otimes\sym^{2}\tau)\oplus \sym^{4}\tau,\,\pi\in\Pi_{\alg}^{\lietype{G}{2}}(\pgl_{7}),\tau\in\Pi_{\alg}^{\perp}(\pgl_{2}).\]
\end{enumerate}
\begin{spprop}\label{prop multiplicity A1+G2}
    For a discrete global Arthur parameter $\psi\in\Psi_{\mathrm{AJ}}(\grpF)$ satisfying $\mathrm{H}(\psi)=H$,
    the multiplicity $\mathrm{m}(\pi_{\psi})=1$ if and only if $\psi$ is one of the following parameters: 
    \begin{itemize}
        \item $\pi[3]\oplus [5]$, where $\pi\in\Pi_{\alg}^{\lietype{G}{2}}(\pgl_{7})$ has Hodge weights $w+v>w>v$ such that $v>4$;
        \item $(\pi\otimes\sym^{2}\tau)\oplus \sym^{4}\tau$, where $\pi\in\Pi_{\alg}^{\lietype{G}{2}}(\pgl_{7})$ has Hodge weights $w+v>w>v$ and $\tau\in\Pi_{\alg}^{\perp}(\pgl_{2})$ satisfies $\mathrm{w}(\tau)\notin\{\frac{w+v}{2},\frac{w}{2},\frac{v}{2}\}$.
    \end{itemize}
\end{spprop}
\begin{proof}
    This follows from the condition $\psi\in\Psi_{\mathrm{AJ}}(\grpF)$ and the fact that $\mathrm{C}_{\psi}$ is trivial.
\end{proof}

\subsubsection{\texorpdfstring{$H=\left(\mathrm{A}_{1}^{[2^{6},1^{14}]}\times \mathrm{A}_{1}^{[2^{6},1^{14}]}\times \symp(2)\right)/\mu_{2}^{\Delta}$}{}}
\label{section Arthur A1+A1+Sp(2)}
 By \Cref{section info A1+A1+Sp(2)},
the restriction of the $26$-dimensional irreducible representation $\jord_{0}$ of $\lietype{F}{4}$ to $H$ is isomorphic to 
\[\triv+
\st\otimes\st\otimes \triv+
\st\otimes\triv\otimes \vrep{4}+
\triv\otimes\st\otimes \vrep{4}+
\triv\otimes\triv\otimes \wedge^{*}\vrep{4},
\]
where $\vrep{4}$ is the standard representation of $\symp(2)$ and $\wedge^{*}\vrep{4}$ is the $5$-dimensional irreducible representation of $\symp(2)$.
The centralizer of $H$ in $\lietype{F}{4}$ is $\mathrm{Z}(H)\simeq \Z/2\Z\times\Z/2\Z$.

For any $\pi\in\Pi_{\alg}^{\symp_{4}}(\pgl_{4})$,
we denote by $\wedge^{*}\pi$ the representation in $\Pi_{\alg,\reg}^{\mathrm{o}}(\pgl_{5})$ corresponding to the following irreducible representation of $\mathcal{L}_{\Z}$:
\[\mathcal{L}_{\Z}\xrightarrow{\psi_{\pi}} \symp(2)\xrightarrow{\wedge^{*}} \SL_{5}(\C).\]

For $\psi\in\Psi_{\mathrm{AJ}}(\grpF)$ satisfying $\mathrm{H}(\psi)=H$ and $\mathrm{m}(\pi_{\psi})=1$,
there are two possible endoscopic types:
\begin{enumerate}[label=(\roman*)]
    \item $(5,(8,1),(5,1),(4,2),(2,2),(1,1))$.
    A global Arthur parameter of this type is of the form:
        \[\wedge^{*}\pi\oplus\left(\pi\otimes\tau\right)\oplus\pi[2]\oplus\tau[2]\oplus[1],\,\pi\in\Pi_{\alg}^{\symp_{4}}(\pgl_{4}),\tau\in\Pi_{\alg}^{\perp}(\pgl_{2}).\]
    \item $(5,(8,1),(8,1),(5,1),(4,1),(1,1))$.
    A global Arthur parameter of this type is of the form:
        \[\wedge^{*}\pi\oplus\left(\pi\otimes\tau_{1}\right)\oplus\left(\pi\otimes\tau_{2}\right)\oplus\left(\tau_{1}\otimes\tau_{2}\right)\oplus[1],\,\pi\in\Pi_{\alg}^{\symp_{4}}(\pgl_{4}),\tau_{1},\tau_{2}\in\Pi_{\alg}^{\perp}(\pgl_{2}).\]
\end{enumerate}

For this subgroup $H$ of $\lietype{F}{4}$,
the restriction of the adjoint representation $\mathfrak{f}_{4}$ of $\lietype{F}{4}$ to $H$ is isomorphic to 
\begin{gather*}
    \left(\sym^{2}\st\otimes\triv+\triv\otimes\sym^{2}\st\right)\otimes \triv 
    +\left(\st\otimes\triv+\triv\otimes\st\right)\otimes\vrep{4}\\
    +\st\otimes\st\otimes\wedge^{*}\vrep{4}
    +\triv\otimes\triv\otimes\sym^{2}\vrep{4}.
\end{gather*}
\begin{spprop}\label{prop multiplicity A1+A1+Sp(2)}
    For a discrete global Arthur parameter $\psi\in\Psi_{\mathrm{AJ}}(\grpF)$ satisfying $\mathrm{H}(\psi)=H$,
    the multiplicity $\mathrm{m}(\pi_{\psi})=1$ if and only if $\psi$ is one of the following parameters: 
   \begin{itemize}
    \item $\wedge^{*}\pi\oplus\left(\pi\otimes\tau\right)\oplus\pi[2]\oplus\tau[2]\oplus[1]$, where $\pi\in\Pi_{\alg}^{\symp_{4}}(\pgl_{4})$ has Hodge weights $w_{1}>w_{2}>1$
    and $\tau\in\Pi_{\alg}^{\perp}(\pgl_{2})$ has motivic weight $v$ satisfying one of the following conditions:
    \begin{itemize}
        \item $w_{1}<v<w_{1}+w_{2}-1,w_{1}+w_{2}\equiv 0\modulo 4,v\equiv 1\modulo 4$;
        \item $w_{1}-w_{2}+1<v<w_{2},w_{1}+w_{2}\equiv 0\modulo 4,v\equiv 1\modulo 4$;
        \item $w_{2}<v<w_{1}-w_{2}-1,w_{1}+w_{2}\equiv 2\modulo 4,v\equiv 1\modulo 4$;
        \item $v>w_{1}+w_{2}+1,w_{1}+w_{2}\equiv 0\modulo 4,v\equiv 3\modulo 4$;
        \item $v<\min(w_{1}-w_{2}-1,w_{2}),w_{1}+w_{2}\equiv 0\modulo 4,v\equiv 3\modulo 4$;
        \item $\max(w_{1}-w_{2}+1,w_{2})<v<w_{1},w_{1}+w_{2}\equiv 2\modulo 4,v\equiv 3\modulo 4$.
    \end{itemize}
    \item $\wedge^{*}\pi\oplus\left(\pi\otimes\tau_{1}\right)\oplus\left(\pi\otimes\tau_{2}\right)\oplus\left(\tau_{1}\otimes\tau_{2}\right)\oplus[1]$, 
    where $\pi\in\Pi_{\alg}^{\symp_{4}}(\pgl_{4})$ has Hodge weights $w_{1}>w_{2}$ 
    and $\tau_{1},\tau_{2}\in\Pi_{\alg}^{\perp}(\pgl_{2})$ have motivic weights $v_{1}>v_{2}$ respectively satisfying one of the following conditions:
    \begin{itemize}
        \item $v_{2}<w_{2}<v_{1}$ and $w_{1}-w_{2}-v_{2}<v_{1}<w_{1}-w_{2}+v_{2}$;
        \item $w_{2}<v_{2}<w_{1}$ and $v_{1}>w_{1}+w_{2}+v_{2}$;
        \item $v_{2}<w_{1}<v_{1}<w_{1}-w_{2}+v_{2}$.
    \end{itemize}
\end{itemize}
\end{spprop}
\begin{proof}
    We take a set of generators $\{\sigma=(1,1,-1),\sigma_{1}=(-1,1,1)\}$ of $\mathrm{C}_{\psi}=\mathrm{Z}(H)\simeq \Z/2\Z\times\Z/2\Z$.
    Let $\chi_{1},\chi_{2}$ be two generators of the character group of $\mathrm{C}_{\psi}$ such that $\chi_{1}(\sigma)=\chi_{2}(\sigma_{1})=-1$ and $\chi_{1}(\sigma_{1})=\chi_{2}(\sigma)=1$.
    
    \textbf{Case (i)}: $\psi=\wedge^{*}\pi\oplus(\pi\otimes\tau)\oplus\pi[2]\oplus\tau[2]\oplus[1]$,
    where $\pi\in\Pi_{\alg}^{\symp_{4}}(\mathrm{PGL}_{4})$ has Hodge weights $w_{1}>w_{2}>1$
    and $\tau\in\Pi_{\alg}^{\perp}(\pgl_{2})$ has motivic weight $v$.
    Here we assume that Arthur's $\SL_{2}(\C)$ is sent to the first $\lietype{A}{1}$-factor of $H_{\C}$.
    In this case,
    the restriction of $\mathfrak{f}_{4}$ along $\psi$ is isomorphic to 
    \[\sym^{2}\pi\oplus\left(\wedge^{*}\pi\otimes\tau[2]\right)\oplus\left(\pi\otimes\tau\right)\oplus\pi[2]\oplus\sym^{2}\tau\oplus[3].\]
    By \Cref{prop epsilon character of Arthur} we have:
    \begin{gather*}
        \varepsilon_{\psi}(\sigma)=\varepsilon(\pi)=\varepsilon(\mathbf{I}_{w_{1}}+\mathbf{I}_{w_{2}})=(-1)^{(w_{1}+w_{2})/2+1},\\
        \varepsilon_{\psi}(\sigma_{1})=\varepsilon(\wedge^{*}\pi\otimes\tau)=(-1)^{\max(w_{1}+w_{2},v)+\max(w_{1}-w_{2},v)+(v+1)/2}.
    \end{gather*}
    On the other side, the group $\mathrm{C}_{\psi}$ acts on $\wedge^{*}\pi,\pi\otimes\tau,\pi[2],\tau[2]$ by $1,\chi_{1}\chi_{2},\chi_{1},\chi_{2}$ respectively.
    The largest element $\mu_{1}$ must be $\frac{w_{1}+w_{2}}{2}$ or $\frac{w_{1}+v}{2}$.
    \begin{enumerate}[label= (\arabic*)]
        \item If $w_{2}>v$, then $\mu_{1}=\frac{w_{1}+w_{2}}{2}$. Now $\mu_{4}=\frac{w_{1}\pm 1}{2}$ and $\rho_{\psi}^{\vee}=\chi_{1}$. 
        \item If $w_{2}<v$, then $\mu_{1}=\frac{w_{1}+v}{2}$. Now $\mu_{2}$ is $\frac{w_{1}+w_{2}}{2}$ or $\frac{w_{2}+v}{2}$.
        \begin{enumerate}
            \item If $w_{1}>v$, then $\mu_{2}=\frac{w_{1}+w_{2}}{2}$. Now $\mu_{4}=\frac{w_{1}\pm 1}{2}$ and $\rho^{\vee}_{\psi}=\chi_{2}$.
            \item If $w_{1}<v$, then $\mu_{2}=\frac{w_{2}+v}{2}$. Now $\mu_{4}=\frac{v\pm 1}{2}$ and $\rho^{\vee}_{\psi}=\chi_{1}$.
        \end{enumerate}
    \end{enumerate}
    By Arthur's multiplicity formula, $\mathrm{m}(\pi_{\psi})=1$ if and only if $\pi$ and $\tau$ satisfy one of the conditions listed in the proposition.

    \textbf{Case (ii)}: $\psi=\wedge^{*}\pi\oplus(\pi\otimes\tau_{1})\oplus(\pi\otimes\tau_{2})\oplus(\tau_{1}\otimes\tau_{2})\oplus[1]$,
    where $\pi\in\Pi_{\alg}^{\symp_{4}}(\mathrm{PGL}_{4})$ has Hodge weights $w_{1}>w_{2}$ 
    and $\tau_{1},\tau_{2}\in\Pi_{\alg}^{\perp}(\pgl_{2})$ have motivic weights $v_{1}>v_{2}$ respectively.
    In this case $\varepsilon_{\psi}$ is a trivial character.
    On the other side, since $\mathrm{C}_{\psi}$ acts on four non-trivial irreducible summands of $\psi$ by four different characters,
    $\rho_{\psi}^{\vee}=1$ if and only if $\mu_{1}$ and $\mu_{4}$ come from the same irreducible summand.
    Now $\mu_{1}$ must be $\frac{w_{1}+w_{2}}{2}$ or $\frac{w_{1}+v_{1}}{2}$ or $\frac{v_{1}+v_{2}}{2}$.
    \begin{enumerate}[label= (\arabic*)]
        \item If $w_{2}>v_{1}$, then $\mu_{1}=\frac{w_{1}+w_{2}}{2}$ and $\mu_{4}$ can not be $\frac{w_{1}-w_{2}}{2}$, thus $\rho^{\vee}_{\psi}$ is not trivial.
        \item If $v_{1}>w_{2}$ and $w_{1}>v_{2}$, then $\mu_{1}=\frac{w_{1}+v_{1}}{2}$. Now $\rho_{\psi}^{\vee}$ is trivial if and only if $\mu_{4}=\frac{w_{2}+v_{1}}{2}$ or $\frac{v_{1}-w_{2}}{2}$.
        \begin{enumerate}
            \item $\mu_{4}=\frac{v_{1}-w_{2}}{2}$ is equivalent to that $v_{1}-w_{2}>\max(v_{1}-v_{2},w_{1}+w_{2},w_{1}+v_{2})$. This holds if and only if $v_{2}>w_{2}$ and $v_{1}>w_{1}+w_{2}+v_{2}$.
            \item $\mu_{4}=\frac{w_{2}+v_{1}}{2}$ is equivalent to that $w_{2}+v_{1}>\max(w_{1}-w_{2},w_{1}-v_{2})$ and $w_{2}+v_{1}$ is smaller than exactly two of $\{w_{1}+w_{2},v_{1}+v_{2},w_{1}+v_{2}\}$. 
                  This holds in two cases: $w_{1}<v_{1}<w_{1}-w_{2}+v_{2}$ or 
                  \[w_{2}>v_{2},w_{1}>v_{1},w_{1}-w_{2}-v_{2}<v_{1}<w_{1}-w_{2}+v_{2}.\] 
        \end{enumerate}
        \item If $v_{2}>w_{1}$, $\mu_{1}=\frac{v_{1}+v_{2}}{2}$. We have
        \[\frac{v_{1}-v_{2}}{2}<\frac{v_{1}-w_{1}}{2}<\frac{v_{1}-w_{2}}{2}<\frac{v_{1}+w_{2}}{2}<\frac{v_{1}+w_{1}}{2}<\frac{v_{1}+v_{2}}{2},\]
        thus $\mu_{4}$ can not be $\frac{v_{1}-v_{2}}{2}$ and $\rho_{\psi}^{\vee}$ is not trivial.
    \end{enumerate}
    In conclusion, by Arthur's multiplicity formula $\mathrm{m}(\pi_{\psi})=1$ if and only if one of the following conditions holds:
    \begin{itemize}
        \item $v_{2}<w_{2}<v_{1}$ and $w_{1}-w_{2}-v_{2}<v_{1}<w_{1}-w_{2}+v_{2}$;
        \item $w_{2}<v_{2}<w_{1}$ and $v_{1}>w_{1}+w_{2}+v_{2}$;
        \item $v_{2}<w_{1}<v_{1}<w_{1}-w_{2}+v_{2}$.\qedhere
    \end{itemize}
\end{proof}

\subsubsection{\texorpdfstring{$H=\left(\mathrm{A}_{1}^{[2^{6},1^{14}]}\times\symp(3)\right)/\mu_{2}^{\Delta}$}{}}
\label{section Arthur A1+Sp(3)}
 By \Cref{section info A1+Sp(3)},
the restriction of the $26$-dimensional irreducible representation $\jord_{0}$ of $\lietype{F}{4}$ to $H$ is isomorphic to 
\[\st\otimes \vrep{6}+\triv\otimes \vrep{14},\] 
where $\vrep{6}$ is the standard $6$-dimensional representation of $\symp(3)$,
$\vrep{14}=\wedge^{*}\vrep{6}$ is the $14$-dimensional irreducible representation of $\symp(3)$ that is a sub-representation of $\wedge^{2}\vrep{6}$.
The centralizer of $H$ in $\lietype{F}{4}$ is $\mathrm{Z}(H)\simeq \Z/2\Z$.

For any $\pi\in\Pi_{\alg}^{\symp_{6}}(\pgl_{6})$,
we denote by $\wedge^{*}\pi$ the representation in $\Pi_{\alg,\reg}^{\mathrm{o}}(\pgl_{14})$ corresponding to the following irreducible representation of $\mathcal{L}_{\Z}$:
\[\mathcal{L}_{\Z}\xrightarrow{\psi_{\pi}}\symp(3)\xrightarrow{\wedge^{*}}\SL_{14}(\C).\]

For $\psi\in\Psi_{\mathrm{AJ}}(\grpF)$ satisfying $\mathrm{H}(\psi)=H$ and $\mathrm{m}(\pi_{\psi})=1$,
there are two possible endoscopic types:
\begin{enumerate}[label=(\roman*)]
    \item $(2,(14,1),(6,2))$.
    A global Arthur parameter of this type is of the form:
        \[\wedge^{*}\pi\oplus\pi[2],\,\pi\in\Pi_{\alg}^{\symp_{6}}(\pgl_{6}).\]
    \item $(2,(14,1),(12,1))$.
    A global Arthur parameter of this type is of the form:
        \[\wedge^{*}\pi\oplus(\pi\otimes\tau),\,\pi\in\Pi_{\alg}^{\symp_{6}}(\pgl_{6}),\tau\in\Pi_{\alg}^{\perp}(\pgl_{2}).\]
\end{enumerate}

For this subgroup $H$ of $\lietype{F}{4}$,
the restriction of the adjoint representation $\mathfrak{f}_{4}$ of $\lietype{F}{4}$ to $H$ is isomorphic to 
\[\sym^{2}\st\otimes\triv + \st\otimes \vrep{14}^{\prime}+ \triv\otimes\sym^{2}\vrep{6},\] 
where $\vrep{14}^{\prime}$ is another $14$-dimensional irreducible representation of $\symp(3)$ that is not equivalent to $\vrep{14}=\wedge^{*}\vrep{6}$.

\begin{spprop}\label{prop multiplicity A1+Sp(3)}
    For a discrete global Arthur parameter $\psi\in\Psi_{\mathrm{AJ}}(\grpF)$ satisfying $\mathrm{H}(\psi)=H$,
    the multiplicity $\mathrm{m}(\pi_{\psi})=1$ if and only if $\psi$ is one of the following parameters: 
   \begin{itemize}
    \item $\wedge^{*}\pi\oplus\pi[2]$, where $\pi\in\Pi_{\alg}^{\symp_{6}}(\mathrm{PGL}_{6})$ has Hodge weights $w_{1}>w_{2}>w_{3}>1$ 
    and one of the following conditions holds:
    \begin{itemize}
        \item $w_{1}>w_{2}+w_{3}+1$ and $w_{1}+w_{2}+w_{3}\equiv 3\modulo 4$;
        \item $w_{1}<w_{2}+w_{3}-1$ and $w_{1}+w_{2}+w_{3}\equiv 1\modulo 4$.
    \end{itemize}
    \item $\wedge^{*}\pi\oplus(\pi\otimes\tau)$, where $\pi\in\Pi_{\alg}^{\symp_{6}}(\mathrm{PGL}_{6})$ has Hodge weights $w_{1}>w_{2}>w_{3}$  
    and $\tau\in\Pi_{\alg}^{\perp}(\pgl_{2})$ has motivic weight $v$ satisfying one of the following conditions:
    \begin{itemize}
        \item $|w_{1}-w_{2}-w_{3}|<v<w_{3}$;
        \item $w_{1}-w_{2}+w_{3}<v<w_{2}$;
        \item $w_{3}<v<\min(w_{2},w_{1}-w_{2}-w_{3})$;
        \item $\max(w_{2},w_{1}-w_{2}-w_{3})<v<w_{1}-w_{2}+w_{3}$;
        \item $w_{1}<v<w_{1}+w_{2}-w_{3}$;
        \item $v>w_{1}+w_{2}+w_{3}$.
    \end{itemize}
\end{itemize}
\end{spprop}
\begin{proof}
    We denote the generator $(-1,1)\in \mathrm{Z}(H)=\mathrm{C}_{\psi}$ by $\gamma$.

    \textbf{Case (i)}: $\psi=\wedge^{*}\pi\oplus\pi[2]$, where $\pi\in\Pi_{\alg}^{\symp_{6}}(\mathrm{PGL}_{6})$ has Hodge weights $w_{1}>w_{2}>w_{3}>1$.
    In this case, 
    the restriction of $\mathfrak{f}_{4}$ along $\psi$ is isomorphic to 
    \[\sym^{2}\pi\oplus \pi^{\prime}[2]\oplus[3],\]
    where $\pi^{\prime}\in\Pi_{\alg}^{\perp}(\pgl_{14})$ corresponds to 
    \[\mathcal{L}_{\Z}\xrightarrow{\psi_{\pi}}\symp(3)\xrightarrow{\vrep{14}^{\prime}}\SL_{14}(\C).\]
    Notice that $\wedge^{3}\vrep{6}\simeq \vrep{14}^{\prime}\oplus \vrep{6}$,
    thus the Hodge weights of $\pi^{\prime}$ are 
    \[\pm w_{1},\pm w_{2},\pm w_{3},\pm w_{1}\pm w_{2}\pm w_{3}.\]
    By \Cref{prop epsilon character of Arthur} we have:
    \begin{align*}
        \varepsilon_{\psi}(\gamma)&=\varepsilon\left(\mathbf{I}_{w_{1}}+\mathbf{I}_{w_{2}}+\mathbf{I}_{w_{3}}+\mathbf{I}_{w_{1}+w_{2}+w_{3}}+\mathbf{I}_{w_{1}+w_{2}-w_{3}}+\mathbf{I}_{w_{1}-w_{2}+w_{3}}+\mathbf{I}_{|w_{1}-w_{2}-w_{3}|}\right)\\
        &=(-1)^{(w_{1}+w_{2}+w_{3}+1)/2+\max(w_{1},w_{2}+w_{3})}.
    \end{align*}
    On the other side, $\gamma$ acts on $\wedge^{*}\pi$ by $1$ and on $\pi[2]$ by $-1$.
    The largest element $\mu_{1}$ must be $\frac{w_{1}+w_{2}}{2}$. 
    Now $\mu_{4}=\frac{w_{1}\pm 1}{2}$, thus $\rho_{\psi}^{\vee}(\gamma)=-1$.
    By Arthur's multiplicity formula, $\mathrm{m}(\pi_{\psi})=1$ if and only if one of the following conditions holds:
    \begin{itemize}
        \item $w_{1}>w_{2}+w_{3}+1$ and $w_{1}+w_{2}+w_{3}\equiv 3\modulo 4$;
        \item $w_{1}<w_{2}+w_{3}-1$ and $w_{1}+w_{2}+w_{3}\equiv 1\modulo 4$.
    \end{itemize}

    \textbf{Case (ii)}: $\psi=\wedge^{*}\pi\oplus(\pi\otimes\tau)$, 
    where $\pi\in\Pi_{\alg}^{\symp_{6}}(\mathrm{PGL}_{6})$ has Hodge weights $w_{1}>w_{2}>w_{3}$ and $\tau\in\Pi_{\alg}^{\perp}(\pgl_{2})$ has motivic weight $v$.
    In this case $\varepsilon_{\psi}$ is trivial.
    On the other side, the largest $\mu_{1}$ must be $\frac{w_{1}+w_{2}}{2}$ or $\frac{w_{1}+v}{2}$.
    \begin{enumerate}[label= (\arabic*)]
        \item If $v<w_{2}$, then $\mu_{1}=\frac{w_{1}+w_{2}}{2}$.
        \begin{enumerate}
            \item If $v<w_{3}$, then $\mu_{4}$ is the middle one in $\{\frac{w_{2}+w_{3}}{2},\frac{w_{1}+v}{2},\frac{w_{1}-v}{2}\}$. 
            Hence $\rho^{\vee}_{\psi}=1$ if and only if $\mu_{4}=\frac{w_{2}+w_{3}}{2}$, which is equivalent to $v>|w_{1}-w_{2}-w_{3}|$.
            \item If $v>w_{3}$, then $\mu_{4}$ is the middle one in $\{\frac{w_{2}+v}{2},\frac{w_{1}+w_{3}}{2},\frac{w_{1}-w_{3}}{2}\}$. 
            Hence $\rho^{\vee}_{\psi}=1$ if and only if $\mu_{4}=\frac{w_{1}\pm w_{3}}{2}$, which is equivalent to $v>w_{1}-w_{2}+w_{3}$ or $v<w_{1}-w_{2}-w_{3}$.
        \end{enumerate}
        \item If $v>w_{2}$, then $\mu_{1}=\frac{w_{1}+v}{2}$.
        \begin{enumerate}
            \item If $v<w_{1}$, then $\mu_{4}$ is the middle one in $\{\frac{w_{2}+v}{2},\frac{w_{1}+w_{3}}{2},\frac{w_{1}-w_{3}}{2}\}$. 
            Hence $\rho^{\vee}_{\psi}=1$ if and only if $\mu_{4}=\frac{w_{2}+v}{2}$, which is equivalent to $w_{1}-w_{2}-w_{3}<v<w_{1}-w_{2}+w_{3}$.
            \item If $v>w_{1}$, then $\mu_{4}$ is the middle one in $\{\frac{w_{1}+w_{2}}{2},\frac{v+w_{3}}{2},\frac{v-w_{3}}{2}\}$. 
            Hence $\rho^{\vee}_{\psi}=1$ if and only if $\mu_{4}=\frac{v\pm w_{3}}{2}$, which is equivalent to $v>w_{1}+w_{2}+w_{3}$ or $v<w_{1}+w_{2}-w_{3}$.
        \end{enumerate}
    \end{enumerate}
    In conclusion, $\mathrm{m}(\pi_{\psi})=1$ if and only if one of the conditions on $\pi,\tau$ listed in the proposition is satisfied.
\end{proof}

\subsubsection{\texorpdfstring{$H=\spin(8)$}{}}
\label{section Arthur Spin(8)}
 By \Cref{section info Spin(8)},
the restriction of the $26$-dimensional irreducible representation $\jord_{0}$ of $\lietype{F}{4}$ to $H$ is isomorphic to 
\[\triv^{\oplus 2}+\vrep{8}+\vrep{\spin}^{+}+\vrep{\spin}^{-},\]
where $\vrep{8}$ is the $8$-dimensional vector representation of $\spin(8)$,
i.e. the composition of $\spin(8)\rightarrow\sorth(8)$ with the standard $8$-dimensional representation of $\sorth(8)$,
and $\vrep{\spin}^{\pm}$ are two $8$-dimensional spinor representations.
The centralizer of $H$ in $\lietype{F}{4}$ is $\mathrm{Z}(H)\simeq \Z/2\Z\times\Z/2\Z$.

For $\psi\in\Psi_{\mathrm{AJ}}(\grpF)$ satisfying $\mathrm{H}(\psi)=H$ and $\mathrm{m}(\pi_{\psi})=1$,
there is only one possible endoscopic type: $(5,(8,1),(8,1),(8,1),(1,1),(1,1))$.
A global Arthur parameter of this type is of the form:
\[\psi=\pi\oplus \spin^{+}\pi\oplus\spin^{-}\pi\oplus[1]\oplus[1],\,\pi\in\Pi_{\alg}^{\sorth_{8}}(\pgl_{8}),\]
where we lift $\psi_{\pi}:\mathcal{L}_{\Z}\twoheadrightarrow\sorth(8)\rightarrow \sorth_{8}(\C)$ to $\widetilde{\psi_{\pi}}:\mathcal{L}_{\Z}\rightarrow \spin_{8}(\C)$,
and $\spin^{*}\pi,*=\pm$ is the representation corresponding to 
\[\mathcal{L}_{\Z}\xrightarrow{\widetilde{\psi_{\pi}}}\spin_{8}(\C)\xrightarrow{V_{\spin}^{*}} \SL_{8}(\C).\]
\begin{spprop}\label{prop multiplicity Spin(8)}
    For any discrete global Arthur parameter $\psi\in\Psi_{\mathrm{AJ}}(\grpF)$ satisfying $\mathrm{H}(\psi)=H$,
    we have $\mathrm{m}(\pi_{\psi})=0$.
\end{spprop}
\begin{proof}
    Let 
    $\psi=\pi\oplus \spin^{+}\pi\oplus\spin^{-}\pi\oplus[1]\oplus[1]$,
    where $\pi\in\Pi_{\alg}^{\sorth_{8}}(\mathrm{PGL}_{8})$ has Hodge weights $2w_{1}>2w_{2}>2w_{3}>2w_{4}$.
    The global component group $\mathrm{C}_{\psi}\simeq\Z/2\Z\times\Z/2\Z$ and it acts on $\pi,\spin^{+}\pi,\spin^{-}\pi$ by three different characters.

    Since $\varepsilon_{\psi}$ is trivial,
    by Arthur's trace formula $\mathrm{m}(\pi_{\psi})=1$ if and only if $\rho^{\vee}_{\psi}=1$,
    which is equivalent to that $\mu_{1}$ and $\mu_{4}$ come from the same irreducible summand of $\psi$ by \Cref{prop calculation of rho character for F4}.

    In this case, the largest element $\mu_{1}$ must be $w_{1}$ or $\frac{w_{1}+w_{2}+w_{3}+w_{4}}{2}$.
    \begin{enumerate}[label= (\arabic*)]
        \item If $w_{1}>w_{2}+w_{3}+w_{4}$, then $\mu_{1}=w_{1}$. Now we have 
             \[w_{2}<\frac{w_{1}+w_{2}-w_{3}+w_{4}}{2}<\frac{w_{1}+w_{3}+w_{3}-w_{4}}{2}<\frac{w_{1}+w_{2}+w_{3}+w_{4}}{2}<\mu_{1},\]
             thus $\mu_{4}$ does not come from $\pi$. 
             Hence $\rho_{\psi}^{\vee}$ is not trivial.
        \item If $w_{1}<w_{2}+w_{3}+w_{4}$, then $\mu_{1}=\frac{w_{1}+w_{2}+w_{3}+w_{4}}{2}$. Now we have
             \[\frac{w_{1}-w_{2}+w_{3}-w_{4}}{2}<\frac{w_{1}+w_{2}-w_{3}-w_{4}}{2}<\min\left(w_{2},\frac{w_{1}+w_{2}\pm(w_{3}-w_{4})}{2}\right)<\mu_{1}\]
             and \[\frac{|w_{1}-w_{2}-w_{3}+w_{4}|}{2}\leq \max\left(w_{4},\frac{-w_{1}+w_{2}+w_{3}+w_{4}}{2}\right)\]
             is also smaller than at least $4$ weights, hence \[\mu_{4}\notin\Big\{\frac{w_{1}-w_{2}+w_{3}-w_{4}}{2},\frac{w_{1}+w_{2}-w_{3}-w_{4}}{2},\frac{|w_{1}-w_{2}-w_{3}+w_{4}|}{2}\Big\}.\] 
             So $\mu_{4}$ does not come from $\spin^{+}\pi$ and $\rho_{\psi}^{\vee}$ is not trivial. 
    \end{enumerate}
    In conclusion, by Arthur's multiplicity formula the multiplicity $\mathrm{m}(\pi_{\psi})$ is always $0$.
\end{proof}

\subsubsection{\texorpdfstring{$H=\spin(9)$}{}}
\label{section Arthur Spin(9)}
 By \Cref{section info Spin(9)},
the restriction of the $26$-dimensional irreducible representation $\jord_{0}$ of $\lietype{F}{4}$ to $H$ is isomorphic to 
\[\triv+\vrep{9}+\vrep{\spin},\]
where $\vrep{9}$ is the standard representation of $\spin(9)$,
$\vrep{\spin}$ is the $16$-dimensional spinor representations.
The centralizer of $H$ in $\lietype{F}{4}$ is $\mathrm{Z}(H)\simeq \Z/2\Z$.

For $\psi\in\Psi_{\mathrm{AJ}}(\grpF)$ satisfying $\mathrm{H}(\psi)=H$ and $\mathrm{m}(\pi_{\psi})=1$,
there is only one possible endoscopic type: $(3,(16,1),(9,1),(1,1))$.
A global Arthur parameter of this type is of the form:
\[\psi=\pi\oplus \spin\pi\oplus[1],\,\pi\in\Pi_{\alg}^{\sorth_{9}}(\pgl_{9}),\]
where we lift $\psi_{\pi}:\mathcal{L}_{\Z}\twoheadrightarrow\sorth(9)\rightarrow \sorth_{9}(\C)$ to $\widetilde{\psi_{\pi}}:\mathcal{L}_{\Z}\rightarrow \spin_{9}(\C)$,
and $\spin\pi$ is the representation corresponding to 
\[\mathcal{L}_{\Z}\xrightarrow{\widetilde{\psi_{\pi}}}\spin_{9}(\C)\xrightarrow{\vrep{\spin}} \SL_{16}(\C).\]

\begin{spprop}\label{prop multiplicity Spin(9)}
    A discrete global Arthur parameter $\psi\in\Psi_{\mathrm{AJ}}(\grpF)$ satisfies $\mathrm{H}(\psi)=H$ and $\mathrm{m}(\pi_{\psi})=1$ 
   if and only if 
   $\psi=\pi\oplus\spin\pi\oplus[1]$, 
   where $\pi\in\Pi_{\alg}^{\sorth_{9}}(\pgl_{9})$ has Hodge weights $w_{1}>w_{2}>w_{3}>w_{4}$ satisfying $w_{2}+w_{3}-w_{4}<w_{1}<w_{2}+w_{3}+w_{4}$.
\end{spprop}
\begin{proof}
    Let $\psi=\pi\oplus \spin\pi\oplus[1]$,
    where $\pi\in\Pi_{\alg}^{\sorth_{9}}(\mathrm{PGL}_{9})$ has Hodge weights $w_{1}>w_{2}>w_{3}>w_{4}$.
    The global component group $\mathrm{C}_{\psi}$ is a cyclic $2$-group,
    and it acts on $\pi$ trivially and on $\spin\pi$ by its non-trivial character. 

    Since the parameter is tempered, $\varepsilon_{\psi}$ is trivial.
    By Arthur's multiplicity formula, 
    $\mathrm{m}(\pi_{\psi})=1$ if and only if $\rho_{\psi}^{\vee}=1$,
    which is equivalent to that $\mu_{1}$ and $\mu_{4}$ come from the same irreducible summand of $\psi$ by \Cref{prop calculation of rho character for F4}.
    In this case, the largest element $\mu_{1}=\frac{w_{1}}{2}$ or $\frac{w_{1}+w_{2}+w_{3}+w_{4}}{4}$.
    \begin{enumerate}[label= (\arabic*)]
        \item If $w_{1}>w_{2}+w_{3}+w_{4}$, then $\mu_{1}=\frac{w_{1}}{2}$. 
        By our discussion in the proof of \Cref{prop multiplicity Spin(8)}, $\mu_{4}$ does not come from $\pi$, thus $\rho_{\psi}^{\vee}$ is not trivial.
        \item If $w_{1}<w_{2}+w_{3}+w_{4}$, then $\mu_{1}=\frac{w_{1}+w_{2}+w_{3}+w_{4}}{4}$. 
        Now $\mu_{4}=\max\left(\frac{w_{2}}{2},\frac{w_{1}+w_{2}-w_{3}+w_{4}}{4}\right)$. 
        Hence $\rho_{\psi}^{\vee}$ is trivial if and only if $w_{1}+w_{4}>w_{2}+w_{3}$.
    \end{enumerate}
    In conclusion, $\mathrm{m}(\pi_{\psi})=1$ if and only if $w_{2}+w_{3}-w_{4}<w_{1}<w_{2}+w_{3}+w_{4}$.
\end{proof}

\subsubsection{\texorpdfstring{$H=\lietype{F}{4}$}{PDFstring}}
\label{section Arthur F4 tempered}
 For stable tempered parameters,
the component group is trivial and as a direct consequence we have:
\begin{spprop}\label{prop multiplicity F4 tempered}
    For any discrete global Arthur parameter $\psi\in\Psi_{\mathrm{AJ}}(\grpF)$ satisfying $\mathrm{H}(\psi)=\lietype{F}{4}$,
    we have $\mathrm{m}(\pi_{\psi})=1$.
\end{spprop}

\subsection{Classification of representations contributing to \texorpdfstring{$\mathcal{A}_{\mathrm{V}_{\lambda}}(\grpF)$}{PDFstring}}
\label{section classify representations in the multiplicity space}
 Recall that in \Cref{section review of automorphic representations},
for each irreducible representation $\mathrm{V}_{\lambda}$ with highest weight $\lambda$ of $\lietype{F}{4}=\grpF(\R)$,
we have defined its multiplicity space in $\mathcal{L}_{\disc}(\grpF)$:
\[\mathcal{A}_{\mathrm{V}_{\lambda}}(\grpF)=\Hom_{\grpF(\R)}(\mathrm{V}_{\lambda},\mathcal{L}_{\disc}(\grpF)^{\mathcal{F}_{4,\mathrm{I}}(\widehat{\Z})}),\]
which parametrizes level one discrete automorphic representation $\pi$ of $\grpF$ such that $\pi_{\infty}\simeq \mathrm{V}_{\lambda}$.
We have a dimension formula \Cref{cor dimension formula for the multiplicity space} for this space.
Now with results in \Cref{section classification of Arthur parameters of F4},
we can study the discrete global Arthur parameters $\psi\in\Psi_{\mathrm{AJ}}(\grpF)$ 
whose corresponding representation $\pi_{\psi}\in\Pi(\grpF)$ has multiplicity $1$ in $\mathcal{L}_{\disc}(\grpF)$ and contributes to $\mathcal{A}_{\vrep{\lambda}}(\grpF)$. 

According to \Cref{lemma dimension of isotypic spaces},
we have:
\begin{align*}
    \dim\mathcal{A}_{\vrep{\lambda}}(\grpF)=\sum_{\pi\in \Pi(\grpF),\,\pi_{\infty}\simeq \vrep{\lambda}}\mathrm{m}(\pi).
\end{align*}
Using discrete global Arthur parameters,
we rewrite this formula as 
\begin{align*}
    \dim\mathcal{A}_{\vrep{\lambda}}(\grpF)=\sum_{\psi\in\Psi_{\mathrm{AJ}}(\grpF),\,\mathrm{c}_{\infty}(\psi)=\mathrm{c}_{\infty}(\vrep{\lambda})}\mathrm{m}(\pi_{\psi})=\sum_{\psi\in\Psi_{\mathrm{AJ}}(\grpF),\,\mathrm{c}_{\infty}(\psi)=\lambda+\rho}\mathrm{m}(\pi_{\psi}),
\end{align*}
where $\rho$ is the half sum of positive roots of $\grpF$.

If the endoscopic type of $\psi\in\Psi_{\mathrm{AJ}}(\grpF)$ is not stable,
i.e. $\mathrm{H}(\psi)$ is the conjugacy class of a proper subgroup of $\lietype{F}{4}=\grpF(\R)$,
then it must have one of the types listed in \Cref{section classification of Arthur parameters of F4}.
For each subgroup $H$ of $\lietype{F}{4}$ listed in \Cref{thm classification result of subgroups satisfying three conditions},
we can determine the discrete global Arthur parameters $\psi\in\Psi_{\mathrm{AJ}}(\grpF)$ satisfying $\mathrm{H}(\psi)=H$ and $\mathrm{m}(\pi_{\psi})=1$.
The difference 
\begin{align}\label{eqn dimension of tempered stable parameters}
\dim \mathcal{A}_{\mathrm{V}_{\lambda}}(\grpF)-\#\left\{\psi\in\Psi_{\mathrm{AJ}}(\grpF)\,|\,\mathrm{H}(\psi)\neq \lietype{F}{4}, \mathrm{c}_{\infty}(\psi)=\rho+\lambda,\mathrm{m}(\pi_{\psi})=1\right\}
\end{align}
is the number of discrete automorphic representations $\pi$ of $\grpF$ with archimedean component $\pi_{\infty}\simeq \mathrm{V}_{\lambda}$
whose global Arthur parameter is tempered and stable.
In other words:
\begin{spprop}\label{prop conjectural number of F4 image representations}
    Let $\lambda$ be a dominant weight of $\lietype{F}{4}$,
    we define the number 
    \[\lietype{F}{4}(\lambda):=\#\set{\pi\in\Pi_{\cusp}(\pgl_{26})}{\mathrm{c}_{\infty}(\pi)=\mathrm{r}_{0}(\lambda+\rho)\in\mathfrak{sl}_{26,\mathrm{ss}},\mathrm{H}(\pi)\simeq \lietype{F}{4}},\]
    where $\mathrm{r}_{0}:\mathfrak{f}_{4}\rightarrow \mathfrak{sl}_{26}$ is the $26$-dimensional irreducible representation of $\mathfrak{f}_{4}$,
    and define $\mathrm{w}(\lambda)$ to be twice the maximal eigenvalue of $\lambda+\rho$.  
    Then we have a formula for the number $\lietype{F}{4}(\lambda)$,
    and we list nonzero $\mathrm{F}_{4}(\lambda)$ for all the dominant weights $\lambda$ such that $\mathrm{w}(\lambda)\leq 44$ in \Cref{TableMult}. 
\end{spprop}
\begin{proof}
    The formula for $\lietype{F}{4}(\lambda)$ follows from \eqref{eqn dimension of tempered stable parameters}
    and our classifications in \Cref{section classification of Arthur parameters of F4}.
    This formula involves the numbers of elements in one of the following sets with certain Hodge weights:
    \[\Pi_{\alg}^{\perp}(\pgl_{2}),\Pi_{\alg}^{\symp_{4}}(\pgl_{4}),\Pi_{\alg}^{\symp_{6}}(\pgl_{6}),\Pi_{\alg}^{\lietype{G}{2}}(\pgl_{7}),\Pi_{\alg}^{\sorth_{9}}(\pgl_{9}).\]
    For $\Pi_{\alg}^{\perp}(\pgl_{2})$,
    this number is related to the dimension of cusp forms for $\SL_{2}(\Z)$,
    as explained in \Cref{ex algebraic representation of GL(2)}.
    For other four sets,
    we can find some tables in \cite{TableAutorepG2} and \cite{TableAutorep}. 
    A \cite{PARI2} program to compute $\mathrm{F}_{4}(\lambda)$ for dominant weights $\lambda$ satisfying $\mathrm{w}(\lambda)\leq 60$
    is provided in \cite{Codes}.
\end{proof}
\begin{rmk}\label{rmk simpler version of formula}
    The formula for $\lietype{F}{4}(\lambda)$ has too many terms,
    thus it is not reasonable to write it down here.
    However, under some hypothesis on $\lambda$,
    many terms vanish and this formula becomes much simpler.
    For example,
    if
    \begin{itemize}
        \item $\lambda_{i}>0$ for $i=1,2,3,4$,
        \item $\lambda_{1}>\lambda_{2}+\lambda_{3}+\lambda_{4}+3$,
        \item and  $\lambda_{3},\lambda_{4}$ are both odd,
    \end{itemize} 
    then we have the following formula: 
    \begin{align*}\label{eqn simpler formula}
        \lietype{F}{4}(\lambda)=\dim\mathcal{A}_{\vrep{\lambda}}(\grpF)-
        \mathrm{O}^{*}(\lambda_{1}^{\prime},\lambda_{2}^{\prime},\lambda_{3}^{\prime},\lambda_{4}^{\prime}),
    \end{align*}   
    where $\mathrm{O}^{*}(w_{1},w_{2},w_{3},w_{4})$ is the number of equivalence classes of level one cuspidal orthogonal representations of $\pgl_{9}$
    with Hodge weights $w_{1}>w_{2}>w_{3}>w_{4}>0$,
    and 
    \begin{gather*}
        \lambda_{1}^{\prime}=2\lambda_{1}+6\lambda_{2}+4\lambda_{3}+2\lambda_{4}+14,
        \,\lambda_{2}^{\prime}=2\lambda_{1}+2\lambda_{2}+2\lambda_{3}+2\lambda_{4}+8,\\
        \lambda_{3}^{\prime}=2\lambda_{1}+2\lambda_{2}+2\lambda_{3}+6,
        \,\lambda_{4}^{\prime}=2\lambda_{1}+2\lambda_{2}+4.
    \end{gather*}
\end{rmk}
In \Cref{TableMult},
we find that the smallest $\mathrm{w}(\lambda)$ for $\lambda$ such that $\lietype{F}{4}(\lambda)\neq 0$
is $36$ and the corresponding dominant weight is $\lambda=\varpi_{1}+2\varpi_{2}+2\varpi_{4}$.
We are now going to prove this fact 
without using \Cref{thm arthur parameter with nonzero multiplicity},
in order to give readers who skip the proof of \Cref{thm arthur parameter with nonzero multiplicity} an example of how we apply Arthur's conjectures.
\begin{spprop}\label{prop smallest motivic weight 36 tempered stable representation}
    There is a level one cuspidal automorphic representation $\pi$ of $\pgl_{26}$ with motivic weight $36$,
    such that the Sato-Tate group $\mathrm{H}(\pi)$ of $\pi$ is isomorphic to the compact Lie group $\lietype{F}{4}$.     
\end{spprop}
\begin{proof}
    We fix $\lambda=\varpi_{1}+2\varpi_{2}+2\varpi_{4}$.
    In \Cref{nonzeroAMF},
    we find that $\dim\mathcal{A}_{\vrep{\lambda}}(\grpF)=1$.
    We denote the unique automorphic representation contributing to $\mathcal{A}_{\mathrm{V}_{\lambda}}(\grpF)$ by $\pi_{0}$ 
    and its corresponding discrete global Arthur parameter by $\psi_{0}$.
    The eigenvalues of $\mathrm{c}_{\infty}(\pi_{0})=\lambda+\rho$ are: 
    \[-18,-16,-13,-12,-9,-9,-7,-6,-5,-4,-3,-2,0,0,2,3,4,5,6,7,9,9,12,13,16,18.\]
    Now it suffices to show that $\mathrm{H}(\psi_{0})=\lietype{F}{4}$.

    We can exclude some possibilities of $\mathrm{H}(\psi_{0})$ and endoscopic types by an argument of motivic weights.
    For example, if $\mathrm{H}(\psi_{0})=\mathrm{A}_{1}^{[17,9]}$ and $\psi_{0}=\sym^{16}\pi\oplus\sym^{8}\pi$ for some $\pi\in\Pi_{\alg}^{\perp}(\pgl_{2})$,
    then $\mathrm{w}(\pi_{0})=16\mathrm{w}(\pi)\geq 16\times11=176$,
    which contradicts with $\mathrm{w}(\pi_{0})=36$.
    We also notice that $1$ is not an eigenvalue of $\mathrm{c}_{\infty}(\pi_{0})$, 
    thus $\psi_{0}$ does not have irreducible summands of the form 
    \[\pi[d],\text{ where }\pi\in\Pi_{\alg}^{\perp}(\pgl_{n}),n\equiv 1\modulo2 \text{ and } d\geq 3.\] 
    Now we list all possible types for $\psi_{0}$:
    \begin{enumerate}[label= (\arabic*)]
        \item $\psi_{0}$ is a stable and tempered parameter;
        \item $\psi_{0}=(\bigoplus_{1\leq i<j\leq 3}\pi_{i}\otimes\pi_{j})\oplus(\bigoplus_{1\leq i\leq 3}\pi_{i}[2])\oplus[1]\oplus[1], \pi_{1},\pi_{2},\pi_{3}\in\Pi_{\alg}^{\perp}(\pgl_{2})$;
        \item $\psi_{0}=(\bigoplus_{1\leq i<j\leq 4}\pi_{i}\otimes\pi_{j})\oplus[1]\oplus[1], \pi_{1},\pi_{2},\pi_{3},\pi_{4}\in\Pi_{\alg}^{\perp}(\pgl_{2})$;
        \item $\psi_{0}=\wedge^{*}\pi\oplus(\pi\otimes\tau)\oplus\pi[2]\oplus\tau[2]\oplus[1],\pi\in\Pi_{\alg}^{\symp_{4}}(\pgl_{4}),\tau\in\Pi_{\alg}^{\perp}(\pgl_{2})$;
        \item $\psi_{0}=\wedge^{*}\pi\oplus(\pi\otimes\tau_{1})\oplus(\pi\otimes\tau_{2})\oplus(\tau_{1}\otimes\tau_{2})\oplus[1],\pi\in\Pi_{\alg}^{\symp_{4}}(\pgl_{4}),\tau_{1},\tau_{2}\in\Pi_{\alg}^{\perp}(\pgl_{2})$;
        \item $\psi_{0}=\wedge^{*}\pi\oplus\pi[2],\pi\in\Pi_{\alg}^{\symp_{6}}(\pgl_{6})$;
        \item $\psi_{0}=\wedge^{*}\pi\oplus(\pi\otimes\tau),\pi\in\Pi_{\alg}^{\symp_{6}}(\pgl_{6}),\tau\in\Pi_{\alg}^{\perp}(\pgl_{2})$;
        \item $\psi_{0}=\pi\oplus\spin^{+}\pi\oplus\spin^{-}\pi\oplus[1]\oplus[1],\pi\in\Pi_{\alg}^{\sorth_{8}}(\pgl_{8})$;
        \item $\psi_{0}=\pi\oplus\spin\pi\oplus[1],\pi\in\Pi_{\alg}^{\sorth_{9}}(\pgl_{9})$.
    \end{enumerate}
    The definitions of some notations like $\wedge^{*},\spin^{\pm}$ can be found in \Cref{section classification of Arthur parameters of F4}.
    Now we are going to show that $\psi_{0}$ can not be of the types listed above except $(1)$. 

    \textbf{Type $(2)$}: The Hodge weights of the irreducible summand $\pi_{i}[2],i=1,2,3$ are $\mathrm{w}(\pi_{i})\pm 1$,
    thus there are two consecutive integers $\frac{\mathrm{w}(\pi_{i})\pm 1}{2}$ in the eigenvalues of $\mathrm{c}_{\infty}(\pi_{0})$.
    The possible $\mathrm{w}(\pi_{i})$'s are $5,7,9,11,13,25$.
    However, $\Pi_{\alg}^{\perp}(\pgl_{2})$ contains no representations with motivic weights $5,7,9,13$, 
    thus we are unable to find three different $\mathrm{w}(\pi_{i})$.
    If $\pi_{i}\simeq\pi_{j}$ for some $i,j$,
    then $\pi_{i}\otimes\pi_{j}$ has two zero weights,
    which is a contradiction!

    \textbf{Type $(3)$}: By the same argument for type $(2)$, $\psi_{0}$ can not be of this type.

    \textbf{Type $(4)$}: Denote the Hodge weights of $\pi\in\Pi_{\alg}^{\symp_{4}}(\pgl_{4})$ by $w_{1}>w_{2}$. 
    By a similar argument for type $(2)$, we can see that $w_{1},w_{2}\in\{5,7,9,11,13,25\}$.
    Via the help of \cite[Table 5]{ChenevierRenard},  we have $w_{1}=25$ and $w_{2}\in\{5,7,9\}$, thus $\mathrm{w}(\tau)$ must be $11$.
    Since $(w_{1}+w_{2})/2$ has to be an eigenvalue of $\mathrm{c}_{\infty}(\pi_{\infty})$, 
    the smaller Hodge weight $w_{2}$ can only be $7$.

    Now we use Arthur's multiplicity formula.
    In this case 
    \[\mathrm{H}(\psi_{0})=\left(\mathrm{A}_{1}^{[2^{6},1^{14}]}\times \mathrm{A}_{1}^{[2^{6},1^{14}]}\times \symp(2)\right)/\mu_{2}^{\Delta},\] 
    and by \Cref{section info A1+A1+Sp(2)}
    the global component group $\mathrm{C}_{\psi_{0}}\simeq\Z/2\Z\times\Z/2\Z$. 
    We take a set of generators $\{\sigma=(1,1,-1),\sigma_{1}=(-1,1,1)\}$ of $\mathrm{C}_{\psi_{0}}$.
    The restriction of the adjoint representation $\mathfrak{f}_{4}$ of $\lietype{F}{4}$ along $\psi_{0}$ is isomorphic to 
    \[\sym^{2}\pi\oplus\left(\wedge^{*}\pi\otimes\tau[2]\right)\oplus\left(\pi\otimes\tau\right)\oplus\pi[2]\oplus\sym^{2}\tau\oplus[3].\]
    By \Cref{prop epsilon character of Arthur} we have:
    \begin{gather*}
        \varepsilon_{\psi_{0}}(\sigma)=\varepsilon(\pi)=\varepsilon(\mathbf{I}_{7})\cdot\varepsilon(\mathbf{I}_{25})=-1.
    \end{gather*}
    On the other side $\mu_{1}=36$ comes from $\pi\otimes\tau$ 
    and $\mu_{4}=24$ comes from $\pi[2]$.
    The element $\sigma$ acts on $\pi\otimes \tau$ and $\pi[2]$ both by $-1$, 
    thus $\rho_{\psi_{0}}^{\vee}(\sigma)=1$ by \Cref{prop calculation of rho character for F4}.
    By Arthur's multiplicity formula,
    the corresponding representation has multiplicity $0$ in $\mathcal{L}_{\disc}(\grpF)$.

    \textbf{Type $(5)$}: Denote the Hodge weights of $\pi\in\Pi_{\alg}^{\symp_{4}}(\pgl_{4})$ by $w_{1}>w_{2}$,
    and assume that $\mathrm{w}(\tau_{1})>\mathrm{w}(\tau_{2})$.
    Since $36\geq w_{1}+\mathrm{w}(\tau_{1})\geq w_{1}+15$, 
    we have $w_{1}\leq 21$, 
    thus $(w_{1},w_{2})=(19,7)$ or $(21,5),(21,9),(21,13)$ by \cite[Table 5]{ChenevierRenard}.
    We also need $(w_{1}\pm w_{2})/2$ to be eigenvalues of $\mathrm{c}_{\infty}(\pi_{0})$, 
    so $(w_{1},w_{2})=(19,7)$.
    However, the equalities $36=w_{1}+\mathrm{w}(\tau_{1})$ and $32=w_{1}+\mathrm{w}(\tau_{2})$ imply that 
    $\mathrm{w}(\tau_{1})=17,\mathrm{w}(\tau_{2})=13$, 
    which contradicts with the non-existence of representations in $\Pi_{\alg}^{\perp}(\pgl_{2})$ with Hodge weight $13$.
    
    \textbf{Type $(6)$}: Denote the Hodge weights of $\pi\in\Pi_{\alg}^{\symp_{6}}(\pgl_{6})$ by $w_{1}>w_{2}>w_{3}$.
    We have three pairs of consecutive integers $\frac{w_{i}\pm 1}{2}$ in the eigenvalues of $\mathrm{c}_{\infty}(\pi_{0})$, 
    thus for $i=1,2,3$ we have $w_{i}\in\{5,7,9,11,13,25\}$.
    By \cite[Table 6]{ChenevierRenard}, $(w_{1},w_{2},w_{3})$ must be $(25,13,7)$. 
    However, $\wedge^{*}\pi$ has $38$ as its weight, 
    which is a contradiction.

    \textbf{Type $(7)$}: Denote the Hodge weights of $\pi\in\Pi_{\alg}^{\symp_{6}}(\pgl_{6})$ by $w_{1}>w_{2}>w_{3}$.
    Since $36\geq w_{1}+\mathrm{w}(\tau)\geq w_{1}+11$, we have $23\leq w_{1}\leq 25$.
    Combining $36\geq w_{1}+w_{2}$
    with \cite[Table 6]{ChenevierRenard},
    we get $(w_{1},w_{2},w_{3})=(23,13,5)$.
    However, $\mathrm{w}(\tau)=32-w_{1}=9<11$, which is a contradiction.

    \textbf{Type $(8)$}: Denote the Hodge weights of $\pi\in\Pi_{\alg}^{\sorth_{8}}(\pgl_{8})$ by $w_{1}>w_{2}>w_{3}>w_{4}$. 
    The multiset \[\{\pm w_{1}/2,\pm w_{2}/2,\pm w_{3}/2,\pm w_{4}/2,\frac{\pm w_{1}\pm w_{2}\pm w_{3}\pm w_{4}}{4},0,0\}\]
    coincides with the multiset of eigenvalues of $\mathrm{c}_{\infty}(\pi_{0})$.
    The solutions to this system of equations are \[(w_{1},w_{2},w_{3},w_{4})=(26,24,18,4),(32,18,12,10),(36,14,8,6).\]
    By the method of Chenevier-Ta\"ibi in \cite{Chenevier2020DiscreteSM}, 
    there are no representations in $\Pi_{\alg}^{\sorth_{8}}(\pgl_{8})$ with these Hodge weights.

    \textbf{Type $(9)$}: By the same argument for type $(9)$, we get the Hodge weights of $\pi\in\Pi_{\alg}^{\sorth_{9}}(\pgl_{9})$:
     \[(w_{1},w_{2},w_{3},w_{4})=(26,24,18,4),(32,18,12,10),(36,14,8,6).\]
    Again by the method in \cite{Chenevier2020DiscreteSM}, 
    there are no representations in $\Pi_{\alg}^{\sorth_{9}}(\pgl_{9})$ with these Hodge weights.

    In conclusion, the discrete global Arthur parameter $\psi_{0}$ is a stable and tempered parameter, 
    i.e. $\mathrm{H}(\psi_{0})=\lietype{F}{4}$.
    Composing this $\psi_{0}$ with the $26$-dimensional irreducible representation $r_{0}:\widehat{\grpF}(\C)\rightarrow \SL_{26}(\C)$,
    we get an irreducible $26$-dimensional representation of $\mathcal{L}_{\Z}$,
    and its corresponding cuspidal representation of $\pgl_{26}$ is the desired one.
\end{proof}
For each dominant weight $\lambda$ of $\grpF$,
we define $\Psi_{\lambda}(\grpF)$ to be the set
\[\set{\psi\in\Psi_{\mathrm{AJ}}(\grpF)}{\pi_{\psi}\in\Pi_{\disc}(\grpF)\text{ and }(\pi_{\psi})_{\infty}\simeq \vrep{\lambda}}.\]
In \Cref{table packet 34} and \Cref{table packet 36}, 
we list the elements of $\Psi_{\lambda}(\grpF)$ for weights $\lambda$ such that $\mathrm{w}(\lambda)\leq 36$ and $\Psi_{\lambda}(\grpF)\neq\emptyset$, 
where we use the following notations: 
\begin{notation}\label{notation representations in the table}
    For a representation $\pi$ in $\Pi_{\alg}^{\symp_{2n}}(\pgl_{2n}), n=1,2,3$ with Hodge weights $w_{1}>w_{2}>\cdots>w_{n}$, 
    we denote it by $\Delta_{w_{1},\ldots,w_{n}}$.
    If there are $k\geq 1$ equivalence classes of cuspidal representations with these Hodge weights, 
    we give them a superscript $\Delta_{w_{1},\ldots,w_{n}}^{(k)}$,
    meaning that in this case we have $k$ different choices of cuspidal representations. 
    Similarly, for $k$ different representations $\pi$ in $\Pi_{\alg}^{\sorth_{9}}(\pgl_{9})$ or $\Pi_{\alg}^{\lietype{G}{2}}(\pgl_{7})$
    with Hodge weights $w_{1}> \cdots > w_{n}$,
    where $n=3$ or $4$, 
    we denote them by $\Delta_{w_{1},\ldots,w_{n},0}^{(k)}$ and omit the superscript when $k=1$,
    i.e. the cuspidal representation with these Hodge weights is unique up to equivalence.      
\end{notation}

\subsection{Some related problems}
\label{section explain the table of lower motivic weights}
 In this subsection we explain some representation-theoretic problems motivated by our conjectural classification of discrete global Arthur parameters for $\grpF$.
\subsubsection{Theta correspondence between \texorpdfstring{$\pgl_{2}\text{ and }\grpF$}{PDFstring}}\label{section parameters corresponding to A1F4 theta}
 Inside an exceptional group $\mathbf{E}_{7,3}$ of Lie type $\lietype{E}{7}$ and $\Q$-rank $3$,
which is split over every finite prime $p$,
there is a reductive dual pair $\pgl_{2}\times\grpF$,
so we have an exceptional theta correspondence between representations of $\pgl_{2}$ and $\grpF$.

For a level one cuspidal automorphic representation $\pi\in\Pi_{\alg}^{\perp}(\pgl_{2})$,
by Savin's work on this exceptional theta correspondence \cite{Savin1994},
if the theta lift $\Theta(\pi)$ of $\pi$ to $\grpF$ is nonzero,
then its corresponding discrete global Arthur parameter is $\psi=\pi[6]\oplus[9]\oplus[5]$.
By \Cref{prop multiplicity principal A1 in Sp(3) with Sp(1)}, 
we see that $\mathrm{m}(\pi_{\psi})$ is always $1$,
admitting Arthur's conjectures.
This predicts that the global theta lift $\Theta(\pi)$ is nonzero for any $\pi\in\Pi_{\alg}^{\perp}(\pgl_{2})$,
and we will prove this result in another paper in progress.
\begin{rmk}\label{rmk table for theta A1F4}
    For $\pi\in\Pi_{\alg}^{\perp}(\pgl_{2})$,
    the archimedean theta lift of $\pi_{\infty}$ is isomorphic to the irreducible representation $\vrep{n\varpi_{4}}$ of $\lietype{F}{4}$ for some $n$.
    For readers interested in this exceptional theta correspondence,
    we list in \Cref{TableA1F4} the dimensions of $\vrep{n\varpi_{4}}^{\mathcal{F}_{4,\mathrm{I}}(\Z)}$ and $\vrep{n\varpi_{4}}^{\mathcal{F}_{4,\mathrm{E}}(\Z)}$
    for $n\leq 40$.
\end{rmk}
\subsubsection{Theta correspondence between \texorpdfstring{$\mathbf{G}_{2}\text{ and }\grpF$}{PDFstring}}
 Inside an exceptional group $\mathbf{E}_{8,4}$ of Lie type $\lietype{E}{8}$ and $\Q$-rank $4$,
there is a reductive dual pair $\mathbf{G}_{2}\times\grpF$,
where $\mathbf{G}_{2}$ is the generic fiber of the split Chevalley group of Lie type $\lietype{G}{2}$.

In \cite{dalal2023counting}, 
Dalal classifies level one quaternionic discrete automorphic representations of $\mathbf{G}_{2}$. 
The exceptional theta correspondence from $\mathbf{G}_{2}$ to $\grpF$ is functorial,
so for a level one quaternionic discrete automorphic representation of $\mathbf{G}_{2}$,
if its global theta lift to $\grpF$ is nonzero,
then we can describe the corresponding discrete global Arthur parameters in $\Psi_{\mathrm{AJ}}(\grpF)$.
The discrete global Arthur parameters of $\grpF$ involved in this correspondence are:
\begin{itemize}
    \item $\sym^{2}\pi[3]\oplus\pi[4]\oplus\pi[2]\oplus[5],\,\pi\in\Pi_{\alg}^{\perp}(\pgl_{2})$,
    \item $\sym^{2}\pi_{1}[3]\oplus (\pi_{1}\otimes\pi_{2}[3])\oplus[5]$, where $\pi_{1},\pi_{2}\in\Pi_{\alg}^{\perp}(\pgl_{2})$ satisfy $\mathrm{w}(\pi_{2})=3\mathrm{w}(\pi_{1})+2$,
    \item and $\pi[3]\oplus[5]$, where $\pi\in\Pi_{\alg}^{\lietype{G}{2}}(\pgl_{7})$.
\end{itemize}
According to \Cref{prop multiplicity principal A1 of Sp(1)+SO(3) with diagonal Sp(1)}, 
\Cref{prop multiplicity a bizzare A1+A1+A1 via A1+Sp(3)}
and \Cref{prop multiplicity A1+G2},
for every $\psi$ among these discrete global Arthur parameters,
we have $\mathrm{m}(\pi_{\psi})=1$.
This predicts that the global theta lift of any level one quaternionic discrete automorphic representation of $\mathbf{G}_{2}$ to $\grpF$ is nonzero,
which is proved by Pollack in \cite[\S 8]{pollack2023exceptional}.
\begin{rmk}\label{rmk table for theta G2F4}
    For any quaternionic discrete series $\pi$ of $\mathbf{G}_{2}(\R)$,
    the archimedean theta lift of $\pi$ 
    is isomorphic to the irreducible representation $\vrep{n\varpi_{3}}$ of $\lietype{F}{4}$ for some $n$.
    For readers interested in this exceptional theta correspondence,
    we list in \Cref{TableG2F4} the dimensions of $\vrep{n\varpi_{3}}^{\mathcal{F}_{4,\mathrm{I}}(\Z)}$ and $\vrep{n\varpi_{3}}^{\mathcal{F}_{4,\mathrm{E}}(\Z)}$
    for $n\leq 30$.
\end{rmk}

\pagebreak
\appendix
\section{Figures and tables}\label{TableAppendix}
\begin{figure}[ht]
    \centering
    \[\left(\scalemath{0.75}{\begin{matrix}
        4&2&2&-3&-1&-1&-1&-3&-3&-3&-2&-2&-2&-2&-2&-4&-4&-4&-4&-2&-2&-2&-2&-4&-4&-4&-4\\
        2&4&2&-2&-2&-2&-2&-4&-4&-4&-4&-3&-1&-1&-1&-3&-3&-3&-2&-2&-2&-2&-2&-4&-4&-4&-4\\
        2&2&4&-2&-2&-2&-2&-4&-4&-4&-4&-2&-2&-2&-2&-4&-4&-4&-4&-3&-1&-1&-1&-3&-3&-3&-2\\
        -3&-2&-2&5&1&1&1&4&4&4&2&2&2&2&2&4&4&4&4&2&2&2&2&4&4&4&4\\
        -1&-2&-2&1&5&1&1&4&4&4&4&2&0&2&2&2&2&2&3&2&0&0&0&2&2&2&0\\
        -1&-2&-2&1&1&5&1&4&2&2&4&2&0&0&2&2&2&2&1&2&2&0&0&2&3&3&2\\
        -1&-2&-2&1&1&1&5&2&4&2&4&2&0&0&0&2&2&2&1&2&2&2&0&3&2&3&2\\
        -3&-4&-4&4&4&4&2&8&6&6&6&4&2&2&3&5&5&6&5&4&2&2&2&5&6&5&4\\
        -3&-4&-4&4&4&2&4&6&8&6&6&4&2&3&2&6&5&5&5&4&2&2&2&5&5&6&4\\
        -3&-4&-4&4&4&2&2&6&6&8&6&4&2&3&3&5&6&5&5&4&2&2&2&6&5&5&4\\
        -2&-4&-4&2&4&4&4&6&6&6&8&4&0&2&2&4&4&4&3&4&3&1&1&5&5&5&3\\
        -2&-3&-2&2&2&2&2&4&4&4&4&5&1&1&1&4&4&4&2&2&2&2&2&4&4&4&4\\
        -2&-1&-2&2&0&0&0&2&2&2&0&1&5&1&1&4&4&4&4&2&0&2&2&2&2&2&3\\
        -2&-1&-2&2&2&0&0&2&3&3&2&1&1&5&1&4&2&2&4&2&0&0&2&2&2&2&1\\
        -2&-1&-2&2&2&2&0&3&2&3&2&1&1&1&5&2&4&2&4&2&0&0&0&2&2&2&1\\
        -4&-3&-4&4&2&2&2&5&6&5&4&4&4&4&2&8&6&6&6&4&2&2&3&5&5&6&5\\
        -4&-3&-4&4&2&2&2&5&5&6&4&4&4&2&4&6&8&6&6&4&2&3&2&6&5&5&5\\
        -4&-3&-4&4&2&2&2&6&5&5&4&4&4&2&2&6&6&8&6&4&2&3&3&5&6&5&5\\
        -4&-2&-4&4&3&1&1&5&5&5&3&2&4&4&4&6&6&6&8&4&0&2&2&4&4&4&3\\
        -2&-2&-3&2&2&2&2&4&4&4&4&2&2&2&2&4&4&4&4&5&1&1&1&4&4&4&2\\
        -2&-2&-1&2&0&2&2&2&2&2&3&2&0&0&0&2&2&2&0&1&5&1&1&4&4&4&4\\
        -2&-2&-1&2&0&0&2&2&2&2&1&2&2&0&0&2&3&3&2&1&1&5&1&4&2&2&4\\
        -2&-2&-1&2&0&0&0&2&2&2&1&2&2&2&0&3&2&3&2&1&1&1&5&2&4&2&4\\
        -4&-4&-3&4&2&2&3&5&5&6&5&4&2&2&2&5&6&5&4&4&4&4&2&8&6&6&6\\
        -4&-4&-3&4&2&3&2&6&5&5&5&4&2&2&2&5&5&6&4&4&4&2&4&6&8&6&6\\
        -4&-4&-3&4&2&3&3&5&6&5&5&4&2&2&2&6&5&5&4&4&4&2&2&6&6&8&6\\
        -4&-4&-2&4&0&2&2&4&4&4&3&4&3&1&1&5&5&5&3&2&4&4&4&6&6&6&8
    \end{matrix}}\right)\]
    \caption{The gram matrix of $(\jord_{\Z},\langle\,,\,\rangle_{\mathrm{E}})$ in the basis $\mathcal{B}$ given in \eqref{eqn basis of second Jordan algebra}}
    \label{Gram2ndModel}
\end{figure}
\pagebreak

\begin{figure}[H]
    \centering
    \[\sigma_{1}=\left(\scalemath{0.7}{\begin{matrix}
            1&2&1&0&-1&0&-2&-1&-2&-1&-2&-1&0&-1&0&-1&-1&-1&-1&-1&-1&-1&-1&-2&-2&-1&-1\\
            1&1&1&-1&0&-1&-1&-1&-1&-1&-1&0&0&-1&-1&-1&-1&0&-1&-1&-1&-1&0&-2&-1&-1&-1\\
            1&2&1&-1&0&-2&-1&-2&-1&-1&-2&-1&0&0&0&-1&-1&-1&0&0&-2&-1&-1&-2&-2&-1&-2\\
            -1&-2&-1&1&0&2&2&2&1&1&2&1&0&0&1&1&2&1&1&1&2&2&0&3&2&1&2\\
            -1&-1&-1&1&-1&1&1&1&0&1&1&0&1&0&1&1&2&1&1&1&2&1&1&2&2&1&2\\
            0&0&-1&0&-1&1&1&0&0&0&1&-1&0&0&0&0&0&0&0&1&1&-1&0&0&1&1&0\\
            0&0&-1&1&0&0&1&1&1&1&1&0&1&1&0&1&1&1&1&1&0&0&1&0&1&0&0\\
            1&1&1&-1&1&-2&-1&-1&0&0&-1&0&0&0&-1&-1&-1&-1&-1&-1&-2&-1&0&-2&-2&-2&-2\\
            0&1&1&-1&0&-1&-1&-2&-1&-1&-1&-1&-1&0&-1&-1&-2&-1&-1&-1&0&-1&-1&-1&-1&0&-1\\
            1&1&0&0&-1&-1&0&-1&0&-1&-1&-1&1&0&-1&0&-1&0&0&0&-1&-1&0&-2&-1&0&-1\\
            0&0&1&-1&1&0&-1&0&-1&-1&-1&1&-1&-1&1&-1&0&0&0&-1&-1&1&-1&0&-1&-1&0\\
            -1&-3&-1&1&0&2&3&2&3&1&3&2&0&1&-1&2&1&1&0&1&2&2&2&3&3&2&3\\
            -1&-2&-1&1&2&2&2&3&3&1&3&2&0&1&0&2&1&2&1&1&2&1&1&2&3&2&2\\
            -1&-1&0&2&2&0&0&2&2&2&1&1&0&1&1&1&1&1&1&0&1&0&1&1&2&1&1\\
            0&0&1&0&1&0&-1&0&0&-1&-1&0&0&0&0&0&-1&0&0&-1&0&0&0&-1&0&0&0\\
            1&3&1&-1&-2&-2&-2&-3&-3&-3&-4&-2&1&-1&0&-1&-1&-1&0&-1&-2&-1&-1&-3&-3&-2&-2\\
            0&1&0&0&0&-1&-1&0&-1&1&-1&0&0&-1&1&-1&1&0&0&0&-1&0&-1&0&-1&-1&-1\\
            0&1&1&0&1&-1&-1&-1&-1&0&-1&-1&-1&0&1&-1&-1&-1&0&-1&0&-1&-1&-1&-1&0&-1\\
            1&0&-1&-1&-2&1&1&-1&0&0&1&0&0&0&-1&0&0&-1&-1&1&-1&0&0&0&-1&0&0\\
            -1&-2&0&1&-1&0&2&0&1&1&1&0&0&0&0&0&1&0&0&1&2&2&1&3&2&1&2\\
            0&-1&0&0&-1&0&1&-1&1&0&0&0&0&0&-1&0&0&-1&-1&0&1&1&0&1&0&1&1\\
            1&0&1&-1&-1&0&-1&-1&-1&-2&-2&0&0&-2&-1&-1&-1&-1&-2&-1&0&0&-1&-1&-1&0&0\\
            1&0&0&-2&-1&1&0&-1&-1&-1&0&0&-1&-1&-1&-1&-1&-1&-2&0&0&-1&-1&-1&-1&0&-1\\
            0&1&0&-1&1&1&-1&0&-1&0&1&0&-1&1&0&0&-1&0&0&0&0&-2&0&-1&0&0&-1\\
            0&1&0&1&0&0&-1&1&0&-1&-1&0&1&0&0&1&0&1&1&0&-1&0&0&-1&0&0&0\\
            1&0&0&-1&1&1&-1&1&0&0&0&1&0&-1&0&0&0&0&-1&0&-1&-1&-1&-1&-1&0&-1\\
            -1&0&0&2&0&-2&0&0&0&2&0&-1&1&1&1&0&1&1&2&0&0&1&2&1&1&-1&1   
        \end{matrix}}\right)\]
        \[\sigma_{2}=\left(\scalemath{0.7}{\begin{matrix}
            2&2&1&-2&0&-1&-1&-2&-2&-1&-1&-1&-1&-1&-1&-2&-2&-2&-2&-1&-1&-2&-2&-2&-3&-2&-3\\
            2&2&2&-1&-1&0&-2&-2&-3&-2&-2&-1&-2&-2&0&-3&-2&-3&-3&-2&0&-2&-2&-2&-2&-2&-2\\
            1&1&1&0&0&0&-1&0&-1&0&-1&0&0&-1&0&-1&0&-1&-1&0&-1&0&-1&0&-1&-1&-1\\
            -1&-1&-1&0&0&-1&2&0&2&-1&0&-1&1&1&0&1&0&1&2&0&-1&1&1&-1&0&0&0\\
            -1&0&-1&0&1&0&2&0&2&0&1&0&0&1&1&1&1&1&2&1&0&0&-1&0&0&1&-1\\
            0&-1&0&0&2&1&1&1&2&1&2&1&-1&0&0&0&0&0&0&1&1&0&-1&1&1&2&0\\
            0&0&1&0&0&0&1&-1&0&0&0&0&-1&-1&0&-1&0&-1&-1&0&1&1&-2&1&-1&1&0\\
            1&1&1&-1&-2&1&-1&-1&-3&-1&-1&0&0&-2&0&-1&0&-1&-2&0&0&0&-1&0&-1&-1&0\\
            0&0&0&1&1&1&-1&2&0&1&1&1&0&0&0&0&0&1&0&0&1&0&1&1&2&0&1\\
            0&0&1&1&0&1&-1&0&0&1&0&1&-1&0&0&0&0&-1&-1&0&1&0&-1&1&0&2&1\\
            1&1&0&-1&-1&-2&-2&-1&-2&-2&-3&-2&1&0&0&-1&-1&-1&0&-2&-3&-1&1&-3&-2&-4&-2\\
            -1&-2&0&1&-1&-1&1&0&1&0&0&0&1&1&-1&1&0&1&1&0&1&2&3&1&2&1&3\\
            0&0&1&0&0&-2&-1&-1&-1&-1&-2&-1&0&0&0&-1&-1&-1&0&-2&-1&1&1&-1&-1&-2&0\\
            0&1&1&-1&0&-1&0&-1&-2&-1&-1&-1&-1&-1&0&-2&-1&-1&-1&-1&0&1&-1&0&-1&-2&-1\\
            -1&1&0&1&0&-1&-1&0&-1&0&-1&-1&0&1&1&0&0&0&1&-1&0&0&1&0&0&-1&0\\
            0&1&-1&1&0&0&0&0&1&1&0&0&0&1&1&1&1&0&1&1&-1&-1&-1&0&-1&1&-1\\
            1&1&0&-2&0&1&0&0&-1&-1&0&0&0&-2&0&-1&0&0&-1&1&-1&-1&-2&-1&-1&-1&-2\\
            0&1&0&0&2&0&0&1&0&1&1&0&-1&0&1&-1&0&0&0&0&0&-1&-1&0&0&-1&-2\\
            1&-2&0&0&-1&1&0&0&1&1&1&1&1&0&-2&1&0&0&-1&1&1&0&1&1&1&2&2\\
            -3&-2&-2&2&-1&1&3&2&3&1&2&2&2&2&0&4&3&4&3&2&2&3&3&3&4&3&4\\
            -2&-1&-1&1&-1&1&0&1&1&1&1&1&2&2&0&3&2&2&2&2&1&1&3&2&3&2&3\\
            -1&1&0&0&-1&-1&-2&-1&-1&-1&-2&-1&1&1&0&1&0&0&1&0&-1&-1&1&-1&0&0&0\\
            0&0&0&-1&0&1&-1&0&0&-1&0&0&1&0&0&1&0&0&0&1&0&-2&1&-1&1&1&0\\
            2&0&1&-1&2&0&0&0&0&0&1&0&-2&-1&0&-2&-2&-2&-2&-1&0&-2&-1&-1&-1&-1&-2\\
            1&1&1&0&0&-2&-1&-2&-1&0&-2&-1&-1&-1&0&-2&-1&-2&-1&-2&-1&0&-2&-1&-3&-1&-1\\
            2&2&2&-3&-1&-1&-3&-3&-4&-3&-3&-2&-1&-2&-1&-3&-3&-3&-3&-2&-2&-3&-1&-4&-3&-3&-3\\
            0&0&-1&1&0&1&3&1&1&1&2&1&-1&0&0&0&1&1&0&0&2&2&-2&2&0&0&0
        \end{matrix}}
        \right)\]
        \caption{Generators $\sigma_{1}$ and $\sigma_{2}$ of $\mathcal{F}_{4,\mathrm{E}}(\Z)$ as $27\times 27$ matrices in the basis $\mathcal{B}$ of $\jord_{\Z}$}
        \label{Generators2ndModel}
\end{figure}
\pagebreak
\begin{table}[H]
    \renewcommand{\arraystretch}{1.37}
    \scalebox{0.88}{
    \centering
    \begin{tabular}{|c|c|c||c|c|c||c|c|c|}
        \hline
        $s$ & $\mathrm{o}(c_{s})$ & $\mathrm{i}(c_{s})$ &$s$ & $\mathrm{o}(c_{s})$ & $\mathrm{i}(c_{s})$ &$s$ & $\mathrm{o}(c_{s})$ & $\mathrm{i}(c_{s})$ \\ \hline
        (1,0,0,0,0)&  1&   (27,351,2925,52) &             (2,1,1,0,1)&  9  &(3,3,0,1)        &            (4,4,2,0,1)&  20& (4,3,-8,0)            \\ \hline
        (0,0,0,0,1)&  2&   (-5,-1,45,20)    &             (0,1,0,1,2)&  10 &(-2,1,0,6)       &            (7,0,1,1,3)&  20& (4,9,16,4)            \\ \hline
        (0,1,0,0,0)&  2&   (3,-9,-35,-4)    &             (0,2,0,1,1)&  10 &(0,-1,0,0)       &            (2,1,3,1,2)&  21& (0,0,2,0)            \\ \hline
        (0,0,1,0,0)&  3&   (0,0,9,-2)       &             (4,2,0,0,1)&  10 &(10,49,160,10)   &            (4,2,1,2,1)&  21& (2,1,-1,0)             \\ \hline
        (1,0,0,0,1)&  3&   (0,0,9,7)        &             (0,0,0,1,4)&  12 &(-4,0,21,15)     &            (0,4,0,1,6)&  24& (-2,0,3,7)          \\ \hline
        (1,1,0,0,0)&  3&   (9,36,90,7)      &             (0,1,0,2,1)&  12 &(-1,2,-2,1)      &            (0,6,0,1,4)&  24& (0,-2,-1,1)           \\ \hline
        (0,0,0,1,0)&  4&   (-1,3,-3,0)      &             (0,2,0,1,2)&  12 &(-1,0,0,3)       &            (1,2,3,2,1)&  24& (0,0,3,-1)          \\ \hline
        (0,1,0,0,1)&  4&   (-1,-1,1,4)      &             (0,4,0,1,0)&  12 &(2,-6,-15,-3)    &            (2,4,2,1,2)&  24& (1,-2,-2,-1)             \\ \hline
        (1,0,1,0,0)&  4&   (3,3,1,0)        &             (1,0,3,0,1)&  12 &(0,0,5,-1)       &            (3,1,3,1,3)&  24& (0,0,1,1)        \\ \hline
        (2,0,0,0,1)&  4&   (7,27,77,8)      &             (1,1,1,1,1)&  12 &(0,0,1,0)        &            (3,5,1,1,2)&  24& (2,-2,-7,-1)         \\ \hline
        (2,1,0,0,0)&  4&   (15,111,545,20)  &             (1,3,1,0,1)&  12 &(2,-4,-11,-2)    &            (4,0,2,1,5)&  24& (-1,0,2,5)                \\ \hline 
        (1,1,0,0,1)&  5&   (2,1,0,2)        &             (1,4,1,0,0)&  12 &(3,-6,-26,-3)    &            (4,2,2,1,3)&  24& (1,0,0,1)           \\ \hline
        (0,0,0,1,1)&  6&   (-2,2,-3,5)      &             (2,0,0,1,3)&  12 &(-2,0,5,8)       &            (4,2,4,1,0)&  24& (2,0,-1,-1)          \\ \hline
        (0,1,0,0,2)&  6&   (-3,0,10,11)     &             (2,0,2,1,0)&  12 &(1,0,2,-1)       &            (6,2,0,3,1)&  24& (3,4,2,1)           \\ \hline
        (0,1,0,1,0)&  6&   (0,0,1,-1)       &             (2,1,0,1,2)&  12 &(0,0,1,3)        &            (6,2,4,0,1)&  24& (4,6,3,1)        \\ \hline
        (0,2,0,0,1)&  6&   (1,-4,-6,-1)     &             (2,2,0,1,1)&  12 &(2,0,-3,0)       &            (7,2,1,1,3)&  24& (4,8,11,3)           \\ \hline
        (1,0,1,0,1)&  6&   (0,0,1,2)        &             (2,4,0,0,1)&  12 &(4,0,-19,-1)     &            (2,4,2,1,4)&  28& (0,-1,0,1)          \\ \hline
        (1,1,1,0,0)&  6&   (3,0,-8,-1)      &             (3,0,1,1,1)&  12 &(2,2,1,1)        &            (3,4,1,3,1)&  28& (1,-1,-1,-1)         \\ \hline
        (2,0,0,1,0)&  6&   (4,8,9,2)        &             (3,3,1,0,0)&  12 &(6,12,5,2)       &            (2,4,6,0,1)&  30& (1,-2,1,-2)        \\ \hline
        (2,1,0,0,1)&  6&   (6,18,37,5)      &             (4,0,2,0,1)&  12 &(5,12,18,3)      &            (3,6,1,1,4)&  30& (1,-2,-3,0)           \\ \hline
        (3,0,1,0,0)&  6&   (12,72,289,14)   &             (4,1,0,0,3)&  12 &(3,6,14,5)       &            (6,1,0,5,1)&  30& (1,1,0,0)             \\ \hline
        (4,0,0,0,1)&  6&   (16,128,681,23)  &             (5,0,1,1,0)&  12 &(8,32,85,7)      &            (6,4,2,2,1)&  30& (3,2,-3,0)               \\ \hline
        (4,1,0,0,0)&  6&   (21,216,1450,35) &             (6,1,0,0,2)&  12 &(11,62,238,13)   &            (8,0,2,1,6)&  30& (1,1,4,4)                 \\ \hline  
        (1,0,0,1,1)&  7&   (-1,1,-1,3)      &             (2,1,1,1,1)&  13 &(1,0,0,0)        &           (12,1,0,3,2)&  30& (7,25,60,6)         \\ \hline
        (2,1,1,0,0)&  7&   (6,15,20,3)      &             (2,2,2,0,1)&  14 &(2,-1,-4,-1)     &            (1,4,3,4,1)&  36& (0,0,2,-1)            \\ \hline
        (0,0,0,1,2)&  8&   (-3,1,5,10)      &             (4,1,0,1,2)&  14 &(3,5,7,3)        &            (2,8,2,1,4)&  36& (1,-3,-4,-1)         \\ \hline
        (0,1,0,1,1)&  8&   (-1,1,-1,2)      &             (1,0,2,1,2)&  15 &(-1,1,0,2)       &            (4,6,2,1,7)&  40& (0,-1,0,2)          \\ \hline
        (0,2,0,1,0)&  8&   (1,-3,-3,-2)     &             (4,2,1,1,0)&  15 &(5,10,9,2)       &            (8,2,6,1,3)&  40& (2,1,0,0)           \\ \hline
        (1,1,1,0,1)&  8&   (1,-1,-1,0)      &             (1,1,3,1,1)&  18 &(0,0,4,-1)       &            (1,6,5,1,5)&  42& (0,-1,1,0)          \\ \hline
        (1,2,1,0,0)&  8&   (3,-3,-17,-2)    &             (2,2,2,1,1)&  18 &(1,-1,0,-1)      &           (10,2,4,1,6)&  42& (2,2,2,2)             \\ \hline
        (2,0,0,1,1)&  8&   (1,1,1,2)        &             (4,1,0,1,4)&  18 &(0,0,4,5)        &           (1,12,7,2,3)&  60& (1,-3,-2,-2)       \\ \hline
        (2,2,0,0,1)&  8&   (5,9,5,2)        &             (6,2,2,0,1)&  18 &(7,23,48,5)      &           (6,4,6,1,12)&  60& (-1,0,1,4)         \\ \hline
        (3,1,1,0,0)&  8&   (9,39,111,8)     &             (2,4,2,1,0)&  20 &(2,-3,-8,-2)     &          (10,2,10,1,6)&  60& (1,0,1,0)             \\ \hline
        (1,1,0,1,1)&  9&   (0,0,0,1)        &             (3,0,1,3,1)&  20 &(0,1,0,0)        &          (11,12,1,3,5)&  60& (3,1,-6,0)                \\ \hline 
    \end{tabular}}
    \caption{Kac coordinates, Orders and invariants $\mathrm{i}$ (defined in \Cref{section comparison of conj classes}) of the rational torsion conjugacy classes of $\lietype{F}{4}$}
    \label{Table all rational Kac}
\end{table}
\pagebreak

\begin{table}[H]
    \centering
    \renewcommand{\arraystretch}{1.08}
    \begin{tabular}{|c|c|c||c|c|c|}
        \hline
    $s$    & $n_{1}(s)$ &$n_{2}(s)$ & $s$    & $n_{1}(s)$ &$n_{2}(s)$  \\ \hline
    (1,0,0,0,0) & 1 & 1 &(1,1,1,1,1) & 435456000 & 105670656   \\ \hline
    (0,0,0,0,1) & 723 & 819 &(1,3,1,0,1) & 101606400 & 0 \\ \hline
    (0,1,0,0,0) & 459900 & 68796 &(2,0,0,1,3) & 1612800 & 0 \\ \hline
    (0,0,1,0,0) & 6540800 & 2283008  &(2,0,2,1,0) & 24192000 & 13208832   \\ \hline
    (1,0,0,0,1) & 121920 & 139776 &(2,1,0,1,2) & 43545600 & 0 \\ \hline
    (1,1,0,0,0) & 268800 & 34944&(2,2,0,1,1) & 14515200 & 17611776 \\ \hline
    (0,0,0,1,0) & 249480 & 137592 &(2,4,0,0,1) & 4112640 & 0   \\ \hline
    (0,1,0,0,1) & 2835000 & 0 &(3,0,1,1,1) & 7257600 & 0 \\ \hline
    (1,0,1,0,0) & 14968800 & 3302208&(3,3,1,0,0) & 4838400 & 0 \\ \hline
    (2,0,0,0,1) & 23400 & 58968 &(4,0,2,0,1) & 14515200 & 4402944   \\ \hline
    (2,1,0,0,0) & 37800 & 0 &(5,0,1,1,0) & 3628800 & 0 \\ \hline
    (1,1,0,0,1) & 1741824 & 0&(2,1,1,1,1) & 0 & 48771072 \\ \hline
    (0,0,0,1,1) & 497280 & 0 &(2,2,2,0,1) & 223948800 & 11321856   \\ \hline
    (0,1,0,1,0) & 44150400 & 8805888 &(4,2,1,1,0) & 34836480 & 0 \\ \hline
    (0,2,0,0,1) & 10483200 & 2201472   &(1,1,3,1,1) & 232243200 & 0 \\ \hline
    (1,0,1,0,1) & 74995200 & 17611776 &(2,2,2,1,1) & 154828800 & 105670656   \\ \hline
    (1,1,1,0,0) & 67737600 & 8805888 &(6,2,2,0,1) & 19353600 & 0 \\ \hline
    (2,0,0,1,0) & 1881600 & 2935296   &(2,4,2,1,0) & 87091200 & 0 \\ \hline
    (2,1,0,0,1) & 604800 & 0 &(4,4,2,0,1) & 52254720 & 0   \\ \hline
    (3,0,1,0,0) & 806400 & 0 &(2,1,3,1,2) & 199065600 & 30191616 \\ \hline
    (4,0,0,0,1) & 6720 & 0   &(4,2,1,2,1) & 0 & 60383232 \\ \hline
    (1,0,0,1,1) & 0 & 4313088 &(0,4,0,1,6) & 7257600 & 0   \\ \hline
    (2,1,1,0,0) & 24883200 & 539136 &(0,6,0,1,4) & 21772800 & 0 \\ \hline
    (0,0,0,1,2) & 272160 & 0   &(1,2,3,2,1) & 174182400 & 0 \\ \hline
    (0,1,0,1,1) & 10886400 & 0 &(2,4,2,1,2) & 174182400 & 52835328   \\ \hline
    (0,2,0,1,0) & 22680000 & 6604416 &(3,1,3,1,3) & 261273600 & 0 \\ \hline
    (1,1,1,0,1) & 342921600 & 0   &(3,5,1,1,2) & 87091200 & 0 \\ \hline
    (1,2,1,0,0) & 32659200 & 0 &(4,2,2,1,3) & 58060800 & 52835328   \\ \hline
    (2,0,0,1,1) & 5443200 & 6604416 &(4,2,4,1,0) & 65318400 & 0 \\ \hline
    (2,2,0,0,1) & 5715360 & 0   &(6,2,4,0,1) & 50803200 & 0 \\ \hline
    (3,1,1,0,0) & 5443200 & 0 &(2,4,2,1,4) & 149299200 & 22643712   \\ \hline
    (1,1,0,1,1) & 77414400 & 0 &(2,4,6,0,1) & 34836480 & 0 \\ \hline
    (2,1,1,0,1) & 19353600 & 35223552   &(6,4,2,2,1) & 139345920 & 0 \\ \hline
    (0,2,0,1,1) & 38320128 & 0 &(2,8,2,1,4) & 116121600 & 0   \\ \hline
    (4,2,0,0,1) & 1741824 & 0 &(4,6,2,1,7) & 104509440 & 0 \\ \hline
    (0,2,0,1,2) & 29030400 & 8805888   &(8,2,6,1,3) & 104509440 & 0 \\ \hline
    (0,4,0,1,0) & 10886400 & 0 &(6,4,6,1,12) & 69672960 & 0   \\ \hline
    (1,0,3,0,1) & 47174400 & 0           &     &    &        \\ \hline
    \end{tabular}
    \caption{Kac coordinates of the conjugacy classes of $\lietype{F}{4}$ whose intersections with $\mathcal{F}_{4,\mathrm{I}}(\Z)$ and $\mathcal{F}_{4,\mathrm{E}}(\Z)$ are not both empty}
    \label{KacTable}
\end{table}
\pagebreak

\begin{table}[H]
    \centering
    \renewcommand{\arraystretch}{1.03}
    \begin{tabular}{|c|c||c|c||c|c||c|c||c|c|}
    \hline
    $\lambda$ & $d(\lambda)$ & $\lambda$ & $d(\lambda)$ &$\lambda$ & $d(\lambda)$ &$\lambda$ & $d(\lambda)$ &$\lambda$ & $d(\lambda)$\\ \hline
    (0,0,0,2)&1&(0,0,1,9)&7&(0,1,1,7)&7 & (0,0,0,13) & 8 & (2,0,4,1) & 13\\ \hline
    (0,0,0,3)&1&(0,0,2,7)&6&(0,1,2,5)&9 & (0,0,1,11) & 15 & (2,1,0,6) & 16\\ \hline
    (0,0,0,4)&1&(0,0,3,5)&6&(0,1,3,3)&14 & (0,0,2,9) & 20 & (2,1,1,4) & 17\\ \hline
    (0,0,2,0)&1&(0,0,4,3)&4&(0,1,4,1)&4 & (0,0,3,7) & 27 & (2,1,2,2) & 25\\ \hline
    (0,0,0,5)&1&(0,0,5,1)&1&(0,2,0,6)&11 & (0,0,4,5) & 34 & (2,1,3,0) & 8\\ \hline
    (0,0,1,3)&1&(0,1,0,8)&2&(0,2,1,4)&9 & (0,0,5,3) & 30 & (2,2,0,3) & 4\\ \hline
    (0,0,0,6)&3&(0,1,1,6)&3&(0,2,2,2)&15 & (0,0,6,1) & 14 & (2,2,1,1) & 9\\ \hline
    (0,0,2,2)&1&(0,1,2,4)&4&(0,2,3,0)&2 & (0,1,0,10) & 11 & (2,3,0,0) & 6\\ \hline
    (0,0,0,7)&1&(0,1,3,2)&3&(0,3,0,3)&3 & (0,1,1,8) & 23 & (3,0,0,7) & 1\\ \hline
    (0,0,1,5)&1&(0,1,4,0)&1&(0,3,1,1)&3 & (0,1,2,6) & 39 & (3,0,1,5) & 9\\ \hline
    (0,0,2,3)&1&(0,2,0,5)&1&(0,4,0,0)&6 & (0,1,3,4) & 44 & (3,0,2,3) & 7\\ \hline
    (0,0,0,8)&4&(0,2,1,3)&3&(1,0,0,10)&3 & (0,1,4,2) & 37 & (3,0,3,1) & 8\\ \hline
    (0,0,1,6)&1&(0,2,2,1)&1&(1,0,1,8)&7 & (0,1,5,0) & 13 & (3,1,0,4) & 12\\ \hline
    (0,0,2,4)&1&(0,3,0,2)&2&(1,0,2,6)&10 & (0,2,0,7) & 11 & (3,1,1,2) & 7\\ \hline
    (0,0,4,0)&2&(1,0,0,9)&1&(1,0,3,4)&11 & (0,2,1,5) & 32 & (3,1,2,0) & 8\\ \hline
    (0,0,0,9)&4&(1,0,1,7)&3&(1,0,4,2)&8 & (0,2,2,3) & 36 & (4,0,0,5) & 2\\ \hline
    (0,0,1,7)&2&(1,0,2,5)&2&(1,0,5,0)&4 & (0,2,3,1) & 26 & (4,0,1,3) & 3\\ \hline
    (0,0,2,5)&1&(1,0,3,3)&3&(1,1,0,7)&2 & (0,3,0,4) & 21 & (4,0,2,1) & 2\\ \hline
    (0,0,3,3)&2&(1,0,4,1)&1&(1,1,1,5)&9 & (0,3,1,2) & 21 & (4,1,0,2) & 4\\ \hline
    (0,1,3,0)&1&(1,1,0,6)&3&(1,1,2,3)&8 & (0,3,2,0) & 14 & (4,1,1,0) & 1\\ \hline
    (0,3,0,0)&1&(1,1,1,4)&2&(1,1,3,1)&9 & (0,4,0,1) & 5 & (5,0,1,1) & 1\\ \hline
    (1,1,0,4)&1&(1,1,2,2)&4&(1,2,0,4)&8 & (1,0,0,11) &3 & (5,1,0,0) & 3\\ \hline
    (3,1,0,0)&1&(1,2,1,1)&2&(1,2,1,2)&5 & (1,0,1,9) & 13 & & \\ \hline
    (0,0,0,10)&5&(1,3,0,0)&1&(1,2,2,0)&5 & (1,0,2,7) & 20 & & \\ \hline
    (0,0,1,8)&4&(2,0,0,7)&1&(1,3,0,1)&1 & (1,0,3,5) & 32 & & \\ \hline
    (0,0,2,6)&6&(2,0,1,5)&2&(2,0,0,8)&5 & (1,0,4,3) & 26 & & \\ \hline
    (0,0,4,2)&3&(2,0,2,3)&1&(2,0,1,6)&4 & (1,0,5,1) & 21 & & \\ \hline
    (0,0,5,0)&1&(2,0,3,1)&1&(2,0,2,4)&10 & (1,1,0,8) & 18 & & \\ \hline
    (0,1,1,5)&1&(2,1,0,4)&2&(2,0,3,2)&4 & (1,1,1,6) & 27 & &\\ \hline
    (0,1,3,1)&1&(2,1,1,2)&1&(2,0,4,0)&5 & (1,1,2,4) & 46 & &\\ \hline
    (0,2,0,4)&1&(2,1,2,0)&1&(2,1,1,3)&5 & (1,1,3,2) & 31 & &\\ \hline
    (0,2,2,0)&1&(3,0,1,3)&1&(2,1,2,1)&2 & (1,1,4,0) & 20 & &\\ \hline
    (1,0,0,8)&1&(3,1,0,2)&1&(2,2,0,2)&8 & (1,2,0,5) & 10 & &\\ \hline
    (1,0,1,6)&1&(0,0,0,12)&13&(3,0,0,6)&4 & (1,2,1,3) & 28 & &\\ \hline
    (1,0,2,4)&1&(0,0,1,10)&6&(3,0,1,4)&3 & (1,2,2,1) & 16 & &\\ \hline
    (1,0,3,2)&1&(0,0,2,8)&15&(3,0,2,2)&3 & (1,3,0,2) & 18 & &\\ \hline
    (1,2,0,2)&1&(0,0,3,6)&15&(3,0,3,0)&2 & (1,3,1,0) & 2 & &\\ \hline
    (2,0,0,6)&2&(0,0,4,4)&15&(3,2,0,0)&2 & (2,0,0,9) & 4 & &\\ \hline
    (2,0,2,2)&1&(0,0,5,2)&4&(4,0,0,4)&3 & (2,0,1,7) & 12 & &\\ \hline
    (2,2,0,0)&1&(0,0,6,0)&11&(4,0,2,0)&2 & (2,0,2,5) & 16 & &\\ \hline
    (0,0,0,11)&5&(0,1,0,9)&2&(6,0,0,0)&3 & (2,0,3,3) & 21 & &\\ \hline
    \end{tabular}
    \caption{The nonzero $d(\lambda)$ for $\lambda=(\lambda_{1},\lambda_{2},\lambda_{3},\lambda_{4})$ such that $2\lambda_{1}+3\lambda_{2}+2\lambda_{3}+\lambda_{4}\leq 13$}
    \label{nonzeroAMF}
\end{table}
\pagebreak
\begin{table}[H]
    \centering
    \renewcommand{\arraystretch}{1.2}
    \begin{tabular}{|c|c|c||c|c|c||c|c|c||c|c|c|}
    \hline
    $n$ & $d_{1}(n)$ & $d_{2}(n)$ & $n$ & $d_{1}(n)$ & $d_{2}(n)$ &  $n$ & $d_{1}(n)$ & $d_{2}(n)$ &  $n$ & $d_{1}(n)$ & $d_{2}(n)$\\  \hline
    1 & 0 & 0  &            11 & 4 & 1  & 21 & 83 & 209 & 31& 4112& 24425\\ \hline     
    2 & 1 & 0  &            12 & 8 & 5  & 22 & 130 & 413 & 32& 6294& 38234\\ \hline     
    3 & 1 & 0  &            13 & 6 & 2  & 23 & 169 & 590 & 33& 8904& 54760 \\ \hline     
    4 & 1 & 0  &            14 & 12 & 8 & 24 & 280 & 1138 & 34& 13284& 82989\\ \hline     
    5 & 1 & 0  &            15 & 13 & 8  & 25 & 368 & 1629& 35& 18664& 117447\\ \hline     
    6 & 2 & 1  &            16 & 20 & 18  & 26 & 601 & 2915 & 36& 27332& 173760\\ \hline     
    7 & 1 & 0  &            17 & 22 & 22  & 27 & 835 & 4253 & 37& 38024& 242971\\ \hline     
    8 & 3 & 1  &            18 & 37 & 58  & 28 & 1323 & 7161 & 38& 54627& 351485\\ \hline     
    9 & 3 & 1  &            19 & 39 & 63  & 29 & 1868 & 10455& 39& 75354& 486013\\ \hline     
    10 & 4 & 1  &           20 & 67 & 150  & 30 & 2919 & 16962& 40& 106332& 689219\\ \hline      
    \end{tabular}
    \caption{Dimensions $d_{1}(n)=\dim \vrep{n\varpi_{4}}^{\mathcal{F}_{4,\mathrm{I}}(\Z)}$ and $d_{2}(n)=\dim \vrep{n\varpi_{4}}^{\mathcal{F}_{4,\mathrm{E}}(\Z)}$ for $n\leq 40$}
    \label{TableA1F4}
\end{table}

\begin{table}[H]
    \centering
    \renewcommand{\arraystretch}{1.2}
    \scalebox{1}{
    \begin{tabular}{|c|c|c||c|c|c|}
    \hline
    $n$ & $d_{1}(n)$ & $d_{2}(n)$ & $n$ & $d_{1}(n)$ & $d_{2}(n)$ 
    \\  \hline
    1& 0& 0&            16& 699558& 4607562             \\ \hline
    2& 1& 0&            17& 1899450& 12528178         \\ \hline
    3& 0& 0&            18& 4951537& 32636950        \\ \hline
    4& 1& 1&            19& 12298529& 81088431       \\ \hline
    5& 0& 1&            20& 29444006& 194120684    \\ \hline
    6& 4& 7&            21& 67821302& 447181025       \\ \hline
    7& 2& 14&           22& 151304284& 997568542       \\ \hline
    8& 32& 136&         23& 326873722& 2155210696      \\ \hline
    9& 84& 583&         24& 686811782& 4528418428      \\ \hline
    10& 497& 2936&      25& 1404333622& 9259307898      \\ \hline
    11& 1765& 11764&    26& 2802604042& 18478677233   \\ \hline
    12& 7111& 46299&    27& 5463354204& 36021961176   \\ \hline
    13& 24173& 159701&  28& 10425639768& 68740584631  \\ \hline
    14& 80166& 526081&  29& 19491910968& 128517811865 \\ \hline
    15& 241776& 1594526&30& 35762551274& 235797459916   \\ \hline
\end{tabular}}
    \caption{Dimensions $d_{1}(n)=\dim \vrep{n\varpi_{3}}^{\mathcal{F}_{4,\mathrm{I}}(\Z)}$ and $d_{2}(n)=\dim \vrep{n\varpi_{3}}^{\mathcal{F}_{4,\mathrm{E}}(\Z)}$ for $n\leq 30$}
    \label{TableG2F4}
\end{table}
\pagebreak
\begin{table}[H]
    \centering
    \renewcommand{\arraystretch}{1.425}
    \begin{tabular}{|c|c|c|c|}
        \hline
        $\mathrm{w}(\lambda)$ & $\lambda$ & $\dim\mathcal{A}_{\vrep{\lambda}}(\grpF)$ & $\Psi_{\lambda}(\grpF)$ \\ \hline
        \multirow{2}*{16} & \multirow{2}*{(0,0,0,0)} & \multirow{2}*{2}& $[9]\oplus[17]$\\ \cline{4-4}
        ~& ~ & ~ & $\Delta_{11}[6]\oplus[5]\oplus[9]$ \\ \hline
        20 &(0,0,0,2) & 1& $\Delta_{15}[6]\oplus[5]\oplus[9]$ \\ \hline
        22 &(0,0,0,3) & 1& $\Delta_{17}[6]\oplus[5]\oplus[9]$ \\ \hline
        \multirow{2}*{24} &(0,0,0,4) & 1& $\Delta_{19}[6]\oplus[5]\oplus[9]$ \\ \cline{2-4}
        ~&(0,0,2,0) & 1& $\sym^{2}\Delta_{11}[3]\oplus \Delta_{11}[4]\oplus \Delta_{11}[2]\oplus [5]$\\ \hline
        \multirow{2}*{26} &(0,0,0,5) & 1& $\Delta_{21}[6]\oplus[5]\oplus[9]$ \\ \cline{2-4}
        ~ &(0,0,1,3) & 1& $\Delta_{24,16,8,0}[3]\oplus [5]$\\ \hline
        \multirow{3}*{28}&\multirow{2}*{(0,0,0,6)} & \multirow{2}*{3} & $\Delta_{23}^{(2)}[6]\oplus[5]\oplus[9]$ \\ \cline{4-4}
        ~&~ & ~ & $\Delta_{26,20,6,0}[3]\oplus[5]$ \\ \cline{2-4}
        ~&(0,0,2,2) & 1& $\Delta_{26,16,10,0}[3]\oplus[5]$\\ \hline
        \multirow{3}*{30}&(0,0,0,7) & 1& $\Delta_{25}[6]\oplus[5]\oplus[9]$ \\ \cline{2-4}
        ~&(0,0,1,5) & 1& $\Delta_{28,20,8,0}[3]\oplus[5]$ \\ \cline{2-4}
        ~&(0,0,2,3) & 1& $\Delta_{28,18,10,0}[3]\oplus[5]$  \\ \hline
        \multirow{6}*{32}&\multirow{2}*{(0,0,0,8)} & \multirow{2}*{4} & $\Delta_{27}^{(2)}[6]\oplus[5]\oplus[9]$  \\ \cline{4-4}
        ~&~ & ~ & $\Delta^{(2)}_{30,24,6,0}[3]\oplus[5]$ \\ \cline{2-4}
        ~&(0,0,1,6) & 1& $\Delta_{30,22,8,0}[3]\oplus[5]$  \\ \cline{2-4}
        ~&(0,0,2,4) & 1& $\Delta_{30,20,10,0}[3]\oplus[5]$  \\ \cline{2-4}
        ~&\multirow{2}*{(0,0,4,0)} & \multirow{2}*{2} & $\sym^{2}\Delta_{15}[3]\oplus\Delta_{15}[4]\oplus\Delta_{15}[2]\oplus[5]$ \\ \cline{4-4}
        ~&~ & ~ & $\Delta_{30,16,14,0}[3]\oplus[5]$ \\ \hline
        \multirow{9}*{34}&\multirow{2}*{(0,0,0,9)} & \multirow{2}*{4} & $\Delta_{29}^{(2)}[6]\oplus[5]\oplus[9]$ \\ \cline{4-4}
        ~&~ & ~ & $\Delta^{(2)}_{32,26,6,0}[3]\oplus[5]$ \\ \cline{2-4}
        ~&(0,0,1,7) & 2& $\Delta^{(2)}_{32,24,8,0}[3]\oplus[5]$  \\ \cline{2-4}
        ~&(0,0,2,5) & 1& $\Delta_{32,22,10,0}[3]\oplus[5]$  \\ \cline{2-4}
        ~&(0,0,3,3) & 2& $\Delta^{(2)}_{32,20,12,0}[3]\oplus[5]$  \\ \cline{2-4}
        ~&(0,1,3,0) & 1& $\Delta_{32,16,14,6,0}\oplus\spin\Delta_{32,16,14,6,0}\oplus[1]$  \\ \cline{2-4}
        ~&(0,3,0,0) & 1& $\sym^{3}\Delta_{11}[2]\oplus\sym^{2}\Delta_{11}[3]\oplus\Delta_{11}[4]\oplus[1]$  \\ \cline{2-4}
        ~&(1,1,0,4) & 1& $\Delta_{30,20,10,8,0}\oplus\spin\Delta_{30,20,10,8,0}\oplus[1]$  \\ \cline{2-4}
        ~&(3,1,0,0) & 1& $\wedge^{*}\Delta_{19,7}\oplus\left(\Delta_{19,7}\otimes\Delta_{15}\right)\oplus\Delta_{19,7}[2]\oplus\Delta_{15}[2]\oplus[1]$  \\ \hline
    \end{tabular}
    \caption{Elements of nonempty $\Psi_{\lambda}(\grpF)$ for the weights $\lambda$ such that $\mathrm{w}(\lambda)\leq 34$}
    \label{table packet 34}
\end{table}
\pagebreak
\begin{table}[H]
    \centering
    \renewcommand{\arraystretch}{1.8}
    \begin{tabular}{|c|c|c|} \hline
        $\lambda$ & $\dim\mathcal{A}_{\vrep{\lambda}}(\grpF)$ & $\Psi_{\lambda}(\grpF)$ \\ \hline
        \multirow{2}*{(0,0,0,10)} & \multirow{2}*{5}& $\Delta_{31}^{(2)}[6]\oplus[5]\oplus[9]$ \\ \cline{3-3}
        ~ & ~ & $\Delta^{(3)}_{34,28,6,0}[3]\oplus[5]$ \\ \hline
        \multirow{2}*{(0,0,1,8)} & \multirow{2}*{4}&  $\wedge^{*}\Delta_{21,13}\oplus\left(\Delta_{21,13}\otimes\Delta_{15}\right)\oplus\left(\Delta_{21,13}\otimes\Delta_{11}\right)\oplus\left(\Delta_{15}\otimes\Delta_{11}\right)\oplus[1]$ \\ \cline{3-3}
        ~ & ~ & $\Delta^{(3)}_{34,26,8,0}[3]\oplus[5]$ \\ \hline
        \multirow{2}*{(0,0,2,6)} & \multirow{2}*{6}&  $\Delta^{(5)}_{34,24,10,0}[3]\oplus[5]$ \\ \cline{3-3}
        ~ & ~ & $\Delta_{34,24,10,4,0}\oplus\spin\Delta_{34,24,10,4,0}\oplus[1]$ \\ \hline
        \multirow{2}*{(0,0,4,2)} & \multirow{2}*{3}&  $\Delta^{(2)}_{34,20,14,0}[3]\oplus[5]$ \\ \cline{3-3}
        ~ & ~ & $\Delta_{34,20,14,4,0}\oplus\spin\Delta_{32,20,14,4,0}\oplus[1]$ \\ \hline
        (0,0,5,0) & 1&  $\sym^{2}\Delta_{17}[3]\oplus\Delta_{17}[4]\oplus\Delta_{17}[2]\oplus[5]$ \\ \hline
        (0,1,1,5) & 1&  $\Delta_{34,22,10,6,0}\oplus\spin\Delta_{34,22,10,6,0}\oplus[1]$ \\ \hline
        (0,1,3,1) & 1&  $\Delta_{34,18,14,6,0}\oplus\spin\Delta_{34,18,14,6,0}\oplus[1]$ \\ \hline
        (0,2,0,4) & 1&  $\Delta_{34,20,10,8,0}\oplus\spin\Delta_{32,16,14,6,0}\oplus[1]$ \\ \hline
        (0,2,2,0) & 1&  $\wedge^{*}\Delta_{21,13}\oplus\left(\Delta_{21,13}\otimes\Delta_{15}\right)\oplus\Delta_{21,13}[2]\oplus\Delta_{15}[2]\oplus[1]$ \\ \hline
        (1,0,0,8) & 1&  $\Delta_{32,26,8,6,0}\oplus\spin\Delta_{32,26,8,6,0}\oplus[1]$ \\ \hline
        (1,0,1,6) & 1&  $\Delta_{32,24,10,6,0}\oplus\spin\Delta_{32,24,10,6,0}\oplus[1]$ \\ \hline
        (1,0,2,4) & 1&  $\Delta_{32,22,12,6,0}\oplus\spin\Delta_{32,22,12,6,0}\oplus[1]$ \\ \hline
        (1,0,3,2) & 1&  $\Delta_{32,20,14,6,0}\oplus\spin\Delta_{32,20,14,6,0}\oplus[1]$ \\ \hline
        (1,2,0,2) & 1&  $\psi_{0}$ \\ \hline
        (2,0,0,6) & 2&  $\Delta^{(2)}_{30,24,10,8,0}\oplus\spin\Delta_{30,24,10,8,0}\oplus[1]$ \\ \hline
        (2,0,2,2) & 1&  $\Delta_{30,20,14,8,0}\oplus\spin\Delta_{30,20,14,8,0}\oplus[1]$ \\ \hline
        (2,2,0,0) & 1&  $\wedge^{*}\Delta_{21,9}\oplus\left(\Delta_{21,9}\otimes\Delta_{15}\right)\oplus\Delta_{21,9}[2]\oplus\Delta_{15}[2]\oplus[1]$ \\ \hline
    \end{tabular}
    \caption{Elements of nonempty $\Psi_{\lambda}(\grpF)$ for the weights $\lambda$ such that $\mathrm{w}(\lambda)=36$}
    \label{table packet 36}
\end{table}

\pagebreak
\begin{table}[H]
    \centering
    \renewcommand{\arraystretch}{1.2}
    \begin{tabular}{|c|c||c|c||c|c||c|c||c|c|}
    \hline 
    $\lambda$& $\mathrm{F}_{4}(\lambda)$ &$\lambda$& $\mathrm{F}_{4}(\lambda)$ &$\lambda$& $\mathrm{F}_{4}(\lambda)$ &$\lambda$& $\mathrm{F}_{4}(\lambda)$ & $\lambda$ & $\mathrm{F}_{4}(\lambda)$ \\ \hline
    (1,2,0,2)&1&(1,2,2,0)&5&(1,1,3,2)&22&(0,1,3,5)&70&(2,0,2,6)&28 \\ \hline
    (0,1,2,4)&2&(2,0,2,4)&2&(1,1,4,0)&11&(0,1,4,3)&68&(2,0,3,4)&32 \\ \hline
    (0,1,4,0)&1&(2,0,3,2)&2&(1,2,0,5)&7&(0,1,5,1)&49&(2,0,4,2)&35 \\ \hline
    (0,2,1,3)&2&(2,1,1,3)&3&(1,2,1,3)&22&(0,2,0,8)&31&(2,0,5,0)&12 \\ \hline
    (0,3,0,2)&2&(2,1,2,1)&2&(1,2,2,1)&13&(0,2,1,6)&61&(2,1,0,7)&10 \\ \hline
    (1,0,3,3)&1&(2,2,0,2)&4&(1,3,0,2)&12&(0,2,2,4)&92&(2,1,1,5)&42 \\ \hline
    (1,1,1,4)&1&(3,0,0,6)&1&(1,3,1,0)&2&(0,2,3,2)&74&(2,1,2,3)&46 \\ \hline
    (1,1,2,2)&2&(3,0,2,2)&2&(2,0,1,7)&2&(0,2,4,0)&35&(2,1,3,1)&41 \\ \hline
    (1,2,1,1)&2&(3,2,0,0)&1&(2,0,2,5)&3&(0,3,0,5)&26&(2,2,0,4)&39 \\ \hline
    (2,1,0,4)&2&(0,0,3,7)&3&(2,0,3,3)&9&(0,3,1,3)&61&(2,2,1,2)&34 \\ \hline
    (2,1,2,0)&1&(0,0,4,5)&6&(2,0,4,1)&5&(0,3,2,1)&40&(2,2,2,0)&24 \\ \hline
    (0,0,3,6)&1&(0,0,5,3)&8&(2,1,0,6)&11&(0,4,0,2)&28&(2,3,0,1)&2 \\ \hline
    (0,0,4,4)&1&(0,0,6,1)&4&(2,1,1,4)&9&(0,4,1,0)&8&(3,0,0,8)&5 \\ \hline
    (0,0,5,2)&1&(0,1,0,10)&2&(2,1,2,2)&21&(1,0,0,12)&1&(3,0,1,6)&6 \\ \hline
    (0,0,6,0)&1&(0,1,1,8)&6&(2,1,3,0)&2&(1,0,1,10)&4&(3,0,2,4)&21 \\ \hline
    (0,1,1,7)&1&(0,1,2,6)&19&(2,2,0,3)&1&(1,0,2,8)&23&(3,0,3,2)&13 \\ \hline
    (0,1,2,5)&3&(0,1,3,4)&18&(2,2,1,1)&8&(1,0,3,6)&36&(3,0,4,0)&14 \\ \hline
    (0,1,3,3)&6&(0,1,4,2)&25&(2,3,0,0)&4&(1,0,4,4)&50&(3,1,0,5)&2 \\ \hline
    (0,1,4,1)&2&(0,1,5,0)&4&(3,0,1,5)&2&(1,0,5,2)&34&(3,1,1,3)&21 \\ \hline
    (0,2,0,6)&4&(0,2,0,7)&2&(3,0,2,3)&2&(1,0,6,0)&24&(3,1,2,1)&13 \\ \hline
    (0,2,1,4)&4&(0,2,1,5)&20&(3,0,3,1)&3&(1,1,0,9)&6&(3,2,0,2)&20 \\ \hline
    (0,2,2,2)&8&(0,2,2,3)&21&(3,1,0,4)&4&(1,1,1,7)&50&(3,2,1,0)&2 \\ \hline
    (0,2,3,0)&2&(0,2,3,1)&19&(3,1,1,2)&5&(1,1,2,5)&69&(4,0,0,6)&2 \\ \hline
    (0,3,0,3)&3&(0,3,0,4)&19&(3,1,2,0)&3&(1,1,3,3)&86&(4,0,1,4)&3 \\ \hline
    (0,3,1,1)&2&(0,3,1,2)&10&(4,1,0,2)&3&(1,1,4,1)&57&(4,0,2,2)&7 \\ \hline
    (0,4,0,0)&1&(0,3,2,0)&13&(0,0,2,10)&4&(1,2,0,6)&56&(4,0,3,0)&1 \\ \hline
    (1,0,2,6)&2&(0,4,0,1)&2&(0,0,3,8)&13&(1,2,1,4)&72&(4,1,1,1)&6 \\ \hline
    (1,0,3,4)&2&(1,0,2,7)&4&(0,0,4,6)&27&(1,2,2,2)&93&(4,2,0,0)&1 \\ \hline
    (1,0,4,2)&4&(1,0,3,5)&11&(0,0,5,4)&26&(1,2,3,0)&17&(5,0,0,4)&2 \\ \hline
    (1,1,1,5)&4&(1,0,4,3)&9&(0,0,6,2)&24&(1,3,0,3)&18&(5,0,2,0)&2 \\ \hline
    (1,1,2,3)&4&(1,0,5,1)&11&(0,0,7,0)&8&(1,3,1,1)&34& & \\ \hline
    (1,1,3,1)&6&(1,1,0,8)&7&(0,1,0,11)&1&(1,4,0,0)&9& & \\ \hline
    (1,2,0,4)&7&(1,1,1,6)&15&(0,1,1,9)&21&(2,0,0,10)&3& & \\ \hline
    (1,2,1,2)&3&(1,1,2,4)&27&(0,1,2,7)&44&(2,0,1,8)&9& &  \\ \hline
    \end{tabular}
    \caption{The nonzero $\mathrm{F}_{4}(\lambda)$ for the weights $\lambda$ such that $\mathrm{w}(\lambda)\leq 44$ }
    \label{TableMult}
\end{table}
\pagebreak

\printbibliography

@book{ATLAS,
  shorthand = {ATLAS},
  AUTHOR = {Conway, J. H. and Curtis, R. T. and Norton, S. P. and Parker,
              R. A. and Wilson, R. A.},
     TITLE = {{$\mathbb{ATLAS}$} of finite groups},
      NOTE = {Maximal subgroups and ordinary characters for simple groups,
              With computational assistance from J. G. Thackray},
 PUBLISHER = {Oxford University Press, Eynsham},
      YEAR = {1985},
     PAGES = {xxxiv+252},
}

@book{chandra1968automorphic,
  AUTHOR = {Harish-Chandra},
     TITLE = {Automorphic forms on semisimple {L}ie groups},
    SERIES = {Lecture Notes in Mathematics},
    VOLUME = {No. 62},
      NOTE = {Notes by J. G. M. Mars},
 PUBLISHER = {Springer-Verlag, Berlin-New York},
      YEAR = {1968},
     PAGES = {x+138},
}

@book{ChenevierLannes,
  AUTHOR = {Chenevier, Ga\"etan and Lannes, Jean},
     TITLE = {Automorphic forms and even unimodular lattices},
    SERIES = {Ergebnisse der Mathematik und ihrer Grenzgebiete. 3. Folge. A
              Series of Modern Surveys in Mathematics [Results in
              Mathematics and Related Areas. 3rd Series. A Series of Modern
              Surveys in Mathematics]},
    VOLUME = {69},
      PUBLISHER = {Springer, Cham},
      YEAR = {2019},
     PAGES = {xxi+417},
}

@book{ChenevierRenard,
  AUTHOR = {Chenevier, Ga\"etan and Renard, David},
     TITLE = {Level one algebraic cusp forms of classical groups of small
              rank},
   JOURNAL = {Mem. Amer. Math. Soc.},
  FJOURNAL = {Memoirs of the American Mathematical Society},
    VOLUME = {237},
      YEAR = {2015},
    NUMBER = {1121},
     PAGES = {v+122},
}

@article{ConjEleConj,
  AUTHOR = {Fang, Yingjue and Han, Gang and Sun, Binyong},
     TITLE = {Conjugacy and element-conjugacy of homomorphisms of compact
              {L}ie groups},
   JOURNAL = {Pacific J. Math.},
  FJOURNAL = {Pacific Journal of Mathematics},
    VOLUME = {283},
      YEAR = {2016},
    NUMBER = {1},
     PAGES = {75--83},
}

@book{ConwayOct,
  AUTHOR = {Conway, John H. and Smith, Derek A.},
     TITLE = {On quaternions and octonions: their geometry, arithmetic, and
              symmetry},
 PUBLISHER = {A K Peters, Ltd., Natick, MA},
      YEAR = {2003},
     PAGES = {xii+159},
}

@article{Dynkin1952,
      AUTHOR = {Dynkin, E. B.},
     TITLE = {Semisimple subalgebras of semisimple {L}ie algebras},
   JOURNAL = {Mat. Sbornik N.S.},
  FJOURNAL = {Mat. Sbornik N.S.},
    VOLUME = {30/72},
      YEAR = {1952},
     PAGES = {349--462 (3 plates)},
}

@article{G96,
  AUTHOR = {Gross, Benedict H.},
     TITLE = {Groups over {${\bf Z}$}},
   JOURNAL = {Invent. Math.},
  FJOURNAL = {Inventiones Mathematicae},
    VOLUME = {124},
      YEAR = {1996},
    NUMBER = {1-3},
     PAGES = {263--279},
}

@manual{GAP4,
  shorthand    = {GAP},
  key          = {GAP},
  organization = {The GAP~Group},
  title        = {{GAP -- Groups, Algorithms, and Programming,
                  Version 4.11.1}},
  year         = {2021},
  url          = {https://www.gap-system.org}
}

@book{Lie,
  AUTHOR = {Bourbaki, Nicolas},
     TITLE = {Lie groups and {L}ie algebras. {C}hapters 4--6},
    SERIES = {Elements of Mathematics (Berlin)},
 PUBLISHER = {Springer-Verlag, Berlin},
      YEAR = {2002},
     PAGES = {xii+300},
}

@book{NilOrbit,
   AUTHOR = {Collingwood, David H. and McGovern, William M.},
     TITLE = {Nilpotent orbits in semisimple {L}ie algebras},
    SERIES = {Van Nostrand Reinhold Mathematics Series},
 PUBLISHER = {Van Nostrand Reinhold Co., New York},
      YEAR = {1993},
     PAGES = {xiv+186},
}

@incollection{NonRed,
   AUTHOR = {Conrad, Brian},
     TITLE = {Non-split reductive groups over {${\bf Z}$}},
 BOOKTITLE = {Autours des sch\'emas en groupes. {V}ol. {II}},
    SERIES = {Panor. Synth\`eses},
    VOLUME = {46},
     PAGES = {193--253},
 PUBLISHER = {Soc. Math. France, Paris},
      YEAR = {2015},}

@book{OctSpr,
 AUTHOR = {Springer, Tonny A. and Veldkamp, Ferdinand D.},
     TITLE = {Octonions, {J}ordan algebras and exceptional groups},
    SERIES = {Springer Monographs in Mathematics},
 PUBLISHER = {Springer-Verlag, Berlin},
      YEAR = {2000},
     PAGES = {viii+208},
}

@manual{PARI2,
  shorthand    = {PARI/GP},
  organization = {{The PARI~Group}},
  title        = {{PARI/GP version \texttt{2.13.2}}},
  year         = {2021},
  address      = {Univ. Bordeaux},
  note         = {available from \url{http://pari.math.u-bordeaux.fr/}}
}

@article{Reeder2010TorsionAO,
  AUTHOR = {Reeder, Mark},
     TITLE = {Torsion automorphisms of simple {L}ie algebras},
   JOURNAL = {Enseign. Math. (2)},
  FJOURNAL = {L'Enseignement Math\'ematique. Revue Internationale. 2e
              S\'erie},
    VOLUME = {56},
      YEAR = {2010},
    NUMBER = {1-2},
     PAGES = {3--47},
}

@book{WeilGroupeTopo,
   AUTHOR = {Weil, Andr\'e},
     TITLE = {L'int\'egration dans les groupes topologiques et ses
              applications},
    SERIES = {Actualit\'es Scientifiques et Industrielles [Current
              Scientific and Industrial Topics]},
    VOLUME = {No. 869},
 PUBLISHER = {Hermann \& Cie, Paris},
      YEAR = {1940},
     PAGES = {158},
}

@book{arthur2013endoscopic,
  AUTHOR = {Arthur, James},
     TITLE = {The endoscopic classification of representations},
    SERIES = {American Mathematical Society Colloquium Publications},
    VOLUME = {61},
      NOTE = {Orthogonal and symplectic groups},
 PUBLISHER = {American Mathematical Society, Providence, RI},
      YEAR = {2013},
     PAGES = {xviii+590},
}

@article{NgoAFL,
     author = {Ng\^o, Bao Ch\^au},
     title = {Le lemme fondamental pour les alg\`ebres de {Lie}},
     journal = {Publications Math\'ematiques de l'IH\'ES},
     publisher = {Springer-Verlag},
     volume = {111},
     year = {2010},
}

@article{Chenevier2020DiscreteSM,
  AUTHOR = {Chenevier, Ga\"etan and Ta\"ibi, Olivier},
     TITLE = {Discrete series multiplicities for classical groups over {$\bf
              Z$} and level 1 algebraic cusp forms},
   JOURNAL = {Publ. Math. Inst. Hautes \'Etudes Sci.},
  FJOURNAL = {Publications Math\'ematiques. Institut de Hautes \'Etudes
              Scientifiques},
    VOLUME = {131},
      YEAR = {2020},
     PAGES = {261--323},
}

@article{AlgModForm,
  AUTHOR = {Gross, Benedict H.},
     TITLE = {Algebraic modular forms},
   JOURNAL = {Israel J. Math.},
  FJOURNAL = {Israel Journal of Mathematics},
    VOLUME = {113},
      YEAR = {1999},
     PAGES = {61--93},
}

@article{Leech,
    AUTHOR = {Elkies, Noam D. and Gross, Benedict H.},
     TITLE = {The exceptional cone and the {L}eech lattice},
   JOURNAL = {Internat. Math. Res. Notices},
  FJOURNAL = {International Mathematics Research Notices},
      YEAR = {1996},
    NUMBER = {14},
     PAGES = {665--698},
}

@incollection{Satake, 
   AUTHOR = {Gross, Benedict H.},
     TITLE = {On the {S}atake isomorphism},
 BOOKTITLE = {Galois representations in arithmetic algebraic geometry
              ({D}urham, 1996)},
    SERIES = {London Math. Soc. Lecture Note Ser.},
    VOLUME = {254},
     PAGES = {223--237},
 PUBLISHER = {Cambridge Univ. Press, Cambridge},
      YEAR = {1998},
}

@Book{Enveloping,
author = { Dixmier, Jacques},
title = { Enveloping algebras },
publisher = { North-Holland Pub. Co. },
year = { 1977 },
}

@book{Knapp86,
    AUTHOR = {Knapp, Anthony W.},
     TITLE = {Representation theory of semisimple groups},
    SERIES = {Princeton Mathematical Series},
    VOLUME = {36},
 PUBLISHER = {Princeton University Press},
   ADDRESS = {Princeton, NJ},
      YEAR = {1986},
}

@Misc{GGPS66,
 Author = {Gelfand, I. M. and Graev, M. I. and Pyatetskii-Shapiro, I. I.},
 Title = {Representation theory and automorphic function. {Translated} from the {Russian} by {K}. {A}. {Hirsch}},
 Year = {1969},
}

@article{UniqueALblock,
 author = {Jacquet, Herv\'e and Shalika, Joseph A.},
 journal = {American Journal of Mathematics},
 publisher = {Johns Hopkins University Press},
 title = {On Euler Products and the Classification of Automorphic Forms II},
 volume = {103},
 year = {1981}
}

@incollection{ArthurUnipRep,
   AUTHOR = {Arthur, James},
     TITLE = {Unipotent automorphic representations: conjectures},
      NOTE = {Orbites unipotentes et repr\'esentations, II},
   JOURNAL = {Ast\'erisque},
  FJOURNAL = {Ast\'erisque},
    NUMBER = {171-172},
      YEAR = {1989},
     PAGES = {13--71},
}

@incollection{Langlandsclassification,
   AUTHOR = {Langlands, R. P.},
     TITLE = {On the classification of irreducible representations of real
              algebraic groups},
 BOOKTITLE = {Representation theory and harmonic analysis on semisimple
              {L}ie groups},
    SERIES = {Math. Surveys Monogr.},
    VOLUME = {31},
     PAGES = {101--170},
 PUBLISHER = {Amer. Math. Soc., Providence, RI},
      YEAR = {1989},
}

@incollection{PurityClozel,
  AUTHOR = {Clozel, Laurent},
     TITLE = {Motifs et formes automorphes: applications du principe de
              fonctorialit\'e},
 BOOKTITLE = {Automorphic forms, {S}himura varieties, and {$L$}-functions,
              {V}ol.\ {I} ({A}nn {A}rbor, {MI}, 1988)},
    SERIES = {Perspect. Math.},
    VOLUME = {10},
     PAGES = {77--159},
 PUBLISHER = {Academic Press, Boston, MA},
      YEAR = {1990},
}

@article{Losev2005OnIO,
   AUTHOR = {Losev, Ivan},
     TITLE = {On invariants of a set of elements of a semisimple {L}ie
              algebra},
   JOURNAL = {J. Lie Theory},
  FJOURNAL = {Journal of Lie Theory},
    VOLUME = {20},
      YEAR = {2010},
    NUMBER = {1},
     PAGES = {17--30},
}

@online{TableAutorep,
  author = {Chenevier,Ga{\"e}tan and Ta{\"i}bi,Olivier},
  title = {Discrete series multiplicities for classical groups over $\Z$ and level 1 algebraic cusp forms},
  url = {http://gaetan.chenevier.perso.math.cnrs.fr/levelone/},
}

@online{TableAutorepG2,
  author= {Chenevier, Ga{\"e}tan  and Renard, David},
  title={The home page of level one algebraic cusp forms of classical groups},
  url={http://gaetan.chenevier.perso.math.cnrs.fr/levelone.html},
}

@online{Codes,
  author={Shan, Yi},
  title={Codes and Tables for the computation of level one automorphic forms of $\mathrm{F}_{4}$},
  url={https://www.math.ens.psl.eu/~yshan/Publications/F4TableCode/F4Table.html},
}

@misc{BorcherdsLattice,
   title={The Leech lattice and other lattices}, 
      author={R. E. Borcherds},
      year={1999},
      eprint={math/9911195},
      archivePrefix={arXiv},
      primaryClass={math.NT},
}

@article{dalal2023counting,
      AUTHOR = {Dalal, Rahul},
     TITLE = {C{OUNTING} {DISCRETE}, {LEVEL}-{$1$}, {QUATERNIONIC}
              {AUTOMORPHIC} {REPRESENTATIONS} {ON} {$G_2$}},
   JOURNAL = {J. Inst. Math. Jussieu},
  FJOURNAL = {Journal of the Institute of Mathematics of Jussieu. JIMJ.
              Journal de l'Institut de Math\'ematiques de Jussieu},
    VOLUME = {23},
      YEAR = {2024},
    NUMBER = {5},
     PAGES = {2127--2157},
}

@article{Savin1994,
 AUTHOR = {Savin, Gordan},
     TITLE = {Dual pair $G_{\mathscr{J}}\times\mathrm{PGL}_{2}$ [where] $G_{\mathscr{J}}$ is the automorphism group of the {J}ordan algebra $\mathrm{J}$},
   JOURNAL = {Invent. Math.},
  FJOURNAL = {Inventiones Mathematicae},
    VOLUME = {118},
      YEAR = {1994},
    NUMBER = {1},
     PAGES = {141--160},
}

@book{SGA3,
    shorthand = {SGA3},
    author = {Demazure, Michel and Grothendieck, Alexander},
    title = {Sch\'emas en groupes},
    series = {Lecture Notes in Math},
    number = {151,152,153},
    publisher  = {Springer-Verlag, New York},
    year = {1970}
}

@article{Borel63,
  AUTHOR = {Borel, Armand},
     TITLE = {Some finiteness properties of adele groups over number fields},
   JOURNAL = {Inst. Hautes \'Etudes Sci. Publ. Math.},
  FJOURNAL = {Institut des Hautes \'Etudes Scientifiques. Publications
              Math\'ematiques},
    NUMBER = {16},
      YEAR = {1963},
     PAGES = {5--30},
}

@misc{pollack2023exceptional,
      title={Exceptional theta functions and arithmeticity of modular forms on $G_2$}, 
      author={Pollack, Aaron},
      year={2023},
      eprint={2211.05280},
      archivePrefix={arXiv},
      primaryClass={math.NT}
}

@book{Yokota2009ExceptionalLG,
      AUTHOR = {Yokota, Ichiro},
     TITLE = {Exceptional {L}ie groups},
    SERIES = {Lecture Notes in Mathematics},
    VOLUME = {2369},
    EDITOR = {Imai, Takao and Ishihara, Tetsuo and Mitsuishi, Kazuyuki and
              Miyasaka, Takashi and Miyashita, Toshikazu and Nishio, Akihiro
              and Shukuzawa, Osamu and Yasukura, Osami},
 PUBLISHER = {Springer, Cham},
      YEAR = {2025},
     PAGES = {xix+249},
}

@book{AdamsExceptionalGrp,
  AUTHOR = {Adams, J. F.},
     TITLE = {Lectures on exceptional {L}ie groups},
    SERIES = {Chicago Lectures in Mathematics},
    EDITOR = {Mahmud, Zafer and Mimura, Mamoru},
      NOTE = {With a foreword by J. Peter May},
 PUBLISHER = {University of Chicago Press, Chicago, IL},
      YEAR = {1996},
     PAGES = {xiv+122},
}

@book{AlgGrpandNumberTheory,
 AUTHOR = {Platonov, Vladimir and Rapinchuk, Andrei},
     TITLE = {Algebraic groups and number theory},
    SERIES = {Pure and Applied Mathematics},
    VOLUME = {139},
 PUBLISHER = {Academic Press, Inc., Boston, MA},
      YEAR = {1994},
     PAGES = {xii+614},
}

@incollection {GrossZmodelAMS,
    AUTHOR = {Gross, Benedict H.},
     TITLE = {On simply-connected groups over {${\bf Z}$}, with {$G({\bf R})$} compact},
 BOOKTITLE = {Integral quadratic forms and lattices ({S}eoul, 1998)},
    SERIES = {Contemp. Math.},
    VOLUME = {249},
     PAGES = {113--118},
 PUBLISHER = {Amer. Math. Soc., Providence, RI},
      YEAR = {1999},
}

@incollection{cogdelllectures,
  AUTHOR = {Cogdell, James W.},
     TITLE = {Lectures on {$L$}-functions, converse theorems, and
              functoriality for {${\rm GL}_n$}},
 BOOKTITLE = {Lectures on automorphic {$L$}-functions},
    SERIES = {Fields Inst. Monogr.},
    VOLUME = {20},
     PAGES = {1--96},
 PUBLISHER = {Amer. Math. Soc., Providence, RI},
      YEAR = {2004},
}

@incollection{borel1979automorphic,
  AUTHOR = {Borel, A. and Jacquet, H.},
     TITLE = {Automorphic forms and automorphic representations},
 BOOKTITLE = {Automorphic forms, representations and {$L$}-functions
              ({P}roc. {S}ympos. {P}ure {M}ath., {O}regon {S}tate {U}niv.,
              {C}orvallis, {O}re., 1977), {P}art 1},
    SERIES = {Proc. Sympos. Pure Math.},
    VOLUME = {XXXIII},
     PAGES = {189--207},
     PUBLISHER = {Amer. Math. Soc., Providence, RI},
      YEAR = {1979},
}

@article{TaibiDimension,
  AUTHOR = {Ta\"ibi, Olivier},
     TITLE = {Dimensions of spaces of level one automorphic forms for split
              classical groups using the trace formula},
   JOURNAL = {Ann. Sci. \'Ec. Norm. Sup\'er. (4)},
  FJOURNAL = {Annales Scientifiques de l'\'Ecole Normale Sup\'erieure.
              Quatri\`eme S\'erie},
    VOLUME = {50},
      YEAR = {2017},
    NUMBER = {2},
     PAGES = {269--344},
}

@article{CohRepReal,
 AUTHOR = {Nair, Arvind N. and Prasad, Dipendra},
     TITLE = {Cohomological representations for real reductive groups},
   JOURNAL = {J. Lond. Math. Soc. (2)},
  FJOURNAL = {Journal of the London Mathematical Society. Second Series},
    VOLUME = {104},
      YEAR = {2021},
    NUMBER = {4},
     PAGES = {1515--1571},
}

@incollection{BorelAutomorphicLFunction,
 AUTHOR = {Borel, A.},
     TITLE = {Automorphic {$L$}-functions},
 BOOKTITLE = {Automorphic forms, representations and {$L$}-functions
              ({P}roc. {S}ympos. {P}ure {M}ath., {O}regon {S}tate {U}niv.,
              {C}orvallis, {O}re., 1977), {P}art 2},
    SERIES = {Proc. Sympos. Pure Math.},
    VOLUME = {XXXIII},
     PAGES = {27--61},
 PUBLISHER = {Amer. Math. Soc., Providence, RI},
      YEAR = {1979},
}

@misc{moeglin2014stabilisation,
      title={Stabilisation de la formule des traces tordue X: stabilisation spectrale}, 
      author={Colette Moeglin and Jean-Loup Waldspurger},
      year={2014},
      eprint={1412.2981},
      archivePrefix={arXiv},
      primaryClass={math.RT},
}

@incollection {SerreMotive,
    AUTHOR = {Serre, Jean-Pierre},
     TITLE = {Propri\'{e}t\'{e}s conjecturales des groupes de {G}alois
              motiviques et des repr\'{e}sentations {$l$}-adiques},
 BOOKTITLE = {Motives ({S}eattle, {WA}, 1991)},
    SERIES = {Proc. Sympos. Pure Math.},
    VOLUME = {55, Part 1},
     PAGES = {377-400},
 PUBLISHER = {Amer. Math. Soc., Providence, RI},
      YEAR = {1994},
}

@article{YunMotive,
   AUTHOR = {Yun, Zhiwei},
     TITLE = {Motives with exceptional {G}alois groups and the inverse
              {G}alois problem},
   JOURNAL = {Invent. Math.},
  FJOURNAL = {Inventiones Mathematicae},
    VOLUME = {196},
      YEAR = {2014},
    NUMBER = {2},
     PAGES = {267--337},
}

@article{GrossSavin,
  AUTHOR = {Gross, Benedict H. and Savin, Gordan},
     TITLE = {Motives with {G}alois group of type {$G_2$}: an exceptional
              theta-correspondence},
   JOURNAL = {Compositio Math.},
  FJOURNAL = {Compositio Mathematica},
    VOLUME = {114},
      YEAR = {1998},
    NUMBER = {2},
     PAGES = {153--217},
}

@article{PatrikisDeformation,
   AUTHOR = {Patrikis, Stefan},
     TITLE = {Deformations of {G}alois representations and exceptional monodromy},
   JOURNAL = {Invent. Math.},
  FJOURNAL = {Inventiones Mathematicae},
    VOLUME = {205},
      YEAR = {2016},
    NUMBER = {2},
     PAGES = {269--336},
}

@article{Dettweiler_Reiter_2010, 
AUTHOR = {Dettweiler, Michael and Reiter, Stefan},
     TITLE = {Rigid local systems and motives of type {$G_2$}},
      NOTE = {With an appendix by Michael Dettweiler and Nicholas M. Katz},
   JOURNAL = {Compos. Math.},
  FJOURNAL = {Compositio Mathematica},
    VOLUME = {146},
      YEAR = {2010},
    NUMBER = {4},
     PAGES = {929--963},
}

@article{BOXER_CALEGARI_EMERTON_LEVIN_MADAPUSI_PERA_PATRIKIS_2019, 
AUTHOR = {Boxer, George and Calegari, Frank and Emerton, Matthew and
              Levin, Brandon and Madapusi Pera, Keerthi and Patrikis,
              Stefan},
     TITLE = {Compatible systems of {G}alois representations associated to
              the exceptional group {$E_6$}},
   JOURNAL = {Forum Math. Sigma},
  FJOURNAL = {Forum of Mathematics. Sigma},
    VOLUME = {7},
      YEAR = {2019},
     PAGES = {Paper No. e4, 29},
}

@article{TaylorICM,
     AUTHOR = {Taylor, Richard},
     TITLE = {Galois representations},
   JOURNAL = {Ann. Fac. Sci. Toulouse Math. (6)},
  FJOURNAL = {Annales de la Facult\'e{} des Sciences de Toulouse.
              Math\'ematiques. S\'erie 6},
    VOLUME = {13},
      YEAR = {2004},
    NUMBER = {1},
     PAGES = {73--119},
}

@article{GWTExceptionalSiegelWeil,
    AUTHOR = {Gan, Wee Teck},
     TITLE = {A {S}iegel-{W}eil formula for exceptional groups},
   JOURNAL = {J. Reine Angew. Math.},
  FJOURNAL = {Journal f\"ur die Reine und Angewandte Mathematik. [Crelle's
              Journal]},
    VOLUME = {528},
      YEAR = {2000},
     PAGES = {149--181},
}

@article{SavinMagaardExceptionalTheta,
    AUTHOR = {Magaard, K. and Savin, G.},
     TITLE = {Exceptional {$\Theta$}-correspondences. {I}},
   JOURNAL = {Compositio Math.},
  FJOURNAL = {Compositio Mathematica},
    VOLUME = {107},
      YEAR = {1997},
    NUMBER = {1},
     PAGES = {89--123},
}

@misc{karasiewicz2023dual,
      title={The Dual Pair $\mathrm{Aut}(C)\times F_{4}$ ($p$-adic case)}, 
      author={Edmund Karasiewicz and Gordan Savin},
      year={2023},
      eprint={2312.02853},
      archivePrefix={arXiv},
      primaryClass={math.RT}
}

@incollection {BuzzardGeeConj,
    AUTHOR = {Buzzard, Kevin and Gee, Toby},
     TITLE = {The conjectural connections between automorphic
              representations and {G}alois representations},
 BOOKTITLE = {Automorphic forms and {G}alois representations. {V}ol. 1},
    SERIES = {London Math. Soc. Lecture Note Ser.},
    VOLUME = {414},
     PAGES = {135--187},
 PUBLISHER = {Cambridge Univ. Press, Cambridge},
      YEAR = {2014},
}

@phdthesis{lachausseeThesis,
  TITLE = {{Autour de l'{\'e}num{\'e}ration des repr{\'e}sentations automorphes cuspidales alg{\'e}briques de $\GL_{n}$ sur $\Q$ en conducteur $>$ 1}},
  AUTHOR = {Lachauss{\'e}e, Guillaume},
  SCHOOL = {{Universit{\'e} Paris-Saclay}},
  YEAR = {2020},
  MONTH = Nov,
  TYPE = {phd thesis}, 
}

@phdthesis{PadowitzThesis,
  title={Traces of Hecke Operators},
  author={Padowitz, Seth},
  school={Harvard University},
  year={1998},
  month=apr,
  type={phd thesis},
}

@article{GrossPollack,
 AUTHOR = {Gross, Benedict H. and Pollack, David},
     TITLE = {On the {E}uler characteristic of the discrete spectrum},
   JOURNAL = {J. Number Theory},
  FJOURNAL = {Journal of Number Theory},
    VOLUME = {110},
      YEAR = {2005},
    NUMBER = {1},
     PAGES = {136--163},
}

@article{RobyPolyLaw,
     AUTHOR = {Roby, Norbert},
     TITLE = {Lois polynomes et lois formelles en th\'eorie des modules},
   JOURNAL = {Ann. Sci. \'Ecole Norm. Sup. (3)},
  FJOURNAL = {Annales Scientifiques de l'\'Ecole Normale Sup\'erieure.
              Troisi\`eme S\'erie},
    VOLUME = {80},
      YEAR = {1963},
     PAGES = {213--348},
}
\end{document}